\documentclass[oneside,11pt]{smfbook}

\newcommand{\clearemptydoublepage}{\newpage{\pagestyle{empty}\cleardoublepage}}

\def\baselinestretch{1.3}
\def\mainmatter{\cleardoublepage\setcounter{page}{1}
\def\baselinestretch{1.3}\normalfont\pagenumbering{arabic}}
\def\backmatter{\cleardoublepage
\def\baselinestretch{1.3}\normalfont}

\usepackage[ansinew]{inputenc}
\usepackage[T1]{fontenc}
\usepackage[english,francais]{babel}

\usepackage{amssymb,amsmath,enumerate,fancyvrb,mathtools}
\usepackage{mathptm}
\usepackage{mathrsfs} 
\usepackage{float,url,lscape,xspace,rotating}
\usepackage{setspace} 
\usepackage[all]{xy} 
\usepackage{textcomp} 

\usepackage{pgf,tikz}  
\usepackage{graphics,graphicx} 
\usepackage{epsfig} 

\usepackage{color}

\usepackage[pdfstartview=FitH,
            colorlinks=true,
            linkcolor=blue,
            urlcolor=blue,
            citecolor=red,
            bookmarks,
            bookmarksopen=true,
            bookmarksnumbered=true]{smfhyperref} 
\newcounter{Hequation}

\makeatletter
\g@addto@macro\equation{\stepcounter{Hequation}}
\makeatother

\usepackage{cleveref}

\theoremstyle{plain}

\newtheorem{theoalph}{Theor\`eme}

\newtheorem{thmalph}[theoalph]{Th\'eor\`eme}

\newtheorem{coralph}[theoalph]{Corollaire}

\theoremstyle{definition}

\theoremstyle{remark}

\theoremstyle{plain}
\newtheorem{thmchap}{Th\'eor\`eme}[chapter]
\newtheorem{thm}[thmchap]{Th\'eor\`eme}
\newtheorem{pro}[thmchap]{Proposition}
\newtheorem{lem}[thmchap]{Lemme}
\newtheorem{cor}[thmchap]{Corollaire}
\newtheorem{prob}{Probl\`eme}

\theoremstyle{definition}
\newtheorem{defin}[thmchap]{D\'efinition}

\theoremstyle{remark}
\newtheorem{rem}[thmchap]{Remarque}
\newtheorem{rems}[thmchap]{Remarques}
\newtheorem{eg}[thmchap]{Exemple}

\addtocounter{section}{0}             
\numberwithin{equation}{section}       

\makeatletter
\newcommand{\reqnomode}{\tagsleft@false\let\veqno\@@eqno}
\newcommand{\leqnomode}{\tagsleft@true\let\veqno\@@leqno}
\makeatother

\def\og{\leavevmode\raise.3ex\hbox{$\scriptscriptstyle\langle\!\langle$~}}
\def\fg{\leavevmode\raise.3ex\hbox{~$\!\scriptscriptstyle\,\rangle\!\rangle$}}

\newcommand{\N}{\mathbb{N}}
\newcommand{\Z}{\mathbb{Z}}

\newcommand{\C}{\mathbb{C}}
\newcommand{\sph}{\mathbb{P}^{1}_{\mathbb{C}}}
\newcommand{\pp}{\mathbb{P}^{2}_{\mathbb{C}}}
\newcommand{\pd}{\mathbb{\check{P}}^{2}_{\mathbb{C}}}
\newcommand\pgcd{\mathrm{pgcd}}
\newcommand\X{\mathrm{X}}
\newcommand\T{\mathrm{T}}
\newcommand\Sing{\mathrm{Sing}}
\newcommand\Tang{\mathrm{Tang}}
\newcommand\Leg{\mathrm{Leg}}
\newcommand\IF{\mathrm{I}_{\mathcal{F}}}
\newcommand\IinvF{\mathrm{I}_{\mathcal{F}}^{\mathrm{inv}}}
\newcommand\ItrF{\mathrm{I}_{\mathcal{F}}^{\hspace{0.2mm}\mathrm{tr}}}
\newcommand\IH{\mathrm{I}_{\mathcal{H}}}
\newcommand\IinvH{\mathrm{I}_{\mathcal{H}}^{\mathrm{inv}}}
\newcommand\ItrH{\mathrm{I}_{\mathcal{H}}^{\hspace{0.2mm}\mathrm{tr}}}
\newcommand\Cinv{\mathrm{C}_{\hspace{-0.3mm}\mathcal{H}}}
\newcommand\Dtr{\mathrm{D}_{\hspace{-0.3mm}\mathcal{H}}}
\newcommand\radH{\Sigma_{\mathcal{H}}^{\mathrm{rad}}}
\newcommand\radHd{\check{\Sigma}_{\mathcal{H}}^{\mathrm{rad}}}
\newcommand*{\transp}[2][-3mu]{\ensuremath{\mskip1mu\prescript{\smash{\mathrm t\mkern#1}}{}{\mathstrut#2}}}
\newcommand\F{\mathcal{F}}
\newcommand\W{\mathcal{W}}
\newcommand\G{\mathcal{G}}
\newcommand\Gunderline{{\mspace{2mu}\underline{\mspace{-2mu}\mathcal{G}\mspace{-2mu}}\mspace{2mu}}}

\newcommand\omegaoverline{{\mspace{2mu}\overline{\mspace{-1.4mu}\omega\mspace{-1.4mu}}\mspace{2mu}}}
\newcommand\Omegaoverline{{\mspace{2mu}\overline{\mspace{-1.4mu}\Omega\mspace{-1.4mu}}\mspace{2mu}}}
\newcommand\thetaoverline{{\mspace{2mu}\overline{\mspace{-1.4mu}\theta\mspace{-1.4mu}}\mspace{2mu}}}

\makeatletter
\let\original@addcontentsline\addcontentsline
\newcommand{\dummy@addcontentsline}[3]{}
\newcommand{\DeactivateToc}{\let\addcontentsline\dummy@addcontentsline}
\newcommand{\ActivateToc}{\let\addcontentsline\original@addcontentsline}
\makeatother

\begin{document}

\frontmatter

\begin{titlepage}
\enlargethispage{20\baselineskip}
\includegraphics[trim=3.65cm 0in -0cm 4.7cm]{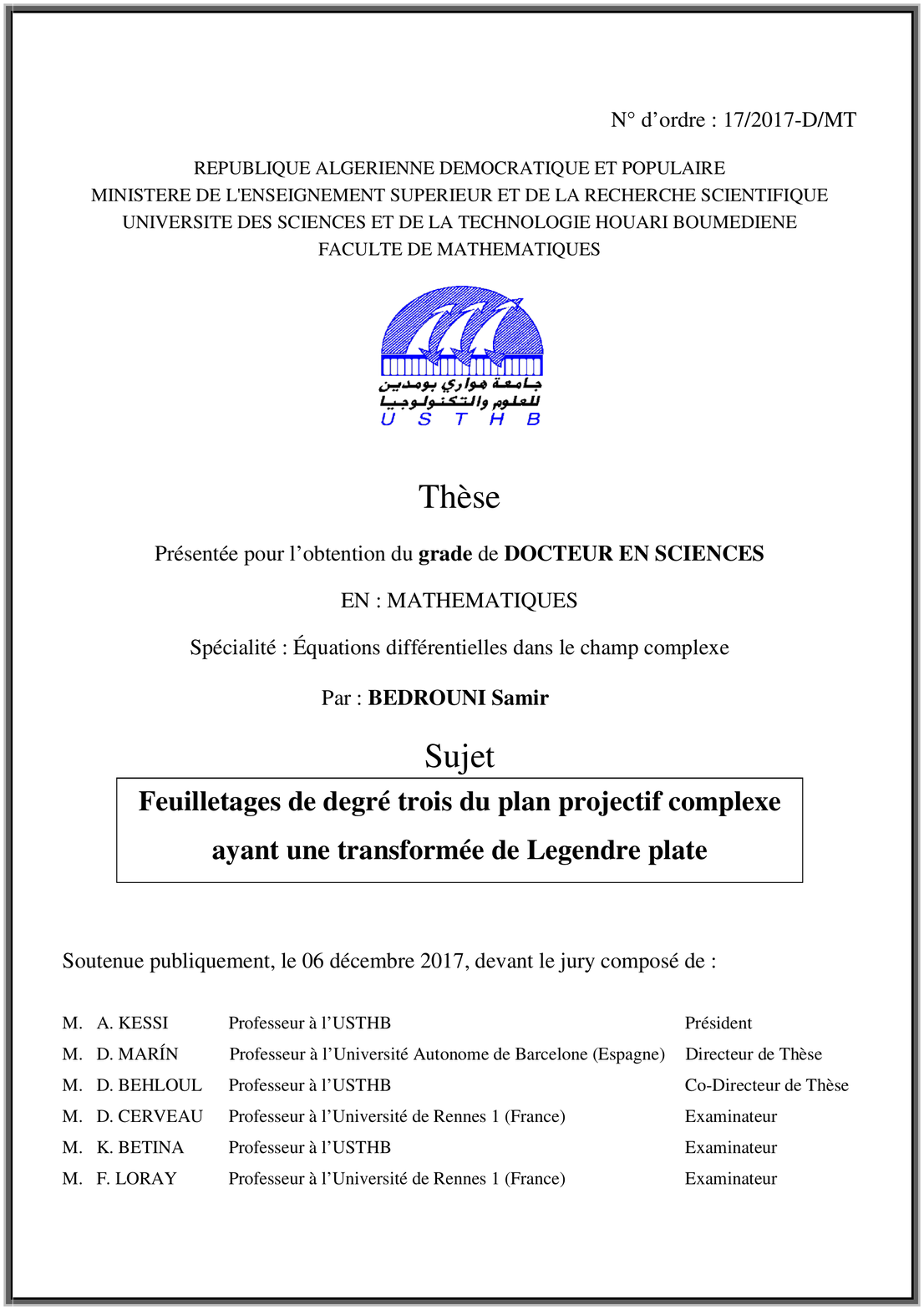}
\end{titlepage}

\title{Feuilletages de degré trois du plan projectif complexe ayant une transformée de Legendre plate}

\author{Samir \textsc{Bedrouni}}

\address{Facult\'e de Math\'ematiques, USTHB, BP $32$, El-Alia, $16111$ Bab-Ezzouar, Alger, Alg\'erie}
\email{sbedrouni@usthb.dz}

\begin{abstract}
\selectlanguage{french}
\noindent L'ensemble $\mathbf{F}(d)$ des feuilletages de degré $d$ du plan projectif complexe s'identifie à un ouvert de Zariski dans un espace projectif de dimension $(d+2)^2-2$ sur lequel agit le groupe $\mathrm{Aut}(\pp)$. Le sous-ensemble $\mathbf{FP}(d)$ de $\mathbf{F}(d)$ formé des feuilletages de $\mathbf{F}(d)$ ayant une transformée de \textsc{Legendre} (tissu dual) plate est un fermé de Zariski de $\mathbf{F}(d)$.

\noindent Dans cette thèse nous étudions les feuilletages de $\mathbf{FP}(d)$ et nous tentons de mieux comprendre la structure topologique de $\mathbf{FP}(3)$.

\noindent Dans un premier temps, nous établissons quelques résultats généraux sur la platitude du $d$-tissu dual d'un feuilletage homogène de degré $d$ et nous décrivons quelques exemples explicites. Nous verrons également qu'il est possible, sous certaines hypothèses, de ramener l'étude de la platitude du tissu dual d'un feuilletage inhomogène au cadre homogène.

\noindent Dans un deuxième temps, nous classifions à automorphisme de $\pp$ près les éléments de $\mathbf{FP}(3).$ Plus précisément, nous montrons qu'à automorphisme près il y a $16$ feuilletages de degré $3$ ayant une transformée de \textsc{Legendre} plate. De cette classification nous déduisons que $\mathbf{FP}(3)$ possède exactement 12 composantes irréductibles.
%
\end{abstract}
\keywords{feuilletage, tissu, platitude, transformation de \textsc{Legendre}.}

\pdfbookmark[0]{D\'{e}dicaces}{dedicaces}
\newpage\thispagestyle{empty}

\hfill
\vspace{9cm}
\begin{minipage}{15.3cm}
\begin{flushright}
$\begin{array}{llll}
\text{\sl{\`A la m\'{e}moire de mon fr\`{e}re \og \hspace{-0.1mm}\textsc{Rabie}\hspace{-0.1mm}\fg }}\\
\text{\sl{J'aurais tant aim\'{e} que tu sois l\`{a} aujourd'hui, mais le destin en a d\'{e}cid\'{e} autrement.}}\\
\text{\sl{Puisse Dieu le Tout-Puissant t'accorder sa sainte miséricorde et t'accueillir dans son vaste paradis,}}
\end{array}$

S. B.
\end{flushright}
\end{minipage}

\pdfbookmark[0]{Remerciements}{remerciements}
\DeactivateToc
\chapter*{Remerciements}
\ActivateToc

Je tiens à exprimer mes vifs remerciements à mon Directeur de thèse, Monsieur David \textsc{Mar\'{\i}n}, Professeur à l'Université Autonome de Barcelone, pour m'avoir proposé ce sujet,~pour son aide précieuse, ses encouragements et ses conseils avisés tout au long de ces six années de thèse. Je le remercie également pour m'avoir accueilli au sein de son laboratoire durant dix-huit mois où j'ai eu la possibilité de travailler dans un excellent environnement.
\vspace{2mm}

Je témoigne ma profonde reconnaissance à mon Co-directeur de thèse, Monsieur Djilali \textsc{Behloul}, Professeur à l'USTHB, pour m'avoir aidé dans les démarches administratives. Merci Monsieur \textsc{Behloul}.
\vspace{2mm}

J'adresse mes sincères remerciements aux membres du jury de ma thèse: Monsieur Arezki \textsc{Kessi}, Professeur à l'USTHB, pour avoir accepté de présider ce jury et Messieurs Dominique \textsc{Cerveau}, Professeur à l'Université de Rennes $1$, Kamel \textsc{Betina}, Professeur à l'USTHB, Frank \textsc{Loray}, Professeur à l'Université de Rennes $1$ et Directeur de recherche au CNRS, pour m'avoir fait l'honneur de rapporter ce travail.
\vspace{2mm}

J'exprime ma profonde gratitude à Monsieur Djamel \textsc{Smaï}, Maître Assistant à l'USTHB, pour sa générosité intellectuelle, ses encouragements et ses conseils précieux lors de la rédaction~de ce manuscrit. Un grand merci Monsieur \textsc{Smaï}.
\vspace{2mm}

Mon séjour de dix-huit mois à l'UAB a été possible grâce à une bourse algérienne, qui m'a été allouée par le programme PNE, auquel j'adresse ma profonde reconnaissance.
\vspace{2mm}

Je ne citerai pas, mais ils se reconnaîtront, tous mes amis et collègues de l'USTHB et de l'UAB, qui m'ont soutenu, même de loin. Je leur en suis reconnaissant.
\vspace{2mm}

Enfin je voudrais remercier ma famille, en particulier mes parents, mes frères \textsc{Yacine}, \textsc{Ali}, \textsc{Khalil}, \textsc{Mohamed}, \textsc{Youcef} et ma s{\oe}ur \textsc{Fatima} pour leur soutien moral et leurs encouragements durant toutes ces années de thèse.

\maketitle

\setcounter{tocdepth}{2}
\pdfbookmark[0]{Table des matières}{tablematieres}

\begin{small}

\tableofcontents

\end{small}

\mainmatter

\chapter*{Introduction}

Un $d$-tissu (régulier) $\W$ de $(\mathbb{C}^2,0)$ est la donnée d'une famille $\{\mathscr{F}_1,\mathscr{F}_2,\ldots,\mathscr{F}_d\}$ de feuilletages holomorphes réguliers de $(\mathbb{C}^2,0)$ deux à deux transverses en l'origine. Le premier résultat significatif dans l'étude des tissus a été obtenu par W. \textsc{Blaschke} et J. \textsc{Dubourdieu} autour des années $1920$. Ils ont montré (\cite{BD28}) que tout $3$-tissu régulier $\mathcal{W}$ de $(\mathbb{C}^2,0)$ est conjugué, via un isomorphisme analytique de $(\mathbb{C}^2,0)$, au $3$-tissu trivial défini par $\mathrm{d}x.\mathrm{d}y.\mathrm{d}(x+y)$, et cela sous l'hypothèse d'annulation d'une $2$-forme différentielle $K(\mathcal{W})$ connue sous le nom de courbure de \textsc{Blaschke} de $\mathcal{W}$. La courbure d'un $d$-tissu $\W$ avec~$d>3$ se définit comme la somme des courbures de \textsc{Blaschke} des sous-$3$-tissus de $\W$. Un tissu de courbure nulle est dit plat. Cette notion est utile pour la classification des tissus de rang maximal; en effet un résultat de N. \textsc{Mih\u{a}ileanu}~montre que la platitude est une condition nécessaire pour la maximalité du rang, \emph{voir} par exemple \cite{Hen06,Rip07}.

\noindent Depuis peu, l'étude des tissus globaux holomorphes définis sur les surfaces complexes a été réactualisée, \emph{voir} par exemple \cite{CL07,PP09,MP13}. Dans cette thèse, nous nous intéressons aux tissus du plan projectif complexe. Un $d$-tissu (global) sur $\pp$ est donné dans une carte affine $(x,y)$ par une équation différentielle algébrique $F(x,y,y')=0$, où $F(x,y,p)=\sum_{i=0}^{d}a_{i}(x,y)p^{d-i}\in\mathbb{C}[x,y,p]$ est un polynôme réduit à coefficient $a_0$ non identiquement nul. Au voisinage de tout point $z_{0}=(x_{0},y_{0})$ tel que $a_{0}(x_{0},y_{0})\Delta(x_{0},y_{0})\neq 0$, où $\Delta(x,y)$ est le $p$-discriminant de $F$, les courbes intégrales de cette équation définissent un $d$-tissu régulier de $(\C^{2},z_{0})$.

\noindent La courbure d'un tissu $\W$ sur $\pp$ est une $2$-forme méromorphe à pôles le long du discriminant $\Delta(\W)$. La platitude d'un tissu $\W$ sur $\pp$ se caractérise par l'holomorphie de sa courbure $K(\W)$ le long des points génériques de $\Delta(\W)$.

\noindent D.~\textsc{Mar\'{\i}n} et J. \textsc{Pereira} ont montré, dans \cite{MP13}, comment on peut associer à tout feuilletage $\F$ de degré $d$ sur $\pp$, un $d$-tissu sur le plan projectif dual $\pd$, appelé {\sl transformée de \textsc{Legendre}} de $\F$ et noté $\Leg\F$; les feuilles de $\Leg\F$ sont les droites tangentes aux feuilles de $\F.$ Plus explicitement, soit $(p,q)$ la carte affine de $\pd$ associée à la droite $\{y=px-q\}\subset{\mathbb{P}^{2}_{\mathbb{C}}}$; si $\F$ est défini par une $1$-forme $\omega=A(x,y)\mathrm{d}x+B(x,y)\mathrm{d}y,$ où $A,B\in\mathbb{C}[x,y],$ $\mathrm{pgcd}(A,B)=1$, alors $\Leg\F$ est donné par l'équation différentielle algébrique
\[
\check{F}(p,q,x):=A(x,px-q)+pB(x,px-q)=0, \qquad \text{avec} \qquad x=\frac{\mathrm{d}q}{\mathrm{d}p}.
\]

\noindent L'ensemble $\mathbf{F}(d)$ des feuilletages de degré $d$ sur $\pp$ s'identifie à un ouvert de \textsc{Zariski} de l'espace projectif $\mathbb{P}_{\C}^{(d+2)^{2}-2}$. Le groupe des automorphismes de $\pp$ agit sur $\mathbf{F}(d)$; l'orbite d'un élément $\F\in\mathbf{F}(d)$ sous l'action de $\mathrm{Aut}(\pp)=\mathrm{PGL}_3(\mathbb{C})$ sera notée $\mathcal{O}(\F)$. Le sous-ensemble $\mathbf{FP}(d)$ de $\mathbf{F}(d)$ formé des $\F\in\mathbf{F}(d)$ tels que $\Leg\F$ soit plat est un fermé de \textsc{Zariski} de $\mathbf{F}(d)$.

\noindent Dans \cite{MP13} les auteurs posent un problème concernant la géométrie des tissus de $\pp$ qui, dans le cadre des feuilletages de $\pp$, consiste en la description de certaines composantes irréductibles de $\mathbf{FP}(d).$ Le premier cas non trivial que l'on rencontre est celui où $d=3$; on dispose actuellement d'une caractérisation théorique (\cite[Théorème~4.5]{BFM13}) des éléments de $\mathbf{FP}(3)$, mais ce résultat reste insuffisant pour mieux comprendre la structure topologique de $\mathbf{FP}(3).$

\noindent C'est pour cela que nous nous proposons, dans un premier temps, d'étudier la platitude de la transformée de \textsc{Legendre} d'un \textsl{feuilletage homogène} de degré $d$ quelconque, feuilletage homogène au sens où il est invariant par homothétie. Plus précisément, il s'agit d'établir, pour un tel feuilletage $\mathcal{H}\in\mathbf{F}(d)$, quelques critères effectifs de l'holomorphie de la courbure de $\Leg\mathcal{H}$ (Théorèmes~\ref{thm:Barycentre} et \ref{thm:Divergence}). Ces critères nous permettront de décrire certains feuilletages homogènes appartenant à $\mathbf{FP}(d)$ pour $d$ arbitraire (Propositions
\ref{pro:omega1-omega2}, \ref{pro:omega3-omega4} et \ref{pro:omega5-omega6}). De plus nous verrons (Proposition~\ref{pro:F-dégénère-H}) que l'étude de la platitude de la transformée de \textsc{Legendre} d'un feuilletage inhomogène se ramène, sous certaines hypothèses, au cadre homogène.

\noindent Dans un deuxième temps, en se basant sur ces résultats et en faisant une étude de la platitude du $3$-tissu dual d'un feuilletage $\F\in\mathbf{F}(3)$ suivant la nature de ses singularités, nous obtenons la classification à automorphisme de $\pp$ pr\`es des éléments de $\mathbf{FP}(3).$
\begin{thmalph}[\rm{Théorème~\ref{thm:classification}}]\label{thmalph:classification}
{\sl \`A automorphisme de $\pp$ pr\`es, il y a seize feuilletages de degré trois $\mathcal{H}_1,\ldots,\mathcal{H}_{11},\mathcal{F}_1,\ldots,\mathcal{F}_{5}$ sur le plan projectif complexe ayant une transformée de \textsc{Legendre} plate. Ils sont d\'ecrits respectivement en carte affine par les $1$-formes suivantes
\begin{itemize}
\item [\texttt{1. }]  \hspace{1mm} $\omega_1\hspace{1mm}=y^3\mathrm{d}x-x^3\mathrm{d}y$;
\item [\texttt{2. }]  \hspace{1mm} $\omega_2\hspace{1mm}=x^3\mathrm{d}x-y^3\mathrm{d}y$;
\item [\texttt{3. }]  \hspace{1mm} $\omega_3\hspace{1mm}=y^2(3x+y)\mathrm{d}x-x^2(x+3y)\mathrm{d}y$;
\item [\texttt{4. }]\hspace{1mm}   $\omega_4\hspace{1mm}=y^2(3x+y)\mathrm{d}x+x^2(x+3y)\mathrm{d}y$;
\item [\texttt{5. }]\hspace{1mm} $\omega_{5}\hspace{1mm}=2y^3\mathrm{d}x+x^2(3y-2x)\mathrm{d}y$;
\item [\texttt{6. }]\hspace{1mm} $\omega_{6}\hspace{1mm}=(4x^3-6x^2y+4y^3)\mathrm{d}x+x^2(3y-2x)\mathrm{d}y$;
\item [\texttt{7. }]\hspace{1mm} $\omega_{7}\hspace{1mm}=y^3\mathrm{d}x+x(3y^2-x^2)\mathrm{d}y$;
\item [\texttt{8. }]\hspace{1mm} $\omega_{8}\hspace{1mm}=x(x^2-3y^2)\mathrm{d}x-4y^3\mathrm{d}y$;
\item [\texttt{9. }]\hspace{1mm} $\omega_{9}\hspace{1mm}=y^{2}\left((-3+\mathrm{i}\sqrt{3})x+2y\right)\mathrm{d}x+
                                                       x^{2}\left((1+\mathrm{i}\sqrt{3})x-2\mathrm{i}\sqrt{3}y\right)\mathrm{d}y$;
\item [\texttt{10. }]\hspace{-1mm} $\omega_{10}=(3x+\sqrt{3}y)y^2\mathrm{d}x+(3y-\sqrt{3}x)x^2\mathrm{d}y$;
\item [\texttt{11. }]\hspace{-1mm} $\omega_{11}=(3x^3+3\sqrt{3}x^2y+3xy^2+\sqrt{3}y^3)\mathrm{d}x+(\sqrt{3}x^3+3x^2y+3\sqrt{3}xy^2+3y^3)\mathrm{d}y$;
\item [\texttt{12. }]\hspace{-1mm} $\omegaoverline_{1}\hspace{1mm}=y^{3}\mathrm{d}x+x^{3}(x\mathrm{d}y-y\mathrm{d}x)$;
\item [\texttt{13. }]\hspace{-1mm} $\omegaoverline_{2}\hspace{1mm}=x^{3}\mathrm{d}x+y^{3}(x\mathrm{d}y-y\mathrm{d}x)$;
\item [\texttt{14. }]\hspace{-1mm} $\omegaoverline_{3}\hspace{1mm}=(x^{3}-x)\mathrm{d}y-(y^{3}-y)\mathrm{d}x$;
\item [\texttt{15. }]\hspace{-1mm} $\omegaoverline_{4}\hspace{1mm}=(x^{3}+y^{3})\mathrm{d}x+x^{3}(x\mathrm{d}y-y\mathrm{d}x)$;
\item [\texttt{16. }]\hspace{-1mm} $\omegaoverline_{5}\hspace{1mm}=y^{2}(y\mathrm{d}x+2x\mathrm{d}y)+x^{3}(x\mathrm{d}y-y\mathrm{d}x)$.
\end{itemize}
}
\end{thmalph}

\noindent Les orbites de $\F_1$ et $\F_2$ sont toutes deux de dimension $6$ qui est la dimension minimale possible, et ce en tout degré supérieur ou égal~à~$2$ (\cite[Proposition~2.3]{CDGBM10}). D.~\textsc{Cerveau}, J.~\textsc{D\'eserti}, D.~\textsc{Garba Belko} et R.~\textsc{Meziani} ont montré qu'en degré $2$ il y a exactement deux orbites de dimension $6$ (\cite[Proposition 2.7]{CDGBM10}). Le Théorème~\ref{thmalph:classification} nous permet d'établir un résultat similaire en degré $3$:
\begin{coralph}[\rm{Corollaire~\ref{cor:dim-min}}]\label{coralph:dim-min}
{\sl \`A automorphisme de $\pp$ près, les feuilletages $\mathcal{F}_{1}$ et $\mathcal{F}_{2}$ sont les seuls feuilletages qui réalisent la dimension minimale des orbites en degré $3.$}
\end{coralph}

\noindent A.~\textsc{Beltr\'{a}n}, M.~\textsc{Falla~Luza} et D.~\textsc{Mar\'{\i}n} ont montré dans \cite{BFM13} que $\mathbf{FP}(3)$ contient l'ensemble des feuilletages $\F\in\mathbf{F}(3)$ dont les feuilles qui ne sont pas des droites n'ont pas de points d'inflexion; ces feuilletages sont dits \textsl{convexes}. De ces travaux (\cite[Corollaire~4.7]{BFM13}) et du Théorème~\ref{thmalph:classification}, nous déduisons la classification à automorphisme de $\pp$ pr\`es des feuilletages convexes de degré $3$ de $\pp.$
\begin{coralph}[\rm{Corollaire~\ref{cor:class-convexe-3}}]\label{coralph:class-convexe-3}
{\sl \`A automorphisme de $\pp$ près, il y a quatre feuilletages convexes de degré trois sur le plan projectif complexe, à savoir les feuilletages  $\mathcal{H}_{1},\mathcal{H}_{\hspace{0.2mm}3},\F_1$ et $\F_3.$
}
\end{coralph}

\noindent Ce corollaire est un analogue en degré $3$ d'un résultat sur les feuilletages de degré $2$ dû à C.~\textsc{Favre} et J.~\textsc{Pereira} (\cite[Proposition~7.4]{FP15}).
\medskip

\noindent On sait d'après \cite[Théorème~3]{MP13} que l'adhérence dans $\mathbf{F}(3)$ de l'orbite $\mathcal{O}(\F_3)$ du feuilletage $\F_3$, dit feuilletage de \textsc{Fermat} de degré $3,$ est une composante irréductible de $\mathbf{FP}(3)$; à l'heure actuelle, à notre connaissance, c'est le seul exemple explicite de composante irréductible de $\mathbf{FP}(3)$ connu dans la littérature. En étudiant les relations d'incidence entre les adhérences des orbites de $\mathcal{H}_{\hspace{0.2mm}i}$ et $\F_j,$ nous obtenons la décomposition de $\mathbf{FP}(3)$ en ses composantes irréductibles.

\begin{thmalph}[\rm{Théorème~\ref{thm:12-Composantes irréductibles}}]\label{thmalph:12-Composantes irréductibles}
{\sl Les adhérences étant prises dans $\mathbf{F}(3),$ nous avons
\begin{align*}
&
\overline{\mathcal{O}(\F_1)}=\mathcal{O}(\F_1),&&
\overline{\mathcal{O}(\F_2)}=\mathcal{O}(\F_2),\\
&
\overline{\mathcal{O}(\F_3)}=\mathcal{O}(\F_1)\cup\mathcal{O}(\mathcal{H}_{1})\cup\mathcal{O}(\mathcal{H}_{\hspace{0.2mm}3})\cup\mathcal{O}(\F_3),&&
\overline{\mathcal{O}(\F_4)}=\mathcal{O}(\F_1)\cup\mathcal{O}(\F_2)\cup\mathcal{O}(\F_4),\\
&
\overline{\mathcal{O}(\mathcal{H}_{1})}=\mathcal{O}(\F_1)\cup\mathcal{O}(\mathcal{H}_{1}),&&
\overline{\mathcal{O}(\mathcal{H}_{\hspace{0.2mm}2})}=\mathcal{O}(\F_2)\cup\mathcal{O}(\mathcal{H}_{\hspace{0.2mm}2}),\\
&
\overline{\mathcal{O}(\mathcal{H}_{\hspace{0.2mm}3})}=\mathcal{O}(\F_1)\cup\mathcal{O}(\mathcal{H}_{\hspace{0.2mm}3}),&&
\overline{\mathcal{O}(\mathcal{H}_{\hspace{0.2mm}8})}=\mathcal{O}(\F_2)\cup\mathcal{O}(\mathcal{H}_{\hspace{0.2mm}8}),\\
&
\overline{\mathcal{O}(\mathcal{H}_{\hspace{0.2mm}5})}=\mathcal{O}(\F_1)\cup\mathcal{O}(\mathcal{H}_{\hspace{0.2mm}5}),&&
\overline{\mathcal{O}(\mathcal{H}_{\hspace{0.2mm}4})}\supset\mathcal{O}(\F_2)\cup\mathcal{O}(\mathcal{H}_{\hspace{0.2mm}4}),\\
&
\overline{\mathcal{O}(\mathcal{H}_{\hspace{0.2mm}7})}\supset\mathcal{O}(\F_1)\cup\mathcal{O}(\mathcal{H}_{\hspace{0.2mm}7}),&&
\overline{\mathcal{O}(\mathcal{H}_{\hspace{0.2mm}6})}\supset\mathcal{O}(\F_2)\cup\mathcal{O}(\mathcal{H}_{\hspace{0.2mm}6}),\\
&
\overline{\mathcal{O}(\mathcal{H}_{\hspace{0.2mm}9})}\subset\mathcal{O}(\F_1)\cup\mathcal{O}(\mathcal{H}_{\hspace{0.2mm}9}),&&
\overline{\mathcal{O}(\mathcal{H}_{\hspace{0.2mm}11})}\subset\mathcal{O}(\F_2)\cup\mathcal{O}(\mathcal{H}_{11}),\\
&
\overline{\mathcal{O}(\mathcal{H}_{\hspace{0.2mm}10})}\subset\mathcal{O}(\F_1)\cup\mathcal{O}(\F_2)\cup\mathcal{O}(\mathcal{H}_{\hspace{0.2mm}10}),&&
\overline{\mathcal{O}(\F_5)}\subset\mathcal{O}(\F_2)\cup\mathcal{O}(\F_5)
\end{align*}
avec
\begin{align*}
&\dim\mathcal{O}(\F_{1})=6,&&&\dim\mathcal{O}(\F_{2})=6,&&& \dim\mathcal{O}(\mathcal{H}_{\hspace{0.2mm}i})=7, i=1,\ldots,11,\\
&
\dim\mathcal{O}(\F_{4})=7,&&& \dim\mathcal{O}(\F_{5})=7,&&& \dim\mathcal{O}(\F_{3})=8.
\end{align*}
En particulier:
\begin{itemize}
  \item  l'ensemble $\mathbf{FP}(3)$ possède exactement douze composantes irréductibles, à savoir $\overline{\mathcal{O}(\F_3)},\,$ $\overline{\mathcal{O}(\F_4)},\,$ $\overline{\mathcal{O}(\F_5)},\,$ $\overline{\mathcal{O}(\mathcal{H}_{\hspace{0.2mm}2})},\,$ $\overline{\mathcal{O}(\mathcal{H}_{\hspace{0.4mm}k})},\,k=4,5,\ldots,11$;
  \item  l'ensemble des feuilletages convexes de degré trois de $\pp$ est exactement l'adhérence $\overline{\mathcal{O}(\F_3)}$ de $\mathcal{O}(\F_3)$ (c'est donc un fermé irréductible de~$\mathbf{F}(3)).$
\end{itemize}
}
\end{thmalph}
\smallskip

\noindent On sait, d'après \cite{Per01}, que tout feuilletage de degré $d\geq1$ sur $\pp$ ne peut avoir plus de $3d$ droites invariantes (distinctes). Lorsque cette borne est atteinte pour $\F\in\mathbf{F}(d)$, alors $\F$ est nécessairement~convexe; dans ce cas on dit que $\F$ est \textsl{convexe réduit}. Dans \cite{MP13} les auteurs ont étudié les feuilletages de $\mathbf{F}(d)$ qui sont convexes réduits; ils ont montré que l'ensemble formé de tels feuilletages est contenu dans $\mathbf{FP}(d)$, \emph{voir} \cite[Théorème 2]{MP13}. \`{A} notre connaissance les seuls feuilletages convexes réduits connus dans la littérature sont ceux qui sont présentés dans \cite[Table~1.1]{MP13}: le feuilletage de \textsc{Fermat} de degré $d$ défini par la $1$-forme $$(x^{d}-x)\mathrm{d}y-(y^{d}-y)\mathrm{d}x$$ et les trois feuilletages donnés par les $1$-formes
\[(2x^{3}-y^{3}-1)y\mathrm{d}x+(2y^{3}-x^{3}-1)x\mathrm{d}y\,,\]
\[(y^2-1)(y^2- (\sqrt{5}-2)^2)(y+\sqrt{5}x) \mathrm{d}x-(x^2-1)(x^2- (\sqrt{5}-2)^2)(x+\sqrt{5}y) \mathrm{d}y \, ,\]
\[(y^3-1)(y^3+7x^3+1) y \mathrm{d}x-(x^3-1)(x^3+7 y^3+1) x \mathrm{d}y\,,\]
qui sont de degré  $4$, $5$ et $7$ respectivement. Dans \cite[Problème~9.1]{MP13} les auteurs posent la question suivante: y a-t-il d'autres feuilletages convexes réduits? Comme application du Théorème~\ref{thmalph:classification} nous donnons une réponse négative en degré trois à ce problème.
\begin{coralph}[\rm{Corollaire~\ref{cor:convexe-réduit-3}}]\label{coralph:convexe-réduit-3}
{\sl Tout feuilletage convexe réduit de degré trois de $\pp$ est linéairement conjugué au feuilletage de \textsc{Fermat} $\F_3$.}
\end{coralph}
\bigskip
\vspace{2mm}

\noindent\textmd{\textbf{Organisation de la thèse}.  --- } Cette thèse est divisée en trois chapitres suivis d'un appendice. Le Chapitre \ref{chap1:Préliminaires} est un rappel sur les feuilletages du plan projectif complexe, la théorie locale des tissus plans et les tissus globaux sur une surface complexe.
\smallskip

\noindent Dans le Chapitre~\ref{chap2:homogène}, nous étudions les feuilletages homogènes du plan projectif complexe ayant une transformée de \textsc{Legendre} plate. Plus précisément, nous établissons quelques résultats~généraux (Théorèmes~\ref{thm:holomo-G(I^tr)}, \ref{thm:Barycentre} et \ref{thm:Divergence}) sur la platitude du $d$-tissu dual d'un feuilletage homogène de degré $d$ et nous décrivons quelques exemples explicites (Propositions
\ref{pro:omega1-omega2}, \ref{pro:omega3-omega4} et \ref{pro:omega5-omega6}).
\smallskip

\noindent Dans le Chapitre~\ref{chap3:classification}, nous étudions les feuilletages de degré $3$ du plan projectif complexe ayant une transformée de \textsc{Legendre} plate et nous démontrons les Théorèmes~\ref{thmalph:classification}, \ref{thmalph:12-Composantes irréductibles} ainsi que les Corollaires~\ref{coralph:dim-min}, \ref{coralph:class-convexe-3}, \ref{coralph:convexe-réduit-3}.
\smallskip

\noindent Enfin, dans l'Appendice~\ref{Dém:pro:cas-nilpotent-noeud-ordre-2}, nous donnons une démonstration d'un résultat (Proposition~\ref{pro:cas-nilpotent-noeud-ordre-2}) utilisé au Chapitre~\ref{chap3:classification} pour démontrer le Théorème~\ref{thmalph:classification}.

\clearemptydoublepage
\chapter{Préliminaires}\label{chap1:Préliminaires}

\section{Feuilletages du plan projectif complexe}\vspace{6mm}

\begin{defin}\label{def:feuilletage}
Un \textbf{\textit{feuilletage holomorphe singulier}} $\mathcal{F}$ sur une surface complexe $S$ est la donnée d'un recouvrement ouvert $(V_{j})_{j\in J}$ de $S$ et d'une collection de $1$-formes holomorphes $\omega_{j}\in \Omega^1 (V_{j}),$ à zéros isolés, tels que sur chaque intersection non triviale \begin{small}$V_{j} \cap V_{k}$\end{small} on ait $\omega_{j}=g_{jk} \omega_{k}$ où les $g_{jk} \in \mathcal O^* (V_{j}\cap V_{k})$ sont des unités holomorphes. Le \textbf{\textit{lieu singulier}} $\mathrm{Sing}\mathcal{F}$ de $\mathcal{F}$ est défini par
\begin{align*}
& \mathrm{Sing}\mathcal{F}\cap\,V_{j}=\mathrm{Sing}\omega_j&& \text{où}&&\mathrm{Sing}\omega_j=\{p\in\,V_j:\omega_j(p)=0\},(j\in\,J).
\end{align*}
\end{defin}

\noindent La surface sur laquelle nous travaillerons dans cette thèse est le plan projectif complexe, \textit{i.e.} $S=\mathbb{P}^{2}_{\mathbb{C}}.$ Soit $$\pi\hspace{1mm}\colon\mathbb{C}^3\setminus\{0\}\to\mathbb{P}^2_{\mathbb{C}}$$ la projection canonique. Si $\mathcal{F}$ est un feuilletage holomorphe sur $\mathbb{P}^{2}_{\mathbb{C}}$ associé aux données $(V_{k},\omega_{k}),$ on peut définir sur $\mathbb{C}^3\smallsetminus\{0\}$ le feuilletage $\pi^{*}\mathcal{F}$ associé aux données $(\pi^{-1}(V_{k}),\pi^{*}\omega_{k}).$ Un résultat de H. \textsc{Cartan} \cite{Car79} assure la trivialité du groupe de cohomologie $\mathrm{H}^1(\mathbb C^{3}\smallsetminus\{0\},\mathcal O^{*})$. De sorte que le feuilletage $\pi^{*} \mathcal F$ est défini par une $1$-forme globale $\omega \in \Omega^{1}(\mathbb C^{3}\smallsetminus\{0\})$ qui, par le théorème d'extension d'\textsc{Hartogs}, s'étend holomorphiquement à l'origine de $\mathbb C^{3}$; ainsi $\pi^{*}\mathcal F$ s'étend en un feuilletage de $\mathbb C^{3}$ défini par une $1$-forme notée encore $\omega\in\Omega^1(\mathbb C^{3})$. Par construction le champ radial $$\mathrm{R}=x\frac{\partial{}}{\partial{x}}+y\frac{\partial{}}{\partial{y}}+z\frac{\partial{}}{\partial{z}},$$ tangent aux fibres de $\pi$, est tangent à $\pi^{*}\mathcal F$ et donc $i_{\mathrm{R}}\omega=0$. On en déduit que, à unité multiplicative près, $\omega$ est homogène: $$\omega=\omega_{d+1}=p(x,y,z)\mathrm{d}x+q(x,y,z)\mathrm{d}y+r(x,y,z)\mathrm{d}z,$$ où $p,q$ et $r$ sont des polynômes homogènes de degré $d+1$, satisfaisant la condition d'\textsc{Euler}: $$xp+yq+zr=0$$ et $\pgcd(p,q,r)=1$ (\textit{i.e.} $\mathrm{Sing}\omega_{d+1}$ est formé de points isolés). Cet énoncé de type GAGA dit qu'un feuilletage holomorphe sur $\mathbb P_{\mathbb C}^{2}$ est en fait alg\'ebrique puisque défini en coordonnées homogènes par une $1$-forme homogène.

\noindent Dualement le feuilletage $\mathcal{F}$ peut aussi être défini par un champ de vecteurs homogène
$$\mathrm{Z}=\mathrm{Z}_{d}=A(x,y,z)\frac{\partial{}}{\partial{x}}+B(x,y,z)\frac{\partial{}}{\partial{y}}+C(x,y,z)\frac{\partial{}}{\partial{z}},$$
les coefficients $A,B$ et $C$ désignent des polynômes homogènes de degré $d$ sans facteur commun. La relation entre $\mathrm{Z}$ et $\omega$ est donnée par $$\omega=i_{\mathrm{R}}i_{\mathrm{Z}}(\mathrm{d}x\wedge\mathrm{d}y\wedge\mathrm{d}z)=A\alpha+B\beta+C\gamma,$$
avec $\alpha=y\mathrm{d}z-z\mathrm{d}y,\quad \beta=z\mathrm{d}x-x\mathrm{d}z,\quad \gamma=x\mathrm{d}y-y\mathrm{d}x.$
\vspace{2mm}

\noindent L'entier $d$ est par définition le \textbf{\textit{degré}} du feuilletage $\mathcal{F}$; on le note $\deg\mathcal{F}.$
\vspace{2mm}

\noindent Soit $\mathcal{C}\subset\pp$ une courbe algébrique d'équation homogène $F(x,y,z)=0.$ On dit que $\mathcal{C}$ est une \textbf{\textit{courbe invariante}} par $\mathcal{F}$ si $\mathcal{C}\smallsetminus\Sing\F$ est une réunion de feuilles, feuilles au sens ordinaire, du feuilletage régulier $\F|_{\pp\smallsetminus\Sing\F}$. Ceci se traduit en termes algébriques par: la $2$-forme $\omega\wedge\mathrm{d}F$ est divisible par $F$, {\it i.e.} s'annule sur toute composante irréductible de $\mathcal{C}.$

\noindent Lorsque toutes les composantes irréductibles de $\mathcal{C}$ ne sont pas $\F$-invariantes et si $p$ est un point quelconque de $\mathcal{C}$, on peut définir un invariant $\Tang(\F,\mathcal{C},p)$, représentant l'\textbf{\textit{ordre de tangence}} de $\F$ avec $\mathcal{C}$ en $p$, comme suit. Fixons une carte locale $(x,y)$ telle que $p=(0,0)$; soient $f(x,y)=0$ une équation locale (réduite) de $\mathcal{C}$ au voisinage de $p$ et $\X$ un champ de vecteurs définissant le germe de $\F$ en $p.$ Désignons par $\X(f)$ la dérivée de \textsc{Lie} de $f$ le long de $\X$ et par $\langle f,\X(f)\rangle$ l'id\'eal de $\C\{x,y\}$ engendr\'e par $f$ et $\X(f)$; alors
$$\Tang(\F,\mathcal{C},p)=\dim_\mathbb{C}\frac{\C\{x,y\}}{\langle f,\X(f)\rangle}.$$
On voit sans peine que cette définition est bien posée, et $\Tang(\F,\mathcal{C},p)<+\infty$ par la non-invariance de~$\mathcal{C}.$ De plus, $\Tang(\F,\mathcal{C},p)=0$ si et seulement si $\mathcal{C}$ est régulière en $p$ et $\F$ est transverse à $\mathcal{C}$ en $p.$ Donc il y a seulement un nombre fini de points $p$ où $\Tang(\F,\mathcal{C},p)>0$ et on peut poser
$$\Tang(\F,\mathcal{C})=\sum_{p\in\mathcal{C}}\Tang(\F,\mathcal{C},p)\in\N.$$
\newpage
\hfill
\vspace{2mm}

\noindent Il est facile de vérifier que si $i\hspace{1mm}\colon\,L\to\mathbb{P}^2_{\mathbb{C}}$ est l'inclusion d'une droite générique du plan projectif $\mathbb{P}^{2}_{\mathbb{C}},$ alors $\deg\mathcal{F}$ coïncide avec le nombre de points singuliers de la restriction $i^{*}\mathcal{F}$ ou, ce qui revient au même, le nombre de points de tangence de $\mathcal{F}$ avec $L.$
\begin{center}
\begin{tabular}{ccc}
\begin{tikzpicture}[x=1.0cm,y=1.0cm,scale=0.26]
\draw[line width=0.8 pt](8,8)--(8,2);
\draw[line width=0.8 pt](8,8)--(3,4.5);
\draw[line width=0.8 pt](8,8)--(4.5,3);
\draw[line width=0.8 pt](8,8)--(6.5,2);
\draw[line width=0.8 pt](8,8)--(11.5,3);
\draw[line width=0.8 pt](8,8)--(13,4.5);
\draw[line width=0.8 pt](8,8)--(9.5,2);
\draw[line width=1.2 pt](3.5,5.5)--(12,6);
\draw[fill=black!100](8,8)circle(0.15cm);
\draw(5.4,0.9)node[right]{degr\'{e} 0};
\end{tikzpicture}
\hspace{1cm}&\hspace{1cm}
\begin{tikzpicture}[x=1.0cm,y=1.0cm,scale=0.4]
\draw[line width=0.8 pt](4.1,4)ellipse(1cm and 0.6cm);
\draw[line width=0.8 pt](4.2,4)ellipse(1.4cm and 0.8cm);
\draw[line width=0.8 pt](4.4,4)ellipse(2cm and 1.2cm);
\draw[line width=0.8 pt](4.4,4)ellipse(2.4cm and 1.5cm);
\draw[line width=0.8 pt](4.7,4)ellipse(3cm and 1.9cm);
\draw[line width=1.2 pt](3,6.)--(2.5,2);
\draw[fill=gray!100](2.76,4) circle (0.1cm) ;
\draw(3,1.4)node[right]{degr\'{e} 1};
\end{tikzpicture}
\hspace{1cm}&\hspace{1cm}
\begin{tikzpicture}[x=1.0cm,y=1.0cm,scale=0.3,rotate=30]
\draw[rotate=-45,line width=0.8 pt,domain=-0.3:4.7,variable=\x,black]plot({\x},{1/3*\x^3-2*\x^2+3*\x+6});
\draw[rotate=-46,line width=0.8 pt,domain=-0.3:4.7,variable=\x,black]plot({\x},{1/3*\x^3-2*\x^2+3*\x+5.5});
\draw[rotate=-47,line width=0.8 pt,domain=-0.3:4.7,variable=\x,black]plot({\x},{1/3*\x^3-2*\x^2+3*\x+5});
\draw[rotate=-48,line width=0.8 pt,domain=-0.3:4.7,variable=\x,black]plot({\x},{1/3*\x^3-2*\x^2+3*\x+4.5});
\draw[rotate=-49,line width=0.8 pt,domain=-0.3:4.7,variable=\x,black]plot({\x},{1/3*\x^3-2*\x^2+3*\x+4});
\draw[rotate=-50,line width=0.8 pt,domain=-0.3:4.7,variable=\x,black]plot({\x},{1/3*\x^3-2*\x^2+3*\x+3.5});
\draw[rotate=2,line width=1.2 pt](2.6,3.81)--(9.1,0.43);
\draw(2.5,0.1)node[right]{degr\'{e} 2};
\end{tikzpicture}
\end{tabular}
\end{center}
En d'autres termes, pour toute droite $L\subset\pp$ non invariante par $\F,$ on a
\begin{equation}\label{equa:Tang(F,L)=d}
\deg\F=\Tang(\F,L)
\end{equation}

\noindent Soient $\F$ un feuilletage de degré $d$ sur $\pp$ et $s\in\Sing\F$ un point singulier de $\F.$ Nous allons introduire plusieurs notions locales attachées au couple $(\mathcal{F},s).$ Le germe de $\F$ en $s$ est défini, à multiplication par une unité de l'anneau local $\mathcal{O}(\mathbb{C}^2,s)$ en $s$ près, par un champ de vecteurs $\X=A(x,y)\frac{\partial{}}{\partial{x}}+B(x,y)\frac{\partial{}}{\partial{y}}$.
\begin{enumerate}
\item [\textbf{\textit{1.}}] L'\textbf{\textit{ordre d'annulation}} $\nu(\mathcal{F},s)$ de $\mathcal{F}$ en $s,$ encore appelé \textbf{\textit{multiplicité algébrique}}, est donné par $$\nu(\mathcal{F},s)=\min\{\nu(A,s),\nu(B,s)\},$$ où $\nu(g,s)$ désigne l'ordre d'annulation de la fonction $g$ en $s$.

\item [\textbf{\textit{2.}}] Notons $\mathfrak{L}_s$ la famille des droites non invariantes par $\F$ et qui passent par $s.$ Pour toute droite $\ell_s$ de~$\mathfrak{L}_s,$ on a l'encadrement $1\leq\Tang(\F,\ell_s,s)\leq d$. Ceci nous permet d'associer au couple~$(\mathcal{F},s)$ les entiers naturels (invariants) suivants
    \begin{align*}
    &\hspace{0.8cm}\tau(\mathcal{F},s)=\min\{\Tang(\F,\ell_s,s)\hspace{1mm}\vert\hspace{1mm}\ell_s\in\mathfrak{L}_s\},&&\hspace{0.5cm}
    \kappa(\mathcal{F},s)=\max\{\Tang(\F,\ell_s,s)\hspace{1mm}\vert\hspace{1mm}\ell_s\in\mathfrak{L}_s\}.
    \end{align*}
    L'invariant $\tau(\mathcal{F},s)$ représente l'ordre de tangence de $\mathcal{F}$ avec une droite générique passant par $s$. Il est facile de montrer que
    $$\tau(\mathcal{F},s)=\min\{k\geq\nu(\mathcal{F},s)\hspace{1mm}\colon\det(J^{k}_{s}\,\mathrm{X},\mathrm{R}_{s})\neq0\},$$ où $J^{k}_{s}\,\mathrm{X}$ désigne le $k$-jet de $\mathrm{X}$ en $s$ et $\mathrm{R}_{s}$ le champ radial centré en $s.$

\item [\textbf{\textit{3.}}] La singularité $s$ de $\mathcal{F}$ est dite \textbf{\textit{radiale}} si $\nu(\mathcal{F},s)=1$ et si de plus $\tau(\mathcal{F},s)\geq2$. Si tel est le cas, l'entier naturel $\tau(\mathcal{F},s)-1$, compris entre $1$ et $d-1,$ est appelé l'\textbf{\textit{ordre de radialité}} de $s.$ Autrement dit, la singularité $s$ est radiale d'ordre $n-1\geq1$ si le champ décrivant $\mathcal{F}$ est du type $$\X=C_{n-2}(x,y)\cdot \mathrm{R}_{s}+\X_{n}+\text{termes de plus haut degré},$$ où $C_{n-2}$ est un polynôme de degré~$\leq n-2$, $\X_{n}$ est un champ homogène de degré $n$, avec $C_{n-2}(s)\neq0$\, et \,$\X_{n}\nparallel\mathrm{R}_{s}$.
    \begin{rem}\label{rem:sum-tau<d}
    Soit $L$ une droite non invariante par $\F.$ Il résulte de la formule~(\ref{equa:Tang(F,L)=d}) que l'on a
    \reqnomode
    \begin{equation}\label{equa:sum-tau<d}
    d\geq\sum_{s\in\Sing\F\cap L}\tau(\F,s)
    \end{equation}
    En particulier, si $\F$ possède deux singularités radiales $s_1,s_2$ d'ordre maximal $d-1$, alors la droite $(s_1s_2)$ est nécessairement invariante par $\F$.
    \end{rem}

\item [\textbf{\textit{4.}}] Le \textbf{\textit{nombre de \textsc{Milnor}}} $\mu(\mathcal{F},s)$ de $\mathcal{F}$ en $s$ est l'entier $$\mu(\mathcal{F},s)=\dim_\mathbb{C}\frac{\mathcal{O}(\mathbb{C}^2,s)}{\langle A,B\rangle},$$ où $\langle A,B\rangle$ désigne l'id\'eal de $\mathcal{O}(\mathbb{C}^2,s)$ engendr\'e par $A$ et $B$.
    \begin{rems}\label{rems:mu>nu^2-Darboux}
    \begin{itemize}
    \item []\hspace{-0.6cm}(i)\hspace{1mm} On a toujours l'inégalité $\mu(\F,s)\geq \nu(A,s)\cdot\nu(B,s),$ \emph{voir} par
    exemple \cite[Pag. 158]{Fis01}. On en déduit en particulier l'inégalité
    \reqnomode
    \begin{equation}\label{equa:mu>nu^2}
    \mu(\mathcal{F},s)\geq\nu(\F,s)^2
    \end{equation}
    \item [(ii)] Le feuilletage $\F$ possède $d^{2}+d+1$ singularités comptées avec multiplicité:
    \begin{equation}\label{equa:Darboux}
    \sum_{s\in\mathrm{Sing}\mathcal{F}}\mu(\mathcal{F},s)=d^{2}+d+1
    \end{equation}
    Cela implique, en particulier, qu'il n'y a pas de feuilletage régulier ({\it i.e.} sans singularité) sur $\mathbb{P}^{2}_{\mathbb{C}}.$ Pour une preuve de cette formule, \emph{voir} par exemple \cite[Proposition~9.2]{CCD13}.
    \item [(iii)] Tout feuilletage sur $\pp$ de degré $d\geq2$ a au plus une singularité de multiplicité algébrique maximale $d.$ Ceci résulte immédiatement des formules (\ref{equa:mu>nu^2}) et (\ref{equa:Darboux}).
    \end{itemize}
    \end{rems}

\item [\textbf{\textit{5.}}] La singularité $s$ est dite \textbf{\textit{non-dégénérée}} si $\mu(\F,s)=1$, c'est équivalent de dire que la partie linéaire $J^{1}_{s}\mathrm{X}$ de $\mathrm{X}$ possède deux valeurs propres $\lambda,\mu$ non nulles. La quantité $\mathrm{BB}(\F,s)=\frac{\lambda}{\mu}+\frac{\mu}{\lambda}+2$ est appelée l'\textbf{\textit{invariant de \textsc{Baum-Bott}}} de $\F$ en $s$ (\emph{voir} \cite{BB72}). D'après \cite{CS82} il passe par $s$ au moins un germe de courbe $\mathcal{C}$ invariante par $\F$; à isomorphisme local près, on peut se ramener à $s=(0,0)$,\, $\mathrm{T}_{s}\mathcal{C}=\{y=\,0\,\}$ et $J^{1}_{s}\mathrm{X}=\lambda x\frac{\partial}{\partial x}+(\varepsilon x+\mu\hspace{0.1mm}y)\frac{\partial}{\partial y}$, où l'on peut prendre $\varepsilon=0$ si $\lambda\neq\mu$. La quantité $\mathrm{CS}(\F,\mathcal{C},s)=\frac{\lambda}{\mu}$ est appelée l'\textbf{\textit{indice de \textsc{Camacho-Sad}}} de $\F$ en $s$ par rapport à $\mathcal{C}$.
    \begin{rems}\label{rems:BB-CS}
    \begin{itemize}
    \item []\hspace{-0.6cm}(i)\hspace{1mm} L'invariant de \textsc{Baum-Bott} peut se définir via la géométrie différentielle pour n'importe quelle
    singularité (même dégénérée avec deux valeurs propres nulles). Même remarque pour l'indice de \textsc{Camacho-Sad}.
    \item [(ii)] Le feuilletage $\F$ satisfait la formule de \textsc{Baum-Bott} (\cite[Pag. 34]{Bru00})
    \reqnomode
    \begin{equation}\label{equa:BB}
    \sum_{s\in\mathrm{Sing}\F}\mathrm{BB}(\F,s)=(d+2)^2
    \end{equation}
    \item [(iii)] Si $\F$ admet une droite invariante $L$, alors (\cite[Pag. 37]{Bru00})
    \reqnomode
    \begin{equation}\label{equa:sum-CS=1}
    1=\sum_{s\in\Sing\F\cap L}\mathrm{CS}(\F,L,s)
    \end{equation}
    \end{itemize}
    \end{rems}
\end{enumerate}
\bigskip

\noindent Soit $\mathcal{F}$ un feuilletage de degré $d$ sur $\mathbb{P}^{2}_{\mathbb{C}}$ donné par un champ de vecteurs $\mathrm{Z}$ homogène de degré~$d.$ Le \textbf{\textit{diviseur d'inflexion}} de $\mathcal{F},$ noté $\IF,$ est le diviseur défini par l'équation
\leqnomode
\begin{equation}\label{equa:ext1}
\left| \begin{array}{ccc}
x   &  \mathrm{Z}(x)   &  \mathrm{Z}^2(x) \\
y   &  \mathrm{Z}(y)   &  \mathrm{Z}^2(y)  \\
z   &  \mathrm{Z}(z)   &  \mathrm{Z}^2(z)
\end{array} \right|=0.
\end{equation}
Ce diviseur a été étudié dans \cite{Per01} dans un contexte plus général. En particulier, les propriétés suivantes ont été prouvées.

\begin{enumerate}
\item [\textbf{\textit{1.}}] Sur $\mathbb{P}^{2}_{\mathbb{C}}\smallsetminus\mathrm{Sing}\mathcal{F},$ $\IF$ coïncide avec la courbe décrite par les points d'inflexion des feuilles de $\mathcal{F}$;
\item [\textbf{\textit{2.}}] Si $\mathcal{C}$ est une courbe algébrique irréductible invariante par $\mathcal{F},$ alors $\mathcal{C}\subset \IF$ si et seulement si $\mathcal{C}$ est une droite invariante;
\item [\textbf{\textit{3.}}] $\IF$ peut se décomposer en $\IF=\IinvF+\ItrF,$ où le support de $\IinvF$ est constitué de l'ensemble des droites invariantes par $\mathcal{F}$ et où le support de $\ItrF$ est l'adhérence des points d'inflexion qui sont isolés le long des feuilles de $\mathcal{F}$;
\item [\textbf{\textit{4.}}] Le degré du diviseur $\IF$ est $3d.$
\end{enumerate}

\begin{defin}\label{def:convexe}
Un feuilletage $\mathcal{F}$ sur $\pp$ sera dit \textbf{\textit{convexe}} si son diviseur d'inflexion $\IF$ est invariant par $\mathcal{F},$ {\it i.e.} si $\IF$ est le produit de droites invariantes.
\end{defin}

\begin{eg}
Le feuilletage $\mathcal{F}$ de degré $3$ sur $\mathbb{P}^{2}_{\mathbb{C}}$ défini par le champ $\mathrm{Z}=x^3\frac{\partial{}}{\partial{x}}+y^3\frac{\partial{}}{\partial{y}}+z^3\frac{\partial{}}{\partial{z}}$ est convexe car $\IF=\{xyz(x-y)(x+y)(x-z)(x+z)(y-z)(y+z)=0\}=\IinvF.$
\end{eg}
\begin{eg}
Le feuilletage $\mathcal{F}$ de degré $3$ sur $\mathbb{P}^{2}_{\mathbb{C}}$ donné par le champ $\mathrm{Z}=y^3\frac{\partial{}}{\partial{x}}+x^3\frac{\partial{}}{\partial{y}}+z^3\frac{\partial{}}{\partial{z}}$ n'est pas convexe car $\ItrF=\{(xy-z^2)(xy+z^2)=0\}\neq\emptyset.$
\end{eg}
\bigskip

\noindent Soit $\mathcal{F}$ un feuilletage sur $\mathbb{P}^{2}_{\mathbb{C}}$ et notons $\mathbb{\check{P}}^{2}_{\mathbb{C}}$ le plan projectif dual. L'\textbf{\textit{application de \textsc{Gauss}}} associée à $\mathcal{F}$ est l'application rationnelle $\mathcal{G}_{\mathcal{F}}\hspace{1mm}\colon\mathbb{P}^{2}_{\mathbb{C}}\dashrightarrow \mathbb{\check{P}}^{2}_{\mathbb{C}}$ définie par $\mathcal{G}_{\mathcal{F}}(p)=\mathrm{T}_{\hspace{-0.4mm}p}\mathcal{F},$ où $\mathrm{T}_{\hspace{-0.4mm}p}\mathcal{F}$ désigne la droite tangente à la feuille de $\mathcal{F}$ en $p.$ Si le feuilletage $\mathcal{F}$ est donné par une $1$-forme $\omega=a(x,y,z)\mathrm{d}x+b(x,y,z)\mathrm{d}y+c(x,y,z)\mathrm{d}z,$ l'application de \textsc{Gauss} est donnée par $$\mathcal{G}_{\mathcal{F}}(p)=[a(p):b(p):c(p)].$$
Posons la définition suivante, qui nous sera utile plus loin. Soit $\mathcal{C}\subset\pp$ une courbe passant par certains points singuliers de $\F$; on définit $\mathcal{G}_{\mathcal{F}}(\mathcal{C})$ comme étant l'adhérence de $\mathcal{G}_{\F}(\mathcal{C}\setminus\Sing\F)$.

\begin{rem}
Soient $(x,y)$ une carte affine de $\mathbb{P}^{2}_{\mathbb{C}}$ et $(p,q)$ la carte affine de $\mathbb{\check{P}}^{2}_{\mathbb{C}}$ correspondant à la droite $\{y=px+q\}\subset\mathbb{P}^{2}_{\mathbb{C}}.$ Si $\mathcal{F}$ est un feuilletage défini par le champ de vecteurs $\X=A(x,y)\frac{\partial{}}{\partial{x}}+B(x,y)\frac{\partial{}}{\partial{y}},$ l'équation de la droite tangente à $\mathcal{F}$ en un point $(x_{0},y_{0})\not\in\mathrm{Sing}\mathcal{F}$ est $B(x_{0},y_{0})x-A(x_{0},y_{0})y+(A(x_{0},y_{0})y_{0}-B(x_{0},y_{0})x_{0})=0$; ainsi l'application de \textsc{Gauss} $\mathcal{G}_{\mathcal{F}}$ associée à $\mathcal{F}$ s'écrit dans ces coordonnées $$(x,y)\longmapsto\left(\frac{B(x,y)}{A(x,y)},y-\frac{B(x,y)}{A(x,y)}x\right).$$
\end{rem}
\newpage
\section{\'{E}léments de la théorie locale des tissus plans}\vspace{6mm}

\Subsection{Germes de tissus plans réguliers}\vspace{3mm}

\begin{defin}
Un \textbf{\textit{germe de $k$-tissu régulier}} $\mathcal{W}=\mathcal{F}_{1}\boxtimes\cdots\boxtimes\mathcal{F}_{k}$ de $(\mathbb{C}^{2},0)$ est une collection de $k$ germes de feuilletages holomorphes réguliers de $(\mathbb{C}^{2},0)$ deux à deux en \textbf{\textit{position générale}} en l'origine, c'est-à-dire $\mathrm{T}_{0}\mathcal{F}_{i}\neq \mathrm{T}_{0}\mathcal{F}_{j}$ pour $1\leq i<j\leq k.$
\end{defin}

\noindent Soit $\mathcal{W}=\mathcal{F}_{1}\boxtimes\cdots\boxtimes\mathcal{F}_{k}$ un germe de $k$-tissu régulier de $(\mathbb{C}^{2},0).$ Usuellement les germes des feuilletages $\mathcal{F}_{i}$ sont définis par des germes de $1$-formes holomorphes $\omega_i\in\Omega^{1}(\mathbb{C}^{2},0)$ uniques à des inversibles près de $\mathcal{O}(\mathbb{C}^{2},0)$, non nuls en $0.$ L'hypothèse sur les $\mathcal{F}_{i}$ d'être en position générale se traduit alors par $$\omega_i\wedge\omega_j(0)\neq0 \quad \text{pour} \quad 1\leq i<j\leq k.$$ Le tissu $\mathcal{W}$ peut donc être vu comme un élément de $\mathrm{Sym}^{k}\Omega^{1}(\mathbb{C}^{2},0)$ que l'on notera $\mathcal{W}=\mathcal{W}(\omega_1\cdot\omega_2\cdot\cdots\cdot\omega_k)$ ou encore $\mathcal{W}=\mathcal{W}(\omega_1,\ldots,\omega_k).$

\noindent Grâce au théorème de \textsc{Frobenius} (\emph{voir} par exemple \cite[Théorème 2.4]{CCD13}), les $k$ feuilles du tissu $\mathcal{W}$ sont les courbes de niveau $\{F_{i}(x,y)=\text{cte}\}$ d'éléments $F_{i}\in \mathcal{O}(\mathbb{C}^{2},0)$ vérifiant $F_{i}(0)=0$ et tels que l'on ait $$\mathrm{d}F_{i}\wedge\mathrm{d}F_{j}(0)\neq0\quad \text{pour} \quad 1\leq i<j\leq k.$$ On désignera encore par $\mathcal{W}=\mathcal{W}(F_{1},\ldots,F_{k})$ un tissu donné par ses feuilles.
\vspace{2mm}

\noindent La théorie locale des tissus plans s'interesse principalement à la classification des germes de tissus plans modulo l'action naturelle de $\mathrm{Diff}(\mathbb{C}^{2},0),$ le groupe des germes de biholomorphismes à l'origine de $\mathbb{C}^{2}.$ Si $\varphi\in\mathrm{Diff}(\mathbb{C}^{2},0)$ est un germe de biholomorphisme alors l'action naturelle dont on vient de parler est donnée par $$\varphi^{*}\mathcal{W}(\omega_1\cdot\cdots\cdot\omega_k)=\mathcal{W}(\varphi^{*}(\omega_1\cdot\cdots\cdot\omega_k)).$$
\vspace{2mm}

\noindent Soient $\mathcal{W}=\mathcal{W}(\omega_1\cdots\omega_k)$ et $\mathcal{W'}=\mathcal{W}(\omega_1'\cdots\omega_k')$ deux germes de $k$-tissus de $(\mathbb{C}^{2},0).$ On dira qu'ils sont \textbf{\textit{équivalents}} s'il existe un germe de biholomorphisme $\varphi\in\mathrm{Diff}(\mathbb{C}^{2},0)$ et une unité $u\in\mathcal{O}^*(\mathbb{C}^{2},0)$ tels que $$\varphi^{*}(\omega_1\cdots\omega_k)=u.\omega_1'\cdots\omega_k'.$$
Le théorème d'inversion local implique que tout $1$-tissu (resp. $2$-tissu) régulier de $(\mathbb{C}^{2},0)$ est équivalent au tissu des droites $x=$cte (resp. $x=$cte et $y=$cte).
\begin{center}
\begin{tabular}{c}
\hspace{1.2cm}
\begin{tikzpicture}[x=2.5cm,y=1.cm,scale=0.3,rotate=0]
\begin{scope}[shift={(-3.1,0)}]
\draw[line width=0.7pt,color=black] (-1,4) arc (-138:-177:-14cm);
\begin{scope}[shift={(0.3,0)}]
\draw[line width=0.7pt,color=black] (-1,4) arc (-138:-177:-13cm);
\end{scope}
\begin{scope}[shift={(0.6,0)}]
\draw[line width=0.7pt,color=black] (-1,4) arc (-138:-177:-12cm);
\end{scope}
\begin{scope}[shift={(0.9,0)}]
\draw[line width=0.7pt,color=black] (-1,4) arc (-138:-177:-11cm);
\end{scope}
\begin{scope}[shift={(-0.1,-0.7)}]
\draw[line width=0.7pt,color=black] (-1,4) arc (-138:-177:-12.5cm);
\end{scope}
\begin{scope}[shift={(-0.35,-1)}]
\draw[line width=0.7pt,color=black] (-1,4) arc (-138:-177:-12cm);
\end{scope}
\draw[rotate=115,line width=0.7pt,color=black] (-0.5,4) arc (-160:-177:-30cm);
\begin{scope}[shift={(-0.1,0.7)}]
\draw[rotate=115,line width=0.7pt,color=black] (-0.5,3.8) arc (-160:-177:-28cm);
\end{scope}
\begin{scope}[shift={(-0.2,1.4)}]
\draw[rotate=115,line width=0.7pt,color=black] (-0.5,3.6) arc (-160:-177:-26cm);
\end{scope}
\begin{scope}[shift={(-0.3,2.1)}]
\draw[rotate=115,line width=0.7pt,color=black] (-0.5,3.4) arc (-160:-177:-24cm);
\end{scope}
\begin{scope}[shift={(0.1,-0.7)}]
\draw[rotate=115,line width=0.7pt,color=black] (-0.5,3.8) arc (-160:-177:-28cm);
\end{scope}
\begin{scope}[shift={(0.2,-1.4)}]
\draw[rotate=115,line width=0.7pt,color=black] (-0.5,3.6) arc (-160:-177:-26cm);
\end{scope}
\end{scope}
\begin{scope}[shift={(3.5,0)}]
\draw[line width=0.7pt](0,4.5)    -- (0,-4.5);
\draw[line width=0.7pt](0.3,4.4)  -- (0.3,-4.4);
\draw[line width=0.7pt](0.6,4.2)  -- (0.6,-4.2);
\draw[line width=0.7pt](0.9,4.3)  -- (0.9,-4.3);
\draw[line width=0.7pt](-0.3,4.1) -- (-0.3,-4.1);
\draw[line width=0.7pt](-0.6,4.4) -- (-0.6,-4.4);
\draw[line width=0.7pt](-0.9,4.3) -- (-0.9,-4.3);
\draw[line width=0.7pt](-1.2,4.1) -- (-1.2,-4.1);
\draw[line width=0.7pt](1.7,0)    -- (-2.,0);
\draw[line width=0.7pt](1.6,0.8)  -- (-1.9,0.8);
\draw[line width=0.7pt](1.65,1.6)  -- (-1.95,1.6);
\draw[line width=0.7pt](1.55,2.4)  -- (-1.85,2.4);
\draw[line width=0.7pt](1.65,-0.8) -- (-1.95,-0.8);
\draw[line width=0.7pt](1.6,-1.6) -- (-1.9,-1.6);
\draw[line width=0.7pt](1.55,-2.4) -- (-1.85,-2.4);
\end{scope}
\begin{scope}[shift={(0.,0.8)}]
\draw[rotate=30,thick,->,color=black] (-1.3,0.5) arc (-100:-140:-8.cm);
\end{scope}
\draw(-0.4,0.5)node[right]{$\sim$};
\draw(-7.5,-7.5)node[right]{Tous les germes de $2$-tissus plans r\'{e}guliers sont \'{e}quivalents (modulo $\mathrm{Diff}(\mathbb{C}^{2},0)).$};
\end{tikzpicture}
\end{tabular}
\end{center}
\noindent Le cas des $3$-tissus est plus intéressant. Un $3$-tissu $\mathcal{W}$ de $(\mathbb{C}^{2},0)$ est dit \textbf{\textit{parallélisable}} s'il est équivalent au tissu des droites $x=$cte, $y=$cte, $x+y=$cte. Se pose alors la question: sous quelle condition peut-on affirmer qu'un tel tissu est parallélisable? La réponse à cette question est donnée par le \og théorème de structure des $3$-tissus parallélisables \fg, que nous énoncerons ci-dessous sans démonstration.

\Subsection{Courbure de \textsc{Blaschke}, hexagonalité et structure des $3$-tissus parallélisables}\label{subsec:structure-3-tissu}\vspace{3mm}

\subsubsection{Courbure de \textsc{Blaschke}}

Soit $\mathcal{W}=\mathcal{W}(\omega_{1},\omega_{2},\omega_{3})$ un $3$-tissu régulier de $(\mathbb{C}^{2},0).$ L'hypothèse de position générale assure alors que les $1$-formes $\omega_{1},\omega_{2}$ et $\omega_{3}$ définissant ce tissu sont indépendantes deux à deux en $0$; il existe donc $\rho_{1}$ et $\rho_{2}$ nécessairement inversibles dans $\mathcal{O}(\mathbb{C}^{2},0)$ tels que $$\omega_{3}=\rho_{1}\omega_{1}+\rho_{2}\omega_{2}.$$ Comme les $1$-formes $\omega$ et $\rho\omega$ définissent le même feuilletage si $\rho$ est inversible dans $\mathcal{O}(\mathbb{C}^{2},0),$ on supposera ici que $$\omega_{1}+\omega_{2}+\omega_{3}=0$$ et on dira que ces trois $1$-formes sont \textbf{\textit{normalisées}}. On définit alors la $2$-forme (non singulière) $\Omega=\omega_{1}\wedge\omega_{2}$ et dans ce cas, on a les égalités $$\Omega=\omega_{1}\wedge\omega_{2}=\omega_{2}\wedge\omega_{3}=\omega_{3}\wedge\omega_{1}.$$ Si l'on pose $\omega'_{i}=\rho_{i}\omega_{i},i=1,2,3,$ où $\rho_{i}$ est inversible dans $\mathcal{O}(\mathbb{C}^{2},0),$ avec $\sum_{i=1}^{3}\omega'_{i}=0$ alors $$\Omega':=\omega'_{1}\wedge\omega'_{2}=\rho_{1}\rho_{2}\Omega=\rho_{2}\rho_{3}\Omega=\rho_{1}\rho_{3}\Omega,$$ ce qui prouve que $\rho_{1}=\rho_{2}=\rho_{3}$ et si $\rho$ désigne cette valeur commune, on a $$\Omega'=\rho^{2}\Omega.$$ Puisque dans $\mathbb{C}^{2},$ on a pour toute $1$-forme $\omega$ l'identité $\omega\wedge\mathrm{d}\omega=0,$ le théorème de \textsc{Frobenius} donne l'existence pour $i=1,2$ et $3$ d'une fonction inversible $g_{i}\in\mathcal{O}^*(\mathbb{C}^{2},0),$ et d'une fonction $F_{i}\in \mathcal{O}(\mathbb{C}^{2},0)$ vérifiant $F_{i}(0)=0$ et $\mathrm{d}F_{i}(0)\neq0$ telles que $$\omega_{i}=g_{i}\mathrm{d}F_{i}.$$ On peut donc écrire $$\mathrm{d}\omega_{i}=\mathrm{d}g_{i}\wedge\mathrm{d}F_{i}=\frac{\mathrm{d}g_{i}}{g_{i}}\wedge\omega_{i}$$ et puisque le système des $\omega_{i}$ est normalisé, on a $$\sum_{i=1}^{3}\frac{\mathrm{d}g_{i}}{g_{i}}\wedge\omega_{i}=0.$$ Puisque $-\omega_{3}=\omega_{1}+\omega_{2},$ on a $$\left(\frac{\mathrm{d}g_{1}}{g_{1}}-\frac{\mathrm{d}g_{3}}{g_{3}}\right)\wedge\omega_{1}+\left(\frac{\mathrm{d}g_{2}}{g_{2}}-\frac{\mathrm{d}g_{3}}{g_{3}}\right)
\wedge\omega_{2}=0.$$ Les deux formes $\omega_{1}$ et $\omega_{2}$ étant indépendantes, le lemme de \'{E}.~\textsc{Cartan} (\emph{voir} par exemple \cite{Car94}) assure l'existence de $\alpha,\beta$ et $\delta$ dans $\mathcal{O}(\mathbb{C}^{2},0)$ telles que
$$
\left\{
\begin{array}{c}
\dfrac{\mathrm{d}g_{1}}{g_{1}}-\dfrac{\mathrm{d}g_{3}}{g_{3}}=\alpha\omega_{1}+\beta\omega_{2}\\
\raisebox{-3mm}{$\dfrac{\mathrm{d}g_{2}}{g_{2}}-\dfrac{\mathrm{d}g_{3}}{g_{3}}=\beta\omega_{1}+\delta\omega_{2}$}
\end{array}
\right.
$$
On vérifie alors que l'on a $$\frac{\mathrm{d}g_{1}}{g_{1}}+(\beta-\alpha)\omega_{1}=\frac{\mathrm{d}g_{2}}{g_{2}}+(\beta-\delta)
\omega_{2}=\frac{\mathrm{d}g_{3}}{g_{3}}-\beta\omega_{3}\hspace{1mm};$$
la $1$-forme ainsi définie sera notée $\gamma$ et vérifie
$$\hspace{3cm} \mathrm{d}\omega_{i}=\gamma\wedge\omega_{i},\qquad  \forall\hspace{0.1cm}i\in\{1,2,3\}.$$
Pour une autre normalisation $\sum_{i=1}^{3}\omega'_{i}=0,$ où l'on a vu que $\omega'_{i}=\rho\omega_{i},$ on a $$\mathrm{d}\omega'_{i}=\gamma'\wedge\omega'_{i}=\mathrm{d}\rho\wedge\omega_{i}+\rho\mathrm{d}\omega_{i}=\rho\gamma'\wedge\omega_{i}.$$ Comme les $1$-formes $\omega_{1}$ et $\omega_{2}$ engendrent $\Omega^{1}(\mathbb{C}^{2},0),$ on a $$\gamma'=\gamma+\frac{\mathrm{d}\rho}{\rho}\hspace{1mm};$$ ainsi, par différentiation de cette dernière égalité, on obtient
\begin{equation}\label{normalisation}
\mathrm{d}\gamma'=\mathrm{d}\gamma
\end{equation}
\begin{defin}
On appelle \textbf{\textit{courbure de \textsc{Blaschke}}} du $3$-tissu $\mathcal{W}$ la $2$-forme $K(\mathcal{W})=\mathrm{d}\gamma.$
\end{defin}
\noindent Il résulte immédiatement de l'égalité \eqref{normalisation} que la courbure de \textsc{Blaschke} ne dépend que du tissu et non du choix de la normalisation qui a permis de la définir.
\begin{rem}\label{rem:calcul-kw}
La $2$-forme $\Omega$ que nous avons construite étant non singulière, on peut écrire $\mathrm{d}\omega_{i}=h_{i}\Omega,$ où $h_{i}\in\mathcal{O}(\mathbb{C}^{2},0)$ pour $i=1,2,3.$ On vérifie alors que $$\gamma=h_{2}\omega_{1}-h_{1}\omega_{2}=h_{3}\omega_{2}-h_{2}\omega_{3}=h_{1}\omega_{3}-h_{3}\omega_{1}.$$
\end{rem}
\begin{eg}
Le $3$-tissu $\mathcal{W}(x,y,x+y)$ est de courbure nulle. En effet, avec la normalisation $\mathrm{d}x+\mathrm{d}y-\mathrm{d}(x+y)=0,$ on obtient $\Omega=\mathrm{d}x\wedge\mathrm{d}y$; ainsi $h_{1}=h_{2}=h_{3}=0$ ce qui prouve que $\gamma=0.$
\end{eg}
\begin{eg}
Pour un $3$-tissu $\mathcal{W}=\mathcal{W}(F_{1},F_{2},F_{3}),$ où $F_{1},F_{2}$ et $F_{3}$ sont dans $\mathcal{O}(\mathbb{C}^{2},0),$ l'hypothèse de position générale et le théorème d'inversion locale permettent de montrer que l'application $$\varphi=(F_{1},F_{2})$$ de $\mathbb{C}^{2}$ dans lui-même est un difféomorphisme local. Il existe donc une fonction $f$ dans $\mathcal{O}(\mathbb{C}^{2},0),$ nulle en $0,$ telle que $\mathcal{W}$ soit défini, via $\varphi,$ par le triplet $$\left(x=\text{cte},y=\text{cte},f(x,y)=\text{cte}\right).$$ On peut alors calculer la courbure d'un tel tissu \og rectifié \fg: $\mathcal{W}_{f}=\mathcal{W}(x,y,f(x,y)).$ En effet, on a la normalisation suivante $$\partial_{x}(f)\mathrm{d}x+\partial_{y}(f)\mathrm{d}y-\mathrm{d}f=0\hspace{1mm};$$ par suite $\Omega=\partial_{x}(f)\partial_{y}(f)\mathrm{d}x\wedge\mathrm{d}y,$ et avec les notations de la Remarque \ref{rem:calcul-kw}, on a $$h_{1}=-h_{2}=\frac{\partial_{x}\partial_{y}(f)}{\partial_{x}(f)\partial_{y}(f)} \qquad\text{et}\qquad h_{3}=0 .$$ On en déduit que $$\gamma=\frac{\partial_{x}\partial_{y}(f)}{\partial_{y}(f)}\mathrm{d}x+\frac{\partial_{x}\partial_{y}(f)}{\partial_{x}(f)}\mathrm{d}y$$ et donc la courbure $K(\mathcal{W}_{f})$ de $\mathcal{W}_{f}$ s'exprime sous la forme $$K(\mathcal{W}_{f})=\mathrm{d}\gamma=\left\{\partial_{x}\left(\frac{\partial_{x}\partial_{y}(f)}{\partial_{x}(f)}\right)-\partial_{y}\left(\frac{\partial_{x}
\partial_{y}(f)}{\partial_{y}(f)}\right)\right\}\mathrm{d}x\wedge\mathrm{d}y$$ et on vérifie que l'on peut l'écrire abusivement $$K(\mathcal{W}_{f})=\partial_{x}\partial_{y}\log\frac{\partial_{x}(f)}{\partial_{y}(f)}\mathrm{d}x\wedge\mathrm{d}y.$$
\end{eg}
\begin{rem}\label{rem:inv-kw-nul}
La nullité de la courbure de \textsc{Blaschke} est invariante sous l'action de $\mathrm{Diff}(\mathbb{C}^{2},0).$ En effet, si $\varphi$ désigne un isomorphisme analytique de $(\mathbb{C}^{2},0),$ alors $$K(\varphi^{*}\mathcal{W})=\varphi^{*}K(\mathcal{W}).$$ Ceci résulte du fait que l'image réciproque de la normalisation est une normalisation.
\end{rem}

\subsubsection{Tissus hexagonaux}

Soit $\mathcal{W}=\mathcal{F}_{1}\boxtimes\mathcal{F}_{2}\boxtimes\mathcal{F}_{3}$ un germe de $3$-tissu régulier de $(\mathbb{C}^{2},0).$ Soient $L_{1},L_{2}$ et $L_{3}$ les trois feuilles passant par $0$ de $\mathcal{F}_{1},\mathcal{F}_{2}$ et $\mathcal{F}_{3}$ respectivement. Si $p$ est un point de $L_{1}$ suffisamment proche de l'origine, la feuille de $\mathcal{F}_{3}$ passant par $p$ intersecte $L_{2}$ en un point qu'on notera $h_{12}(p).$ On vérifie que l'application $p\longmapsto h_{12}(p)$ définit un germe holomorphe de $(L_{1},0)$ dans $(L_{2},0).$

\noindent Plus généralement, quand $i,j$ et $k$ sont tels que $\{i,j,k\}=\{1,2,3\},$ en se déplaçant le long des feuilles de $\mathcal{F}_{i},$ on peut associer à tout point $p$ de $L_{j}$ (suffisamment proche de l'origine) un point $h_{jk}(p)$ de $L_{k}$; cela nous définit un germe d'application holomorphe $h_{jk}\hspace{1mm}\colon(L_{j},0)\longrightarrow (L_{k},0).$ Par composition on obtient un germe d'application
$$H_{1}=h_{31}\circ h_{23}\circ h_{12}\circ h_{31}\circ h_{23}\circ h_{12}\hspace{1mm}\colon(L_{1},0)\longrightarrow(L_{1},0).$$
Pour $p\in L_{1}$ (toujours supposé suffisamment proche de l'origine), l'image $q=H_{1}(p)$ de $p$ par $H_{1}$ est obtenue en traçant un \og hexagone \fg autour de l'origine en se déplaçant le long des feuilles de $\mathcal{W}:$
\begin{center}
\begin{tikzpicture}[x=1.cm,y=1.cm,scale=0.142,rotate=0]
\draw[line width=1.2pt]plot[smooth, tension=.7] coordinates {(-15,3)(-8,1)(0,0)(8,1)(15,3)};
\draw [line width=1.2pt]plot[smooth, tension=.7] coordinates {(-10,-10)(-4,-6)(0,0) (4,6) (10,10)};
\draw [line width=1.2pt]plot[smooth, tension=.7] coordinates {(11,-9)(4,-6)(0,0)(-4,6)(-5,12)};
\begin{scope}[shift={(9,3.2)}]
\draw[densely dotted,line width=0.8pt]plot[smooth, tension=.7] coordinates {(11,-9)(4,-6)(0,0)(-4,6)(-5,14)};
\end{scope}
\begin{scope}[shift={(0,7)}]
\draw[densely dotted,line width=0.8pt]plot[smooth, tension=.7] coordinates {(-15,3)(-8,1)(0,0)(8,1)(15,3)};
\end{scope}
\begin{scope}[shift={(-5.6,5.5)}]
\draw[densely dotted,line width=0.8pt]plot[smooth, tension=.7] coordinates {(-10,-10)(-4,-6)(0,0) (4,6) (11,11)};
\end{scope}
\begin{scope}[shift={(-4.7,-4)}]
\draw[densely dotted,line width=0.8pt]plot[smooth, tension=.7] coordinates {(11,-9)(4,-6)(0,0)(-4,6)(-5,12)(-5.1,15)};
\end{scope}
\begin{scope}[shift={(0,-6)}]
\draw[densely dotted,line width=0.8pt]plot[smooth, tension=.7] coordinates {(-15,3)(-8,1)(0,0)(8,1)(15,3)};
\end{scope}
\begin{scope}[shift={(5,-3)}]
\draw[densely dotted,line width=0.8pt]plot[smooth, tension=.7] coordinates {(-10,-10)(-4,-6)(0,0) (4,6) (10,10)};
\end{scope}
\draw[color=red, line width=1.5pt]plot[smooth,tension=.7] coordinates {(10,1.5)(9,3.2)(5.8,7.4)};
\draw[color=red, line width=1.5pt]plot[smooth,tension=.7] coordinates {(5.8,7.4)(0,7)(-4.5,7.4)};
\draw[color=red, line width=1.5pt]plot[smooth,tension=.7] coordinates {(-4.5,7.4)(-5.6,5.5)(-8.3,1.1)};
\draw[color=red, line width=1.5pt]plot[smooth,tension=.7] coordinates {(-8.3,1.1)(-4.7,-4)(-3.6,-5.8)};
\draw[color=red, line width=1.5pt]plot[smooth,tension=.7] coordinates {(-3.6,-5.8)(0,-6)(3.5,-5.8)};
\draw[color=red, line width=1.5pt]plot[smooth,tension=.7] coordinates {(3.5,-5.8)(5,-3)(7.2,0.8)};
\draw[color=black,fill=black!100] (0,0) circle (0.6cm);
\draw[color=red,fill=red!100] (10,1.5) circle (0.6cm);
\draw[color=red,fill=red!100] (5.8,7.4) circle (0.6cm);
\draw[color=red,fill=red!100] (-4.5,7.4) circle (0.6cm);
\draw[color=red,fill=red!100] (-8.3,1.1) circle (0.6cm);
\draw[color=red,fill=red!100] (-3.6,-5.8) circle (0.6cm);
\draw[color=red,fill=red!100] (3.5,-5.8) circle (0.6cm);
\draw[color=red,fill=red!100] (7.2,0.8) circle (0.6cm);
\draw(14.8,3.4)node[right]{$L_{1}$};
\draw(9.5,11.1)node[right]{$L_{2}$};
\draw(-6.5,13.1)node[right]{$L_{3}$};
\draw(-1.6,1.9)node[right]{$0$};
\draw(9.1,3)node[right]{$p$};
\draw(5.6,2.5)node[right]{$q$};
\end{tikzpicture}
\end{center}

\begin{defin}
Le $3$-tissu $\mathcal{W}$ est dit \textbf{\textit{hexagonal}} si tout hexagone tracé en partant de suffisamment près de l'origine se referme ({\it i.e.} si le germe $H_{1}$ défini ci-dessus est l'identité). Plus généralement un $k$-tissu est dit \textbf{\textit{hexagonal}} si tous ses sous-$3$-tissus le sont.
\end{defin}
\begin{eg}
Le $3$-tissu $\mathcal{W}(x,y,x+y)$ est hexagonal comme le montre la figure qui suit.
\vspace{0.35cm}

\begin{center}
\begin{tikzpicture}[x=1.1cm,y=1.cm,scale=0.3]
\draw[line width=1.pt](6,-7) -- (6,6);
\draw[line width=1.pt](6,6) -- (-6,6);
\draw[line width=1.pt](-6,6) -- (-6,-7);
\draw[line width=1.pt](-6,-7) -- (6,-7);
\begin{scope}[shift={(-1,0)}]
\draw[line width=0.6pt](6,-7) -- (6,6);
\end{scope}
\begin{scope}[shift={(-2,0)}]
\draw[line width=0.6pt](6,-7) -- (6,6);
\end{scope}
\begin{scope}[shift={(-3,0)}]
\draw[line width=0.6pt](6,-7) -- (6,6);
\end{scope}
\begin{scope}[shift={(-4,0)}]
\draw[line width=0.6pt](6,-7) -- (6,6);
\end{scope}
\begin{scope}[shift={(-5,0)}]
\draw[line width=0.6pt](6,-7) -- (6,6);
\end{scope}
\begin{scope}[shift={(-6,0)}]
\draw[line width=0.6pt](6,-7) -- (6,6);
\end{scope}
\begin{scope}[shift={(-7,0)}]
\draw[line width=0.6pt](6,-7) -- (6,6);
\end{scope}
\begin{scope}[shift={(-8,0)}]
\draw[line width=0.6pt](6,-7) -- (6,6);
\end{scope}
\begin{scope}[shift={(-9,0)}]
\draw[line width=0.6pt](6,-7) -- (6,6);
\end{scope}
\begin{scope}[shift={(-10,0)}]
\draw[line width=0.6pt](6,-7) -- (6,6);
\end{scope}
\begin{scope}[shift={(-11,0)}]
\draw[line width=0.6pt](6,-7) -- (6,6);
\end{scope}
\begin{scope}[shift={(0,-1)}]
\draw[line width=0.6pt](6,6) -- (-6,6);
\end{scope}
\begin{scope}[shift={(0,-2)}]
\draw[line width=0.6pt](6,6) -- (-6,6);
\end{scope}
\begin{scope}[shift={(0,-3)}]
\draw[line width=0.6pt](6,6) -- (-6,6);
\end{scope}
\begin{scope}[shift={(0,-4)}]
\draw[line width=0.6pt](6,6) -- (-6,6);
\end{scope}
\begin{scope}[shift={(0,-5)}]
\draw[line width=0.6pt](6,6) -- (-6,6);
\end{scope}
\begin{scope}[shift={(0,-6)}]
\draw[line width=0.6pt](6,6) -- (-6,6);
\end{scope}
\begin{scope}[shift={(0,-7)}]
\draw[line width=0.6pt](6,6) -- (-6,6);
\end{scope}
\begin{scope}[shift={(0,-8)}]
\draw[line width=0.6pt](6,6) -- (-6,6);
\end{scope}
\begin{scope}[shift={(0,-9)}]
\draw[line width=0.6pt](6,6) -- (-6,6);
\end{scope}
\begin{scope}[shift={(0,-10)}]
\draw[line width=0.6pt](6,6) -- (-6,6);
\end{scope}
\begin{scope}[shift={(0,-11)}]
\draw[line width=0.6pt](6,6) -- (-6,6);
\end{scope}
\begin{scope}[shift={(0,-12)}]
\draw[line width=0.6pt](6,6) -- (-6,6);
\end{scope}
\draw[line width=0.6pt](5,6) -- (6,5);
\draw[line width=0.6pt](3,6) -- (6,3);
\draw[line width=0.6pt](1,6) -- (6,1);
\draw[line width=0.6pt](-1,6) -- (6,-1);
\draw[line width=0.6pt](-3,6) -- (6,-3);
\draw[line width=0.6pt](-5,6) -- (6,-5);
\draw[line width=0.6pt](-4,-7) -- (-6,-5);
\draw[line width=0.6pt](-2,-7) -- (-6,-3);
\draw[line width=0.6pt](0,-7) -- (-6,-1);
\draw[line width=0.6pt](2,-7) -- (-6,1);
\draw[line width=0.6pt](4,-7) -- (-6,3);
\draw[line width=0.6pt](6,-7) -- (-6,5);
\draw[color=red,fill=red!100] (0,0) circle (0.35cm);
\draw[color=red, line width=1.75pt]plot[smooth,tension=0.] coordinates { (1,-1) (1,0) (0,1) (-1,1) (-1,0) (0,-1) (1,-1)};
\draw[color=red, line width=1.75pt]plot[smooth,tension=0.] coordinates { (3,-3) (3,0) (0,3) (-3,3) (-3,0) (0,-3) (3,-3)};
\draw[color=red, line width=1.75pt]plot[smooth,tension=0.] coordinates { (5,-5) (5,0) (0,5) (-5,5) (-5,0) (0,-5) (5,-5)};
\draw(-7,-9.7)node[right]{$\mathcal{W}(x,y,x+y)$ est hexagonal.};
\end{tikzpicture}
\end{center}
\end{eg}
\vspace{5mm}

\noindent Le théorème classique suivant caractérise les $3$-tissus réguliers de $(\mathbb{C}^{2},0)$ (\emph{cf.} \cite{PP09}).

\begin{thm}[Structure des $3$-tissus parallélisables]\label{th:stru-3-tiss-parall}
{\sl Soit $\mathcal{W}=\mathcal{F}_{1}\boxtimes\mathcal{F}_{2}\boxtimes\mathcal{F}_{3}$ un $3$-tissu régulier de $(\mathbb{C}^{2},0).$ Les assertions suivantes sont équivalentes:
\begin{enumerate}
\item [\textit{(1)}] Le tissu $\mathcal{W}$ est parallélisable;
\item [\textit{(2)}] Le tissu $\mathcal{W}$ est hexagonal;
\item [\textit{(3)}] Le tissu $\mathcal{W}$ est de courbure de \textsc{Blaschke} nulle, {\it i.e.} $K(\mathcal{W})=0$;
\item [\textit{(4)}] Pour $i=1,2,3,$ il existe une $1$-forme fermée $\eta_{i}$ définissant $\mathcal{F}_{i}$ telle que $\eta_{1}+\eta_{2}+\eta_{3}=0.$
\end{enumerate}}
\end{thm}

\noindent Deux tissus topologiquement équivalents ne peuvent être hexagonaux qu'en même temps. L'hexagonalité est donc une condition de fermeture de nature topologique. Le théorème précédent montre qu'elle est équivalente à l'annulation de la courbure de \textsc{Blaschke}, qui est une condition de nature différentielle. C'est ce fait remarquable qui a amené \textsc{Blaschke} à désigner sous l'appellation ``Topologie Fragen der Differentialgeometrie'' (``Questions topologiques de géométrie différentielle'') l'étude de ces questions de fermeture pour les tissus et de leurs liens avec la classification analytique des tissus. Par extension, c'est sous cette terminologie que furent publiés dans les années trente les travaux de l'école hambourgeoise sur les tissus.

\newpage

\section{Tissus globaux sur une surface complexe}\vspace{6mm}

\Subsection{Définitions et exemples}\label{subsec:tissu-global}\vspace{3mm}

\begin{defin}\label{def:germe-tissu-singulier}
Soit $k\geq1$ un entier. Un \textbf{\textit{germe de $k$-tissu singulier}} $\mathcal{W}$ de $(\mathbb{C}^{2},0)$ est défini par un germe de $k$-forme symétrique $\omega\in\mathrm{Sym}^{k}\Omega^{1}(\mathbb{C}^{2},0)$ satisfaisant les deux conditions:
\begin{enumerate}
\item [\textit{(1)}] le lieu singulier $\mathrm{Sing}\omega$ de $\omega$ est formé de points isolés;
\item [\textit{(2)}] en tout point générique $p\in(\mathbb{C}^{2},0),$ $\omega(p)$ se factorise en produit de $k$ formes linéaires deux à deux non colinéaires.
\end{enumerate}
Ajoutons que par définition deux telles $k$-formes symétriques $\omega$ et $\omega'$ définissent le même tissu si elles sont égales à unité multiplicative près: $\omega'=u\omega$ avec $u\in\mathcal{O}^*(\mathbb{C}^{2},0).$
\end{defin}

\begin{defin}
Soient $S$ une surface complexe et $k\geq1$ un entier. Un \textbf{\textit{$k$-tissu global}} $\mathcal{W}$ sur $S$ est la donnée d'un recouvrement ouvert $(U_{i})_{i\in I}$ de $S$ et d'une collection de $k$-formes symétriques $\omega_{i}\in \mathrm{Sym}^{k}\Omega^{1}_{S}(U_{i})$ tels que les propriétés suivantes soient satisfaites:
\begin{enumerate}
\item [($\mathfrak{a}$)] sur chaque intersection non triviale $U_{i}\cap U_{j}$ on a $\omega_{i}=g_{ij}\omega_{j}$ où les $g_{ij}\in \mathcal{O}^{*}_{S}(U_{i}\cap U_{j})$ sont des unités holomorphes;
\item [($\mathfrak{b}$)] le lieu singulier $\mathrm{Sing}\omega_{i}$ de $\omega_{i}$ est formé de points isolés;
\item [($\mathfrak{c}$)] en tout point générique $p$ de $U_{i},$ $\omega_{i}(p)$ se factorise en produit de $k$ formes linéaires deux à deux non colinéaires.
\end{enumerate}
L'ensemble des points de $S$ qui ne vérifient pas la propriété ($\mathfrak{c}$) est appelé le \textbf{\textit{discriminant}} de $\mathcal{W}$ et est noté $\Delta(\mathcal{W}).$ Le \textbf{\textit{lieu singulier}} $\mathrm{Sing}\mathcal{W}$ de $\mathcal{W}$ est défini par $\mathrm{Sing}\mathcal{W}\cap\,U_{i}=\mathrm{Sing}\omega_i$ et il est contenu dans $\Delta(\mathcal{W}).$
\end{defin}
\begin{rems}\label{rems:tissu-global}
\begin{itemize}
\item []\hspace{-0.8cm}(i) Lorsque $k=1$ on retrouve la définition habituelle d'un feuilletage holomorphe singulier $\mathcal{F}$ (Définition \ref{def:feuilletage}); dans ce cas $\Delta(\mathcal{F})=\mathrm{Sing}\mathcal{F}.$
\smallskip
  \item [(ii)] Les fonctions $g_{ij}$ forment un $1$-cocycle à valeurs dans $\mathcal{O}^{*}_{S}$ et déterminent un fibré en droites $L$ sur $S,$ appelé le \textbf{\textit{fibré normal}} de $\mathcal{W}.$ La collection $\{\omega_{i}\}$ définit une section de $\mathrm{Sym}^{k}\Omega^{1}_{S}\otimes L,$ \textit{i.e.} $\omega=\{\omega_{i}\}$ peut être vue comme un élément de $\mathrm{H}^0(S,\mathrm{Sym}^{k}\Omega^{1}_{S}\otimes L).$ Deux sections globales $\omega,\omega'\in\mathrm{H}^0(S,\mathrm{Sym}^{k}\Omega^{1}_{S}\otimes L)$ décrivent le même tissu si et seulement si elles sont égales à multiplication près par un élément $g\in\mathrm{H}^0(S,\mathcal{O}^{*}_{S}).$
  \smallskip
  \item [(iii)] Si $S$ est compacte, se donner un $k$-tissu global sur $S$ revient à se donner un élément $\omega$ de $\mathbb{P}\mathrm{H}^0(S,\mathrm{Sym}^{k}\Omega^{1}_{S}\otimes L)$ pour un certain fibré en droites $L\to S,$ tel qu'en tout point générique $p\in S$ le germe de $\omega$ en $p$ définisse un germe de $k$-tissu singulier, \textit{i.e.} que $\omega_{,p}$ vérifie les conditions \textit{(1)} et \textit{(2)} de la Définition \ref{def:germe-tissu-singulier}.
\end{itemize}
\end{rems}

\noindent Soit $\mathcal{W}$ un $k$-tissu global sur une surface complexe $S.$ Nous dirons que $\mathcal{W}$ est \textbf{\textit{décomposable}} s'il existe des tissus globaux $\mathcal{W}_{1},\mathcal{W}_{2}$ sur $S$ n'ayant pas de sous-tissus communs tels que $\mathcal{W}$ soit la superposition de $\mathcal{W}_{1}$ et $\mathcal{W}_{2},$ \textit{i.e.} que $\mathcal{W}=\mathcal{W}_{1}\boxtimes\mathcal{W}_{2}.$ Le $k$-tissu $\mathcal{W}$ sera dit \textbf{\textit{complètement décomposable}} si on peut l'écrire sous la forme $\mathcal{W}=\mathcal{F}_{1}\boxtimes\cdots\boxtimes\mathcal{F}_{k}$ pour certains feuilletages globaux $\mathcal{F}_{1},\ldots,\mathcal{F}_{k}$ sur $S.$ Notons que la restriction de $\mathcal{W}$ à un voisinage suffisamment petit de tout point générique $p\in S$ est complètement décomposable.

\noindent Soit $\Gamma\subset\Delta(\mathcal{W})$ une composante irréductible du discriminant de $\mathcal{W}.$ Nous dirons que $\Gamma$ est \textbf{\textit{invariante}} (resp. \textbf{\textit{totalement invariante}}) par $\mathcal{W}$ si, sur la partie régulière de $\Gamma,$ on a $\mathrm{T}\hspace{0.2mm}\Gamma\subset\mathrm{T}\mathcal{W}|_{\Gamma}$ (resp. $\mathrm{T}\hspace{0.2mm}\Gamma=\mathrm{T}\mathcal{W}|_{\Gamma}$). Lorsque $\mathcal{W}$ est un germe de tissu irréductible, ces deux notions coïncident.
\vspace{-4mm}

\begin{eg}[Tissus algébriques]
Soit $\mathcal{C}$ une courbe algébrique réduite de degré $k$ sur $\mathbb{P}^{2}_{\mathbb{C}},$ non nécessairement irréductible et éventuellement singulière. \`{A} cette courbe on va associer un $k$-tissu linéaire sur le dual $\mathbb{\check{P}}^{2}_{\mathbb{C}},$ linéaire au sens où ses feuilles sont des morceaux de droites. Cette construction est centrale et très classique en géométrie des tissus (\emph{voir} \cite[Chapitre 23]{BB38}).

\noindent On procède localement, en commençant par fixer une droite $\ell_{0}$ qui intersecte $\mathcal{C}_{\text{reg}}$ transversalement en $k$ points distincts. Sur la courbe le $0$-cycle $\mathcal{C}.\ell_{0}$ s'écrit donc \begin{small}$\mathcal{C}.\ell_{0}=p_{1}(\ell_{0})+\cdots+p_{k}(\ell_{0}),$\end{small} où les $p_{i}(\ell_{0})$ sont $k$ points distincts. Cette condition d'intersecter transversalement la partie régulière de $\mathcal{C}$ en $k$ points distincts est ouverte. Il va donc exister un voisinage $V_{0}$ de $\ell_{0}$ dans $\mathbb{\check{P}}^{2}_{\mathbb{C}}$ tel que cette condition de transversalité soit vérifiée dessus. Quitte à prendre ce voisinage simplement connexe et \og suffisamment petit \fg, on va pouvoir suivre de façon holomorphe les points d'intersection de $\ell$ avec $\mathcal{C}$ lorsque $\ell$ varie dans $V_{0}.$ Plus précisément, il existe $k$ applications holomorphes $p_{i}\hspace{1mm}\colon V_{0}\to \mathcal{C}_{\text{reg}},$ telles que, pour toute droite $\ell\in V_{0},$ on ait en tant que $0$-cycle: $$\mathcal{C}.\ell=p_{1}(\ell)+\cdots+p_{k}(\ell).$$

\noindent Pour $i=1,\ldots,k,$ notons alors $\mathcal{F}_{i}$ le feuilletage holomorphe de $V_{0}$ dont les feuilles sont les lignes de niveaux \{$p_{i}=$cte\}. Pour $\ell\in V_{0},$ la feuille de $\mathcal{F}_{i}$ passant par $\ell$ est le segment de droite $$\{\ell\in V_{0}\hspace{1mm}\colon p_{i}(\ell)\in\ell\}.$$

\noindent Ainsi \{$\mathcal{F}_{1},\ldots,\mathcal{F}_{k}$\} est un $k$-tissu linéaire sur le voisinage $V_{0}$ de $\ell_{0}$. Cette construction peut se
faire au voisinage de toute droite $\ell_{0}\in\mathbb{\check{P}}^{2}_{\mathbb{C}}\setminus\hspace{1mm}\mathcal{\check{\hspace{-1mm}C}}$ où $\hspace{1mm}\mathcal{\check{\hspace{-1mm}C}}$ désigne la courbe duale de $\mathcal{C}$ (c'est-à-dire la courbe de $\mathbb{\check{P}}^{2}_{\mathbb{C}}$ formée des droites qui sont tangentes à $\mathcal{C}$). Il est clair que les tissus ainsi obtenus se recollent pour former un $k$-tissu linéaire singulier sur $\mathbb{\check{P}}^{2}_{\mathbb{C}}.$ Ce $k$-tissu s'appelle le \textbf{\textit{tissu algébrique}} associé à la courbe $\mathcal{C}$; on le note $\mathcal{W}_{\mathcal{C}}.$

\noindent On peut montrer que le lieu singulier de $\mathcal{W}_{\mathcal{C}}$ est exactement la courbe duale $\hspace{1mm}\mathcal{\check{\hspace{-1mm}C}}.$

\noindent Il est possible de définir $\mathcal{W}_{\mathcal{C}}$ d'autres façons. Par exemple, dans le cas où aucune des composantes irréductibles de $\mathcal{C}$ n'est une droite, le fait que $\mathcal{C}$ est de degré $k$ équivaut à celui que sa courbe duale $\hspace{1mm}\mathcal{\check{\hspace{-1mm}C}}$ est de classe $k$, c'est-à-dire que par un point générique du dual $\mathbb{\check{P}}^{2}_{\mathbb{C}}$ passent $k$ tangentes distinctes à $\hspace{1mm}\mathcal{\check{\hspace{-1mm}C}}$. Ce sont exactement les feuilles de $\mathcal{W}_{\mathcal{C}}$ passant par ce point. Ainsi $\mathcal{W}_{\mathcal{C}}$ est le tissu sur $\mathbb{\check{P}}^{2}_{\mathbb{C}}$ formé des tangentes à la courbe de classe $k$ qu'est $\hspace{1mm}\mathcal{\check{\hspace{-1mm}C}}.$

\begin{center}
\begin{tikzpicture}[y=0.45pt,x=0.45pt,yscale=-1]
\path[fill=black] (-0.37583864,12.994694) node[above right] (text2991) {};
\path[draw=black,line join=round,line cap=round,line width=1.pt]
(311.4763,0.3943  ) .. controls (319.9890,15.5486 ) and (339.6642,50.5892 ) ..
(319.9326,58.8517 ) .. controls (300.2011,67.1142 ) and (254.7810,25.5011 ) ..
(223.4737,35.3786 ) .. controls (192.1475,45.2749 ) and (173.4119,82.6440 ) ..
(179.4254,133.7778) .. controls (185.4388,184.9303) and (208.4401,215.7833) ..
(238.6200,218.8442) .. controls (268.8186,221.9238) and (302.6441,176.9682) ..
(329.3286,175.0340) .. controls (356.0132,173.0998) and (372.0427,180.7239) ..
(398.8024,207.3330);
\path[draw=blue,line join=round,line cap=round,line width=0.pt]
(21.7611,154.6218) -- (267,5.0890);
\path[draw=blue,line join=round,line cap=round,line width=0.pt]
(10.3731,140.4817) -- (359.50,50.3451);
\path[draw=blue,line join=round,line cap=round,line width=0.pt]
(0.4134,110.6239 ) -- (401.4,183.7285);
\path[draw=blue,line join=round,line cap=round,line width=0.pt]
(6.8591,84.8317  ) -- (275.5,248.7677);
\node at (5.45cm,0.28cm) {$\mathcal{C}$};
\fill [color=black] (1.15cm,1.96cm) circle (2.5pt) ;
\node at (1.15cm,2.36cm) {\large$p$};
\draw(-80,290)node[right]{Les feuilles passant par $p$ du tissu associé à une courbe $\mathcal{C}$ de classe $4.$};
\end{tikzpicture}
\end{center}

\noindent Cette dernière construction permet d'obtenir de jolis dessins réels de tissus algébriques (\emph{voir} par exemple \cite[Pag. 29--30]{BB38}).
\end{eg}
\begin{eg}[Tissu associé à une équation différentielle implicite $F(x,y,y')=0$]
On considère une équation différentielle du premier ordre de la forme suivante:
\begin{equation}\label{tissu-equa-diff}
F(x,y,y'):=a_{0}(x,y)\cdot(y')^{k}+a_{1}(x,y)\cdot(y')^{k-1}+\cdots+a_{k}(x,y)=0,
\end{equation}
où les $a_{i},i=0,\ldots,k,$ sont des fonctions holomorphes sur un ouvert $U$ de $\mathbb{C}^{2}.$

\noindent On suppose que $F(x,y,p)\in\mathcal{O}(U)[p]$ est sans facteurs multiples et que $a_{0}\neq0\in\mathcal{O}(U).$ On note $R\in\mathcal{O}(U)$ le $p$-résultant de $F,$ \textit{i.e.}
\begin{align*}
& R:=\text{Result}(F,\partial_{p}(F))=(-1)^{\frac{k(k-1)}{2}}\cdot a_{0}\cdot \Delta
\end{align*}
où $\Delta$ est son $p$-discriminant.

\noindent En vertu du théorème de \textsc{Cauchy}, les $k$ courbes intégrales d'une équation différentielle de la forme \eqref{tissu-equa-diff} définissent un $k$-tissu régulier en dehors du lieu singulier défini par $\{R=0\}.$ En effet, si $m$ est un point de $U\setminus\{R=0\},$ il existe un voisinage $V_{m}$ de $m$ dans $U$ et $k$ applications holomorphes $(x,y)\mapsto p_{i}(x,y)$ telles que, sur $V_{m},$ on ait
\begin{align*}
& F(x,y,p)=a_{0}(x,y)\prod_{i=1}^{k}(p-p_{i}(x,y)) \qquad \text{avec} \qquad p_{i}(m)\neq p_{j}(m) \hspace{2mm}\text{si}\hspace{2mm}i\neq j.
\end{align*}
\noindent Au voisinage de $m,$ l'équation $F(x,y,y')=0$ admet $k$ solutions $(x,y)\mapsto y_{i}(x,y)$ qui sont des intégrales premières des feuilletages définis par les champs de vecteurs $\X_{i}=\partial_{x}+p_{i}\partial_{y}.$ On peut donc considérer le tissu $\mathcal{W}(y_{1},\ldots,y_{k})$ au voisinage de $m.$

\noindent La construction locale que l'on vient d'esquisser peut se faire au voisinage de tout point $m$ générique. Les tissus ainsi obtenus se recollent pour former un $k$-tissu global (singulier) sur $U.$ On peut ainsi entreprendre l'étude géométrique de l'équation $F(x,y,y')=0$ par celle du tissu qui lui est associé. Ce point de vue se généralise au système d'EDP du premier ordre (\emph{cf.} l'article \cite{Nak01} de I. \textsc{Nakai}, qui a par ailleurs obtenu certains résultats avec cette approche).

\noindent Il est clair que tout germe de $k$-tissu est susceptible d'être défini au moyen d'une équation différentielle du premier ordre et de degré $k$ en $y'.$ Cette \og définition implicite \fg a l'avantage de ne privilégier aucun des feuilletages qui composent un tissu. Ce point de vue dans l'étude des tissus a été développé par A. \textsc{Hénaut} \cite{Hen00,Hen06}.
\end{eg}

\Subsection{Barycentre d'un tissu}\vspace{3mm}

Soient $\mathcal{W}$ un $k$-tissu sur une surface complexe $S$ et $\mathcal{F}$ un feuilletage sur $S$ génériquement transverse à $\mathcal{W}.$ Nous allons construire un nouveau feuilletage sur $S$ noté $\beta_{\mathcal{F}}(\mathcal{W})$ et appelé \textbf{\textit{barycentre de}} $\mathcal{W}$ \textbf{\textit{par rapport à}} $\mathcal{F}.$

\noindent Soit $p$ un point générique de $S$; alors les directions tangentes $L_{1},\ldots,L_{k}$ des feuilles de $\mathcal{W}$ passant par $p$ sont toutes distinctes de $\mathrm{T}_{\hspace{-0.4mm}p}\mathcal{F}.$ Elles peuvent donc être considérées comme $k$ points dans $\mathbb{P}\mathrm{T}_{\hspace{-0.4mm}p}S\setminus \mathrm{T}_{\hspace{-0.4mm}p}\mathcal{F}$ qui admet une structure affine canonique. On peut donc considérer leur barycentre dans cette droite affine, que l'on note $\beta_{\mathrm{T}_{\hspace{-0.4mm}p}\mathcal{F}}(L_{1},\ldots,L_{k}).$ En faisant varier $p,$ on obtient ainsi un champ de droites sur $S.$ Celui-ci s'intègre alors, pour des raisons de dimension, en un feuilletage qui se prolonge en un feuilletage holomorphe singulier sur $S$ tout entière: par définition, c'est le feuilletage $\beta_{\mathcal{F}}(\mathcal{W}).$

\noindent On peut élargir cette définition du barycentre en remplaçant le feuilletage $\mathcal{F}$ par un $n$-tissu $\mathcal{W}'.$ L'idée est simple: si on écrit $\mathcal{W}'=\mathcal{F}_{1}\boxtimes\cdots\boxtimes\mathcal{F}_{n},$ on définit le $\mathcal{W}'$\textbf{\textit{-barycentre de}} $\mathcal{W}$ comme étant $$\beta_{\mathcal{W}'}(\mathcal{W})=\beta_{\mathcal{F}_{1}}(\mathcal{W})\boxtimes\cdots\boxtimes\beta_{\mathcal{F}_{n}}(\mathcal{W}).$$

\Subsection{Tissus sur $\mathbb{P}^{2}_{\mathbb{C}}$ et transformation de \textsc{Legendre}}\label{subsec:Legendre}\vspace{3mm}

Nous concentrerons maintenant notre attention sur les $k$-tissus globaux sur le plan projectif complexe. D'après la remarque (iii) du paragraphe \S\ref{subsec:tissu-global}, un $k$-tissu $\mathcal{W}$ sur $\mathbb{P}^{2}_{\mathbb{C}}$ est défini par un élément $\omega$ de $\mathbb{P}\mathrm{H}^0(\mathbb{P}^{2}_{\mathbb{C}},\mathrm{Sym}^{k}\Omega^{1}_{\mathbb{P}^{2}_{\mathbb{C}}}\otimes L)$ pour un certain fibré en droite $L\to\mathbb{P}^{2}_{\mathbb{C}},$ tel qu'en tout point générique $p\in \mathbb{P}^{2}_{\mathbb{C}}$ le germe $\omega_{,p}$ de $\omega$ en $p$ détermine un germe de $k$-tissu singulier au sens de la Définition \ref{def:germe-tissu-singulier}. Il est naturel d'écrire $L$ comme étant $\mathcal{O}_{\mathbb{P}^{2}_{\mathbb{C}}}(d+2k)$ vu que l'image réciproque de $\omega$ par une droite $\ell\subset \mathbb{P}^{2}_{\mathbb{C}}$ est une section de $\mathrm{Sym}^{k}\Omega^{1}_{\mathbb{P}^{1}_{\mathbb{C}}}(d+2k)=\mathcal{O}_{\mathbb{P}^{1}_{\mathbb{C}}}(d)$ et par conséquent pour $\ell$ générique l'entier $d,$ appelé le \textbf{\textit{degré}} du tissu, représente le nombre de tangences entre $\ell$ et le $k$-tissu $\mathcal{W}$; on le notera $\deg\mathcal{W}.$ L'espace des $k$-tissus sur $\mathbb{P}^{2}_{\mathbb{C}}$ de degré $d$ est alors un ouvert de $\mathbb{P}H^{0}(\mathbb{P}^{2}_{\mathbb{C}},\mathrm{Sym}^{k}\Omega^{1}_{\mathbb{P}^{2}_{\mathbb{C}}}(d+2k))$, qui sera noté $\mathbf{T}(k;d)$.

\noindent D'autre part, de la suite d'\textsc{Euler}
\[
0\longrightarrow \mathcal{O}_{\mathbb{P}^{2}_{\mathbb{C}}}\longrightarrow \mathcal{O}_{\mathbb{P}^{2}_{\mathbb{C}}}(1)^{\oplus3}\longrightarrow \mathrm{T}\mathbb{P}^{2}_{\mathbb{C}} \longrightarrow 0,
\]
\noindent on déduit la suite exacte suivante
\[
0\longrightarrow \mathrm{Sym}^{k-1}(\mathcal{O}_{\mathbb{P}^{2}_{\mathbb{C}}}(1)^{\oplus3})\otimes \mathcal{O}_{\mathbb{P}^{2}_{\mathbb{C}}} \longrightarrow \mathrm{Sym}^{k}(\mathcal{O}_{\mathbb{P}^{2}_{\mathbb{C}}}(1)^{\oplus3})\longrightarrow \mathrm{Sym}^{k}\mathrm{T}\mathbb{P}^{2}_{\mathbb{C}}\longrightarrow 0.
\]
\noindent Cela signifie qu'un $k$-tissu de degré $d$ sur $\pp$ peut se décrire par un polynôme bihomogène $P(x,y,z;a,b,c)$ de degré $d$ en $(x,y,z)$ et de degré $k$ en $(a,b,c).$ Plus précisément:
\begin{enumerate}
  \item [(i)] \begin{small}$\X=P(x,y,z;\frac{\partial{}}{\partial{x}},\frac{\partial{}}{\partial{y}},\frac{\partial{}}{\partial{z}})$\end{small} définit une section globale de \begin{small}$\mathrm{Sym}^{k}\mathrm{T}\mathbb{P}^{2}_{\mathbb{C}}(d-k)$;\end{small}
  \item [(ii)] \begin{small}$\omega=P(x,y,z;y\mathrm{d}z-z\mathrm{d}y,z\mathrm{d}x-x\mathrm{d}z,x\mathrm{d}y-y\mathrm{d}x)$\end{small} définit une section globale de \begin{small}$\mathrm{Sym}^{k}\Omega^{1}_{\mathbb{P}^{2}_{\mathbb{C}}}(d+2k).$\end{small}
\end{enumerate}

\noindent Deux tels polynômes $P(x,y,z;a,b,c)$ et $P'(x,y,z;a,b,c)$ définissent le même tissu si et seulement s'il existe $\lambda,\lambda'\in\mathbb{C}^{*}$ et un polyn\^ome bihomog\`ene $Q(x,y,z;a,b,c)$ de bidegr\'e $(d-1,k-1)$ tels que $$\lambda P(x,y,z;a,b,c)-\lambda' P'(x,y,z;a,b,c)=(ax+by+cz)Q(x,y,z;a,b,c).$$

\noindent En utilisant des coordonnées homogènes $(a:b:c)$ dans le plan projectif dual $\mathbb{\check{P}}^{2}_{\mathbb{C}}$ qui correspondent à la droite $\{ax+by+cz=0\}\subset{\mathbb{P}^{2}_{\mathbb{C}}}$, on peut écrire
\begin{align*}
\mathrm{T}^*_{(x:y:z)}\mathbb{P}^{2}_{\mathbb{C}}&=\{\omega=a\mathrm{d}x+b\mathrm{d}y+c\mathrm{d}z\in \mathrm{T}^*\mathbb{C}^{3}\hspace{1mm}\colon i_{R}\omega=0\} \\
&=\{a\mathrm{d}x+b\mathrm{d}y+c\mathrm{d}z\hspace{1mm}\colon ax+by+cz=0\}
\end{align*}
et la variété de contact $\mathbb{P}\mathrm{T}^{*}\mathbb{P}^{2}_{\mathbb{C}}$ s'identifie de façon naturelle à la variété d'incidence
\[
\mathcal{I}=\{((x:y:z),(a:b:c))\hspace{1mm}\vert\hspace{1mm} ax+by+cz=0\}\subset\mathbb{P}^{2}_{\mathbb{C}}\times\mathbb{\check{P}}^{2}_{\mathbb{C}}.
\]

\noindent Soient $\mathcal{W}$ un $k$-tissu de degré $d$ sur $\mathbb{P}^{2}_{\mathbb{C}}$ et \begin{small}$P(x,y,z;a,b,c)$\end{small} un polynôme bihomogène définissant $\mathcal{W}$. Alors, avec l'identification ci-dessus de $\mathbb{P}\mathrm{T}^{*}\mathbb{P}^{2}_{\mathbb{C}}$ à $\mathcal{I},$ le graphe de $\mathcal{W}$ dans $\mathbb{P}\mathrm{T}^{*}\mathbb{P}^{2}_{\mathbb{C}}$ est la surface
\[
S_{\mathcal{W}}=\{((x:y:z),(a:b:c))\in\mathbb{P}^{2}_{\mathbb{C}}\times\mathbb{\check{P}}^{2}_{\mathbb{C}}\hspace{1mm}\vert\hspace{1mm} ax+by+cz=0,P(x,y,z;a,b,c)=0\}\subset \mathcal{I}.
\]

\noindent Supposons que le tissu $\mathcal{W}$ soit irréductible de degré $d>0$ et considérons les restrictions $\pi$ et $\check{\pi}$ à $S_{\mathcal{W}}$ des projections naturelles de $\mathbb{P}^{2}_{\mathbb{C}}\times\mathbb{\check{P}}^{2}_{\mathbb{C}}$ sur $\mathbb{P}^{2}_{\mathbb{C}}$ et $\mathbb{\check{P}}^{2}_{\mathbb{C}}$ respectivement. Ces projections $\pi$ et $\check{\pi}$ sont des applications rationnelles de degrés $k$ et $d$ respectivement. La distribution de contact sur $\mathbb{P}\mathrm{T}^{*}\mathbb{P}^{2}_{\mathbb{C}}$ s'identifie à
\[
\mathcal{D}=\ker(a\mathrm{d}x+b\mathrm{d}y+c\mathrm{d}z)=\ker(x\mathrm{d}a+y\mathrm{d}b+z\mathrm{d}c).
\]
\noindent Le feuilletage $\mathcal{F}_{\mathcal{W}}$ induit par $\mathcal{D}$ sur $S_{\mathcal{W}}$ se projette par $\pi$ sur le $k$-tissu $\mathcal{W}$ et se projette par $\check{\pi}$ en un $d$-tissu $\hspace{1mm}\mathcal{\check{\hspace{-1mm}W}}$ sur $\mathbb{\check{P}}^{2}_{\mathbb{C}}.$ Le $d$-tissu $\hspace{1mm}\mathcal{\check{\hspace{-1mm}W}}$ s'appelle la \textbf{\textit{transformée de \textsc{Legendre}}} de $\mathcal{W}$ et se note $\mathrm{Leg}\mathcal{W}$.

\noindent Si $\mathcal{W}$ est décrit par \begin{Small}$P(x,y,z;\frac{\partial{}}{\partial{x}},\frac{\partial{}}{\partial{y}},\frac{\partial{}}{\partial{z}}),$\end{Small} ou respectivement par \begin{Small}$P(x,y,z;y\mathrm{d}z-z\mathrm{d}y,z\mathrm{d}x-x\mathrm{d}z,x\mathrm{d}y-y\mathrm{d}x),$\end{Small} alors sa transformée de \textsc{Legendre} $\mathrm{Leg}\mathcal{W}$ est décrite par \begin{Small}$P(\frac{\partial{}}{\partial{a}},\frac{\partial{}}{\partial{b}},\frac{\partial{}}{\partial{c}};a,b,c),$\end{Small} ou respectivement par \begin{Small}$P(b\mathrm{d}c-c\mathrm{d}b,c\mathrm{d}a-a\mathrm{d}c,a\mathrm{d}b-b\mathrm{d}a;a,b,c).$\end{Small}

\noindent En utilisant ces formules, on peut procéder pour définir la transformée de \textsc{Legendre} d'un $k$-tissu arbitraire de degré $d$ arbitraire. Notons que si $\mathcal{W}$ se décompose en produit de deux tissus $\mathcal{W}_{1}\boxtimes \mathcal{W}_{2}$, alors sa transformée de \textsc{Legendre} est le produit de $\mathrm{Leg}\mathcal{W}_{1}$ par $\mathrm{Leg}\mathcal{W}_{2}.$
\vspace{2mm}

\noindent En consultant des ouvrages classiques sur les équations différentielles, on peut trouver la transformation de \textsc{Legendre} comme une transformation involutive entre les équations différentielles polynomiales (\emph{voir} par exemple \cite{Inc44}).

\begin{rems}\label{rem:Propriétés-Leg}
Il est facile de vérifier les trois propriétés suivantes de la transformation de \textsc{Legendre}:
\begin{enumerate}
  \item [(i)] Fixons une droite générique $\ell$ de $\pp$. Alors $\Tang(\mathcal{W},\ell)=p_{1}+\cdots+p_{d}$, où $p_{i}\in\pp$. On peut penser à $\ell$ comme étant un point de $\pd$ et les $p_{i}$ comme étant des droites de $\pd$ passant par $\ell$. Alors $\T_{\hspace{-0.3mm}\ell}\Leg\W=\underset{i=1}{\overset{d}{\cup }}\T_{\hspace{-0.3mm}\ell}p_{i}$;
  \item [(ii)] Si $\mathcal{L}$ est une feuille de $\mathcal{W}$ qui n'est pas une droite, l'union des droites tangentes à $\mathcal{L}$ est une feuille de $\mathrm{Leg}\mathcal{W}$;
 \item [(iii)] Tout $k$-tissu sur $\pp$ peut se lire dans une carte affine donnée $(x,y)$ de $\pp$ par une équation différentielle $F(x,y;y')=0$ de degré $k$ en $y'$, à coefficients polynômiaux. Dans la carte affine $(p,q)$ de $\pd$ correspondant à la droite $\{y=px-q\}\subset{\mathbb{P}^{2}_{\mathbb{C}}},$ la transformée de \textsc{Legendre} d'un tel tissu est donnée par l'équation différentielle
\[
\check{F}(p,q;x):=F(x,px-q;p)=0, \qquad \text{avec} \qquad x=\frac{\mathrm{d}q}{\mathrm{d}p}.
\]
En particulier, pour un feuilletage défini par une $1$-forme $\omega=A(x,y)\mathrm{d}x+B(x,y)\mathrm{d}y,$ on peut prendre $F(x,y;y')=B(x,y)y'+A(x,y)$ et donc sa transformée de \textsc{Legendre} est décrite par
\[
 \check{F}(p,q;x)=A(x,px-q)+pB(x,px-q).
\]
\end{enumerate}
\end{rems}

\begin{lem}[\cite{BFM13}, \rm{Lemme~2.2}]\label{lem:Delta-Leg}
{\sl
Soit $\F$ un feuilletage sur $\pp.$ Notons $\check{\Sigma}_{\F}$ l'ensemble des droites duales des singularités spéciales $\footnote{\hspace{0.1cm} $\Sigma_{\F}$ peut aussi être défini comme l'ensemble des $s\in\Sing\F$ tels que $\nu(\F,s)\geq2$ ou $s$ est une singularité radiale de $\F$.}\Sigma_{\F}=\{s\in\Sing\F\hspace{0.8mm}:\hspace{0.8mm}\tau(\F,s)\geq2\}.$ Alors $$\Delta(\Leg\F)=\mathcal{G}_{\F}(\ItrF)\cup\check{\Sigma}_{\F}.$$
}
\end{lem}

\begin{proof}[\sl D\'emonstration]
D'après la Remarque \ref{rem:Propriétés-Leg} (i), le support du discriminant de $\Leg\F$ contient l'image par l'application de \textsc{Gauss} de toute composante du diviseur d'inflexion transverse; de plus une composante irréductible de $\Delta(\Leg\F)$ qui n'est pas contenue dans $\G_{\F}(\ItrF)$ est nécessairement la droite duale d'un certain point singulier de $\F.$ Soit $s$ un point singulier de $\F$ tel que $\tau(\F,s)=1$; alors la seule possibilité pour que la droite $\check{s}$ duale de $s$ soit contenue dans $\Delta(\Leg\F)$ est que $\check{s}=\mathcal{G}_{\F}(\mathcal{C}),$ pour une certaine composante $\mathcal{C}$ du diviseur $\ItrF.$
\end{proof}

\begin{pro}[\cite{MP13}]\label{pro:MP13-3.3}
{\sl
Soit $\F$ un feuilletage de degré $d$ sur $\pp$ possédant une singularité radiale $s$ d'ordre $n-1.$ Notons $\ell$ la droite duale de $s.$ Alors, au voisinage d'un point générique de $\ell,$ le tissu $\Leg\F$ peut se décomposer comme le produit $\W_{n}\boxtimes\W_{d-n},$ où $\W_{n}$ est un $n$-tissu irréductible laissant $\ell$ invariante et $\W_{d-n}$ est un $(d-n)$-tissu transverse à $\ell.$ De plus, au voisinage d'un point générique de $\ell,$ on a $$\Delta(\Leg\F)=(n-1)\ell+\Delta(\W_{d-n}).$$
}
\end{pro}

\noindent Cette proposition implique, en particulier, que la multiplicité de $\Delta(\Leg\F)$ le long de la droite $\ell$ est $\geq n-1;$ l'égalité est réalisée si et seulement si le tissu $\W_{d-n}$ est régulier au voisinage de tout point générique de $\ell.$ Voici un exemple où cette inégalité est stricte.

\begin{eg}
Considérons le feuilletage $\F$ de degré $4$ sur $\pp$ défini en carte affine par $$\omega=x\mathrm{d}y-y\mathrm{d}x+y^{2}(y+1)^{2}\mathrm{d}x\hspace{1mm};$$ on observe que l'origine $O=(0,0)$ est une singularité radiale de $\F$ d'ordre $1.$ Dans la carte affine $(p,q)$ de $\pd$ associée à la droite $\{y=px-q\}\subset{\pp},$ le tissu $\Leg\F$ est implicitement présenté par l'équation
\begin{Small}
\begin{align*}
p^{4}\cdot x^{4}-2p^{3}(2q-1)\cdot x^{3}+p^{2}(6q^{2}-6q+1)\cdot x^{2}-2qp(q-1)(2q-1)\cdot x+q(q^{3}-2q^{2}+q+1)=0, \qquad x=\frac{\mathrm{d}q}{\mathrm{d}p}.
\end{align*}
\end{Small}
Son discriminant est $$\Delta(\Leg\F)=16p^{12}q^{2}(16q+1)\hspace{1mm};$$ on remarque que la droite duale de l'origine $\check{O}=(q=0)$ est de multiplicité $2>1$ dans le discriminant de $\Leg\F.$
\end{eg}

\Subsection{Courbure et platitude}\label{subsec:Courbure-platitude}\vspace{3mm}

\hspace{-1.34mm}On commence par rappeler la définition de la courbure d'un $k$-tissu $\mathcal{W}.$ On suppose dans un premier temps que $\mathcal{W}$ est un germe de $k$-tissu de $(\mathbb{C}^{2},0)$ complètement décomposable, $\mathcal{W}=\mathcal{F}_{1}\boxtimes\cdots\boxtimes\mathcal{F}_{k}.$ Soit, pour tout $1\leq i\leq k,$ une $1$-forme $\omega_{i}$ à singularité isolée en $0$ définissant le feuilletage $\mathcal{F}_{i}.$ D'après \cite{PP08}, pour tout triplet $(r,s,t)$ avec $1\leq r<s<t\leq k,$ on définit $\eta_{rst}~=\eta(\mathcal{F}_{r}\boxtimes\mathcal{F}_{s}\boxtimes\mathcal{F}_{t})$ comme l'unique $1$-forme méromorphe satisfaisant les égalités suivantes:
\begin{equation}\label{equa:eta-rst}
{\left\{\begin{array}[c]{lll}
\mathrm{d}(\delta_{st}\,\omega_{r}) &=& \eta_{rst}\wedge\delta_{st}\,\omega_{r}\\
\mathrm{d}(\delta_{tr}\,\omega_{s}) &=& \eta_{rst}\wedge\delta_{tr}\,\omega_{s}\\
\mathrm{d}(\delta_{rs}\,\omega_{t}) &=& \eta_{rst}\wedge\delta_{rs}\,\omega_{t}
\end{array}
\right.}
\end{equation}
où $\delta_{ij}$ désigne la fonction définie par $\omega_{i}\wedge\omega_{j}=\delta_{ij}\,\mathrm{d}x\wedge\mathrm{d}y.$ Comme chacune des $1$-formes $\omega_{i}$ n'est définie qu'à multiplication près par un inversible de $\mathcal{O}(\mathbb{C}^{2},0),$ il en résulte que chacune des $1$-formes $\eta_{rst}$ est bien déterminée à l'addition près d'une $1$-forme holomorphe fermée. Ainsi la $1$-forme
\begin{equation}\label{equa:eta}
\hspace{7mm}\eta(\mathcal{W})=\eta(\mathcal{F}_{1}\boxtimes\cdots\boxtimes\mathcal{F}_{k})=\sum_{1\le r<s<t\le k}\eta_{rst}
\end{equation}
est bien définie à l'addition près d'une $1$-forme holomorphe fermée. La {\sl courbure} du tissu \, $\mathcal{W}=\mathcal{F}_{1}\boxtimes\cdots\boxtimes\mathcal{F}_{k}$ est par définition la $2$-forme
\begin{align*}
&K(\mathcal{W})=K(\mathcal{F}_{1}\boxtimes\cdots\boxtimes\mathcal{F}_{k})=\mathrm{d}\,\eta(\mathcal{W}).
\end{align*}

\noindent On peut vérifier que $K(\mathcal{W})$ est une $2$-forme méromorphe à pôles le long du discriminant $\Delta(\mathcal{W})$ de $\mathcal{W},$ canoniquement associée à $\mathcal{W}$; plus précisément, pour toute application holomorphe dominante $\varphi,$ on a $K(\varphi^{*}\mathcal{W})=\varphi^{*}K(\mathcal{W}).$

\noindent Si maintenant $\mathcal{W}$ est un $k$-tissu sur une surface complexe $S$ (non forcément complètement décomposable)\,, alors on peut le transformer en un $k$-tissu complètement décomposable au moyen d'un revêtement galoisien ramifié. L'invariance de la courbure de ce nouveau tissu par l'action du groupe de \textsc{Galois} permet de la redescendre en une $2$-forme méromorphe globale sur $S,$ à pôles le long du discriminant de $\mathcal{W}.$
\begin{defin}
Un $k$-tissu $\mathcal{W}$ est dit \textbf{\textit{plat}} si sa courbure $K(\mathcal{W})$ est identiquement nulle.
\end{defin}

\noindent L'intérêt de cette notion apparaît déjà pour les germes de $3$-tissus puisque dans ce cas la platitude est équivalente au fait d'être parallélisable (\emph{voir} \S\ref{subsec:structure-3-tissu}). Cette notion est également utile pour la classification des tissus de rang maximal. En effet, un résultat de \textsc{Mih\u{a}ileanu} montre que la platitude est une condition nécessaire pour la maximalité du rang. Pour plus de détails et des informations récentes, \emph{voir} \cite{Hen06,Rip07}.

\noindent Reprenons l'exemple du $k$-tissu $\W$ associé à l'équation différentielle (\ref{tissu-equa-diff}) avec $k=3$:
\begin{align*}
& F(x,y,p):=a_{0}(x,y)p^{3}+a_{1}(x,y)p^{2}+a_{2}(x,y)p+a_{3}(x,y)=0, \qquad   p=\frac{\mathrm{d}y}{\mathrm{d}x}.
\end{align*}
Rappelons une formule due à A.~\textsc{Hénaut} \cite{Hen00} qui donne la courbure du $3$-tissu $\W$ en fonction des coefficients $a_i.$ Posons
\begin{align*}
&
\hspace{-2.2cm}R:=\text{Result}(F,\partial_{p}(F))=
\left| \begin{array}{ccccc}
a_0  & a_1  &  a_2  & a_3  & 0 \\
0    & a_0  &  a_1  & a_2  & a_3   \\
3a_0 & 2a_1 &  a_2  & 0    & 0 \\
0    & 3a_0 & 2a_1  & a_2  & 0  \\
0    & 0    & 3a_0  & 2a_1 & a_2
\end{array} \right|\not\equiv0
\end{align*}
\begin{align*}
&
\alpha_1=
\left| \begin{array}{ccccc}
\partial_{y}(a_{0})                     & a_0 & -a_0  &  0    &  0   \\
\partial_{x}(a_{0})+\partial_{y}(a_{1}) & a_1 &  0    & -2a_0 &  0    \\
\partial_{x}(a_{1})+\partial_{y}(a_{2}) & a_2 &  a_2  & -a_1  & -3a_0  \\
\partial_{x}(a_{2})+\partial_{y}(a_{3}) & a_3 &  2a_3 &  0    & -2a_1   \\
\partial_{x}(a_{3})                     & 0   &  0    &  a_3  & -a_2
\end{array} \right|
\end{align*}
et
\begin{align*}
&
\alpha_2=
\left| \begin{array}{ccccc}
0   & \partial_{y}(a_{0})                     & -a_0  &  0    &  0  \\
a_0 & \partial_{x}(a_{0})+\partial_{y}(a_{1}) &  0    & -2a_0 &  0   \\
a_1 & \partial_{x}(a_{1})+\partial_{y}(a_{2}) &  a_2  & -a_1  & -3a_0  \\
a_2 & \partial_{x}(a_{2})+\partial_{y}(a_{3}) &  2a_3 &  0    & -2a_1   \\
a_3 & \partial_{x}(a_{3})                     &  0    &  a_3  & -a_2
\end{array} \right|\hspace{0.5mm};
\end{align*}
alors la courbure du $3$-tissu $\W$ est donnée par (\cite{Hen00})
\begin{equation}\label{equa:Formule-Henaut}
K(\W)=\left(\partial_{y}\left(\frac{\alpha_1}{R}\right)-\partial_{x}\left(\frac{\alpha_2}{R}\right)\right)\mathrm{d}x\wedge\mathrm{d}y.
\end{equation}
%
%
\begin{eg}
Soit $\mathcal{W}$ le $3$-tissu de $\mathbb{C}^{2}$ défini par l'équation différentielle $$(y')^{3}+4xy\cdot y'-8y^{2}=0\hspace{1mm};$$ cette équation est un exemple pris par \textsc{Cauchy} pour illustrer la difficulté à cerner les solutions singulières d'une équation différentielle. Un calcul, utilisant la formule~(\ref{equa:Formule-Henaut}), montre que $$K(\mathcal{W})\equiv0.$$ Le tissu $\mathcal{W}$ est donc plat. Cependant, le $3$-tissu $\mathcal{W}'$ défini sur $\mathbb{C}^{2}$ par l'équation différentielle
\begin{equation*}\footnote{\hspace{0.1cm} Cette équation provient de la physique, formulée dans la thèse de G. \textsc{Mignard} (\cite[Pag. 10]{Mig99}).}
(y')^{3}-4y\cdot y'-4x=0
\end{equation*}
ne l'est pas puisque $$K(\mathcal{W}')=\frac{12\,x(54x^{2}-27y+36y^{2}+64y^{3})}{(27x^{2}-16y^{3})^{2}}\,\mathrm{d}x\wedge\mathrm{d}y\not\equiv0.$$
\end{eg}
\vspace{2mm}

\noindent Notons qu'un $k$-tissu $\mathcal{W}$ sur $\mathbb{P}^{2}_{\mathbb{C}}$ est plat si et seulement si sa courbure est holomorphe le long des points génériques des composantes irréductibles de $\Delta(\mathcal{W})$. Ceci résulte de la définition de $K(\mathcal{W})$ et du fait qu'il n'existe pas de $2$-forme holomorphe sur $\mathbb{P}^{2}_{\mathbb{C}}$ autre que la $2$-forme nulle.

\noindent Dans \cite[\S\hspace{0.1mm}2]{MP13}, les auteurs ont étudié l'holomorphie de la courbure d'un germe de tissu plan. Nous rappelons brièvement ci-dessous quelques-uns de leurs résultats qui nous seront utiles dans les chapitres ultérieurs.

\begin{thm}[\cite{MP13}]\label{thm:MP13-1}
{\sl
Soit $\W$ un germe de $(k+2)$-tissu de $(\mathbb{C}^{2},0)$ à discriminant non vide et lisse (mais pas nécessairement réduit). Supposons que $\W=\W_{2}\boxtimes\W_{k}$ où $\W_{2}$ est un $2$-tissu satisfaisant $\Delta(\W_{2})=\Delta(\W)$ et $\W_{k}$ est un $k$-tissu. La courbure de $\W$ est holomorphe le long de $\Delta(\W)$ si et seulement si $\Delta(\W)$ est invariant par $\W_{2}$ ou par $\beta_{\W_{2}}(\W_{k}),$ où $\beta_{\W_{2}}(\W_{k})$ désigne le $\W_{2}$-barycentre de $\W_{k}$.
}
\end{thm}

\noindent On en tire en particulier le corollaire suivant.

\begin{cor}\label{cor:MP13}
{\sl
Soit $\W$ un germe de $3$-tissu à l'origine de $\mathbb{C}^{2}$ de discriminant non vide et lisse. Supposons que $\W=\W_{2}\boxtimes\F$ où $\W_{2}$ est un $2$-tissu vérifiant $\Delta(\W_{2})=\Delta(\W)$ et $\F$ est un feuilletage. La courbure de $\W$ est holomorphe le long de $\Delta(\W)$ si et seulement si $\Delta(\W)$ est invariant par $\W_{2}$ ou par $\F.$
}
\end{cor}

\begin{eg}
Considérons le $3$-tissu $\W$ présenté par l'équation
\begin{align*}
&\underbrace{\mathrm{d}y}_{\mathcal F} \cdot\underbrace{\left((y-x^{2})y\mathrm{d}y^{2}-2y\mathrm{d}x\mathrm{d}y+\mathrm{d}x^{2}\right)}_{\mathcal W_2}=0.
\end{align*}
Pour ce tissu nous avons
\begin{align*}
&&\Delta(\W)=\Delta(\W_{2})=4y\hspace{0.1mm}x^{2}  \hspace{1cm} \text{et} \hspace{1cm} K(\W)=\dfrac{y}{x^{2}}\mathrm{d}x\wedge\mathrm{d}y.
\end{align*}
On voit que la courbure de $\W$ est holomorphe sur la composante $\{y=0\}\subset\Delta(\W),$ mais elle ne l'est pas sur la composante $\{x=0\}\subset\Delta(\W).$ Cela provient du fait que la droite $(y=0)$ est invariante par $\F$ et que la droite $(x=0)$ n'est ni $\F$-invariante, ni $\W_{2}$-invariante.
\end{eg}

\noindent Considérons un germe de $n$-tissu $\W_{n}$ à l'origine de $\mathbb{C}^{2}.$ Supposons que $\Delta(\W_{n})$ possède une composante irréductible $C$ totalement invariante par $\W_{n}.$ Il existe donc un système de coordonnées locales $(x,y)$ dans lequel $C=\{y=0\}$ et tel que $\W_{n}$ soit donné par
\begin{align*}
& \omega=\mathrm{d}y^{n}+y^{m}(a_{n-1}(x,y)\mathrm{d}y^{n-1}\mathrm{d}x+\cdots+a_{0}(x,y)\mathrm{d}x^{n})
\end{align*}
pour un certain entier $m\geq1.$

\begin{lem}[\cite{MP13}]\label{lem:MP13-2.5}
{\sl
Avec les notations ci-dessus on a $\hspace{0.5mm}\Delta(\W_{n})\geq (n-1)C.$ L'égalité est réalisée si et seulement si $m=1$ et $a_{0}(x,0)\neq0$; auquel cas $\W_{n}$ est irréductible.
}
\end{lem}

\begin{pro}[\cite{MP13}]\label{pro:MP13-2.6}
{\sl
Soit $\W_{n}$ un germe de $n$-tissu de $(\mathbb{C}^{2},0),\, n\geq2.$ Supposons que $\Delta(\W_{n})$ possède une composante irréductible $C$ totalement invariante par $\W_{n}$ et de multiplicité minimale $n-1.$ Soit $\W_{k-n}$ un germe de $(k-n)$-tissu régulier de $(\mathbb{C}^{2},0)$ transverse à $C.$ Alors la courbure de $\W=\W_{n}\boxtimes\W_{k-n}$ est holomorphe le long de $C.$
}
\end{pro}

\noindent L'hypothèse que la composante $C\subset\Delta(\W_{n})$ est \textsl{de multiplicité minimale} $n-1$ est essentielle pour la validité de la Proposition \ref{pro:MP13-2.6}, comme le montre l'exemple ci-dessous.

\begin{eg}
Le $3$-tissu $\W$ de $(\mathbb{C}^{2},0)$ donné par $\omega=\mathrm{d}y^{3}+3y^{2}\mathrm{d}x^{2}\mathrm{d}y-2xy^{3}\mathrm{d}x^{3}$ a pour discriminant $\Delta(\W)=-108y^{6}(x^{2}+1)$ et pour courbure la $2$-forme $K(\W)=\dfrac{\raisebox{-0.4mm}{$2x$}}{3y(x^{2}+1)^{2}}\mathrm{d}x\wedge\mathrm{d}y$ (remarquons que $K(\W)$ n'est pas holomorphe sur la droite $y=0)$. La courbe $\{y=0\}\subset\Delta(\W)$ est totalement invariante par $\W$ mais elle n'est pas de multiplicité minimale $2$: on a pris ici $n=k=3.$

\end{eg}

\clearemptydoublepage
\chapter{Platitude du $d$-tissu dual d'un feuilletage homogène de degré $d$ sur $\mathbb{P}^{2}_{\hspace{-0.5mm}\mathbb{C}}$}\label{chap2:homogène}

\hspace{6mm}
\vspace{-1cm}
\section{Géométrie des feuilletages homogènes}

\begin{defin}\label{def:feuil-homog}
Un feuilletage de degré $d$ sur $\mathbb{P}^{2}_{\mathbb{C}}$ est dit \textbf{\textit{homogène}} s'il existe une carte affine $(x,y)$ de $\mathbb{P}^{2}_{\mathbb{C}}$ dans laquelle il est invariant sous l'action du groupe des homothéties complexes $(x,y)\longmapsto \lambda(x,y),\hspace{1mm} \lambda\in \mathbb{C}^{*}.$
\end{defin}

\noindent Un tel feuilletage $\mathcal{H}$ est alors défini par une $1$-forme $$\omega=A(x,y)\mathrm{d}x+B(x,y)\mathrm{d}y,$$ où $A$ et $B$ sont des polynômes homogènes de degré $d$ sans composante commune. Cette $1$-forme s'écrit en coordonnées homogènes
\begin{align*}
&\hspace{2mm} z\hspace{0.3mm}A(x,y)\mathrm{d}x+z\hspace{0.2mm}B(x,y)\mathrm{d}y-\left(x\hspace{0.2mm}A(x,y)+yB(x,y)\right)\mathrm{d}z\hspace{1mm};
\end{align*}
ainsi le feuilletage $\mathcal{H}$ a au plus $d+2$ singularités dont l'origine $O$ de la carte affine $z=1$ est le seul point singulier de $\mathcal{H}$ qui n'est pas situé sur la droite à l'infini $L_{\infty}=(z=0)$; de plus $\nu(\mathcal{H},O)=d.$

\noindent Dorénavant nous supposerons que $d\geq2$; dans ce cas, d'après la Remarque~\ref{rems:mu>nu^2-Darboux}~(iii), le point $O$ est la seule singularité de $\mathcal{H}$ de multiplicité algébrique $d.$

\noindent Le champ de vecteurs homogène $-B(x,y)\frac{\partial}{\partial x}+A(x,y)\frac{\partial}{\partial y}+0\frac{\partial}{\partial z}$ définit aussi le feuilletage $\mathcal{H}$ car est dans le noyau de la $1$-forme précédente. Notons, pour abréger, $A_x=\frac{\partial A}{\partial x},$ $A_y=\frac{\partial A}{\partial y}$, $B_x=\frac{\partial B}{\partial x},$ $B_y=\frac{\partial B}{\partial y}$; d'après la formule (\ref{equa:ext1}), le diviseur d'inflexion $\IH$ de $\mathcal{H}$ est donné par
\begin{Small}
\begin{equation*}
0=\left| \begin{array}{ccc}
x & \hspace{-1.5mm}-B &  BB_{x}-AB_{y} \\
y &  A &  AA_{y}-BA_{x}  \\
z &  0 &  0
\end{array} \right|
=z
\left| \begin{array}{cc}
-\frac{1}{d}(xB_{x}+yB_{y})                &  BB_{x}-AB_{y} \vspace{1.5mm}\\
\hspace{2.4mm}\frac{1}{d}(xA_{x}+yA_{y})   &  AA_{y}-BA_{x}
\end{array} \right|
=\frac{z}{d}(xA+yB)(A_{x}B_{y}-A_{y}B_{x})
=\frac{z}{d}\Cinv\Dtr,
\end{equation*}
\end{Small}
\hspace{-1mm}où $\Dtr=A_{x}B_{y}-A_{y}B_{x}\in\mathbb{C}[x,y]_{2d-2}$\, et \,$\Cinv=xA+yB\in\mathbb{C}[x,y]_{d+1}$ désigne le {\sl cône tangent} de $\mathcal{H}$ en l'origine $O$.

\noindent Il en résulte que:
\begin{enumerate}
\item [(i)] le support du diviseur $\IinvH$ est constitué des droites du cône tangent $\Cinv=0$ et de la droite à l'infini $L_{\infty}$;
\item [(ii)] le diviseur $\ItrH$ se décompose sous la forme $\ItrH=\prod_{i=1}^{n}T_{i}^{\rho_{i}-1}$ pour un certain nombre $n\leq \deg\Dtr=2d-2$ de droites $T_{i}$ passant par $O,$ $\rho_{i}-1$ étant l'ordre d'inflexion de la droite $T_{i}.$ Lorsque $\rho_{i}=2$ on parle d'une droite d'inflexion simple pour $\mathcal{H},$ lorsque $\rho_{i}=3$ d'une droite d'inflexion double, etc.
\end{enumerate}

\begin{pro}\label{pro:SingH}
{\sl Avec les notations précédentes, pour tout point singulier $s\in\mathrm{Sing}\mathcal{H}\cap L_{\infty},$ nous avons

\noindent\textbf{\textit{1.}} $\nu(\mathcal{H},s)=1;$

\noindent\textbf{\textit{2.}} la droite $L_{s}$ passant par l'origine $O$ et le point $s$ est invariante par $\mathcal{H}$ et elle apparaît avec multiplicité $\tau(\mathcal{H},s)-1$ dans le diviseur $\Dtr=0,$ {\it i.e.}
$$\Dtr=\ItrH\prod_{s\in\mathrm{Sing}\mathcal{H}\cap L_{\infty}}L_{s}^{\tau(\mathcal{H},s)-1}.$$
}
\end{pro}

\begin{proof}[\sl D\'emonstration]
Soit $s$ un point singulier de $\mathcal{H}$ sur $L_{\infty}=(z=0)$. Sans perte de généralité, nous pouvons supposer que les coordonnées homogènes de $s$ sont de la forme $[x_0:1:0],\,x_0\in\mathbb{C}.$ Dans la carte affine $y=1$, $\mathcal{H}$ est décrit par la $1$-forme
\begin{align*}
&\theta=z\hspace{0.3mm}A(x,1)\mathrm{d}x-\left(x\hspace{0.2mm}A(x,1)+B(x,1)\right)\mathrm{d}z\hspace{1mm};
\end{align*}
la condition $s\in\Sing\mathcal{H}$ est équivalente à $B(x_0,1)=-x_0\hspace{0.4mm}A(x_0,1).$ L'égalité $\pgcd(A,B)=1$ implique alors que $A(x_{0},1)\neq0$; d'où $\nu(\mathcal{H},s)=1.$

\noindent Montrons la seconde assertion. Le fait que $$\theta=A(x,1)\left(z\mathrm{d}(x-x_0)-(x-x_0)\mathrm{d}z\right)-\left(x_0\hspace{0.4mm}A(x,1)+B(x,1)\right)\mathrm{d}z$$
entraîne que $$\tau:=\tau(\mathcal{H},s)=\min\{k\geq1:J^{k}_{x_0}(x_0\hspace{0.4mm}A(x,1)+B(x,1))\neq0\},$$
cela permet d'écrire $x_0\hspace{0.4mm}A(x,1)+B(x,1)=\sum_{k=\tau}^{d}c_{k}(x-x_0)^{k}$, avec $c_{\tau}\neq0$. Par suite
\begin{align*}
&& B(x,y)=(x-x_0y)^{\tau}P(x,y)-x_0\hspace{0.4mm}A(x,y), \quad \text{où} \quad P(x,y)=\sum_{k=0}^{d-\tau}c_{k+\tau}(x-x_0y)^{k}y^{d-\tau-k}.
\end{align*}
Un calcul élémentaire montre que $\Dtr=A_{x}B_{y}-A_{y}B_{x}$ est de la forme $\Dtr=-(x-x_0y)^{\tau-1}Q(x,y),$ avec $Q\in\mathbb{C}[x,y]$ et $$Q(x_0,1)=\tau P(x_0,1)\left(xA_{x}+yA_{y}\right)\Big|_{(x,y)=(x_0,1)}.$$
Comme $P(x_0,1)=c_{\tau}$\, et \,$xA_{x}+yA_{y}=dA$,\, $Q(x_0,1)=\tau\hspace{0.2mm}c_{\tau}\hspace{0.2mm}dA(x_0,1)\neq0.$
\end{proof}

\begin{defin}\label{def:type-homog}
Soit $\mathcal{H}$ un feuilletage homogène de degré $d$ sur $\pp$ ayant un certain nombre $m\leq d+1$ de singularités radiales $s_{i}$ d'ordre $\tau_{i}-1,$ $2\leq\tau_{i}\leq d$ pour $i=1,2,\ldots,m.$ Le support du diviseur $\ItrH$ est constitué d'un certain nombre $n\leq2d-2$ de droites d'inflexion transverses $T_{j}$ d'ordre $\rho_{j}-1,$ $2\leq\rho_{j}\leq d$ pour $j=1,2,\ldots,n.$ On définit le \textbf{\textit{type du feuilletage}} $\mathcal{H}$ par $$\mathcal{T}_\mathcal{H}=\sum\limits_{i=1}^{m}\mathrm{R}_{\tau_{i}-1}+\sum\limits_{j=1}^{n}\mathrm{T}_{\rho_{j}-1}=
\sum\limits_{k=1}^{d-1}(r_{k}\cdot\mathrm{R}_k+t_{k}\cdot\mathrm{T}_k)
\in\Z\left[\mathrm{R}_1,\mathrm{R}_2,\ldots,\mathrm{R}_{d-1},\mathrm{T}_1,\mathrm{T}_2,\ldots,\mathrm{T}_{d-1}\right]$$ et le \textbf{\textit{degré du type}} $\mathcal{T}_\mathcal{H}$ par $\deg\mathcal{T}_\mathcal{H}=\sum_{k=1}^{d-1}(r_{k}+t_{k})\in\N\setminus\{0,1\}$; c'est le nombre de droites distinctes qui composent le diviseur $\Dtr.$
\end{defin}

\begin{eg}
Considérons le feuilletage homogène $\mathcal{H}$ de degré $5$ sur $\pp$ défini par $$\omega=y^5\mathrm{d}x+2x^3(3x^2-5y^2)\mathrm{d}y.$$ Un calcul élémentaire conduit à
\begin{align*}
&& \mathrm{C}_{\mathcal{H}}=xy\left(6x^4-10x^2y^2+y^4\right) \qquad\text{et}\qquad \Dtr=150\hspace{0.15mm}x^2y^4(x-y)(x+y)\hspace{1mm};
\end{align*}
on constate que l'ensemble des singularités radiales de $\mathcal{H}$ est constitué des deux points $[0:1:0]$ et $[1:0:0]$; leurs ordres de radialité sont égaux respectivement à $2$ et $4.$ De plus le support du diviseur $\ItrH$ est formé des deux droites d'équations $x-y=0$ et $x+y=0$; ce sont des droites d'inflexion transverses simples. Donc le feuilletage $\mathcal{H}$ est du type $\mathcal{T}_\mathcal{H}=1\cdot\mathrm{R}_{2}+1\cdot\mathrm{R}_{4}+2\cdot\mathrm{T}_{1}$ et le degré de $\mathcal{T}_\mathcal{H}$ est $\deg\mathcal{T}_\mathcal{H}=4.$
\end{eg}

\noindent \`{A} tout feuilletage homogène $\mathcal{H}$ de degré $d$ sur $\pp$ on peut associer une application rationnelle $\Gunderline_{\mathcal{H}}\hspace{1mm}\colon\mathbb{P}^{1}_{\mathbb{C}}\rightarrow \mathbb{P}^{1}_{\mathbb{C}}$ de la façon suivante: si $\mathcal{H}$ est décrit par $\omega=A(x,y)\mathrm{d}x+B(x,y)\mathrm{d}y,$ $A$ et $B$ désignant des polynômes homogènes de degré $d$ sans facteur commun, on définit $\Gunderline_{\mathcal{H}}$ par $$\Gunderline_{\mathcal{H}}([x:y])=[-A(x,y):B(x,y)]\hspace{1mm};$$
il est clair que cette définition ne dépend pas du choix de la $1$-forme homogène $\omega$ décrivant le feuilletage $\mathcal{H}.$

\noindent Dorénavant nous noterons l'application $\Gunderline_{\mathcal{H}}$ simplement par $\Gunderline.$ Le feuilletage homogène $\mathcal{H}$ ainsi que son tissu dual $\mathrm{Leg}\mathcal{H}$ peuvent être décrits analytiquement en utilisant uniquement l'application $\Gunderline.$ En effet, la pente $p$ de $\mathrm{T}_{\hspace{-0.4mm}(x,y)}\mathcal{H}$ est donnée par $\Gunderline([x:y])=[p:1]$ et les pentes $x_{i}$ $(i=1,\ldots,d)$ de $\mathrm{T}_{\hspace{-0.4mm}(p,q)}\mathrm{Leg}\mathcal{H}$ sont données par $x_{i}=\dfrac{q}{p-p_{i}(p)},$ avec $\Gunderline^{-1}([p:1])=\{[p_{i}(p):1]\}.$

\noindent En carte affine $\mathbb{C}\subset\sph$ cette application s'écrit $\Gunderline\hspace{1mm}\colon z\mapsto-\dfrac{A(1,z)}{B(1,z)}.$ On a
\begin{align*}
\Gunderline(z)-z=-\dfrac{A(1,z)+zB(1,z)}{B(1,z)}=-\frac{\mathrm{C}_\mathcal{H}(1,z)}{B(1,z)}\hspace{1mm};
\end{align*}
\noindent de plus, les identités $dA=xA_{x}+yA_{y}$\hspace{1mm} et \hspace{1mm}$dB=xB_{x}+yB_{y}$
permettent de réécrire le diviseur $\Dtr$ sous la forme  $\Dtr=-\dfrac{\raisebox{-0.5mm}{$d$}}{\raisebox{0.7mm}{$x$}}\left(BA_{y}-AB_{y}\right)$
de sorte que
\begin{align*}
\hspace{5mm}\Gunderline'(z)=-\left(\frac{BA_{y}-AB_{y}}{B^{2}}\right)\Big|_{(x,y)=(1,z)}=\frac{\Dtr(1,z)}{dB^{2}(1,z)}.
\end{align*}
\noindent On en déduit immédiatement les propriétés suivantes:
\vspace{1.2mm}

\begin{itemize}
\item [\textbf{\textit{1.}}] les points fixes de $\Gunderline$ correspondent au cône tangent de $\mathcal{H}$ en l'origine $O$ (\textit{i.e.} $[a:b]\in\mathbb{P}^{1}_{\mathbb{C}}$ est fixe par $\Gunderline$ si et seulement si la droite d'équation $by-a\hspace{0.2mm}x=0$ est invariante par $\mathcal{H}$\hspace{0.25mm});
    \vspace{1.2mm}

\item [\textbf{\textit{2.}}] le point $[a:b]\in\mathbb{P}^{1}_{\mathbb{C}}$ est critique fixe par $\Gunderline$ si et seulement si le point $[b:a:0]\in L_{\infty}$ est singulier radial de $\mathcal{H}$. La multiplicité du point critique $[a:b]$ de $\Gunderline$ est exactement égale à l'ordre de radialité de la singularité à l'infini;
    \vspace{1.2mm}

\item [\textbf{\textit{3.}}] le point $[a:b]\in\mathbb{P}^{1}_{\mathbb{C}}$ est critique non fixe par $\Gunderline$ si et seulement si la droite d'équation $by-a\hspace{0.2mm}x=0$ est une droite d'inflexion transverse pour $\mathcal{H}.$ La multiplicité du point critique $[a:b]$ de $\Gunderline$ est précisément égale à l'ordre d'inflexion de cette droite.
\end{itemize}

\begin{rem}\label{rem:2d-2-R1}
Un feuilletage homogène de degré $d$ supérieur ou égal à $2$ sur $\pp$ ne peut avoir $d+1$ singularités radiales distinctes, ou, ce qui revient au même, ne peut être de type $$\sum_{k=1}^{d-1}(r_{k}\cdot\mathrm{R}_k+t_{k}\cdot\mathrm{T}_k), \qquad\text{avec}\hspace{1mm} \sum_{k=1}^{d-1}r_{k}=d+1.$$ Ceci découle du fait bien connu qu'une application rationnelle de la sphère de \textsc{Riemann} dans elle-même a au moins un point fixe non critique, \emph{voir} par exemple~\cite[Théorème 12.4]{Mil99}.
\end{rem}
\smallskip

\section{\'{E}tude de la platitude du tissu dual d'un feuilletage homogène}\label{sec:étude-platitude-homog}
\vspace{3mm}

La Proposition~$3.2$ de \cite{BFM13} est un critère de la platitude de la transformée de \textsc{Legendre} d'un feuilletage homogène de degré $3$. Notre premier résultat généralise ce critère en degré arbitraire.

\begin{thm}\label{thm:holomo-G(I^tr)}
{\sl Soit $\mathcal{H}$ un feuilletage homogène de degré $d\geq3$ sur $\pp.$ Alors le $d$-tissu $\Leg\mathcal{H}$ est plat si et seulement si sa courbure $K(\Leg\mathcal{H})$ est holomorphe sur $\G_{\mathcal{H}}(\ItrH).$
}
\end{thm}

\noindent La démonstration de ce théorème utilise les Lemmes~\ref{lem:Delta-LegH} et \ref{lem:holomo-O} qui suivent. Dans ces deux lemmes, $\mathcal{H}$ désigne un feuilletage homogène de degré $d\geq3$ sur $\pp$ défini, en carte affine $(x,y)$, par la $1$-forme
\begin{align*}
& \omega=A(x,y)\mathrm{d}x+B(x,y)\mathrm{d}y,\quad A,B\in\mathbb{C}[x,y]_{d},\hspace{2mm}\pgcd(A,B)=1.
\end{align*}
\newpage
\hfill
\medskip

\begin{lem}\label{lem:Delta-LegH}
{\sl Le discriminant de $\Leg\mathcal{H}$ se décompose en $$\Delta(\Leg\mathcal{H})=\G_{\mathcal{H}}(\ItrH)\cup\radHd\cup\check{O},$$ où $\radHd$ désigne l'ensemble des droites duales des points de $$\radH=\{s\in\Sing\mathcal{H}\hspace{0.8mm}:\hspace{0.8mm}\nu(\mathcal{H},s)=1,\,\tau(\mathcal{H},s)\geq2\}.$$
}
\end{lem}

\begin{proof}[\sl D\'emonstration]
Par le Lemme \ref{lem:Delta-Leg}, $\Delta(\Leg\mathcal{H})=\mathcal{G}_{\mathcal{H}}(\ItrH)\cup\check{\Sigma}_{\mathcal{H}}$, où $\check{\Sigma}_{\mathcal{H}}$ est l'ensemble des droites duales des points de $\Sigma_{\mathcal{H}}=\{s\in\Sing\mathcal{H}\hspace{0.8mm}:\hspace{0.8mm}\tau(\mathcal{H},s)\geq2\}$. D'après la première assertion de la Proposition~\ref{pro:SingH}, l'origine $O$ est le seul point singulier de $\mathcal{H}$ de multiplicité algébrique supérieure ou égale à $2$; par conséquent $\Sigma_{\mathcal{H}}=\radH\cup\{O\}.$
\end{proof}

\begin{lem}[\cite{BFM13}, \rm{Lemme~3.1}]\label{lem:holomo-O}
{\sl
Si la courbure de $\Leg\mathcal{H}$ est holomorphe sur $\pd\hspace{-0.3mm}\setminus\hspace{-0.3mm} \check{O},$ alors $\Leg\mathcal{H}$ est plat.
}
\end{lem}

\begin{proof}[\sl D\'emonstration]
Soit $(a,b)$ la carte affine de $\pd$ associée à la droite $\{ax-by+1=0\}\subset{\pp}$; le $d$-tissu $\mathrm{Leg}\mathcal{H}$ est donné par la $d$-forme symétrique $\check{\omega}=bA(\mathrm{d}b,\mathrm{d}a)+aB(\mathrm{d}b,\mathrm{d}a)$. L'homogénéité de $A$ et $B$ implique alors que toute homothétie $h_{\lambda}\hspace{1mm}\colon(a,b)\longmapsto\lambda(a,b)$ laisse invariant $\Leg \mathcal{H}$; par suite $$h_{\lambda}^{*}(K(\Leg \mathcal{H}))=K(\Leg \mathcal{H}).$$
En combinant l'hypothèse de l'holomorphie de la courbure en dehors de $\check{O}$ avec le fait que $\check{O}$ est la droite à l'infini dans la carte $(a,b),$ on constate que $K(\Leg \mathcal{H})=P(a,b)\mathrm{d}a\wedge\mathrm{d}b$ pour un certain $P\in\mathbb{C}[a,b].$ On déduit de ce qui précède que $\lambda^{2}P(\lambda\hspace{0.1mm}a,\lambda\hspace{0.1mm}b)=P(a,b),$ d'où l'énoncé.
\end{proof}

\begin{proof}[\sl D\'emonstration du Théorème~\ref{thm:holomo-G(I^tr)}]
L'implication directe est triviale. Montrons la réciproque; supposons que $K(\Leg\mathcal{H})$ soit holomorphe sur $\G_{\mathcal{H}}(\ItrH).$ D'après les Lemmes \ref{lem:Delta-LegH} et \ref{lem:holomo-O}, il suffit de prouver que $K(\Leg\mathcal{H})$ est holomorphe le long de $\Xi:=\radHd\setminus\G_{\mathcal{H}}(\ItrH).$ Supposons donc $\Xi$ non vide; soit $s$ une singularité radiale de $\mathcal{H}$ d'ordre $n-1$ telle que la droite $\check{s}$ duale de $s$ ne soit pas contenue dans $\G_{\mathcal{H}}(\ItrH).$ La Proposition~\ref{pro:MP13-3.3} implique qu'au voisinage de tout point générique $m$ de $\check{s},$ le tissu $\Leg\mathcal{H}$ peut se décomposer comme le produit $\W_{n}\boxtimes\W_{d-n},$ où $\W_{n}$ est un $n$-tissu irréductible laissant $\check{s}$ invariante et $\W_{d-n}$ est un $(d-n)$-tissu transverse à $\check{s}.$ De plus, la condition $\check{s}\not\subset\G_{\mathcal{H}}(\ItrH)$ assure que le tissu $\W_{d-n}$ est régulier au voisinage de $m.$ Par conséquent $K(\Leg\mathcal{H})$ est holomorphe au voisinage de $m,$ en vertu de la Proposition~\ref{pro:MP13-2.6}.
\end{proof}

\begin{cor}
{\sl Soit $\mathcal{H}$ un feuilletage homogène convexe de degré $d$ sur le plan projectif complexe. Alors le $d$-tissu $\Leg\mathcal{H}$ est plat.}
\end{cor}
\newpage
\hfill
\medskip

\noindent Le théorème suivant est un critère effectif d'holomorphie de la courbure du tissu dual d'un feuilletage homogène le long de l'image par l'application de \textsc{Gauss} d'une droite d'inflexion transverse simple, {\it i.e.} d'ordre d'inflexion minimal.

\begin{thm}\label{thm:Barycentre}
{\sl Soit $\mathcal{H}$ un feuilletage homogène de degré $d\geq3$ sur $\pp$ défini par la $1$-forme
$$\omega=A(x,y)\mathrm{d}x+B(x,y)\mathrm{d}y,\quad A,B\in\mathbb{C}[x,y]_d,\hspace{2mm}\pgcd(A,B)=1.$$
Supposons que $\mathcal{H}$ ait une droite d'inflexion $T=(ax+by=0)$ transverse et simple. Supposons en outre que $[-a:b]\in\mathbb{P}^{1}_{\mathbb{C}}$ soit le seul point critique de $\Gunderline$ dans sa fibre $\Gunderline^{-1}(\Gunderline([-a:b])).$ Posons $T'=\mathcal{G}_{\mathcal{H}}(T)$ et considérons la courbe $\Gamma_{(a,b)}$ de $\pp$ définie par
\begin{align*}
Q(x,y;a,b):=\left| \begin{array}{cc}
\dfrac{\partial{P}}{\partial{x}} &  A(b,-a)
\vspace{2mm}
\\
\dfrac{\partial{P}}{\partial{y}} &  B(b,-a)
\end{array} \right|=0,
\quad\text{où}\quad
P(x,y;a,b):=\frac{1}{(ax+by)^{2}}
\left| \begin{array}{cc}
A(x,y)  &  A(b,-a)
\\
B(x,y)  &  B(b,-a)
\end{array} \right|.
\end{align*}
Alors la courbure de $\mathrm{Leg}\mathcal{H}$ est holomorphe sur $T'$ si et seulement si $T=\{ax+by=0\}\subset\Gamma_{(a,b)},$ {\it i.e.} si et seulement si $Q(b,-a\hspace{0.2mm};a,b)=0.$
}
\end{thm}

\begin{rem}\label{rem:Q(b,-a,a,b)}
L'hypothèse que $T=(ax+by=0)$ est une droite d'inflexion pour $\mathcal{H}$ implique que $P\in\mathbb{C}[x,y]_{d-2}$ et donc $Q\in\mathbb{C}[x,y]_{d-3}.$ En particulier lorsque $d=3$ on a $$Q(b,-a\hspace{0.2mm};a,b)=\frac{\Cinv\left(B(b,-a),-A(b,-a)\right)}{\left(\Cinv(b,-a)\right)^2}\hspace{1mm};$$ en effet si on pose $\tilde{a}=A(b,-a)$, $\tilde{b}=B(b,-a)$ et $P(x,y;a,b)=f(a,b)x+g(a,b)y$ on obtient
\begin{align*}
Q(b,-a\hspace{0.2mm};a,b)=f(a,b)\tilde{b}-g(a,b)\tilde{a}=P(\tilde{b},-\tilde{a}\hspace{0.2mm};a,b)
=\frac{\tilde{b}A(\tilde{b},-\tilde{a})-\tilde{a}B(\tilde{b},-\tilde{a})}{(a\tilde{b}-b\tilde{a})^2}
=\frac{\Cinv\left(\tilde{b},-\tilde{a}\right)}{\left(\Cinv(b,-a)\right)^2}.
\end{align*}
\end{rem}

\begin{proof}[\sl D\'emonstration]
\`{A} isomorphisme linéaire près on peut se ramener à $T=(y=rx)$; si $(p,q)$ est la carte affine de $\pd$ associée à la droite $\{y=px-q\}\subset{\mathbb{P}^{2}_{\mathbb{C}}},$ alors $T'=(p=\Gunderline(r))$ avec $\Gunderline(z)=-\dfrac{A(1,z)}{B(1,z)}.$ Comme l'indice de ramification de $\Gunderline$ en $z=r$ est égal à $2$ et comme $z=r$ est l'unique point critique dans sa fibre $\Gunderline^{-1}(\Gunderline(r))$, cette fibre est formée de $d-1$ points distincts, soit $\Gunderline^{-1}(\Gunderline(r))=\{r,z_{1},z_{2},\ldots,z_{d-2}\}.$ De plus, au voisinage de tout point générique de $T'$, le tissu dual de $\mathcal{H}$ se décompose en $\Leg\mathcal{H}=\W_{2}\boxtimes\W_{d-2}$ avec
\begin{align*}
& \W_{2}\Big|_{T'}=\left(\mathrm{d}q-x_{0}(q)\mathrm{d}p\right)^{2}
\qquad\text{et}\qquad
\W_{d-2}\Big|_{T'}=\prod_{i=1}^{d-2}\left(\mathrm{d}q-x_{i}(q)\mathrm{d}p\right),
\end{align*}
où $x_{0}(q)=\dfrac{q}{\Gunderline(r)-r}$\, et\, $x_{i}(q)=\dfrac{q}{\Gunderline(r)-z_{i}},\,i=1,2,\ldots,d-2.$ D'après le Théorème \ref{thm:MP13-1}, la courbure de $\Leg\mathcal{H}$ est holomorphe le long de $T'$ si et seulement si $T'$ est invariante par le barycentre de $\W_{d-2}$ par rapport à $\W_{2}.$ Or la restriction de $\beta_{\W_{2}}(\W_{d-2})$ à $T'$ est donnée par $\mathrm{d}q-\beta(q)\mathrm{d}p=0$ avec $$\beta=x_{0}+\dfrac{1}{\frac{1}{d-2}\sum\limits_{i=1}^{d-2}\dfrac{1}{x_{i}-x_{0}}}.$$
\noindent Ainsi la courbure de $\Leg\mathcal{H}$ est holomorphe sur $T'$ si et seulement si $\beta=\infty,$ {\it i.e.} si et seulement si $\sum_{i=1}^{d-2}\dfrac{1}{x_{i}-x_{0}}=0,$ car $x_{0}\neq\infty$ ($z=r$ est non fixe par $\Gunderline$). Cette dernière condition se réécrit
\begin{equation}\label{equa:Barycentre-implicite}
\hspace{0.8cm}0=\sum\limits_{i=1}^{d-2}\frac{\Gunderline(r)-z_{i}}{r-z_{i}}=d-2+\left(\Gunderline(r)-r\right)\sum\limits_{i=1}^{d-2}\frac{1}{r-z_{i}}.
\end{equation}
D'autre part les $z_{i}$ sont exactement les racines du polynôme $$F(z):=\dfrac{P(1,z\hspace{0.2mm};-r,1)}{B(1,r)}=\dfrac{A(1,z)+\Gunderline(r)B(1,z)}{(z-r)^{2}}$$ et donc
\begin{align*}
\sum\limits_{i=1}^{d-2}\frac{1}{r-z_{i}}
\hspace{0.3mm}=\hspace{0.3mm}
\sum\limits_{i=1}^{d-2}\left(\frac{1}{F(r)}\prod\limits_{\underset{j\neq i}{j=1}}^{d-2}(r-z_{j})\right)
\hspace{0.3mm}=\hspace{0.3mm}
\frac{1}{F(r)}\sum\limits_{i=1}^{d-2}\hspace{0.5mm}\prod\limits_{\underset{j\neq i}{j=1}}^{d-2}(r-z_{j})
\hspace{0.3mm}=\hspace{0.3mm}
\frac{F'(r)}{F(r)}.
\end{align*}
Ainsi l'équation (\ref{equa:Barycentre-implicite}) est équivalente à $(\Gunderline(r)-r)F'(r)+(d-2)F(r)=0,$ {\it i.e.} à
\begin{equation}\label{equa:Barycentre-explicite}
\hspace{2.2cm}(d-2)P(1,r\hspace{0.2mm};-r,1)+\left(\Gunderline(r)-r\right)\frac{\partial{P}}{\partial{y}}\Big|_{(x,y)=(1,r)}=0\hspace{1mm};
\end{equation}
comme $P\in\mathbb{C}[x,y]_{d-2}$ on peut réécrire (\ref{equa:Barycentre-explicite}) sous la forme
\begin{align*}
\hspace{2.3cm}\left((d-2)P(x,y\hspace{0.2mm};-r,1)-y\frac{\partial{P}}{\partial{y}}+x\Gunderline(r)\frac{\partial{P}}{\partial{y}}\right)\Big|_{y=rx}=0\hspace{1mm};
\end{align*}
celle-ci peut à son tour s'écrire
\begin{align*}
\hspace{-1.3cm}\left(\frac{\partial{P}}{\partial{x}}+\Gunderline(r)\frac{\partial{P}}{\partial{y}}\right)\Big|_{y=rx}=0,
\end{align*}
en vertu de l'identité d'\textsc{Euler}. Il en résulte que $K(\Leg\mathcal{H})$ est holomorphe le long de $T'$ si et seulement si
\begin{align*}
\left(B(1,r)\frac{\partial{P}}{\partial{x}}-A(1,r)\frac{\partial{P}}{\partial{y}}\right)\Big|_{y=rx}=0.
\end{align*}

\noindent \textit{\textbf{Remarque.}}\,---\, En degré $3$ l'équation (\ref{equa:Barycentre-implicite}) s'écrit $\dfrac{\Gunderline(r)-z_{1}}{r-z_{1}}=0$; ainsi la courbure du $3$-tissu $\Leg\mathcal{H}$ est holomorphe sur $T'=(p=\Gunderline(r))$ si et seulement si $\Gunderline(r)=z_{1}$, {\it i.e.} si et seulement si $$\Gunderline(\Gunderline(r))=\Gunderline(r).$$
\end{proof}

\noindent Le théorème suivant est un critère effectif d'holomorphie de la courbure du tissu dual d'un feuilletage homogène le long de l'image par l'application de \textsc{Gauss} d'une droite d'inflexion transverse d'ordre maximal.

\begin{thm}\label{thm:Divergence}
{\sl Soit $\mathcal{H}$ un feuilletage homogène de degré $d\geq3$ sur $\pp$ défini par la $1$-forme
$$\omega=A(x,y)\mathrm{d}x+B(x,y)\mathrm{d}y,\quad A,B\in\mathbb{C}[x,y]_d,\hspace{2mm}\pgcd(A,B)=1.$$
Supposons que $\mathcal{H}$ possède une droite d'inflexion transverse $T$ d'ordre maximal $d-1$ et posons $T'=\mathcal{G}_{\mathcal{H}}(T).$ Alors la courbure de $\mathrm{Leg}\mathcal{H}$ est holomorphe le long de $T'$ si et seulement si la $2$-forme $\mathrm{d}\omega$ s'annule sur la droite $T.$
}
\end{thm}
\noindent La démonstration de ce théorème utilise le lemme technique suivant, qui nous sera aussi utile ultérieurement.

\begin{lem}\label{lem:critique-maximale}
{\sl Soit $f:\mathbb{P}^{1}_{\mathbb{C}}\rightarrow\mathbb{P}^{1}_{\mathbb{C}}$ une application rationnelle de degré $d;f(z)=\dfrac{a(z)}{b(z)}$ avec $a$ et $b$ des polynômes sans facteur commun et $\max(\deg a,\deg b)=d.$ Soit $z_{0}\in\mathbb{C}$ tel que $f(z_0)\neq\infty.$ Alors, $z_0$ est un point critique de $f$ de multiplicité $m-1$ si et seulement s'il existe un polynôme $c\in\mathbb{C}[z]$ de degré $\leq d-m$ vérifiant $c(z_0)\neq0$ et tel que $a(z)=f(z_0)b(z)+c(z)(z-z_0)^{m}.$
}
\end{lem}

\begin{proof}[\sl D\'emonstration]
D'après la formule de \textsc{Taylor}, l'assertion $z=z_0$ est un point critique de $f$ de multiplicité $m-1$ se traduit par  $f(z)=f(z_0)+h(z)(z-z_0)^m,$ avec $h(z_0)\neq0.$ Par suite $$a(z)-f(z_0)b(z)=c(z)(z-z_0)^m$$
avec $c(z):=h(z)b(z)$, $c(z_0)\neq0$; comme le membre de gauche est un polynôme en $z$ de degré $\leq d$ celui de droite aussi. On constate alors que la fonction $c(z)$ est polynomiale en $z$ de degré $\leq d-m$, d'où l'énoncé.
\end{proof}

\begin{proof}[\sl D\'emonstration du Théorème~\ref{thm:Divergence}]
On peut se ramener à $T=(y=rx)$; si $(p,q)$ est la carte affine de $\pd$ associée à la droite $\{y=px-q\}\subset{\mathbb{P}^{2}_{\mathbb{C}}},$ alors $T'=(p=\Gunderline(r))$ avec $\Gunderline(z)=~-\dfrac{A(1,z)}{B(1,z)}.$ De plus, le $d$-tissu $\Leg\mathcal{H}$ est décrit par $\prod_{i=1}^{d}\check{\omega}_i$, où
\begin{align*}
&\check{\omega}_i=\dfrac{\mathrm{d}q}{q}-\lambda_{i}(p)\mathrm{d}p,&&\lambda_{i}(p)=\dfrac{1}{p-p_{i}(p)}&& \text{et} &&\{p_{i}(p)\}=\Gunderline^{-1}(p).
\end{align*}

\noindent En appliquant les formules (\ref{equa:eta-rst}) et (\ref{equa:eta})  à $\Leg\mathcal{H}=\W(\check{\omega}_1,\check{\omega}_2,\ldots,\check{\omega}_d)$ on constate que $\eta(\mathrm{Leg}\mathcal{H})$ s'écrit sous la forme
\begin{align*}
\hspace{-4.22cm}\eta(\mathrm{Leg}\mathcal{H})=\alpha(p)\mathrm{d}p+\dfrac{\mathrm{d}q}{q}\sum_{1\le i<j<k\le d}\beta_{ijk}(p),
\end{align*}
avec
\begin{align*}
\beta_{ijk}(p)=
\dfrac{-\lambda_{i}'}{(\lambda_{i}-\lambda_{j})(\lambda_{i}-\lambda_{k})}+
\dfrac{-\lambda_{j}'}{(\lambda_{j}-\lambda_{i})(\lambda_{j}-\lambda_{k})}+
\dfrac{-\lambda_{k}'}{(\lambda_{k}-\lambda_{i})(\lambda_{k}-\lambda_{j})}.
\end{align*}
Comme le point $z=r$ est critique non fixe pour $\Gunderline$ de multiplicité $d-1,$ il existe un isomorphisme analytique $\varphi:(\mathbb{C},0)\rightarrow(\mathbb{C},0)$ tel qu'au voisinage de $T'$ on ait
\begin{align*}
\hspace{-1.5cm}\lambda_{i}(p)=\frac{1}{\Gunderline(r)-r}+\varphi\left(\zeta^{i}\left(p-\Gunderline(r)\right)^{\frac{1}{d}}\right), \quad\text{avec}\hspace{2mm}
\zeta=\mathrm{e}^{2\mathrm{i}\pi/d}.
\end{align*}
Notons que
\begin{align*}
\hspace{1.5cm}\lambda_{i}'(p)=\frac{1}{d}\left(p-\Gunderline(r)\right)^{\frac{1-d}{d}}\left[\zeta^{i}\varphi'(0)+
\zeta^{2i}\varphi''(0)\left(p-\Gunderline(r)\right)^{\frac{1}{d}}+
o\left((p-\Gunderline(r))^{\frac{1}{d}}\right)\right],
\end{align*}
et
\begin{align*}
\hspace{-1.2cm}\lambda_{i}(p)-\lambda_{j}(p)=\left(p-\Gunderline(r)\right)^{\frac{1}{d}}\varphi'(0)(\zeta^{i}-\zeta^{j})+
o\left((p-\Gunderline(r))^{\frac{1}{d}}\right).
\end{align*}
Il s'en suit que
\begin{align*}
\hspace{1.1mm}\beta_{ijk}(p)=\left(p-\Gunderline(r)\right)^{-1-\frac{1}{d}}\tilde{\beta}_{ijk}\left((p-\Gunderline(r))^{\frac{1}{d}}\right),
\quad\hspace{1mm}\text{avec}\hspace{1mm}\quad
\tilde{\beta}_{ijk}(z)\in\mathbb{C}\{z\}.
\end{align*}
En fait, si $\langle i',j',k'\rangle$ désigne trois permutations circulaires de $i,j$ et $k,$ on a
\begin{align*}
\hspace{-3.5cm}\tilde{\beta}_{ijk}(0)=-\frac{1}{d\varphi'(0)}\underbrace{\sum_{\langle i',\,j',\,k'\rangle}\frac{\zeta^{i'}}{(\zeta^{i'}-\zeta^{j'})(\zeta^{i'}-\zeta^{k'})}}_{0}=0,
\end{align*}
et
\begin{align*}
\hspace{-1.8cm}\tilde{\beta}'_{ijk}(0)=\frac{\varphi''(0)}{2d\varphi'(0)^{2}}\underbrace{\sum_{\langle i',\,j',\,k'\rangle}\frac{\zeta^{i'}(\zeta^{j'}+\zeta^{k'})}{(\zeta^{i'}-\zeta^{j'})(\zeta^{i'}-\zeta^{k'})}}_{-1}=-\frac{\varphi''(0)}{2d\varphi'(0)^{2}}.
\end{align*}
En posant $\beta(z):=\sum_{1\le i<j<k\le d}\beta_{ijk}(z)$ \hspace{2mm}et\hspace{2mm} $\tilde{\beta}(z):=\sum_{1\le i<j<k\le d}\tilde{\beta}_{ijk}(z),$ on obtient que
\begin{align*}
\hspace{-4.7cm}\beta(p)=\left(p-\Gunderline(r)\right)^{-1-\frac{1}{d}}\tilde{\beta}\left((p-\Gunderline(r))^{\frac{1}{d}}\right).
\end{align*}
Comme $K(\Leg\mathcal{H})=\mathrm{d}\hspace{0.1mm}\eta(\mathrm{Leg}\mathcal{H})=\dfrac{\beta'(p)}{q}\mathrm{d}p\wedge\mathrm{d}q$ et comme $\beta(p)\in\mathbb{C}\{p-\Gunderline(r)\}\Big[\dfrac{1}{p-\Gunderline(r)}\Big],$ on en déduit que $K(\Leg\mathcal{H})$ est holomorphe le long de $T'=(p=\Gunderline(r))$ si et seulement si $\tilde{\beta}(z)\in\raisebox{0.3mm}{$z$}\hspace{0.15mm}\mathbb{C}\{\raisebox{0.3mm}{$z^{d}$}\}$ satisfait la condition
\begin{align*}
&0=\tilde{\beta}'(0)=\sum_{1\le i<j<k\le d}\tilde{\beta}'_{ijk}(0)=-\binom{{d}}{{3}}\frac{\varphi''(0)}{2d\varphi'(0)^{2}},
\end{align*}
\textit{i.e.} si et seulement si $\varphi''(0)=0.$

\noindent D'après le Lemme~\ref{lem:critique-maximale}, le fait que $z=r$ est un point critique (non fixe) de $\Gunderline$ de multiplicité $d-1$ se traduit par $-A(1,z)=\Gunderline(r)B(1,z)+c(z-r)^{d},$ pour un certain $c\in\mathbb{C}^{*}.$ Par suite
\begin{align*}
\hspace{0.15cm}A(x,y)=-\Gunderline(r)B(x,y)-c(y-rx)^{d} \qquad\text{et}\qquad B(x,y)=b_{0}x^{d}+\sum_{i=1}^{d}b_{i}(y-rx)^{i}x^{d-i}.
\end{align*}
Puisque $b_{0}=B(1,r)\neq0,$ on peut supposer sans perte de généralité que $b_{0}=1.$ Ainsi
\begin{align*}
\mathrm{d}\omega\Big|_{y=rx}=\left(d+b_{1}(\Gunderline(r)-r)\right)x^{d-1}\mathrm{d}x\wedge\mathrm{d}y.
\end{align*}
D'autre part, $\Gunderline(z)=\Gunderline(r)+\dfrac{c(z-r)^{d}}{1+b_{1}(z-r)+\cdots}$ et, pour tout $p\in\mathbb{P}^{1}_{\mathbb{C}}$ suffisamment voisin de $\Gunderline(r),$ l'équation $\Gunderline(z)=p$ est équivalente à
\begin{align*}
\hspace{-1.45cm}\left(p-\Gunderline(r)\right)^{\frac{1}{d}}=\frac{c^{\frac{1}{d}}(z-r)}{\sqrt[d]{1+b_{1}(z-r)+\cdots}}
=c^{\frac{1}{d}}(z-r)\left[1-\frac{1}{d}b_{1}(z-r)+\cdots\right].
\end{align*}
Par suite les $p_{i}(p)\in\Gunderline^{-1}(p)$ s'écrivent
\begin{small}
\begin{align*}
\hspace{-4.22cm}p_{i}(p)=r+\frac{1}{c^{\frac{1}{d}}}\zeta^{i}\left(p-\Gunderline(r)\right)^{\frac{1}{d}}+
\frac{b_{1}}{dc^{\frac{2}{d}}}\zeta^{2i}\left(p-\Gunderline(r)\right)^{\frac{2}{d}}+\cdots
\end{align*}
\end{small}
\hspace{-1mm}et donc
\begin{small}
\begin{align*}
\hspace{0.8cm}p-p_{i}(p)=\left(\Gunderline(r)-r\right)-\frac{1}{c^{\frac{1}{d}}}\zeta^{i}\left(p-\Gunderline(r)\right)^{\frac{1}{d}}-
\frac{b_{1}}{dc^{\frac{2}{d}}}\zeta^{2i}\left(p-\Gunderline(r)\right)^{\frac{2}{d}}+\cdots+\left(p-\Gunderline(r)\right)+\cdots.
\end{align*}
\end{small}
\hspace{-1mm}Par conséquent
\begin{align*}
\hspace{0.65cm}\lambda_{i}(p)=\dfrac{1}{p-p_{i}(p)}=\frac{1}{\Gunderline(r)-r}+\varphi'(0)\zeta^{i}\left(p-\Gunderline(r)\right)^{\frac{1}{d}}
+\frac{\varphi''(0)}{2}\zeta^{2i}\left(p-\Gunderline(r)\right)^{\frac{2}{d}}+\cdots,
\end{align*}
avec
\begin{align*}
\hspace{0.44cm}\varphi'(0)=\frac{1}{c^{\frac{1}{d}}(\Gunderline(r)-r)^{2}}\neq0
\qquad\text{et}\qquad
\varphi''(0)=\frac{2}{dc^{\frac{2}{d}}(\Gunderline(r)-r)^{3}}\left[d+b_{1}(\Gunderline(r)-r)\right],
\end{align*}
ce qui termine la démonstration.
\end{proof}
\newpage
\hfill
\medskip

\noindent Comme conséquence immédiate des Théorèmes \ref{thm:holomo-G(I^tr)}, \ref{thm:Barycentre}, \ref{thm:Divergence} et de la Remarque \ref{rem:Q(b,-a,a,b)} nous obtenons la caractérisation suivante de la platitude de la transformée de \textsc{Legendre} d'un feuilletage homogène de degré $3$ sur le plan projectif complexe.

\begin{cor}\label{cor:platitude-degre-3}
{\sl Soit $\mathcal{H}$ un feuilletage homogène de degré $3$ sur $\pp$ défini par la $1$-forme $$\omega=A(x,y)\mathrm{d}x+B(x,y)\mathrm{d}y,\quad A,B\in\mathbb{C}[x,y]_3,\hspace{2mm}\pgcd(A,B)=1.$$ Alors, le $3$-tissu $\Leg\mathcal{H}$ est plat si et seulement si les deux conditions suivantes sont satisfaites:
\begin{itemize}
\item [\hspace{-2mm}\textit{(1)}] pour toute droite d'inflexion de $\mathcal{H}$ transverse et simple $T_1=(ax+by=0),$ la droite d'équation $A(b,-a)x+B(b,-a)y=0$ est invariante par $\mathcal{H};$

\item [\hspace{-2mm}\textit{(2)}] pour toute droite d'inflexion de $\mathcal{H}$ transverse et double $T_{2},$ la $2$-forme $\mathrm{d}\omega$ s'annule sur $T_{2}.$
\end{itemize}
En particulier, si le feuilletage $\mathcal{H}$ est convexe alors $\Leg\mathcal{H}$ est plat.
}
\end{cor}
\smallskip

\noindent Ce corollaire jouera un rôle important au Chapitre~\ref{chap3:classification}.
\bigskip
\bigskip
\bigskip

\section{Platitude et feuilletages homogènes de type appartenant à $\Z\left[\mathrm{R}_1,\mathrm{R}_2,\ldots,\mathrm{R}_{d-1},\mathrm{T}_1,\mathrm{T}_{d-1}\right]$}\label{sec:type-min-max}
\vspace{6mm}

Nous nous proposons dans ce paragraphe de décrire certaines feuilletages homogènes de degré $d\geq3$ sur $\pp$, de type appartenant à $\Z\left[\mathrm{R}_1,\mathrm{R}_2,\ldots,\mathrm{R}_{d-1},\mathrm{T}_1,\mathrm{T}_{d-1}\right]$ et dont le $d$-tissu dual est plat. Nous considérons ici un feuilletage homogène $\mathcal{H}$ de degré $d\geq3$ sur $\pp$ défini, en carte affine $(x,y),$ par $$\omega=A(x,y)\mathrm{d}x+B(x,y)\mathrm{d}y,\quad A,B\in\mathbb{C}[x,y]_d,\hspace{2mm}\pgcd(A,B)=1.$$ L'application rationnelle $\Gunderline\hspace{1mm}\colon\mathbb{P}^{1}_{\mathbb{C}}\rightarrow \mathbb{P}^{1}_{\mathbb{C}}$, $\Gunderline(z)=-\dfrac{A(1,z)}{B(1,z)},$ nous sera très utile pour établir les énoncés qui suivent.

\begin{pro}\label{pro:omega1-omega2}
{\sl Si $\deg\mathcal{T}_{\mathcal{H}}=2,$ alors le $d$-tissu $\Leg\mathcal{H}$ est plat si et seulement si $\mathcal{H}$ est linéairement conjugué à l'un des deux feuilletages $\mathcal{H}_{1}^{d}$ et $\mathcal{H}_{2}^{d}$ décrits respectivement par les $1$-formes
\begin{itemize}
\item [\texttt{1. }] $\omega_1^{\hspace{0.2mm}d}=y^d\mathrm{d}x-x^d\mathrm{d}y\hspace{0.5mm};$
\item [\texttt{2. }] $\omega_2^{\hspace{0.2mm}d}=x^d\mathrm{d}x-y^d\mathrm{d}y.$
\end{itemize}
}
\end{pro}

\begin{proof}[\sl D\'emonstration]
L'égalité $\deg\mathcal{T}_{\mathcal{H}}=2$ est réalisée si et seulement si nous sommes dans l'une des situations suivantes
\begin{itemize}
\item [(i)] $\mathcal{T}_{\mathcal{H}}=2\cdot\mathrm{R}_{d-1}\hspace{1mm};$

\item [(ii)] $\mathcal{T}_{\mathcal{H}}=2\cdot\mathrm{T}_{d-1}\hspace{1mm};$

\item [(iii)] $\mathcal{T}_{\mathcal{H}}=1\cdot\mathrm{R}_{d-1}+1\cdot\mathrm{T}_{d-1}.$
\end{itemize}
Commençons par étudier l'éventualité (i). Nous pouvons supposer à conjugaison près que les deux singularités radiales de $\mathcal{H}$ sont $[0:1:0]$ et $[1:0:0]$, ce qui revient à supposer que les points $\infty=[1:0],\,[0:1]\in\mathbb{P}^{1}_{\mathbb{C}}$ sont critiques fixes de $\Gunderline$, de même multiplicité $d-1$. Cela se traduit par le fait que $A(x,y)=ay^d$ et $B(x,y)=bx^d$, avec $ab\neq0,$ en vertu du Lemme~\ref{lem:critique-maximale}. Par suite $\omega=ay^d\mathrm{d}x-(-b)x^d\mathrm{d}y$ et nous pouvons évidemment normaliser les coefficients $a$ et $-b$ à $1.$ Ainsi $\mathcal{H}$ est conjugué au feuilletage $\mathcal{H}_{1}^{d}$ décrit par $\omega_1^{\hspace{0.2mm}d}=y^d\mathrm{d}x-x^d\mathrm{d}y$; le $d$-tissu $\Leg\mathcal{H}_{1}^{d}$ est plat car $\mathcal{H}_{1}^{d}$ est convexe.

\noindent Intéressons-nous à la possibilité (ii). \`{A} isomorphisme linéaire près nous pouvons nous ramener à la situation suivante:
\begin{itemize}
\item [$\bullet$] les points $[0:1],\,[1:1]\in\mathbb{P}^{1}_{\mathbb{C}}$ sont critiques non fixes de $\Gunderline$, de même multiplicité $d-1$;

\item [$\bullet$] $\Gunderline(0)$ et $\Gunderline(1)\neq\infty.$
\end{itemize}
Toujours d'après le Lemme~\ref{lem:critique-maximale}, il existe des constantes $\alpha,\beta\in\mathbb{C}^*$ telles que
\begin{align*}
-A(1,z)=\Gunderline(0)B(1,z)+\alpha\hspace{0.1mm}z^d=\Gunderline(1)B(1,z)+\beta(z-1)^d
\end{align*}
avec $\Gunderline(0)\neq0,\hspace{0.5mm}\Gunderline(1)\neq1$\hspace{0.5mm} et \hspace{0.5mm}$\Gunderline(0)\neq\Gunderline(1).$ L'homogénéité de $A$ et $B$ entraîne alors que
\begin{align*}
\quad \omega=\left(\Gunderline(0)\hspace{0.2mm}s\hspace{0.2mm}(y-x)^d-g(1)\hspace{0.2mm}r\hspace{0.2mm}y^d\right)\mathrm{d}x+\left(ry^d-s(y-x)^d\right)\mathrm{d}y
\end{align*}
avec $r=\dfrac{\alpha}{\Gunderline(1)-\Gunderline(0)}\neq0$\hspace{1mm} et \hspace{1mm}$s=\dfrac{\beta}{\Gunderline(1)-\Gunderline(0)}\neq0.$ D'après les Théorèmes \ref{thm:holomo-G(I^tr)} et \ref{thm:Divergence}, le $d$-tissu $\Leg\mathcal{H}$ est plat si et seulement si $\mathrm{d}\omega$ s'annule sur les deux droites $y(y-x)=0.$ Un calcul immédiat montre que
\begin{align*}
\mathrm{d}\omega\Big|_{y=0}=-sd(\Gunderline(0)-1)x^{d-1}\mathrm{d}x\wedge\mathrm{d}y
\hspace{1mm}\quad\text{et}\hspace{1mm}\quad
\mathrm{d}\omega\Big|_{y=x}=rd\Gunderline(1)x^{d-1}\mathrm{d}x\wedge\mathrm{d}y.
\end{align*}
Ainsi $\Leg\mathcal{H}$ est plat si et seulement si $\Gunderline(0)=1$ et $\Gunderline(1)=0$, auquel cas $$\omega=s(y-x)^d\mathrm{d}x+\left(ry^d-s(y-x)^d\right)\mathrm{d}y\hspace{1mm};$$ quitte à remplacer $\omega$ par $\varphi^*\omega,$ où $\varphi(x,y)=\left(s^{\larger{\frac{-1}{d+1}}}x-r^{\larger{\frac{-1}{d+1}}}y,-r^{\larger{\frac{-1}{d+1}}}y\right),$ on se ramène à $$\omega=\omega_2^{\hspace{0.2mm}d}=x^d\mathrm{d}x-y^d\mathrm{d}y.$$

\noindent Considérons pour finir l'éventualité (iii). Nous pouvons supposer que la singularité radiale de~$\mathcal{H}$ est le point $[0:1:0]$ et que la droite d'inflexion transverse de $\mathcal{H}$ est la droite $(y=0)$; $\Gunderline(0)\neq\Gunderline(\infty)=\infty$ car $\Gunderline^{-1}(\Gunderline(0))=\{0\}.$ Un raisonnement analogue à celui du cas précédent conduit à
\begin{align*}
&\omega=-\left(\Gunderline(0)\beta\hspace{0.3mm}x^d+\alpha y^d\right)\mathrm{d}x+\beta\hspace{0.3mm}x^d\mathrm{d}y,
\qquad\text{avec}\qquad \alpha\beta\Gunderline(0)\neq0.
\end{align*}
\noindent La courbure du tissu associé à cette $1$-forme ne peut pas être holomorphe sur $\mathcal{G}_{\mathcal{H}}(\{y=0\})$ car
$$\mathrm{d}\omega\Big|_{y=0}=d\beta\hspace{0.3mm}x^{d-1}\mathrm{d}x\wedge\mathrm{d}y\not\equiv0\hspace{1mm};$$
il en résulte que $\Leg\mathcal{H}$ ne peut pas être plat lorsque $\mathcal{T}_{\mathcal{H}}=1\cdot\mathrm{R}_{d-1}+1\cdot\mathrm{T}_{d-1}.$
\end{proof}

\begin{pro}\label{pro:omega3-omega4}
{\sl Soit $\nu$ un entier compris entre $1$ et $d-2$. Si le feuilletage $\mathcal{H}$ est de type
$$
\mathcal{T}_{\mathcal{H}}=1\cdot\mathrm{R}_{\nu}+1\cdot\mathrm{R}_{d-\nu-1}+1\cdot\mathrm{R}_{d-1},
\qquad\text{resp}.\hspace{1.5mm}
\mathcal{T}_{\mathcal{H}}=1\cdot\mathrm{R}_{\nu}+1\cdot\mathrm{R}_{d-\nu-1}+1\cdot\mathrm{T}_{d-1},
$$
alors le $d$-tissu $\Leg\mathcal{H}$ est plat si et seulement si $\mathcal{H}$ est linéairement conjugué au feuilletage $\mathcal{H}_{3}^{d,\nu}$, resp. $\mathcal{H}_{4}^{d,\nu},$ donné par
\[
\hspace{-2cm}\omega_{3}^{\hspace{0.2mm}d,\nu}=\sum\limits_{i=\nu+1}^{d}\binom{{d}}{{i}}x^{d-i}y^i\mathrm{d}x-
\sum\limits_{i=0}^{\nu}\binom{{d}}{{i}}x^{d-i}y^i\mathrm{d}y,
\]
\[
\hspace{-0.9cm}\text{resp}.\hspace{1.5mm}
\omega_{4}^{\hspace{0.2mm}d,\nu}=(d-\nu-1)\sum\limits_{i=\nu+1}^{d}\binom{{d}}{{i}}x^{d-i}y^i\mathrm{d}x+
\nu\sum\limits_{i=0}^{\nu}\binom{{d}}{{i}}x^{d-i}y^i\mathrm{d}y.
\]
}
\end{pro}

\begin{proof}[\sl D\'emonstration]
Dans les deux cas, nous pouvons supposer à conjugaison linéaire près que les points $[0:1],\,[1:0],\,[-1:1]\in\mathbb{P}^{1}_{\mathbb{C}}$ sont critiques de $\Gunderline$, de multiplicité $\nu,$ $d-\nu-1,$ $d-1$ respectivement. Les points $[0:1]$ et $[1:0]$ sont évidemment fixes par $\Gunderline$; le feuilletage $\mathcal{H}$ est de type $\mathcal{T}_{\mathcal{H}}=1\cdot\mathrm{R}_{\nu}+1\cdot\mathrm{R}_{d-\nu-1}+1\cdot\mathrm{R}_{d-1}$ (resp.\hspace{1.5mm}$\mathcal{T}_{\mathcal{H}}=1\cdot\mathrm{R}_{\nu}+1\cdot\mathrm{R}_{d-\nu-1}+1\cdot\mathrm{T}_{d-1}$) si et seulement si le point $[-1:1]$ est fixe (resp. non fixe) par $\Gunderline.$ Puisque $\Gunderline^{-1}(\Gunderline(-1))=\{-1\}$ nous avons $\Gunderline(-1)\neq\Gunderline(\infty)=\infty$. Donc, d'après le Lemme~\ref{lem:critique-maximale}, il existe une constante $\alpha\in\mathbb{C}^*$ et un polynôme homogène $B_{\nu}\in\mathbb{C}[x,y]_\nu$ tels que
\begin{align*}
-A(x,y)=\Gunderline(-1)B(x,y)+\alpha(y+x)^d, \quad B(x,y)=x^{d-\nu}B_{\nu}(x,y)
\quad\text{et}\quad
y^{\nu+1}\hspace{0.5mm}\text{divise}\hspace{1.3mm}A(x,y).
\end{align*}
Il en résulte que
\begin{eqnarray*}
-A(x,y)&=&\Gunderline(-1)x^{d-\nu}B_{\nu}(x,y)+\alpha\sum_{i=0}^{d}\binom{{d}}{{i}}x^{d-i}y^{i}
\nonumber\\
&=&\Gunderline(-1)x^{d-\nu}B_{\nu}(x,y)+\alpha\sum_{i=0}^{\nu}\binom{{d}}{{i}}x^{d-i}y^{i}+\alpha\sum_{i=\nu+1}^{d}\binom{{d}}{{i}}x^{d-i}y^{i}
\nonumber\hspace{1mm};
\end{eqnarray*}
par suite $A(x,y)$ est divisible par $y^{\nu+1}$ si et seulement si
\begin{align*}
-A(x,y)=\alpha\sum_{i=\nu+1}^{d}\binom{{d}}{{i}}x^{d-i}y^{i}
\hspace{1mm}\quad\text{et}\hspace{1mm}\quad
\Gunderline(-1)x^{d-\nu}B_{\nu}(x,y)+\alpha\sum_{i=0}^{\nu}\binom{{d}}{{i}}x^{d-i}y^{i}=0.
\end{align*}
Quitte à remplacer $\omega=A(x,y)\mathrm{d}x+B(x,y)\mathrm{d}y$ par\hspace{0.2mm} $-\dfrac{1}{\alpha}\omega\hspace{0.5mm}$ on se ramène à
\begin{align*}
& \omega=\sum_{i=\nu+1}^{d}\binom{{d}}{{i}}x^{d-i}y^{i}\mathrm{d}x+\frac{1}{\Gunderline(-1)}\sum_{i=0}^{\nu}\binom{{d}}{{i}}x^{d-i}y^{i}\mathrm{d}y,
\hspace{1mm}\quad \Gunderline(-1)\neq\Gunderline(0)=0.
\end{align*}
\begin{itemize}
  \item [$\bullet$] Si $\Gunderline(-1)=-1$ nous obtenons le feuilletage $\mathcal{H}_{3}^{d,\nu}$ décrit par
$$\omega_{3}^{\hspace{0.2mm}d,\nu}=\sum\limits_{i=\nu+1}^{d}\binom{{d}}{{i}}x^{d-i}y^i\mathrm{d}x-
\sum\limits_{i=0}^{\nu}\binom{{d}}{{i}}x^{d-i}y^i\mathrm{d}y\hspace{1mm};$$
la transformée de \textsc{Legendre} $\Leg\mathcal{H}_{3}^{d,\nu}$ est plate car $\mathcal{H}_{3}^{d,\nu}$ est convexe.
  \item [$\bullet$] Si $\Gunderline(-1)\neq-1$ alors, d'après les Théorèmes \ref{thm:holomo-G(I^tr)} et \ref{thm:Divergence}, le $d$-tissu $\Leg\mathcal{H}$ est plat si et seulement si
$$\hspace{1.5cm}0\equiv\mathrm{d}\omega\Big|_{y=-x}=
\binom{{d}}{{\nu+1}}\dfrac{(-1)^{\nu+1}(\nu+1)}{(d-1)\Gunderline(-1)}\left[\Gunderline(-1)\nu-d+\nu+1\right]x^{d-1}\mathrm{d}x\wedge\mathrm{d}y,$$
{\it i.e.} si et seulement si $\Gunderline(-1)=\dfrac{d-\nu-1}{\nu},$ auquel cas
$$(d-\nu-1)\omega=\omega_{4}^{\hspace{0.2mm}d,\nu}=(d-\nu-1)\sum\limits_{i=\nu+1}^{d}\binom{{d}}{{i}}x^{d-i}y^i\mathrm{d}x+
\nu\sum\limits_{i=0}^{\nu}\binom{{d}}{{i}}x^{d-i}y^i\mathrm{d}y.$$
\end{itemize}
\end{proof}

\begin{pro}\label{pro:omega5-omega6}
{\sl Si le feuilletage $\mathcal{H}$ est de type
$$
\mathcal{T}_{\mathcal{H}}=1\cdot\mathrm{R}_{d-2}+1\cdot\mathrm{T}_{1}+1\cdot\mathrm{R}_{d-1},
\qquad\text{resp}.\hspace{1.5mm}
\mathcal{T}_{\mathcal{H}}=1\cdot\mathrm{R}_{d-2}+1\cdot\mathrm{T}_{1}+1\cdot\mathrm{T}_{d-1},
$$
alors le $d$-tissu $\Leg\mathcal{H}$ est plat si et seulement si $\mathcal{H}$ est linéairement conjugué au feuilletage $\mathcal{H}_{5}^{d}$, resp. $\mathcal{H}_{6}^{d},$ décrit par
\[
\hspace{-4.2cm}\omega_{5}^{\hspace{0.2mm}d}=2y^d\mathrm{d}x+x^{d-1}(yd-(d-1)x)\mathrm{d}y,
\]
\[
\hspace{1cm}\text{resp}.\hspace{1.5mm}
\omega_{6}^{\hspace{0.2mm}d}=\left((d-1)^2x^d-d(d-1)x^{d-1}y+(d+1)y^d\right)\mathrm{d}x+x^{d-1}\left(yd-(d-1)x\right)\mathrm{d}y.
\]
}
\end{pro}
\begin{proof}[\sl D\'emonstration]
Nous allons traiter ces deux types simultanément. \`{A} isomorphisme linéaire près, nous pouvons nous ramener à la situation suivante: les points $[1:0],\,[1:1],\,[0:1]\in\mathbb{P}^{1}_{\mathbb{C}}$ sont critiques de $\Gunderline$, de multiplicité $d-2,$ $1,$ $d-1$ respectivement. Le point $[1:0]$ (resp.  $[1:1]$) est fixe (resp. non fixe) par $\Gunderline$; le feuilletage $\mathcal{H}$ est de type $\mathcal{T}_{\mathcal{H}}=1\cdot\mathrm{R}_{d-2}+1\cdot\mathrm{T}_{1}+1\cdot\mathrm{R}_{d-1}$ (resp.\hspace{1.5mm}$\mathcal{T}_{\mathcal{H}}=1\cdot\mathrm{R}_{d-2}+1\cdot\mathrm{T}_{1}+1\cdot\mathrm{T}_{d-1}$) si et seulement si le point $[0:1]$ est fixe (resp. non fixe) par $\Gunderline.$ Puisque $\Gunderline^{-1}(\Gunderline(0))=\{0\}$ nous avons $\Gunderline(0)\neq\Gunderline(1)$ et $\Gunderline(0)\neq\Gunderline(\infty)=\infty$; de plus $\Gunderline(1)\neq\Gunderline(\infty)=\infty$ car $\Gunderline^{-1}(\Gunderline(\infty))=\{\infty,z_{0}\}$ pour un certain point $z_{0}\neq\infty,$ non critique de $\Gunderline.$ Par suite, d'après le Lemme~\ref{lem:critique-maximale}, il existe une constante $\alpha\in\mathbb{C}^*$ telle que
\begin{align*}
-A(x,y)=\Gunderline(0)B(x,y)+\alpha y^d
\qquad\text{et}\qquad
B(x,y)=\frac{\raisebox{-0.4mm}{$\alpha$}}{\raisebox{0.4mm}{$s$}}x^{d-1}\left(yd-(d-1)x\right),
\end{align*}
avec $s=\Gunderline(1)-\Gunderline(0)\neq0.$
Quitte à multiplier $\omega=A(x,y)\mathrm{d}x+B(x,y)\mathrm{d}y$ par $\dfrac{\raisebox{-0.4mm}{$s$}}{\raisebox{0.4mm}{$\alpha$}}$ on se ramène à
\begin{align*}
& \omega=-\left(\Gunderline(0)\,x^{d-1}\left(yd-(d-1)x\right)+sy^d\right)\mathrm{d}x+x^{d-1}\left(yd-(d-1)x\right)\mathrm{d}y.
\end{align*}
D'après ce qui précède le point $[1:1]$ est le seul point critique de $\Gunderline$ dans sa fibre $\Gunderline^{-1}(\Gunderline(1)).$ Donc, d'après le Théorème~\ref{thm:Barycentre}, la courbure de $\Leg\mathcal{H}$ est holomorphe sur $\mathcal{G}_{\mathcal{H}}(\{y=x\})$ si et seulement si $$\hspace{2cm}0=Q(1,1;-1,1)=-\frac{1}{6}sd(d-1)(d-2)(\Gunderline(0)+s+2),$$ {\it i.e.} si et seulement si $s=-\Gunderline(0)-2.$
\begin{itemize}
  \item [$\bullet$] Si $\Gunderline(0)=0$ alors la condition $s=-\Gunderline(0)-2=-2$ est suffisante pour que $\Leg\mathcal{H}$ soit plat, en vertu du Théorème~\ref{thm:holomo-G(I^tr)}, auquel cas $$\omega=\omega_{5}^{\hspace{0.2mm}d}=2y^d\mathrm{d}x+x^{d-1}(yd-(d-1)x)\mathrm{d}y.$$

  \item [$\bullet$] Si $\Gunderline(0)\neq0$ alors, d'après les Théorèmes \ref{thm:holomo-G(I^tr)} et \ref{thm:Divergence}, $\Leg\mathcal{H}$ est plat si et seulement si
      $$
      \hspace{1cm}
      s=-\Gunderline(0)-2
      \quad\hspace{1cm} \text{et} \quad\hspace{1cm}
      0\equiv\mathrm{d}\omega\Big|_{y=0}=d(\Gunderline(0)-d+1)x^{d-1}\mathrm{d}x\wedge\mathrm{d}y,
      $$
{\it i.e.} si et seulement si $\Gunderline(0)=d-1$\, et \,$s=-d-1$, auquel cas
$$\hspace{1cm}\omega=\omega_{6}^{\hspace{0.2mm}d}=\left((d-1)^2x^d-d(d-1)x^{d-1}y+(d+1)y^d\right)\mathrm{d}x+x^{d-1}\left(yd-(d-1)x\right)\mathrm{d}y.$$
\end{itemize}
\end{proof}

\clearemptydoublepage
\chapter{Feuilletages de degré trois de $\pp$ ayant une transformée de \textsc{Legendre} plate}\label{chap3:classification}
\hspace{6mm}

Dans ce chapitre, nous noterons $\mathbf{F}(d)$ l'ensemble des feuilletages de degré $d$ sur $\pp$; c'est un ouvert de \textsc{Zariski} de l'espace projectif $\mathbb{P}_{\C}^{(d+2)^{2}-2}$. Le groupe des automorphismes de $\pp$ agit sur $\mathbf{F}(d)$; l'orbite d'un élément $\F\in\mathbf{F}(d)$ sous l'action de $\mathrm{Aut}(\pp)=\mathrm{PGL}_3(\mathbb{C})$ sera notée $\mathcal{O}(\F).$ Nous désignerons par $\mathbf{FP}(d)$ le sous-ensemble de $\mathbf{F}(d)$ formé des $\F\in\mathbf{F}(d)$ tels que $\Leg\F$ soit plat; c'est un fermé de \textsc{Zariski} de $\mathbf{F}(d)$ et donc, si $\F\in\mathbf{FP}(d),$ alors l'adhérence $\overline{\mathcal{O}(\F)}$ de $\mathcal{O}(\F)$ dans $\mathbf{F}(d)$ est contenue dans $\mathbf{FP}(d).$

\noindent Ce chapitre comprend trois sections. Dans la première section, nous établissons la classification à automorphisme de $\pp$ pr\`es des feuilletages homogènes de $\mathbf{FP}(3).$ Dans la deuxième section, nous classifions complètement les feuilletages de~$\mathbf{FP}(3).$ Dans~la~troisième~section,~nous donnons une description des composantes irréductibles de $\mathbf{FP}(3).$

\section{Classification à automorphisme près des feuilletages homogènes de $\mathbf{FP}(3)$}

Cette section est consacrée à la démonstration de l'énoncé suivant.

\begin{thm}\label{thm:Class-Homog3-Plat}
{\sl \`A automorphisme de $\pp$ pr\`es, il y a onze feuilletages homogènes de degré trois $\mathcal{H}_1,\ldots,\mathcal{H}_{11}$ sur le plan projectif complexe ayant une transformée de \textsc{Legendre} plate. Ils sont d\'ecrits respectivement en carte affine par les $1$-formes suivantes
\begin{itemize}
\item [\texttt{1. }] \hspace{1mm}$\omega_1\hspace{1mm}=y^3\mathrm{d}x-x^3\mathrm{d}y$;
\smallskip
\item [\texttt{2. }] \hspace{1mm}$\omega_2\hspace{1mm}=x^3\mathrm{d}x-y^3\mathrm{d}y$;
\smallskip
\item [\texttt{3. }] \hspace{1mm}$\omega_3\hspace{1mm}=y^2(3x+y)\mathrm{d}x-x^2(x+3y)\mathrm{d}y$;
\smallskip
\item [\texttt{4. }] \hspace{1mm}$\omega_4\hspace{1mm}=y^2(3x+y)\mathrm{d}x+x^2(x+3y)\mathrm{d}y$;
\smallskip
\item [\texttt{5. }] \hspace{1mm}$\omega_5\hspace{1mm}=2y^3\mathrm{d}x+x^2(3y-2x)\mathrm{d}y$;
\smallskip
\item [\texttt{6. }] \hspace{1mm}$\omega_6\hspace{1mm}=(4x^3-6x^2y+4y^3)\mathrm{d}x+x^2(3y-2x)\mathrm{d}y$;
\newpage
\hfill
\bigskip
\item [\texttt{7. }] \hspace{1mm}$\omega_7\hspace{1mm}=y^3\mathrm{d}x+x(3y^2-x^2)\mathrm{d}y$;
\smallskip
\item [\texttt{8. }] \hspace{1mm}$\omega_8\hspace{1mm}=x(x^2-3y^2)\mathrm{d}x-4y^3\mathrm{d}y$;
\smallskip
\item [\texttt{9. }] \hspace{1mm}$\omega_9\hspace{1mm}=y^{2}\left((-3+\mathrm{i}\sqrt{3})x+2y\right)\mathrm{d}x+
                                                       x^{2}\left((1+\mathrm{i}\sqrt{3})x-2\mathrm{i}\sqrt{3}y\right)\mathrm{d}y$;
\smallskip
\item [\texttt{10. }] \hspace{-1mm}$\omega_{10}=(3x+\sqrt{3}y)y^2\mathrm{d}x+(3y-\sqrt{3}x)x^2\mathrm{d}y$;
\smallskip
\item [\texttt{11. }]  \hspace{-1mm}$\omega_{11}=(3x^3+3\sqrt{3}x^2y+3xy^2+\sqrt{3}y^3)\mathrm{d}x+(\sqrt{3}x^3+3x^2y+3\sqrt{3}xy^2+3y^3)\mathrm{d}y$.
\end{itemize}
}
\end{thm}
\medskip

\noindent Considérons un feuilletage homogène $\mathcal{H}$ de degré $3$ sur $\pp$ défini, en carte affine $(x,y)$, par
$$\omega=A(x,y)\mathrm{d}x+B(x,y)\mathrm{d}y,$$
où $A$ et $B$ désignent des polynômes homogènes de degré $3$ sans composante commune; la classification menant au Théorème~\ref{thm:Class-Homog3-Plat} est établie au cas par cas suivant que $\deg\mathcal{T}_{\mathcal{H}}=2,3$ ou~$4$, {\it i.e.} suivant la nature du support du diviseur $\Dtr=A_{x}B_{y}-A_{y}B_{x}$ qui peut être deux droites, trois droites ou quatre droites. Pour ce faire commençons par établir les deux lemmes suivants.

\begin{lem}\label{lem:2T1+1TR2}
{\sl Si $\mathcal{T}_{\mathcal{H}}=2\cdot\mathrm{T}_{1}+1\cdot\mathrm{R}_{2}$, resp. $\mathcal{T}_{\mathcal{H}}=2\cdot\mathrm{T}_{1}+1\cdot\mathrm{T}_{2}$, alors, à conjugaison linéaire près, la $1$-forme $\omega$ décrivant $\mathcal{H}$ est du type
\begin{small}
\[
\hspace{-3.25cm}\omega=y^3\mathrm{d}x+\left(\beta\,x^3-3\beta\,xy^2+\alpha\,y^3\right)\mathrm{d}y,\qquad \beta\left((2\beta-1)^2-\alpha^2\right)\neq0,
\]
\[
\text{resp}.\hspace{1.5mm}\omega=\left(x^3-3xy^2+\alpha\,y^3\right)\mathrm{d}x+\left(\delta\,x^3-3\delta\,xy^2+\beta\,y^3\right)\mathrm{d}y,\qquad
(\beta-\alpha\delta)\left((\beta-2)^2-(\alpha-2\delta)^2\right)\neq0.
\]
\end{small}
}
\end{lem}
\vspace{-8mm}

\begin{proof}[\sl D\'emonstration] \`{A} isomorphisme près nous pouvons nous ramener à $\Dtr=cy^{2}(y-x)(y+x)$ pour un certain $c\in \mathbb{C}^*$. Le produit $\Cinv(1,1)\Cinv(1,-1)$ est évidemment non nul; $\mathcal{H}$ est de type $\mathcal{T}_{\mathcal{H}}=2\cdot\mathrm{T}_{1}+1\cdot\mathrm{R}_{2}$ (resp. $\mathcal{T}_{\mathcal{H}}=2\cdot\mathrm{T}_{1}+1\cdot\mathrm{T}_{2}$) si et seulement si $\Cinv(1,0)=0$ (resp. $\Cinv(1,0)\neq0$).
Écrivons les coefficients $A$ et $B$ de $\omega$ sous la forme
$$A(x,y)=a_{0}x^3+a_{1}x^2y+a_{2}xy^2+a_{3}y^3\qquad \text{et} \qquad B(x,y)=b_{0}x^3+b_{1}x^2y+b_{2}xy^2+b_{3}y^3\hspace{1mm};$$
nous avons donc
\begin{align*}
\hspace{-1.3cm}\Cinv=a_0x^4+(a_1+b_0)x^3y+(a_2+b_1)x^2y^2+(a_3+b_2)xy^3+b_3y^4
\end{align*}
et
\begin{align*}
&\hspace{1.2cm}\Dtr=(a_0b_1-a_1b_0)x^4+2(a_0b_2-a_2b_0)x^3y+(3a_0b_3+a_1b_2-a_2b_1-3a_3b_0)x^2y^2\\
&\hspace{2.1cm}+2(a_1b_3-a_3b_1)xy^3+(a_2b_3-a_3b_2)y^4.
\end{align*}
Ainsi $\Cinv(1,0)=a_0$\, et
\newpage
\hfill
\vspace{-9mm}

\begin{equation}\label{equa:2T1+1TR2}
\Dtr=cy^{2}(y-x)(y+x)\hspace{4mm}\Leftrightarrow\hspace{4mm}\left\{\begin{array}{lllll}
a_0\hspace{0.2mm}b_1=a_1b_0\\
a_0\hspace{0.2mm}b_2=a_2b_0\\
a_1b_3=a_3b_1\\
a_2b_3-a_3b_2=c\\
3a_0b_3+a_1b_2-a_2b_1-3a_3b_0=-c
\end{array}\right.
\end{equation}

\begin{itemize}
  \item [$\bullet$] Si $a_0\neq0$ alors le système (\ref{equa:2T1+1TR2}) équivaut à
$$a_1=0,\qquad a_2=-3a_0,\qquad b_1=0,\qquad b_2=-3b_0,\qquad c=-3(a_0b_3-a_3b_0).$$
Posons $a_3=a_0\alpha,\quad b_0=a_0\delta,\quad b_3=a_0\beta$; alors, quitte à diviser $\omega$ par $a_0$, cette forme s'écrit
$$\omega=\left(x^3-3xy^2+\alpha\,y^3\right)\mathrm{d}x+\left(\delta\,x^3-3\delta\,xy^2+\beta\,y^3\right)\mathrm{d}y\hspace{1mm};$$
un calcul direct montre que la condition $c\,\Cinv(1,1)\Cinv(1,-1)\neq0$ est vérifiée si et seulement si $(\beta-\alpha\delta)\left((\beta-2)^2-(\alpha-2\delta)^2\right)\neq0.$
  \item [$\bullet$] Si $a_0=0$ alors le système (\ref{equa:2T1+1TR2}) conduit à
  $$a_1=a_2=b_1=0,\quad b_2=-3b_0,\quad c=3a_3b_0\neq0.$$ Écrivons $b_0=a_3\beta$\, et \,$b_3=a_3\alpha$; alors, quitte à remplacer $\omega$ par $\dfrac{1}{a_3}\omega,$ on se ramène à $$\omega=y^3\mathrm{d}x+\left(\beta\,x^3-3\beta\,xy^2+\alpha\,y^3\right)\mathrm{d}y,$$ et la non nullité du produit $c\,\Cinv(1,1)\Cinv(1,-1)$ est équivalente à $\beta\left((2\beta-1)^2-\alpha^2\right)\neq0.$
\end{itemize}
\end{proof}

\begin{lem}\label{lem:deg-type=4}
{\sl Si le diviseur $\Dtr$ est réduit, {\it i.e.} si $\deg\mathcal{T}_\mathcal{H}=4$, alors $\omega$ est, à conjugaison linéaire près, de l'une des formes suivantes
\begin{itemize}
\item [\texttt{1.}] $y^2\left((2r+3)x-(r+2)y\right)\mathrm{d}x-x^2(x+ry)\mathrm{d}y$,\\
                    \hspace{1cm} \text{où}\hspace{1mm}
                    \begin{small}$r(r+1)(r+2)(r+3)(2r+3)\neq0$\end{small};
                    \vspace{1mm}
\item [\texttt{2.}] $s\hspace{0.12mm}y^2\left((2r+3)x-(r+2)y\right)\mathrm{d}x-x^2(x+ry)\mathrm{d}y$,\\
                    \hspace{1cm} \text{où}\hspace{1mm}
                    \begin{small}$rs(s-1)(r+1)(r+2)(r+3)(2r+3)\left(s(2r+3)^2-r^2\right)\neq0$\end{small};
                    \vspace{1mm}
\item [\texttt{3.}] $t\hspace{0.12mm}y^2\left((2r+3)x-(r+2)y\right)\mathrm{d}x-x^2(x+ry)\mathrm{d}(sy-x)$,\\
                    \hspace{1cm} \text{où}\hspace{1mm}
                    \begin{small}$r\hspace{0.17mm}s\hspace{0.17mm}t(r+1)(r+2)(r+3)(2r+3)(s-t-1)(tu^3-r^2su-r^2v)\neq0$\end{small},\hspace{1mm} \begin{small}$u=2r+3$\end{small}
                    \hspace{1mm}\text{et}\hspace{1mm}
                    \begin{small}$v=r(r+2)$\end{small};
                    \vspace{-5mm}
\item [\texttt{4.}] $uy^2\left((2r+3)x-(r+2)y\right)\mathrm{d}(y-sx)-x^2(x+ry)\mathrm{d}(ty-x)$,\\
                    \hspace{1cm} \text{où}\hspace{1mm}
                    \begin{small}$ur(r+1)(r+2)(r+3)(2r+3)(st-1)(su+t-u-1)(uv^4+suwv^3+r^2twv+r^2w^2)\neq0$\end{small},\\
                    \begin{small}$v=2r+3$\end{small}
                    \hspace{2mm}\text{et}\hspace{2mm}
                    \begin{small}$w=r(r+2)$\end{small}.
\end{itemize}
Ces quatre modèles sont respectivement de types
\begin{small}
$3\cdot\mathrm{R}_1+1\cdot\mathrm{T}_1,\hspace{1mm}2\cdot\mathrm{R}_1+2\cdot\mathrm{T}_1,\hspace{1mm}1\cdot\mathrm{R}_1+3\cdot\mathrm{T}_1,\hspace{1mm}4\cdot\mathrm{T}_1.$
\end{small}
}
\end{lem}

\begin{proof}[\sl D\'emonstration]
D'après la Remarque \ref{rem:2d-2-R1} le feuilletage $\mathcal{H}$ ne peut être de type $4\cdot\mathrm{R}_{1};$ nous sommes donc dans l'une des situations suivantes
\begin{itemize}
\item [(i)] $\mathcal{T}_{\mathcal{H}}=3\cdot\mathrm{R}_{1}+1\cdot\mathrm{T}_{1}\hspace{1mm};$

\item [(ii)] $\mathcal{T}_{\mathcal{H}}=2\cdot\mathrm{R}_{1}+2\cdot\mathrm{T}_{1};$

\item [(iii)] $\mathcal{T}_{\mathcal{H}}=1\cdot\mathrm{R}_{1}+3\cdot\mathrm{T}_{1};$

\item [(iv)] $\mathcal{T}_{\mathcal{H}}=4\cdot\mathrm{T}_{1}.$
\end{itemize}
\`{A} conjugaison linéaire près nous pouvons nous ramener à $\Dtr=cxy(y-x)(y-\alpha x)$ pour certains $c,\alpha\in \mathbb{C}^*,\alpha\neq1$. Dans la dernière éventualité nous avons $$\Cinv(0,1)\Cinv(1,0)\Cinv(1,1)\Cinv(1,\alpha)\neq0$$ et dans les cas (i), resp. (ii), resp. (iii) nous pouvons supposer que
\begin{align*}
\left\{
\begin{array}[c]{c}
\Cinv(0,1)=0
\\
\Cinv(1,0)=0
\\
\Cinv(1,1)=0
\\
\Cinv(1,\alpha)\neq0
\end{array}
\right.
\hspace{2mm}
\text{resp.}\hspace{1.5mm}\left\{
\begin{array}[c]{c}
\Cinv(0,1)=0
\\
\Cinv(1,0)=0
\\
\Cinv(1,1)\neq0
\\
\Cinv(1,\alpha)\neq0
\end{array}
\right.
\hspace{2mm}
\text{resp.}\hspace{1.5mm}\left\{
\begin{array}[c]{c}
\Cinv(0,1)=0
\\
\Cinv(1,0)\neq0
\\
\Cinv(1,1)\neq0
\\
\Cinv(1,\alpha)\neq0
\end{array}
\right.
\end{align*}
Comme dans le lemme précédent, en écrivant
$$A(x,y)=a_{0}x^3+a_{1}x^2y+a_{2}xy^2+a_{3}y^3\qquad \text{et} \qquad B(x,y)=b_{0}x^3+b_{1}x^2y+b_{2}xy^2+b_{3}y^3$$
nous obtenons que
\begin{equation}\label{equa:deg-type=4}
\Dtr=cxy(y-x)(y-\alpha x)\hspace{4mm}\Leftrightarrow\hspace{4mm}\left\{\begin{array}{lllll}
a_0\hspace{0.2mm}b_1=a_1b_0\\
a_2b_3=a_3b_2\\
2(a_1b_3-a_3b_1)=c\\
2(a_0b_2-a_2b_0)=c\alpha\\
3a_0b_3+a_1b_2-a_2b_1-3a_3b_0=-c(\alpha+1)
\end{array}\right.
\end{equation}
Envisageons l'éventualité (iv). Comme $c\neq0,$ $a_0=\Cinv(1,0)\neq0$\, et \,$b_3=\Cinv(0,1)\neq0$, le système (\ref{equa:deg-type=4}) équivaut à
\begin{equation*}
\left\{
\begin{array}{lllll}
b_1=\dfrac{a_1b_0}{a_0}\\
a_2=\dfrac{a_3b_2}{b_3}\\
c=\dfrac{2a_1(a_0b_3-a_3b_0)}{a_0}\\
a_0\hspace{0.2mm}b_2-\alpha a_1b_3=0\\
(3a_0+2\alpha a_1+2a_1)b_3+a_1b_2=0
\end{array}
\right.
\hspace{4mm}\Leftrightarrow\hspace{4mm}
\left\{
\begin{array}{lllll}
b_1=\dfrac{a_1b_0}{a_0}\\
a_2=\dfrac{a_3a_1\alpha}{a_0}\\
c=\dfrac{2a_1(a_0b_3-a_3b_0)}{a_0}\\
b_2=\dfrac{a_1b_3\alpha}{a_0}\\
a_1(a_1+2a_0)\alpha+a_0(2a_1+3a_0)=0
\end{array}
\right.
\end{equation*}
Donc $a_1\neq0$ et puisque $\alpha\neq0$, le produit $(a_1+2a_0)(2a_1+3a_0)$ est non nul. Il s'en suit que
\begin{align*}
& a_2=-\dfrac{a_3(2a_1+3a_0)}{a_1+2a_0},&& b_1=\dfrac{a_1b_0}{a_0},&& b_2=-\dfrac{b_3(2a_1+3a_0)}{a_1+2a_0},\\
& \alpha=-\dfrac{a_0(2a_1+3a_0)}{a_1(a_1+2a_0)},&& c=\dfrac{2a_1(a_0b_3-a_3b_0)}{a_0}.
\end{align*}
Posons $r=\dfrac{a_1}{a_0},\quad s=-\dfrac{a_3}{b_3},\quad t=-\dfrac{b_0}{a_0},\quad u=-\dfrac{b_3}{a_1+2a_0}$; alors
\begin{align*}
\hspace{6mm}& b_0=-ta_0,&& b_1=-rta_0,&& b_2=(2r+3)ua_0,&& b_3=-u(r+2)a_0,\\
\hspace{6mm}& a_1=ra_0,&& a_2=-su(2r+3)a_0,&& a_3=su(r+2)a_0,\\
\hspace{6mm}& \alpha=-\frac{2r+3}{r(r+2)},&& c=2r(r+2)u(st-1)a_0^{2}.
\end{align*}
Quitte à remplacer $\omega$ par $\dfrac{1}{a_0}\omega,$ le coefficient $a_0$ vaut $1$ et $\omega$ s'écrit
\begin{eqnarray*}
\omega\hspace{-1mm}&=&\hspace{-1mm}\left(x^3+rx^2y-su(2r+3)xy^2+su(r+2)y^3\right)\mathrm{d}x+\left(-tx^3-rtx^2y+u(2r+3)xy^2-u(r+2)y^3\right)\mathrm{d}y
\\
\hspace{-1mm}&=&\hspace{-1mm}uy^2\left((2r+3)x-(r+2)y\right)\mathrm{d}(y-sx)-x^2(x+ry)\mathrm{d}(ty-x)
\hspace{1mm};
\end{eqnarray*}
un calcul direct montre que la condition $c\alpha(\alpha-1)\Cinv(0,1)\Cinv(1,0)\Cinv(1,1)\Cinv(1,\alpha)\neq0$ est équivalente à
\begin{align*}
ur(r+1)(r+2)(r+3)(2r+3)(st-1)(su+t-u-1)(uv^4+suwv^3+r^2twv+r^2w^2)\neq0
\end{align*}
avec $v=2r+3$\, et \,$w=r(r+2).$

\noindent Maintenant nous étudions la possibilité (iii). Dans ce cas nous avons $b_3=\Cinv(0,1)=0$ et $a_0=\Cinv(1,0)\neq0$; le système (\ref{equa:deg-type=4}) conduit à
\begin{align*}
& a_2=-\dfrac{a_3(2a_1+3a_0)}{a_1+2a_0},&& b_1=\dfrac{a_1b_0}{a_0},&& b_2=0,&& \alpha=-\dfrac{a_0(2a_1+3a_0)}{a_1(a_1+2a_0)},&& c=-\dfrac{2a_1a_3b_0}{a_0}.
\end{align*}
En posant $r=\dfrac{a_1}{a_0},\quad s=-\dfrac{b_0}{a_0}$\hspace{2mm} et \hspace{2mm}$t=-\dfrac{a_3}{a_1+2a_0},$ nous obtenons que
\begin{align*}
\hspace{6mm}& b_0=-sa_0,&& b_1=-rsa_0,&& b_2=b_3=0,&& c=-2rst(r+2)a_0^{2}, \\
\hspace{6mm}& a_1=ra_0,&& a_2=t(2r+3)a_0,&& a_3=-t(r+2)a_0,&& \alpha=-\frac{2r+3}{r(r+2)}.
\end{align*}
Quitte à diviser $\omega$ par $a_0$ on se ramène à
\begin{eqnarray*}
\omega\hspace{-1mm}&=&\hspace{-1mm}\left(x^3+rx^2y+t(2r+3)xy^2-t(r+2)y^3\right)\mathrm{d}x-sx^2(x+ry)\mathrm{d}y
\\
\hspace{-1mm}&=&t\hspace{0.12mm}y^2\left((2r+3)x-(r+2)y\right)\mathrm{d}x-x^2(x+ry)\mathrm{d}(sy-x),
\end{eqnarray*}
\newpage
\hfill

\noindent et la non nullité du produit $c\alpha(\alpha-1)\Cinv(1,0)\Cinv(1,1)\Cinv(1,\alpha)$ se traduit par
\begin{align*}
r\hspace{0.17mm}s\hspace{0.17mm}t(r+1)(r+2)(r+3)(2r+3)(s-t-1)(tu^3-r^2su-r^2v)\neq0,
\hspace{2mm}\text{avec}\hspace{2mm}
u=2r+3\hspace{2mm} \text{et} \hspace{2mm}v=r(r+2).
\end{align*}
\noindent Les deux premiers cas se traitent de façon analogue.
\end{proof}

\begin{proof}[\sl D\'emonstration du Théorème~\ref{thm:Class-Homog3-Plat}]
\textsl{Premier cas}: $\deg\mathcal{T}_{\mathcal{H}}=2.$ Dans ce cas le $3$-tissu $\Leg\mathcal{H}$ est plat si et seulement si la $1$-forme $\omega$ définissant $\mathcal{H}$ est linéairement conjuguée à l'une des deux $1$-formes
$$\omega_1=y^3\mathrm{d}x-x^3\mathrm{d}y\qquad\hspace{1.5mm} \text{et} \qquad\hspace{1.5mm} \omega_2=x^3\mathrm{d}x-y^3\mathrm{d}y.$$
C'est une application directe de la Proposition~\ref{pro:omega1-omega2} pour $d=3.$

\noindent\textsl{Second cas}: $\deg\mathcal{T}_{\mathcal{H}}=3.$
\begin{itemize}
  \item [$\bullet$] Si $\mathcal{T}_\mathcal{H}=2\cdot\mathrm{R}_1+1\cdot\mathrm{R}_2$, resp. $\mathcal{T}_\mathcal{H}=2\cdot\mathrm{R}_1+1\cdot\mathrm{T}_2$, alors, d'après la Proposition~\ref{pro:omega3-omega4}, $\Leg\mathcal{H}$ est plat si et seulement si $\omega$ est conjuguée à
      \[
        \hspace{1cm}\omega_{3}^{\hspace{0.2mm}3,1}=\,\sum\limits_{i=2}^{3}\binom{{3}}{{i}}x^{3-i}y^i\mathrm{d}x-
        \sum\limits_{i=0}^{1}\binom{{3}}{{i}}x^{3-i}y^i\mathrm{d}y\,=\,y^2(3x+y)\mathrm{d}x-x^2(x+3y)\mathrm{d}y\,=\,\omega_3,
      \]
      \[
        \text{resp}.\hspace{1.5mm}
        \omega_{4}^{\hspace{0.2mm}3,1}=\,\sum\limits_{i=2}^{3}\binom{{3}}{{i}}x^{3-i}y^i\mathrm{d}x+
        \sum\limits_{i=0}^{1}\binom{{3}}{{i}}x^{3-i}y^i\mathrm{d}y\,=\,y^2(3x+y)\mathrm{d}x+x^2(x+3y)\mathrm{d}y\,=\,\omega_4.
      \]
  \item [$\bullet$]  Si $\mathcal{T}_\mathcal{H}=1\cdot\mathrm{R}_1+1\cdot\mathrm{T}_1+1\cdot\mathrm{R}_2$, resp. $\mathcal{T}_\mathcal{H}=1\cdot\mathrm{R}_1+1\cdot\mathrm{T}_1+1\cdot\mathrm{T}_2$, alors, d'après la Proposition~\ref{pro:omega5-omega6}, $\Leg\mathcal{H}$ est plat si et seulement si $\omega$ est conjuguée à
      \[
        \hspace{-1.5cm}\omega_{5}^{\hspace{0.2mm}3}\,=\,2y^3\mathrm{d}x+x^2(3y-2x)\mathrm{d}y\,=\,\omega_5,
      \]
      \[
        \text{resp}.\hspace{1.5mm}\omega_{6}^{\hspace{0.2mm}3}\,=\,(4x^3-6x^2y+4y^3)\mathrm{d}x+x^2(3y-2x)\mathrm{d}y\,=\,\omega_6.
      \]
  \item [$\bullet$] Si $\mathcal{T}_\mathcal{H}=2\cdot\mathrm{T}_1+1\cdot\mathrm{R}_2$, alors, d'après le Lemme~\ref{lem:2T1+1TR2}, la $1$-forme $\omega$ est du type
      \[
       \omega=y^3\mathrm{d}x+\left(\beta\,x^3-3\beta\,xy^2+\alpha\,y^3\right)\mathrm{d}y,\qquad \beta\left((2\beta-1)^2-\alpha^2\right)\neq0,
      \]
      et dans ce cas nous avons $\ItrH=(y-x)(y+x)$. D'après le Corollaire~\ref{cor:platitude-degre-3}, le $3$-tissu $\Leg\mathcal{H}$ est plat si et seulement si $$\hspace{1cm} 0=Q(1,1;-1,1)=(2\beta+2-\alpha)\beta \qquad \text{et} \qquad 0=Q(1,-1;1,1)=-(2\beta+2+\alpha)\beta,$$ {\it i.e.} si et seulement si $\alpha=0$\, et \,$\beta=-1$, auquel cas $\omega=\omega_{\hspace{0.3mm}7}=y^3\mathrm{d}x+x(3y^2-x^2)\mathrm{d}y.$
 \item [$\bullet$] Dans ce deuxième cas, il ne nous reste plus qu'à traiter l'éventualité $\mathcal{T}_\mathcal{H}=2\cdot\mathrm{T}_1+1\cdot\mathrm{T}_2$. Toujours d'après le Lemme~\ref{lem:2T1+1TR2}, $\omega$ est, à conjugaison près, de la forme
     \[
      \hspace{1cm}\omega=\left(x^3-3xy^2+\alpha\,y^3\right)\mathrm{d}x+\left(\delta\,x^3-3\delta\,xy^2+\beta\,y^3\right)\mathrm{d}y,\hspace{3mm} (\beta-\alpha\delta)\left((\beta-2)^2-(\alpha-2\delta)^2\right)\neq0\hspace{1mm};
     \]
     comme $\ItrH=y^2(y-x)(y+x)$ le $3$-tissu $\Leg\mathcal{H}$ est plat si et seulement si
     \[
     \left\{
      \begin{array}[l]{l}
      0\equiv\mathrm{d}\omega\Big|_{y=0}=3\delta\,x^{2}\mathrm{d}x\wedge\mathrm{d}y
      \\
      0=Q(1,1;-1,1)=(4+\beta-2\alpha-2\delta)(\beta-\alpha\delta)
      \\
      0=Q(1,-1;1,1)=(4+\beta+2\alpha+2\delta)(\beta-\alpha\delta),
     \end{array}
     \right.
    \]
     en vertu du Corollaire~\ref{cor:platitude-degre-3}. Il s'en suit que $\Leg\mathcal{H}$ est plat si et seulement si $\alpha=\delta=0$\, et \,$\beta=-4$, auquel cas $\omega=\omega_8=x(x^2-3y^2)\mathrm{d}x-4y^3\mathrm{d}y.$
\end{itemize}

\noindent\textsl{Troisième cas}: $\deg\mathcal{T}_{\mathcal{H}}=4.$ Pour examiner la platitude dans ce dernier cas, nous allons appliquer le Corollaire~\ref{cor:platitude-degre-3} aux différents modèles du Lemme~\ref{lem:deg-type=4}.
\begin{itemize}
  \item [$\bullet$] Si $\mathcal{T}_\mathcal{H}=3\cdot\mathrm{R}_1+1\cdot\mathrm{T}_1$, alors $\omega$ est du type
  \[
    \omega=y^2\left((2r+3)x-(r+2)y\right)\mathrm{d}x-x^2(x+ry)\mathrm{d}y
  \]
  avec $r(r+1)(r+2)(r+3)(2r+3)\neq0$. Nous avons $\ItrH=sx+ty$ où $s=2r+3$\, et \,$t=r(r+2)$; par suite le $3$-tissu $\Leg\mathcal{H}$ est plat si et seulement si $$0=Q(t,-s\hspace{0.2mm};s,t)=r(r+1)^2(r+2)^2(r+3)(2r+3)\left[r^2+3r+3\right],$$ {\it i.e.} si et seulement si $r=-\dfrac{3}{2}\pm\mathrm{ i}\dfrac{\sqrt{3}}{2}$. Dans les deux cas la $1$-forme $\omega$ est linéairement conjuguée à $$\omega_9=y^{2}\left((-3+\mathrm{i}\sqrt{3})x+2y\right)\mathrm{d}x+x^{2}\left((1+\mathrm{i}\sqrt{3})x-2\mathrm{i}\sqrt{3}y\right)\mathrm{d}y\hspace{1mm};$$ en effet si $r=-\dfrac{3}{2}-\mathrm{ i}\dfrac{\sqrt{3}}{2},$\, resp. $r=-\dfrac{3}{2}+\mathrm{ i}\dfrac{\sqrt{3}}{2}$,\, alors
  $$
  \hspace{1cm}\omega_9=-(1+\mathrm{i}\sqrt{3})\omega,
  \hspace{2cm}\text{resp}.\hspace{1.5mm}
  \omega_9=-2\varphi^*\omega,\quad \text{où}\hspace{1.5mm} \varphi(x,y)=(y,x).
  $$
  \item [$\bullet$] Si $\mathcal{T}_\mathcal{H}=2\cdot\mathrm{R}_1+2\cdot\mathrm{T}_1$, alors $\omega$ est de la forme
  \[
   \omega=s\hspace{0.12mm}y^2\left((2r+3)x-(r+2)y\right)\mathrm{d}x-x^2(x+ry)\mathrm{d}y
  \]
   avec $rs(s-1)(r+1)(r+2)(r+3)(2r+3)\left(s(2r+3)^2-r^2\right)\neq0$. Posons $t=2r+3$ et $u=r(r+2)$; nous avons $\ItrH=(y-x)(tx+uy).$ Donc $\Leg\mathcal{H}$ est plat si et seulement si
  \[
   \hspace{2.5cm}
   \left\{
   \begin{array}[l]{l}
   0=Q(1,1;-1,1)=-s(r+1)^2\left[s(r+2)+1\right]
   \\
   0=Q(u,-t;\hspace{0.2mm}t,u)=rs(r+1)^2(r+2)^2(2r+3)\left[s(2r+3)^2+(r+2)r^2\right],
   \end{array}
   \right.
  \]
  {\it i.e.} si et seulement si $r=\pm\sqrt{3}$\, et \,$s=-2+r$, car $rs(r+1)(r+2)(2r+3)\neq0.$ Dans les deux cas $\omega$ est linéairement conjuguée à $$\omega_{10}=(3x+\sqrt{3}y)y^2\mathrm{d}x+(3y-\sqrt{3}x)x^2\mathrm{d}y\hspace{1mm};$$ en effet si $(r,s)=(-\sqrt{3},-2-\sqrt{3})$,\, resp. $(r,s)=(\sqrt{3},-2+\sqrt{3})$, alors
  $$
  \hspace{1cm}\omega_{10}=\sqrt{3}\omega,
  \hspace{2cm}\text{resp}.\hspace{1.5mm}
  \omega_{10}=-\sqrt{3}\hspace{0.2mm}\varphi^*\omega,\quad \text{où}\hspace{1.5mm} \varphi(x,y)=(x,-y).
  $$
  \item [$\bullet$] Si $\mathcal{T}_\mathcal{H}=1\cdot\mathrm{R}_1+3\cdot\mathrm{T}_1$, alors $\omega$ est du type
  \[
   \omega=t\hspace{0.12mm}y^2\left((2r+3)x-(r+2)y\right)\mathrm{d}x-x^2(x+ry)\mathrm{d}(sy-x)
  \]
     avec\hspace{1mm} $rst(r+1)(r+2)(r+3)(2r+3)(s-t-1)(tu^3-r^2su-r^2v)\neq0,$\hspace{1mm} $u=2r+3$\, et \,$v=r(r+2).$ Puisque $\ItrH=y(y-x)(ux+vy)$ la courbure de $\Leg\mathcal{H}$ est holomorphe le long de $\mathcal{G}_{\mathcal{H}}(\{y(y-x)=0\})$ si et seulement si
  \[
   \hspace{2.5cm}
   \left\{
   \begin{array}[l]{l}
   0=Q(1,0\hspace{0.2mm};0,1)=st\left[(2r+3)s-(r+2)\right]
   \\
   0=Q(1,1;-\hspace{0.2mm}1,1)=-st(r+1)^2\left[(r+2)(t+1)+s\right],
   \end{array}
   \right.
 \]
  {\it i.e.} si et seulement si $s=\dfrac{r+2}{2r+3}$\, et \,$t=-\dfrac{2(r+2)}{2r+3},$ auquel cas $K(\Leg\mathcal{H})$ ne peut être holomorphe sur $\mathcal{G}_{\mathcal{H}}(\{ux+vy=0\})$ car $$\hspace{1cm}Q(v,-u\hspace{0.2mm};u,v)=12\,r(r+1)^3(r+2)^5(r+3)(2r+3)^{-2}\neq0.$$ Par conséquent la transformée de \textsc{Legendre} $\Leg\mathcal{H}$ de $\mathcal{H}$ ne peut être plate lorsque $\mathcal{T}_\mathcal{H}=1\cdot\mathrm{R}_1+3\cdot\mathrm{T}_1.$
  \item [$\bullet$] Si $\mathcal{T}_\mathcal{H}=4\cdot\mathrm{T}_1$, alors $\omega$ est de la forme
  \[
   \omega=uy^2\left((2r+3)x-(r+2)y\right)\mathrm{d}(y-sx)-x^2(x+ry)\mathrm{d}(ty-x),
  \]
   où $ur(r+1)(r+2)(r+3)(2r+3)(st-1)(su+t-u-1)(uv^4+suwv^3+r^2twv+r^2w^2)\neq0$, $v=2r+3$\, et \,$w=r(r+2).$ Comme $\ItrH=xy(y-x)(vx+wy)$ la courbure de $\Leg\mathcal{H}$ est holomorphe le long de $\mathcal{G}_{\mathcal{H}}(\{xy(y-x)=0\})$ si et seulement si
  \[
   \hspace{2.5cm}
   \left\{
   \begin{array}[l]{l}
   0=Q(0,-1;\hspace{0.2mm}1,0)=-u^2(r+2)^2(st-1)\left[rs+1\right]
   \\
   0=Q(1,0\hspace{0.2mm};0,1)=-u(st-1)\left[(2r+3)t-r-2\right]
   \\
   0=Q(1,1;\hspace{0.2mm}-1,1)=-u(r+1)^2(st-1)\left[(rs+2s+1)u-t-r-2\right],
   \end{array}
   \right.
  \]
   {\it i.e.} si et seulement si $s=-\dfrac{1}{r}$,\, $t=\dfrac{r+2}{2r+3}$\, et \,$u=-\dfrac{r(r+2)^2}{2r+3},$ auquel cas $$\hspace{-1cm}Q(w,-v\hspace{0.2mm};v,w)=16r(r+1)^5(r+2)^5(r+3)(2r+3)^{-2}\left[r^2+3r+3\right].$$ Par suite $\Leg\mathcal{H}$ est plat si et seulement si nous sommes dans l'un des deux cas suivants
   \begin{itemize}
   \item [(i)]  $r=-\dfrac{3}{2}+\mathrm{ i}\dfrac{\sqrt{3}}{2},\quad s=\dfrac{1}{2}+\mathrm{i}\dfrac{\sqrt{3}}{6},\quad
                 t=\dfrac{1}{2}-\mathrm{ i}\dfrac{\sqrt{3}}{6},\quad u=1\hspace{1mm};$
   \item [(ii)] $r=-\dfrac{3}{2}-\mathrm{ i}\dfrac{\sqrt{3}}{2},\quad s=\dfrac{1}{2}-\mathrm{i}\dfrac{\sqrt{3}}{6},\quad
                 t=\dfrac{1}{2}+\mathrm{ i}\dfrac{\sqrt{3}}{6},\quad u=1.$
   \end{itemize}
   Dans les deux cas la $1$-forme $\omega$ est linéairement conjuguée à $$\omega_{11}=(3x^3+3\sqrt{3}x^2y+3xy^2+\sqrt{3}y^3)\mathrm{d}x+(\sqrt{3}x^3+3x^2y+3\sqrt{3}xy^2+3y^3)\mathrm{d}y\hspace{1mm};$$ en effet dans les cas (i), resp. (ii) nous avons
   \begin{align*}
   &&\omega_{11}=3\varphi_1^*\omega,\quad \text{où}\hspace{1.5mm} \varphi_1=(x,\mathrm{e}^{-5\mathrm{i}\pi/6}\,y),
   \hspace{1cm}\text{resp}.\hspace{1.5mm}
   \omega_{11}=3\varphi_2^*\omega,\quad \text{où}\hspace{1.5mm} \varphi_2=(x,\mathrm{e}^{5\mathrm{i}\pi/6}\,y).
   \end{align*}
\end{itemize}
\noindent Les feuilletages $\mathcal{H}_{\hspace{0.2mm}i},i=1,\ldots,11,$ ne sont pas linéairement conjugués car, par construction, $\mathcal{T}_{\mathcal{H}_{j}}~\neq~\mathcal{T}_{\mathcal{H}_{\hspace{0.2mm}i}}$ pour tout $j\neq i,$ d'où l'énoncé.
\end{proof}
\smallskip

\noindent Une particularité remarquable de la classification obtenue est que toutes les singularités des feuilletages $\mathcal{H}_{\hspace{0.2mm}i},i=1,\ldots,11,$ sur la droite à l'infini sont non-dégénérées. Nous aurons besoin dans la prochaine section des valeurs des indices $\mathrm{CS}(\mathcal{H}_{\hspace{0.2mm}i},L_{\infty},s)$, $s\in\Sing\mathcal{H}_{\hspace{0.2mm}i}\cap L_{\infty}$. Pour~cela, nous avons calculé, pour chaque $i=1,\ldots,11$, le polynôme suivant (dit \textsl{polynôme de \textsc{Camacho-Sad} du feuilletage homogène} $\mathcal{H}_{\hspace{0.2mm}i}$)
\begin{align*}
\mathrm{CS}_{\mathcal{H}_{\hspace{0.2mm}i}}(\lambda)=\prod\limits_{s\in\Sing\mathcal{H}_{\hspace{0.2mm}i}\cap L_{\infty}}(\lambda-\mathrm{CS}(\mathcal{H}_{\hspace{0.2mm}i},L_{\infty},s)).
\end{align*}
\noindent Le tableau suivant résume les types et les polynômes de \textsc{Camacho-Sad} des feuilletages $\mathcal{H}_{\hspace{0.2mm}i}$, $i=1,\ldots,11$.

\begingroup
\renewcommand*{\arraystretch}{1.5}
\begin{table}[h]
\begin{center}
\begin{tabular}{|c|c|c|}\hline
$i$  &  $\mathcal{T}_{\mathcal{H}_{\hspace{0.2mm}i}}$              &  $\mathrm{CS}_{\mathcal{H}_{\hspace{0.2mm}i}}(\lambda)$ \\\hline
$1$  &  $2\cdot\mathrm{R}_2$                                       &  $(\lambda-1)^{2}(\lambda+\frac{1}{2})^{2}$              \\\hline
$2$  &  $2\cdot\mathrm{T}_2$                                       &  $(\lambda-\frac{1}{4})^{4}$                              \\\hline
$3$  &  $2\cdot\mathrm{R}_1+1\cdot\mathrm{R}_2$                    &  $(\lambda-1)^{3}(\lambda+2)$                              \\\hline
$4$  &  $2\cdot\mathrm{R}_1+1\cdot\mathrm{T}_2$                    &  $(\lambda-1)^{2}(\lambda+\frac{1}{2})^{2}$                 \\\hline
$5$  &  $1\cdot\mathrm{R}_1+1\cdot\mathrm{T}_1+1\cdot\mathrm{R}_2$ &  $(\lambda-1)^{2}(\lambda+\frac{1}{5})(\lambda+\frac{4}{5})$ \\\hline
$6$  &  $1\cdot\mathrm{R}_1+1\cdot\mathrm{T}_1+1\cdot\mathrm{T}_2$ &  $(\lambda-1)(\lambda+\frac{2}{7})(\lambda-\frac{1}{7})^{2}$  \\\hline
$7$  &  $2\cdot\mathrm{T}_1+1\cdot\mathrm{R}_2$                    &  $(\lambda-1)(\lambda-\frac{1}{4})(\lambda+\frac{1}{8})^{2}$   \\\hline
$8$  &  $2\cdot\mathrm{T}_1+1\cdot\mathrm{T}_2$                    &  $(\lambda-\frac{1}{10})^{2}(\lambda-\frac{2}{5})^{2}$          \\\hline
$9$  &  $3\cdot\mathrm{R}_1+1\cdot\mathrm{T}_1$                    &  $(\lambda-1)^{3}(\lambda+2)$                                    \\\hline
$10$ &  $2\cdot\mathrm{R}_1+2\cdot\mathrm{T}_1$                    &  $(\lambda-1)^{2}(\lambda+\frac{1}{2})^{2}$                       \\\hline
$11$ &  $4\cdot\mathrm{T}_1$                                       &  $(\lambda-\frac{1}{4})^{4}$                                       \\\hline
\end{tabular}
\end{center}
\bigskip
\caption{Types et polynômes de \textsc{Camacho-Sad} des onze feuilletages homogènes donnés par le Théorème~\ref{thm:Class-Homog3-Plat}.}\label{tab:CS(lambda)}
\end{table}
\endgroup
\newpage
\hfill

\section{Classification complète des feuilletages de $\mathbf{FP}(3)$}\label{sec:Classification-feuilletages-FP(3)}

Il s'agit dans cette deuxième section de démontrer l'un des principaux résultats annoncés dans l'Introduction:

\begin{thm}\label{thm:classification}
{\sl \`A automorphisme de $\pp$ pr\`es, il y a seize feuilletages cubiques $\mathcal{H}_1,\ldots,\mathcal{H}_{11},$ $\mathcal{F}_1,\ldots,\mathcal{F}_{5}$ sur le plan projectif complexe ayant une transformée de \textsc{Legendre} plate. Ils sont d\'ecrits respectivement en carte affine par les $1$-formes suivantes
\begin{itemize}
\item [\texttt{1. }]  \hspace{1mm} $\omega_1\hspace{1mm}=y^3\mathrm{d}x-x^3\mathrm{d}y$;
\smallskip
\item [\texttt{2. }]  \hspace{1mm} $\omega_2\hspace{1mm}=x^3\mathrm{d}x-y^3\mathrm{d}y$;
\smallskip
\item [\texttt{3. }]  \hspace{1mm} $\omega_3\hspace{1mm}=y^2(3x+y)\mathrm{d}x-x^2(x+3y)\mathrm{d}y$;
\smallskip
\item [\texttt{4. }]\hspace{1mm}   $\omega_4\hspace{1mm}=y^2(3x+y)\mathrm{d}x+x^2(x+3y)\mathrm{d}y$;
\smallskip
\item [\texttt{5. }]\hspace{1mm} $\omega_{5}\hspace{1mm}=2y^3\mathrm{d}x+x^2(3y-2x)\mathrm{d}y$;
\smallskip
\item [\texttt{6. }]\hspace{1mm} $\omega_{6}\hspace{1mm}=(4x^3-6x^2y+4y^3)\mathrm{d}x+x^2(3y-2x)\mathrm{d}y$;
\smallskip
\item [\texttt{7. }]\hspace{1mm} $\omega_{7}\hspace{1mm}=y^3\mathrm{d}x+x(3y^2-x^2)\mathrm{d}y$;
\smallskip
\item [\texttt{8. }]\hspace{1mm} $\omega_{8}\hspace{1mm}=x(x^2-3y^2)\mathrm{d}x-4y^3\mathrm{d}y$;
\smallskip
\item [\texttt{9. }]\hspace{1mm} $\omega_{9}\hspace{1mm}=y^{2}\left((-3+\mathrm{i}\sqrt{3})x+2y\right)\mathrm{d}x+
                                                       x^{2}\left((1+\mathrm{i}\sqrt{3})x-2\mathrm{i}\sqrt{3}y\right)\mathrm{d}y$;
\smallskip
\item [\texttt{10. }]\hspace{-1mm} $\omega_{10}=(3x+\sqrt{3}y)y^2\mathrm{d}x+(3y-\sqrt{3}x)x^2\mathrm{d}y$;
\smallskip
\item [\texttt{11. }]\hspace{-1mm} $\omega_{11}=(3x^3+3\sqrt{3}x^2y+3xy^2+\sqrt{3}y^3)\mathrm{d}x+(\sqrt{3}x^3+3x^2y+3\sqrt{3}xy^2+3y^3)\mathrm{d}y$;
\smallskip
\item [\texttt{12. }]  \hspace{-1mm} $\omegaoverline_{1}\hspace{1mm}=y^{3}\mathrm{d}x+x^{3}(x\mathrm{d}y-y\mathrm{d}x)$;
\smallskip
\item [\texttt{13. }]  \hspace{-1mm} $\omegaoverline_{2}\hspace{1mm}=x^{3}\mathrm{d}x+y^{3}(x\mathrm{d}y-y\mathrm{d}x)$;
\smallskip
\item [\texttt{14. }]  \hspace{-1mm} $\omegaoverline_{3}\hspace{1mm}=(x^{3}-x)\mathrm{d}y-(y^{3}-y)\mathrm{d}x$;
\smallskip
\item [\texttt{15. }]  \hspace{-1mm} $\omegaoverline_{4}\hspace{1mm}=(x^{3}+y^{3})\mathrm{d}x+x^{3}(x\mathrm{d}y-y\mathrm{d}x)$;
\smallskip
\item [\texttt{16. }]  \hspace{-1mm} $\omegaoverline_{5}\hspace{1mm}=y^{2}(y\mathrm{d}x+2x\mathrm{d}y)+x^{3}(x\mathrm{d}y-y\mathrm{d}x)$.
\end{itemize}
}
\end{thm}

\noindent La démonstration comporte plusieurs étapes consistant en une analyse fine des singularités des feuilletages de $\mathbf{FP}(3).$

\subsection{Feuilletages de $\mathbf{FP}(3)$ n'ayant que des singularités non-dégénérées}\label{subsec:cas-non-dégénéré}

Parmi les éléments de $\mathbf{FP}(d)$ n'ayant que des singularités non-dégénérées, il y a le \textsl{feuilletage de \textsc{Fermat}} $\F^{d}$ de degré $d$ défini en carte affine par la $1$-forme
$$\omega_{F}^{d}=(x^{d}-x)\mathrm{d}y-(y^{d}-y)\mathrm{d}x\hspace{1mm};$$
en effet, d'une part $\Leg\F^{d}$ est plat car il est algébrisable d'après \cite[Proposition~5.2]{MP13}; d'autre part, un calcul élémentaire montre que toutes les singularités du feuilletage $\F^{d}$ sont non-dégénérées.

\begin{rem}
Dans l'énoncé du Théorème~\ref{thm:classification}, nous avons désigné par $\F_3$ le feuilletage de \textsc{Fermat} de degré trois, {\it i.e.} $\F_3:=\F^3.$
\end{rem}

\noindent Le théorème suivant est le résultat principal de ce paragraphe.

\begin{thm}\label{thm:Fermat}
{\sl Soit $\F$ un feuilletage de degré $3$ sur $\pp$. Supposons que toutes ses singularités soient non-dégénérées et que son $3$-tissu dual $\Leg\F$ soit plat. Alors $\F$ est linéairement conjugué au feuilletage de \textsc{Fermat} $\F_3$ défini par la $1$-forme $\omegaoverline_3=(x^{3}-x)\mathrm{d}y-(y^{3}-y)\mathrm{d}x.$
}
\end{thm}

\begin{rems}
\begin{itemize}
\item []\hspace{-0.8cm}(i) Pour $d=3$ ou $d\geq5,$ l'adhérence $\overline{\mathcal{O}(\F^{d})}$ (dans $\mathbf{F}(d)$) de $\mathcal{O}(\F^{d})$ est une composante irréductible de $\mathbf{FP}(d)$, \emph{voir} \cite[Théorème~3]{MP13}.
  \item [(ii)] L'ensemble $\mathbf{FP}(4)$ contient des feuilletages à singularités non-dégénérées et qui ne sont pas conjugués au feuilletage $\F^{4},$ {\it e.g.} la famille $(\F_{\lambda}^{4})_{\lambda\in\C}$ de feuilletages définis par
$$\omega_{F}^{4}+\lambda((x^{3}-1)y^{2}\mathrm{d}y-(y^{3}-1)x^{2}\mathrm{d}x).$$
En effet, d'après \cite[Théorème 8.1]{MP13}, pour tout $\lambda$ fixé dans $\mathbb{C}$, $\F_{\lambda}^{4}\in\mathbf{FP}(4)$; de plus un calcul facile montre que $\F_{\lambda}^{4}$ est à singularités non-dégénérées. Mais, si $\lambda$ est non nul $\F^{4}$~et~$\F^{4}_{\lambda}$ ne sont pas linéairement conjugués car le premier est convexe, alors que le second ne l'est pas.
\end{itemize}
\end{rems}

\noindent La démonstration du Théorème~\ref{thm:Fermat} repose sur le Théorème~\ref{thm:Class-Homog3-Plat} de classification des feuilletages homogènes appartenant à $\mathbf{FP}(3)$, et sur les trois résultats qui suivent, dont les deux premiers sont valables en degré quelconque.

\noindent Notons d'abord que le feuilletage $\F^{d}$ possède trois singularités radiales d'ordre maximal $d-1$, non alignées. La proposition suivante montre que cette propriété caractérise l'orbite de $\F^{d}$.

\begin{pro}\label{pro:Fermat-d}
{\sl Soit $\F$ un feuilletage de degré $d$ sur $\pp$ ayant trois singularités radiales d'ordre maximal $d-1$, non alignées. Alors $\F$ est linéairement conjugué au feuilletage de \textsc{Fermat} $\F^{d}.$
}
\end{pro}
\newpage
\hfill
\vspace{-1.3cm}

\begin{proof}
Par hypothèse $\F$ possède trois points singuliers $m_j,j=1,2,3,$ non alignés vérifiant $\nu(\F,m_j)=1$ et $\tau(\F,m_j)=d$. D'après la Remarque \ref{rem:sum-tau<d}, les trois droites joignant ces points deux à deux sont invariantes par $\F$. Choisissons des coordonnées homogènes $[x:y:z]$ telles que $m_1=[0:0:1],\,m_2=[0:1:0]$\, et \,$m_3=[1:0:0]$. Les égalités $\nu(\F,m_1)=1$ et $\tau(\F,m_1)=d$, combinées avec le fait que $(m_2m_3)=(z=0)$ est $\F$-invariante, assurent que toute $1$-forme $\omega$ décrivant $\F$ dans la carte affine $z=1$ est du type
\begin{align*}
\omega=(x\mathrm{d}y-y\mathrm{d}x)(\gamma+C_{1}(x,y)+\cdots +C_{d-2}(x,y))+A_{d}(x,y)\mathrm{d}x+B_{d}(x,y)\mathrm{d}y
\end{align*}
avec $\gamma\neq 0,\hspace{3mm}A_{d},B_{d}\in\mathbb{C}[x,y]_d,\hspace{3mm}C_{k}\in\mathbb{C}[x,y]_k$ \hspace{1mm}pour $k=1,\ldots,d-2.$

\noindent Dans la carte affine $y=1$ le feuilletage $\F$ est donné par
\begin{align*}
\theta=-(\gamma\hspace{0.1mm}z^{d}+C_{1}(x,1)z^{d-1}+\cdots +C_{d-2}(x,1)z^{2})\mathrm{d}x+A_{d}(x,1)(z\mathrm{d}x-x\mathrm{d}z)-B_{d}(x,1)\mathrm{d}z\hspace{1mm};
\end{align*}
nous avons $\theta\wedge(z\mathrm{d}x-x\mathrm{d}z)=z\hspace{0.1mm}Q(x,z)\mathrm{d}x\wedge\mathrm{d}z$, avec
$$Q(x,z)=x\left[\gamma\hspace{0.1mm}z^{d-1}+C_{1}(x,1)z^{d-2}+\cdots+C_{d-2}(x,1)z\right]+B_{d}(x,1).$$
L'égalité $\tau(\F,m_2)=d$ entraîne alors que le polynôme $Q\in\mathbb{C}[x,z]$ est homogène de degré $d$, ce qui permet d'écrire $B_d(x,y)=\beta\hspace{0.1mm}x^d$\, et \,$C_k(x,y)=\delta_{k}x^k,$\, $\beta,\delta_k\in\mathbb{C}.$ Par suite nous avons $J^{1}_{(0,0)}\theta=A_{d}(0,1)(z\mathrm{d}x-x\mathrm{d}z)$; alors l'égalité $\nu(\F,m_2)=1$ assure que $A_{d}(0,1)\neq0.$

\noindent De la même manière, en se plaçant dans la carte affine $x=1$ et en écrivant explicitement les égalités $\tau(\F,m_3)=d$\, et \,$\nu(\F,m_3)=1$, nous obtenons que $B_{d}(1,0)\neq0,$\, $A_d(x,y)=\alpha\hspace{0.1mm}y^d$\, et \,$C_k(x,y)=\varepsilon_{k}y^k,$\, $\alpha,\varepsilon_k\in\mathbb{C}.$ Donc $\alpha\beta\neq0,$\, les $C_k$ sont tous nuls et \,$\omega$\, est du type $$\omega=\gamma(x\mathrm{d}y-y\mathrm{d}x)+\alpha y^{d}\mathrm{d}x+\beta x^{d}\mathrm{d}y.$$
Écrivons $\alpha=\gamma\mu^{1-d}$\, et \,$\beta=-\gamma\hspace{0.1mm}\lambda^{\hspace{-0.1mm}1-d}.$ Quitte à remplacer $\omega$ par $\varphi^*\omega,$ où $\varphi(x,y)=(\lambda\hspace{0.1mm}x,\mu\hspace{0.1mm}y),$ le feuilletage $\F$ est défini, dans les coordonnées affines $(x,y),$ par la $1$-forme
\vspace{2mm}

\noindent\hspace{4.5cm}$\omega_{F}^{d}=(x^{d}-x)\mathrm{d}y-(y^{d}-y)\mathrm{d}x.$
\end{proof}

\noindent Le résultat suivant permet de ramener l'étude de la platitude au cadre homogène:
\vspace{-1.2mm}
\begin{pro}\label{pro:F-dégénère-H}
{\sl Soit $\F$ un feuilletage de degré $d\geq1$ sur $\pp$ ayant une droite invariante $L.$ Supposons que toutes les singularités de $\F$ sur $L$ soient non-dégénérées. Il existe un feuilletage homogène $\mathcal{H}$ de degré $d$ sur $\pp$ ayant les propriétés suivantes
\begin{itemize}
\item [$\bullet$] $\mathcal{H}\in\overline{\mathcal{O}(\F)}$;

\item [$\bullet$] $L$ est invariante par $\mathcal{H}$;

\item [$\bullet$] $\Sing\mathcal{H}\cap L=\Sing\F\cap L$;

\item [$\bullet$] $\forall\hspace{1mm}s\in\Sing\mathcal{H}\cap L,\hspace{1mm}
                   \mu(\mathcal{H},s)=1
                   \hspace{2mm}\text{et}\hspace{2mm}
                   \mathrm{CS}(\mathcal{H},L,s)=\mathrm{CS}(\F,L,s)$.
\end{itemize}
Si de plus $\Leg\F$ est plat, alors $\Leg\mathcal{H}$ l'est aussi.
}
\end{pro}

\begin{proof}
Choisissons, dans $\pp,$ un système de coordonnées homogènes $[x:y:z]$ tel que $L=(z=0)$; comme $L$ est $\F$-invariante, $\F$ est défini dans la carte affine $z=1$ par une $1$-forme $\omega$ du type
$$\omega=\sum_{i=0}^{d}(A_i(x,y)\mathrm{d}x+B_i(x,y)\mathrm{d}y),$$
où les $A_i,\,B_i$ sont des polynômes homogènes de degré $i$.

\noindent Montrons par l'absurde que $\mathrm{pgcd}(A_d,B_d)=1$; supposons donc que $\mathrm{pgcd}(A_d,B_d)\neq1.$ Quitte à conjuguer $\omega$ par une transformation linéaire de $\mathbb{C}^2=(z=1)$, nous pouvons nous ramener à
$$A_{d}(x,y)=x\hspace{0.3mm}\widetilde{A}_{d-1}(x,y) \qquad\text{et}\qquad B_{d}(x,y)=x\widetilde{B}_{d-1}(x,y)$$
pour certains $\widetilde{A}_{d-1},\,\widetilde{B}_{d-1}$ dans $\mathbb{C}[x,y]_{d-1}$; alors $s_{0}=[0:1:0]\in L$ est un point singulier de $\F$. Dans la carte affine $y=1$, le feuilletage $\F$ est donné par
\begin{eqnarray*}
\hspace{1.5cm}\theta\hspace{-1mm}&=&\hspace{-1mm}\sum_{i=0}^{d}z^{d-i}[A_i(x,1)(z\mathrm{d}x-x\mathrm{d}z)-B_i(x,1)\mathrm{d}z]
\\
\hspace{1.5cm}\hspace{-1mm}&=&\hspace{-1mm}[A_{d}(x,1)+z\hspace{0.3mm}A_{d-1}(x,1)+\cdots](z\mathrm{d}x-x\mathrm{d}z)-
                [B_{d}(x,1)+zB_{d-1}(x,1)+\cdots]\mathrm{d}z.
\end{eqnarray*}
Le $1$-jet de $\theta$ au point singulier $s_{0}=(0,0)$ s'écrit $-[\widetilde{B}_{d-1}(0,1)x+B_{d-1}(0,1)z]\mathrm{d}z$; ce qui implique que $\mu(\F,s_{0})>1$: contradiction avec l'hypothèse que toute singularité de $\F$ située sur $L$ est non-dégénérée.

\noindent Il s'en suit que la $1$-forme $\omega_d=A_d(x,y)\mathrm{d}x+B_d(x,y)\mathrm{d}y$ définit bien un feuilletage homogène de degré $d$ sur $\pp$, que nous notons $\mathcal{H}$. Il est évident que $L$ est $\mathcal{H}$-invariante et que $\Sing\mathcal{F}\cap L=\Sing\mathcal{H}\cap L$. Considérons la famille d'homothéties $\varphi=\varphi_{\varepsilon}=(\frac{x}{\varepsilon},\frac{y}{\varepsilon}).$ Nous avons $$\varepsilon^{d+1}\varphi^*\omega=\sum_{i=0}^{d}(\varepsilon^{d-i}A_i(x,y)\mathrm{d}x+\varepsilon^{d-i}B_i(x,y)\mathrm{d}y)$$
qui tend vers $\omega_d$ lorsque $\varepsilon$ tend vers $0$; il en résulte que $\mathcal{H}\in\overline{\mathcal{O}(\F)}.$

\noindent Montrons que $\mathcal{H}$ vérifie la quatrième propriété annoncée. Soit $s\in\Sing\mathcal{H}\cap L$. Quitte à conjuguer $\omega$ par un isomorphisme linéaire de $\mathbb{C}^2=(z=0)$, nous pouvons supposer que $s=[0:1:0]$; il existe donc un polynôme $\widehat{B}_{d-1}\in\mathbb{C}[x,y]_{d-1}$ tel que $B_{d}(x,y)=x\widehat{B}_{d-1}(x,y)$. Le feuilletage $\mathcal{H}$ est décrit dans la carte affine $y=1$ par $$\theta_d=A_d(x,1)(z\mathrm{d}x-x\mathrm{d}z)-B_d(x,1)\mathrm{d}z.$$
\noindent Posons $\lambda=A_{d}(0,1)$\, et \,$\nu=A_{d}(0,1)+\widehat{B}_{d-1}(0,1)$. Le $1$-jet de $\theta_d$ en $s=(0,0)$ s'écrit $J^{1}_{(0,0)}\theta_{d}=\lambda\hspace{0.1mm}z\mathrm{d}x-\nu\hspace{0.1mm}x\mathrm{d}z$, et celui de $\theta$ est donné par $J^{1}_{(0,0)}\theta=\lambda\hspace{0.1mm}z\mathrm{d}x-\nu\hspace{0.1mm}x\mathrm{d}z-z\hspace{0.1mm}B_{d-1}(0,1)\mathrm{d}z$. L'hypothèse $\mu(\F,s)=1$ signifie que $\lambda\nu$ est non nul. Par suite $\mu(\mathcal{H},s)=1$\, et \,$\mathrm{CS}(\mathcal{H},L,s)=\mathrm{CS}(\F,L,s)=\frac{\lambda}{\nu}.$

\noindent L'implication\, $K(\Leg\F)\equiv0\hspace{0.3mm}\Longrightarrow\hspace{0.3mm}K(\Leg\mathcal{H})\equiv0$\, découle du fait que $\mathcal{H}\in\overline{\mathcal{O}(\F)}.$
\end{proof}

\noindent Nous illustrons le résultat précédent en l'appliquant au feuilletage $\F^{d}$.

\begin{eg}
Le feuilletage de \textsc{Fermat} $\F^{d}$ est donné en coordonnées homogènes par la $1$-forme
\[x^{d}(y\mathrm{d}z-z\mathrm{d}y)+y^{d}(z\mathrm{d}x-x\mathrm{d}z)+z^{d}(x\mathrm{d}y-y\mathrm{d}x).\] Il possède les $3d$ droites invariantes suivantes :
\begin{enumerate}
\item [(a)] $x=0$,\, $y=0$,\, $z=0$;
\item [(b)] $y=\zeta x$,\, $y=\zeta z$,\, $x=\zeta z$\,  avec $\zeta^{d-1}=1$.
\end{enumerate}
Les droites de la famille (a) (resp. (b)) donnent lieu à $3$ (resp. $3d-3$) feuilletages homogènes appartenant à $\overline{\mathcal{O}(\F^{d})}\subset\mathbf{FP}(d)$ et de type $2\cdot\mathrm{R}_{d-1}$ (resp. $1\cdot\mathrm{R}_{d-1} + (d-1)\cdot\mathrm{R}_{1}$). Ceux qui sont de type $2\cdot\mathrm{R}_{d-1}$ sont tous conjugués au feuilletage $\mathcal{H}^{d}_{1}$ donné par la Proposition~\ref{pro:omega1-omega2}; ceux qui sont de type $1\cdot\mathrm{R}_{d-1}+(d-1)\cdot\mathrm{R}_{1}$ sont tous conjugués au feuilletage $\mathcal{H}_{\hspace{0.2mm}0}^{d}$ défini par la $1$-forme
$$
\omega_{\hspace{0.2mm}0}^{\hspace{0.2mm}d}=(d-1)y^{d}\mathrm{d}x+x(x^{d-1}-dy^{d-1})\mathrm{d}y.
$$
Le feuilletage $\mathcal{H}_{\hspace{0.2mm}0}^{\hspace{0.2mm}3}$ est conjugué au feuilletage $\mathcal{H}_{\hspace{0.2mm}3}$ donné par le Théorème~\ref{thm:Class-Homog3-Plat}, car $\mathcal{H}_{\hspace{0.2mm}0}^{\hspace{0.2mm}3}\in\mathbf{FP}(3)$ et $\mathcal{T}_{\mathcal{H}_{\hspace{0.2mm}0}^{\hspace{0.2mm}3}}=\mathcal{T}_{\mathcal{H}_{\hspace{0.2mm}3}}.$
\end{eg}

\begin{rem}\label{rem:droite-inva}
Si $\F$ est un feuilletage de degré $d$ sur $\pp$ et si $m$ est un point singulier de $\F$, nous avons l'encadrement $\sigma(\F,m)\leq\tau(\F,m)+1\leq d+1,$ où $\sigma(\F,m)$ désigne le nombre de droites (distinctes) invariantes par $\F$ et qui passent par $m.$
\end{rem}

\noindent La première assertion du lemme suivant joue un rôle clé dans la preuve du Théorème~\ref{thm:Fermat}.

\begin{lem}\label{lem:2-droite-invar}
{\sl Soit $\F$ un feuilletage de degré $3$ sur $\pp.$ Supposons que $\F$ ait une singularité $m$ non-dégénérée et que son $3$-tissu dual $\Leg\F$ soit plat. Alors:

\noindent\textbf{\textit{1.}} Si $\mathrm{BB}(\F,m)\not\in\{4,\frac{16}{3}\}$, par cette singularité passent exactement deux droites invariantes par~$\F$.

\noindent\textbf{\textit{2.}} Si $\mathrm{BB}(\F,m)=\frac{16}{3}$, par cette singularité passe une droite $\ell$ invariante par $\F$ et telle que $\mathrm{CS}(\F,\ell,m)=3.$
}
\end{lem}

\begin{proof}
\textbf{\textit{i.}} Supposons que $\mathrm{BB}(\F,m)\neq4$. Par hypothèse $\mu(\F,m)=1$. Ces deux conditions assurent l'existence d'une carte affine $(x,y)$ de $\pp$ dans laquelle
$m=(0,0)$ et $\F$ est défini par une $1$-forme du type $\theta_1+\theta_2+\theta_3+\theta_4,$ où
\begin{align*}
&  \theta_1=\lambda\hspace{0.1mm}y\mathrm{d}x+\mu\hspace{0.1mm}x\mathrm{d}y,\hspace{2mm}\text{avec}\hspace{1mm}\lambda\mu(\lambda+\mu)\neq0,&&\hspace{2mm}
   \theta_2=\left(\sum_{i=0}^{2}\alpha_{i}\hspace{0.1mm}x^{2-i}y^{i}\right)\mathrm{d}x+
                        \left(\sum_{i=0}^{2}\beta_{i}\hspace{0.1mm}x^{2-i}y^{i}\right)\mathrm{d}y,\\
&  \theta_3=\left(\sum_{i=0}^{3}a_{i}\hspace{0.1mm}x^{3-i}y^{i}\right)\mathrm{d}x+
                        \left(\sum_{i=0}^{3}b_{i}\hspace{0.1mm}x^{3-i}y^{i}\right)\mathrm{d}y,&&\hspace{2mm}
   \theta_4=\left(\sum_{i=0}^{3}c_{i}\hspace{0.1mm}x^{3-i}y^{i}\right)(x\mathrm{d}y-y\mathrm{d}x),
\end{align*}

\textbf{\textit{i.1.}} Supposons que $\mathrm{BB}(\F,m)\neq\frac{16}{3}$; cette condition se traduit par $(\lambda+3\mu)(3\lambda+\mu)\neq0.$

\noindent Commençons par montrer que $\alpha_0=0$. Supposons par l'absurde que $\alpha_0\neq0$. Soit $(p,q)$ la carte affine de $\pd$ associée à la droite $\{px-qy=1\}\subset{\mathbb{P}^{2}_{\mathbb{C}}}$; le $3$-tissu $\mathrm{Leg}\mathcal{F}$ est donné par la $3$-forme symétrique
\begin{small}
\begin{align*}
&\check{\omega}=\left[\left(\beta_2\hspace{0.1mm}p+\alpha_2\hspace{0.1mm}q-\lambda\hspace{0.1mm}q^2\right)\mathrm{d}p^{2}+
                \left(\beta_1\hspace{0.1mm}p+\alpha_1\hspace{0.1mm}q+\lambda\hspace{0.1mm}pq-\mu\hspace{0.1mm}pq\right)\mathrm{d}p\mathrm{d}q+
                \left(\beta_0\hspace{0.1mm}p+\alpha_0\hspace{0.1mm}q+\mu\hspace{0.1mm}p^2\right)\mathrm{d}q^{2}\right](p\mathrm{d}q-q\mathrm{d}p)\\
&\hspace{0.7cm}+q\left(a_3\mathrm{d}p^{3}+a_2\mathrm{d}p^{2}\mathrm{d}q+a_1\mathrm{d}p\mathrm{d}q^{2}+a_0\mathrm{d}q^{3}\right)+
               p\left(b_3\mathrm{d}p^{3}+b_2\mathrm{d}p^{2}\mathrm{d}q+b_1\mathrm{d}p\mathrm{d}q^{2}+b_0\mathrm{d}q^{3}\right)\\
&\hspace{0.7cm}+c_3\mathrm{d}p^{3}+c_2\mathrm{d}p^{2}\mathrm{d}q+c_1\mathrm{d}p\mathrm{d}q^{2}+c_0\mathrm{d}q^{3}.
\end{align*}
\end{small}
\hspace{-1mm}Considérons la famille d'automorphismes $\varphi=\varphi_{\varepsilon}=(\alpha_0\hspace{0.1mm}\varepsilon^{-1}\hspace{0.1mm}p,\hspace{0.1mm}\alpha_0\hspace{0.1mm}\varepsilon^{-2}\hspace{0.1mm}q).$ Nous constatons que
\begin{align*}
&\check{\omega}_{0}:
=\lim_{\varepsilon\to 0}\varepsilon^9\alpha_{0}^{-6}\varphi^*\check{\omega}
=(p\mathrm{d}q-q\mathrm{d}p)\left(-\lambda\hspace{0.1mm}q^2\mathrm{d}p^2+pq(\lambda-\mu)\mathrm{d}p\mathrm{d}q+(\mu\hspace{0.1mm}p^2+q)\mathrm{d}q^{2}\right).
\end{align*}

\noindent Puisque $\mu$ est non nul $\check{\omega}_{0}$ définit un $3$-tissu $\mathcal{W}_{0}$, qui appartient évidemment à $\overline{\mathcal{O}(\Leg\F)}$. L'image réciproque de $\W_0$ par l'application rationnelle\, $\psi(p,q)=\left(\lambda(p+q),-\lambda(\lambda+\mu)^2pq\right)$\, s'écrit $\psi^*\mathcal{W}_{0}=\mathcal{F}_{1}\boxtimes\mathcal{F}_{2}\boxtimes\mathcal{F}_{3}$, où
\begin{small}
\begin{align*}
&\F_1\hspace{0.1mm}:\hspace{0.1mm}q^2\mathrm{d}p+p^2\mathrm{d}q=0,&&
\F_2\hspace{0.1mm}:\hspace{0.1mm}\mu q^2\mathrm{d}p+p(\lambda q+\mu q-\lambda p)\mathrm{d}q=0,&&
\F_3\hspace{0.1mm}:\hspace{0.1mm}\mu p^2\mathrm{d}q+q(\lambda p+\mu p-\lambda q)\mathrm{d}p=0.
\end{align*}
\end{small}
\hspace{-1mm}Un calcul direct, utilisant la formule (\ref{equa:eta-rst}), conduit à
\begin{small}
\begin{align*}
\eta(\psi^{*}\mathcal{W}_0)=
\frac{5(\lambda+\mu)p^2-(8\lambda+7\mu)pq+(3\lambda+4\mu)q^2}{(\lambda+\mu)p(p-q)^2}\mathrm{d}p+
\frac{5(\lambda+\mu)q^2-(8\lambda+7\mu)pq+(3\lambda+4\mu)p^2}{(\lambda+\mu)q(p-q)^2}\mathrm{d}q
\end{align*}
\end{small}
\noindent de sorte que
$$
K(\psi^{*}\mathcal{W}_0)=\mathrm{d}\eta(\psi^{*}\mathcal{W}_0)
=-\frac{4\mu(p+q)}{(\lambda+\mu)(p-q)^3}\mathrm{d}p\wedge\mathrm{d}q\not\equiv0\hspace{1mm};
$$
\noindent comme $\Leg\F$ est plat par hypothèse, il en est de même pour $\W_0$; par suite $K(\psi^{*}\mathcal{W}_0)=\psi^{*}K(\mathcal{W}_0)=0,$ ce qui est absurde. D'où l'égalité $\alpha_0=0.$

\noindent Montrons maintenant que $a_0=0$. Raisonnons encore par l'absurde en supposant $a_0\neq0$. Le feuilletage $\F$ est décrit dans la carte affine $(x,y)$ par $\theta=\theta_1+\theta_2+\theta_3+\theta_4$\hspace{0.1mm} avec $\alpha_0=0.$ En faisant agir la transformation linéaire diagonale  $(\varepsilon\hspace{0,1mm}x\hspace{0,1mm},\hspace{0,1mm}a_0\hspace{0,1mm}\varepsilon^{3}y)$ sur $\theta$ puis en passant à la limite lorsque $\varepsilon\to0$ nous obtenons
$$\theta_0=\lambda\hspace{0.1mm}y\mathrm{d}x+\mu\hspace{0.1mm}x\mathrm{d}y+x^3\mathrm{d}x$$
qui définit un feuilletage de degré trois $\F_0\in\overline{O(\F)}.$

\noindent Notons $\mathrm{I}_{0}=\mathrm{I}_{\mathcal{F}_{0}}^{\hspace{0.2mm}\mathrm{tr}}$, $\mathcal{G}_{0}=\mathcal{G}_{\mathcal{F}_0}$ et $\mathrm{I}_{0}^{\perp}=\overline{\mathcal{G}_{0}^{-1}(\mathcal{G}_{0}(\mathrm{I}_{0}))\setminus\mathrm{I}_{0}},$ où l'adhérence est prise au sens ordinaire. Un calcul élémentaire montre que
\begin{small}
\begin{align*}
\mathcal{G}_{0}(x,y)=\left(\dfrac{x^3+\lambda\hspace{0.1mm}y}{x(x^3+\lambda\hspace{0.1mm}y+\mu\hspace{0.1mm}y)},
-\dfrac{\mu}{x^3+\lambda\hspace{0.1mm}y+\mu\hspace{0.1mm}y}\right),\hspace{2mm}
\mathrm{I}_{0}=\{(x,y)\in\mathbb{C}^2\hspace{1mm}\colon(\lambda-2\mu)x^3+\lambda(\lambda+\mu)y=0\}\subset\pp
\end{align*}
\end{small}
\hspace{-1mm}et que la courbe $\mathrm{I}_{0}^{\perp}$ a pour équation affine $f(x,y)=y-\nu\hspace{0.1mm}x^3=0$, où $\nu=-\dfrac{4\lambda+\mu}{4\lambda(\lambda+\mu)}.$
\newpage
\hfill
\vspace{-0.8cm}

\noindent Comme $\Leg\F$ est plat, $\Leg\F_0$ l'est aussi. Or, d'après \cite[Corollaire~4.6]{BFM13}, le $3$-tissu $\Leg\F_0$ est plat si~et~seulement~si $\mathrm{I}_{0}^{\perp}$ est invariante par $\F_0,$ {\it i.e.} si et seulement si
$$0\equiv\mathrm{d}f\wedge\theta_0\Big|_{y=\nu x^3}=3(3\lambda+\mu)\mu\hspace{0.2mm}x^3\mathrm{d}x\wedge\mathrm{d}y\hspace{1mm};$$
d'où $\mu(3\lambda+\mu)=0$: contradiction. Donc $a_0=\alpha_0=0$, ce qui signifie que la droite $(y=0)$ est $\F$-invariante.

\noindent Ce qui précède montre également que l'invariance de la droite $(y=0)$ par $\F$ découle uniquement du fait que $\lambda\mu(\lambda+\mu)(3\lambda+\mu)\neq0$ et de l'hypothèse que $\Leg\F$ est plat. En permutant les coordonnées $x$ et $y$, la condition $\lambda\mu(\lambda+\mu)(\lambda+3\mu)\neq0$ permet de déduire que $\beta_2=b_3=0$, {\it i.e.} que la droite $(x=0)$ est aussi invariante par $\F$.

\noindent La singularité $m$ de $\F$ n'est pas radiale car $\mathrm{BB}(\F,m)\neq4;$ de plus $\nu(\F,m)=1$ car $\mu(\F,m)=1$. Il s'en suit que $\tau(\F,m)=1;$ d'après la Remarque \ref{rem:droite-inva}, nous avons $\sigma(\F,m)\leq\tau(\F,m)+1=2,$ d'où la première assertion.
\vspace{2mm}

\textbf{\textit{i.2.}} Lorsque $\mathrm{BB}(\F,m)=\frac{16}{3}$, nous pouvons supposer que $\lambda=-3\mu$ et $\mu\neq0$. Dans ce cas
\begin{align*}
&\lambda\mu(\lambda+\mu)(3\lambda+\mu)=-48\mu^4\neq0&& \text{et} && \lambda\mu(\lambda+\mu)(\lambda+3\mu)=0\hspace{1mm};
\end{align*}
donc la droite $(y=0)$ est $\F$-invariante, mais nous ne pouvons rien dire a priori sur l'invariance par $\F$ ou non de la droite $(x=0).$ Sur la droite $(y=0)$ nous avons $\mathrm{CS}(\F,(y=0),m)=-\frac{\lambda}{\mu}=3,$ d'où la seconde assertion.
\end{proof}

\begin{proof}[\sl D\'emonstration du Théorème~\ref{thm:Fermat}]
\'{E}crivons $\Sing\F=\Sigma^{0}\cup\Sigma^{1}\cup\Sigma^{2}$ avec
\begin{small}
\begin{align*}
&\Sigma^{0}=\{s\in\Sing\F\hspace{1mm}\colon \mathrm{BB}(\F,s)=\tfrac{16}{3}\},&&
\Sigma^{1}=\{s\in\Sing\F\hspace{1mm}\colon \mathrm{BB}(\F,s)=4\},&&
\Sigma^{2}=\Sing\F\setminus(\Sigma^{0}\cup\Sigma^{1})
\end{align*}
\end{small}
\hspace{-1mm}et notons $\kappa_i=\#\hspace{0.5mm}\Sigma^{i},\,i=0,1,2.$ Par hypothèse, $\F$ est de degré $3$ et toutes ses singularités ont leur nombre de \textsc{Milnor} $1$. Les formules (\ref{equa:Darboux}) et (\ref{equa:BB}) impliquent alors que
\begin{align}\label{equa:Dar-BB-3}
&\#\hspace{0.5mm}\Sing\F=\kappa_0+\kappa_1+\kappa_2=13 &&\text{et} &&\tfrac{16}{3}\kappa_0+4\kappa_1+\sum_{s\in\Sigma^{2}}\mathrm{BB}(\F,s)=25\hspace{1mm};
\end{align}
il en résulte que $\Sigma^{2}$ est non vide. Soit $m$ un point de $\Sigma^{2}$; d'après la première assertion du Lemme~\ref{lem:2-droite-invar} il passe par $m$ exactement deux droites $\ell_{m}^{(1)}$ et $\ell_{m}^{(2)}$ invariantes par $\F$. Alors, pour $i=1,2$, la Proposition~\ref{pro:F-dégénère-H} assure l'existence d'un feuilletage homogène $\mathcal{H}^{(i)}_{m}$ de degré $3$ sur $\pp$ appartenant à $\overline{\mathcal{O}(\F)}$ et tel que la droite $\ell_{m}^{(i)}$ soit $\mathcal{H}^{(i)}_{m}$-invariante. Comme $\Leg\F$ est plat par hypothèse, il en est de même pour $\Leg\mathcal{H}^{(1)}_{m}$ et $\Leg\mathcal{H}^{(2)}_{m}$. Donc chacun des $\mathcal{H}^{(i)}_{m}$ est linéairement conjugué à l'un des onze feuilletages homogènes donnés par le Théorème~\ref{thm:Class-Homog3-Plat}. Pour $i=1,2,$ la Proposition~\ref{pro:F-dégénère-H} assure aussi que
\begin{itemize}
\item [($\mathfrak{a}$)] $\Sing\F\cap\ell_{m}^{(i)}=\Sing\mathcal{H}^{(i)}_{m}\cap\ell_{m}^{(i)}$;
\vspace{0.5mm}
\item [($\mathfrak{b}$)] $\forall\hspace{1mm}s\in\Sing\mathcal{H}^{(i)}_{m}\cap\ell_{m}^{(i)},\quad
                   \mu(\mathcal{H}^{(i)}_{m},s)=1\hspace{2mm}\text{et}\quad
                   \mathrm{CS}(\mathcal{H}^{(i)}_{m},\ell_{m}^{(i)},s)=\mathrm{CS}(\F,\ell_{m}^{(i)},s)$.
\end{itemize}
\newpage
\hfill

\noindent Puisque $\mathrm{CS}(\F,\ell_{m}^{(1)},m)\mathrm{CS}(\F,\ell_{m}^{(2)},m)=1$, nous avons $\mathrm{CS}(\mathcal{H}^{(1)}_{m},\ell_{m}^{(1)},m)\mathrm{CS}(\mathcal{H}^{(2)}_{m},\ell_{m}^{(2)},m)=1$. Cette égalité et la Table \ref{tab:CS(lambda)} impliquent
\begin{align*}
\{\mathrm{CS}(\mathcal{H}^{(1)}_{m},\ell_{m}^{(1)},m),\,\mathrm{CS}(\mathcal{H}^{(2)}_{m},\ell_{m}^{(2)},m)\}=\{-2,-\tfrac{1}{2}\}\hspace{1mm};
\end{align*}
d'où $\mathrm{BB}(\F,m)=-\frac{1}{2}.$ Le point $m\in\Sigma^{2}$ étant arbitraire, $\Sigma^{2}$ est formé des $s\in\Sing\F$ tels que $\mathrm{BB}(\F,s)=-\frac{1}{2}.$ Par suite le système (\ref{equa:Dar-BB-3}) se réécrit
\begin{align*}
&\kappa_0+\kappa_1+\kappa_2=13 &&\text{et} &&\tfrac{16}{3}\kappa_0+4\kappa_1-\tfrac{1}{2}\kappa_2=25
\end{align*}
dont l'unique solution est $(\kappa_0,\kappa_1,\kappa_2)=(0,7,6)$, c'est-à-dire que $\Sing\F=\Sigma^{1}\cup\Sigma^{2},\hspace{2mm}\#\hspace{0.5mm}\Sigma^{1}=7$\hspace{2mm} et \hspace{2mm}$\#\hspace{0.5mm}\Sigma^{2}=6.$

\noindent Pour fixer les idées, nous supposons que $\mathrm{CS}(\mathcal{H}^{(1)}_{m},\ell_{m}^{(1)},m)=-2$ pour n'importe quel choix de $m\in\Sigma_2$, donc $\mathrm{CS}(\mathcal{H}^{(2)}_{m},\ell_{m}^{(2)},m)=-\frac{1}{2}.$ Dans ce cas, l'inspection de la Table \ref{tab:CS(lambda)} ainsi que les relations ($\mathfrak{a}$) et ($\mathfrak{b}$) conduisent à
\begin{align*}
&\#\hspace{0.5mm}(\Sigma^{1}\cap\ell_{m}^{(1)})=3,&&
\#\hspace{0.5mm}(\Sigma^{1}\cap\ell_{m}^{(2)})=2,&&
\Sigma^{2}\cap\ell_{m}^{(1)}=\{m\},&&
\Sigma^{2}\cap\ell_{m}^{(2)}=\{m,m'\}
\end{align*}
pour un certain point $m'\in\Sigma^{2}\setminus\{m\}$ vérifiant $\mathrm{CS}(\F,\ell_{m}^{(2)},m')=-\frac{1}{2}$. Ce point $m\hspace{0.1mm}'$ satisfait à son tour l'égalité $\Sigma^{2}\cap\ell_{m\hspace{0.1mm}'}^{(1)}=\{m\hspace{0.1mm}'\}$. Nous constatons que  $\ell_{m\hspace{0.1mm}'}^{(2)}=\ell_{m}^{(2)}$,\, $\ell_{m\hspace{0.1mm}'}^{(1)}\neq\ell_{m}^{(1)}$,\, $\ell_{m\hspace{0.1mm}'}^{(1)}\neq\ell_{m}^{(2)}$ et que ces trois droites distinctes satisfont $\Sigma^{2}\cap(\ell_{m}^{(1)}\cup\ell_{m}^{(2)}\cup\ell_{m\hspace{0.1mm}'}^{(1)})=\{m,m\hspace{0.1mm}'\}.$ Comme $\#\hspace{0.5mm}\Sigma^{2}=6=2\cdot3$, $\F$ possède $3\cdot3=9$ droites invariantes.

\noindent Posons $\Sigma^{1}\cap\ell_{m}^{(2)}=\{m_1,m_2\}$. Notons $\mathcal{D}_1,\mathcal{D}_2,\ldots,\mathcal{D}_6$ les six droites $\F$-invariantes qui restent; par construction chacune d'elles doit couper $\ell_{m}^{(1)}$ et $\ell_{m}^{(2)}$ en des points de $\Sigma^{1}.$ Par ailleurs, d'après la Remarque \ref{rem:droite-inva}, pour tout $s\in\Sing\F$ nous avons $\sigma(\F,s)\leq\tau(\F,s)+1\leq4.$ Donc par chacun des points $m_1$ et $m_2$ passent exactement trois droites de la famille $\{\mathcal{D}_1,\mathcal{D}_2,\ldots,\mathcal{D}_6\}$. Puisque $\#\hspace{0.5mm}(\Sigma^{1}\cap\ell_{m}^{(1)})=3$, $\Sigma^{1}\cap\ell_{m}^{(1)}$ contient au moins un point, noté $m_3$, par lequel passent précisément trois droites de la famille $\{\ell_{m\hspace{0.1mm}'}^{(1)},\mathcal{D}_1,\mathcal{D}_2,\ldots,\mathcal{D}_6\}$. Ainsi, pour $j=1,2,3$ nous avons $\sigma(\F,m_j)=4,$ ce qui implique que $\tau(\F,m_j)=3.$ L'hypothèse sur les singularités de $\F$ assure que $\nu(\F,m_j)=1$ pour $j=1,2,3.$ Il s'en suit que les singularités $m_1$, $m_2$ et $m_3$ sont radiales d'ordre $2$ de $\F$.

\noindent Par construction ces trois points ne sont pas alignés. Nous concluons en appliquant la Proposition~\ref{pro:Fermat-d}.
\end{proof}
\newpage
\hfill
\vspace{-1.5cm}

\Subsection{Feuilletages de $\mathbf{FP}(3)$ ayant au moins une singularité dégénérée}\label{subsec:cas-dégénéré}

\Subsubsection{Cas d'une singularité dégénérée de multiplicité algébrique $\leq2$}

Dans l'Appendice~\ref{Dém:pro:cas-nilpotent-noeud-ordre-2} nous donnons une démonstration calculatoire de l'énoncé suivant.

\begin{pro}[\emph{voir} Appendice \ref{Dém:pro:cas-nilpotent-noeud-ordre-2}]\label{pro:cas-nilpotent-noeud-ordre-2}
{\sl Soit $\mathcal{F}$ un feuilletage de degré $3$ sur $\pp$ ayant au moins une singularité dégénérée de multiplicité algébrique inférieure ou égale à $2$. Alors le $3$-tissu dual $\Leg\F$ de $\F$ ne peut être plat.}
\end{pro}

\noindent Il s'agit dans cet appendice d'un raisonnement par l'absurde: nous supposons d'abord qu'il existe un feuilletage $\F$ de degré $3$ sur $\pp$ tel que le $3$-tissu $\Leg\F$ soit plat et dont le lieu singulier $\Sing\F$ contient un point $m$ vérifiant $\mu(\F,m)\geq2$\, et \,$\nu(\F,m)\leq2$; puis, nous utilisons ces deux inégalités pour distinguer trois cas, chacun donnant lieu à une section (\S\ref{sec:cas-noeud},~\S\ref{sec:cas-nilpotent},~\S\ref{sec:cas-ordre-2}); enfin, dans chacun de ces cas, en calculant explicitement la courbure de $\Leg\F$ par la Formule~(\ref{equa:Formule-Henaut}) et en tenant compte de la condition $K(\Leg\F)\equiv0,$ nous arrivons à la contradiction: $\deg\F<3$ ({\it i.e.} le lieu singulier de $\F$ contient une droite, une conique, ou une cubique).

\begin{prob}\label{prob:cas-nilpotent-noeud-ordre-2}
{\sl
Fournir une démonstration non calculatoire de la Proposition~\ref{pro:cas-nilpotent-noeud-ordre-2}.
}
\end{prob}

\Subsubsection{Cas d'une singularité dégénérée de multiplicité algébrique $3$}

Dans ce paragraphe nous nous intéressons aux feuilletages $\F\in\mathbf{FP}(3)$ ayant une singularité~$m$ dégénérée de multiplicité algébrique $3.$ Nous allons distinguer deux cas suivant que $\mu(\F,m)=13$ ou $\mu(\F,m)<13$ ({\it i.e.} suivant que $\Sing\F=\{m\}$ ou $\{m\}\varsubsetneq\Sing\F$).
\paragraph{Le cas $\nu(\F,m)=3$ et $\mu(\F,m)=13$}

Nous commençons par établir l'énoncé suivant de classification des feuilletages de $\mathbf{F}(3)$ dont le lieu singulier est réduit à un point de multiplicité algébrique~$3.$
\begin{pro}\label{pro:v=3-mu=13}
{\sl Soient $\mathcal{F}$ un feuilletage de degré trois sur $\pp$ ayant une seule singularité et $\omega$ une $1$-forme décrivant $\F.$ Si cette singularité est de multiplicité algébrique $3$, alors $\omega$ est, à isomorphisme près, de l'un des types suivants
\begin{itemize}
\item[\texttt{1. }]\hspace{1mm} $x^3\mathrm{d}x+y^2(c\hspace{0.1mm}x+y)(x\mathrm{d}y-y\mathrm{d}x),\hspace{1mm}c\in\mathbb{C}$;

\item[\texttt{2. }]\hspace{1mm} $x^3\mathrm{d}x+y(x+c\hspace{0.1mm}xy+y^2)(x\mathrm{d}y-y\mathrm{d}x),\hspace{1mm}c\in\mathbb{C}$;

\item[\texttt{3. }]\hspace{1mm} $x^3\mathrm{d}x+(x^2+c\hspace{0.1mm}xy^2+y^3)(x\mathrm{d}y-y\mathrm{d}x),\hspace{1mm}c\in\mathbb{C}$;

\item[\texttt{4. }]\hspace{1mm} $x^2y\mathrm{d}x+(x^3+c\hspace{0.1mm}xy^2+y^3)(x\mathrm{d}y-y\mathrm{d}x),\hspace{1mm}c\in\mathbb{C}$;

\item[\texttt{5. }]\hspace{1mm} $x^2y\mathrm{d}x+(x^3+\delta\,xy+y^3)(x\mathrm{d}y-y\mathrm{d}x),\hspace{1mm}\delta\in\mathbb{C}^*$;

\item[\texttt{6. }]\hspace{1mm} $x^2y\mathrm{d}y+(x^3+c\hspace{0.1mm}xy^2+y^3)(x\mathrm{d}y-y\mathrm{d}x),\hspace{1mm} c\in\mathbb{C}$;

\item[\texttt{7. }]\hspace{1mm} $xy(x\mathrm{d}y-\lambda\,y\mathrm{d}x)
                                 +(x^3+y^3)(x\mathrm{d}y-y\mathrm{d}x),\hspace{1mm}\lambda\in\mathbb{C}\setminus\{0,1\}$;

\item[\texttt{8. }]\hspace{1mm} $xy(y-x)\mathrm{d}x+(c_0\,x^3+c_1x^2y+y^3)(x\mathrm{d}y-y\mathrm{d}x),\hspace{1mm} c_0(c_0+c_1+1)\neq0$.
\end{itemize}
Ces huit $1$-formes ne sont pas linéairement conjuguées entre elles.
}
\end{pro}

\noindent Cette proposition est un analogue en degré $3$ d'un résultat sur les feuilletages de degré $2$ dû à D.~\textsc{Cerveau}, J.~\textsc{D\'eserti}, D.~\textsc{Garba Belko} et R.~\textsc{Meziani} (\cite[Proposition 1.8]{CDGBM10}). La démonstration que nous allons en donner est très voisine de~celle de~\cite{CDGBM10}; elle résultera des Lemmes \ref{lem:exclure-cône-4-droites}, \ref{lem:cône-3-droites}, \ref{lem:cône-2-droites} et \ref{lem:cône-1-droite} qui suivent. Dans ces quatre lemmes $\mathcal{F}$ désigne un feuilletage de degré trois sur $\pp$ défini par une~$1$-forme $\omega$ et tel que
\begin{itemize}
\item[\texttt{1. }] l'unique singularité de $\F$ soit $O=[0:0:1]$, {\it i.e.} $\mu(\F,O)=13$;
\item[\texttt{2. }] les jets d'ordre $1$ et $2$ de $\omega$ en $(0,0)$ soient nuls, {\it i.e.} $\nu(\F,O)=3$.
\end{itemize}
Dans ce cas
$$\omega=A(x,y)\mathrm{d}x+B(x,y)\mathrm{d}y+C(x,y)(x\mathrm{d}y-y\mathrm{d}x),$$
où $A,$ $B$ et $C$ sont des polynômes homogènes de degré $3.$ Le feuilletage $\mathcal{F}$ étant de degré trois, le cône tangent $xA+yB$ de $\omega$ en $(0,0)$ ne peut être identiquement nul. Le polynôme $C$ n'est pas non plus identiquement nul sinon la droite à l'infini serait invariante par $\mathcal{F}$ qui posséderait donc une singularité sur cette droite, ce qui est exclu. Nous allons raisonner suivant la nature du cône tangent qui, a priori, peut être quatre droites, trois
droites, deux droites ou une droite.
\begin{lem}\label{lem:exclure-cône-4-droites}
{\sl Tout facteur irréductible $L$ de $xA+yB$ divise $\mathrm{pgcd}(A,B)$ et ne divise pas~$C$. En particulier, le cône tangent de $\omega$ en $(0,0)$ ne peut être l'union de quatre droites distinctes.
}
\end{lem}

\begin{proof}[\sl D\'emonstration]
\`{A} isomorphisme près, nous pouvons nous ramener à $L=x$; alors $x$ divise $B.$ Ainsi sur la droite $x=0$ la forme $\omega$ s'écrit $A(0,y)\mathrm{d}x-y\,C(0,y)\mathrm{d}x=y^3\left(A(0,1)-y\,C(0,1)\right)\mathrm{d}x.$ Comme $O$ est l'unique singularit\'e de $\mathcal{F},$ le produit $A(0,1)C(0,1)$ est nul. Le point $[0:1:0]$ étant non singulier, $C(0,1)$ est non nul et par suite $A(0,1)=0$; d'où $x$ divise $A$ mais pas $C.$
\end{proof}

\begin{lem}\label{lem:cône-3-droites}
{\sl Si le cône tangent de $\omega$ en $(0,0)$ est
composé de trois droites distinctes, alors, à isomorphisme près, $\omega$ est du type
\begin{align*}
&\hspace{2cm} xy(y-x)\mathrm{d}x+(c_0\,x^3+c_1x^2y+y^3)(x\mathrm{d}y-y\mathrm{d}x),&& c_0(c_0+c_1+1)\neq0.
\end{align*}
}
\end{lem}

\begin{proof}[\sl D\'emonstration]
Nous pouvons supposer que $xA+yB=\ast x^2y(y-x),\ast\in\mathbb{C}^{*}$; il s'en suit que $\omega$ s'écrit (\cite{CM82})
\begin{Small}
\begin{align*}
& x^2y(y-x)\left(\lambda_0\frac{\mathrm{d}x}{x}+\lambda_1\frac{\mathrm{d}y}{y}+\lambda_2\frac{\mathrm{d}(y-x)}{y-x}
+\delta\hspace{0.1mm}\mathrm{d}\left(\frac{y}{x}\right)\right)+(c_0\,x^3+c_1x^2y+c_2xy^2+c_3y^3)(x\mathrm{d}y-y\mathrm{d}x),&& \delta,\lambda_i,c_i\in\mathbb{C}.
\end{align*}
\end{Small}
\hspace{-1mm}Nous avons donc
\begin{Small}
\begin{align*}
&A(x,y)=y\Big((y-x)(\lambda_0\,x-\delta y)-\lambda_2x^2\Big),&&
B(x,y)=x\Big((y-x)(\lambda_1x+\delta y)+\lambda_2xy\Big),&&
C(x,y)=\sum_{i=0}^{3}c_i\hspace{0.1mm}x^{3-i}y^{i}.
\end{align*}
\end{Small}
\hspace{-1mm}D'après le Lemme \ref{lem:exclure-cône-4-droites}, le polynôme $xy(y-x)$ divise $A$ et $B$ mais ne divise pas $C$; ceci se traduit par
\begin{small}
\begin{align*}
& A(0,1)=A(1,0)=A(1,1)=B(0,1)=B(1,0)=B(1,1)=0&& \text{et} && C(0,1)C(1,0)C(1,1)\neq0.
\end{align*}
\end{small}
\hfill
\vspace{-1cm}

\noindent Il en résulte que $\delta=\lambda_1=\lambda_2=0$\, et que \,$c_0c_3(c_0+c_1+c_2+c_3)\neq0$. Le feuilletage $\F$ étant de degré trois $\lambda_0$~est non nul; nous pouvons donc supposer que~$\lambda_0=1$, d'où
\begin{align*}
& \omega=xy(y-x)\mathrm{d}x+(c_0\,x^3+c_1x^2y+c_2xy^2+c_3y^3)(x\mathrm{d}y-y\mathrm{d}x).
\end{align*}
L'homothétie $\left(\frac{1}{c_3}x,\frac{1}{c_3}y\right)$ nous permet de supposer que $c_3=1$; par conséquent
\begin{align*}
& \omega=xy(y-x)\mathrm{d}x+(c_0\,x^3+c_1x^2y+c_2xy^2+y^3)(x\mathrm{d}y-y\mathrm{d}x),&& c_0(c_0+c_1+c_2+1)\neq0.
\end{align*}
Quitte à conjuguer $\omega$ par le difféomorphisme \begin{small}$\left(\dfrac{x}{1+c_2y},\dfrac{y}{1+c_2y}\right)$\end{small} le coefficient $c_2$ vaut $0$, d'où l'énoncé.
\end{proof}

\begin{lem}\label{lem:cône-2-droites}
{\sl Si le cône tangent de $\omega$ en $(0,0)$ est formé de deux droites distinctes, alors $\omega$ est, à conjugaison près, de l'un des types suivants
\begin{itemize}
\item[\texttt{1. }]\hspace{1mm} $x^2y\mathrm{d}x+(x^3+c\hspace{0.1mm}xy^2+y^3)(x\mathrm{d}y-y\mathrm{d}x),\hspace{1mm} c\in\mathbb{C}$;

\item[\texttt{2. }]\hspace{1mm} $x^2y\mathrm{d}x+(x^3+\delta\,xy+y^3)(x\mathrm{d}y-y\mathrm{d}x),\hspace{1mm} \delta\in\mathbb{C}^*$;

\item[\texttt{3. }]\hspace{1mm} $x^2y\mathrm{d}y+(x^3+c\hspace{0.1mm}xy^2+y^3)(x\mathrm{d}y-y\mathrm{d}x),\hspace{1mm} c\in\mathbb{C}$;

\item[\texttt{4. }]\hspace{1mm} $xy(x\mathrm{d}y-\lambda\,y\mathrm{d}x)+(x^3+y^3)(x\mathrm{d}y-y\mathrm{d}x),\hspace{1mm}
                                 \lambda\in\mathbb{C}\setminus\{0,1\}$.
\end{itemize}
}
\end{lem}

\begin{proof}[\sl D\'emonstration]
\`{A} conjugaison linéaire près nous sommes dans l'une des deux situations suivantes
\begin{itemize}
\item[($\mathfrak{a}$)] $xA+yB=\ast x^3y,\quad \ast\in\mathbb{C}^{*}$;

\item[($\mathfrak{b}$)] $xA+yB=\ast x^2y^2,\quad\hspace{-1.7mm} \ast\in\mathbb{C}^{*}$.
\end{itemize}
Commençons par étudier l'éventualité ($\mathfrak{a}$). Dans ce cas la $1$-forme $\omega$ s'écrit (\cite{CM82})
\begin{small}
\begin{align*}
& x^3y\left(\lambda_0\frac{\mathrm{d}x}{x}+\lambda_1\frac{\mathrm{d}y}{y}+\mathrm{d}\left(\frac{\delta_1xy+\delta_2y^2}{x^2}\right)\right)
+(c_0\,x^3+c_1x^2y+c_2xy^2+c_3y^3)(x\mathrm{d}y-y\mathrm{d}x),&& \lambda_i,\delta_i,c_i\in\mathbb{C}.
\end{align*}
\end{small}
\hspace{-1mm}Nous avons alors
\begin{small}
\begin{align*}
&A(x,y)=y(\lambda_0x^2-\delta_1xy-2\delta_2y^2),&&
B(x,y)=x(\lambda_1x^2+\delta_1xy+2\delta_2y^2),&&
C(x,y)=\sum_{i=0}^{3}c_i\hspace{0.1mm}x^{3-i}y^{i}.
\end{align*}
\end{small}
\hspace{-1mm}D'après le Lemme \ref{lem:exclure-cône-4-droites}, le polynôme $xy$ divise $A$ et $B$ mais ne divise pas $C$. Par suite $\delta_2=\lambda_1=0$\, et \,$c_0c_3\neq0.$ Le feuilletage $\F$ étant de degré trois le coefficient $\lambda_0$ est non nul et nous pouvons le supposer égal à $1.$ Ainsi $\F$ est décrit par \begin{align*}
& \omega=x^2y\mathrm{d}x+\delta_1xy(x\mathrm{d}y-y\mathrm{d}x)+(c_0\,x^3+c_1x^2y+c_2xy^2+c_3y^3)(x\mathrm{d}y-y\mathrm{d}x).
\end{align*}
\hspace{-1mm}La transformation linéaire diagonale \begin{small}$\left(\dfrac{1}{c_0}x,\sqrt[3]{\dfrac{1}{c_3c_0^2}}y\right)$\end{small} nous permet de supposer que $c_0=c_3=1$; par conséquent
\begin{align*}
& \omega=x^2y\mathrm{d}x+(x^3+\delta_1xy+c_1x^2y+c_2xy^2+y^3)(x\mathrm{d}y-y\mathrm{d}x).
\end{align*}

\noindent Si $\delta_1=0$, resp. $\delta_1\neq0$, en conjuguant $\omega$ par
\begin{Small}
\begin{align*}
& \left(\frac{x}{1+c_1y},\frac{y}{1+c_1y}\right),&&
\text{resp.}\hspace{1mm}
\left(
\frac{x}{1-\left(\dfrac{c_2}{\delta_1}\right)y-\left(\dfrac{\delta_1c_1+c_2}{\delta_1^2}\right)x},
\frac{y}{1-\left(\dfrac{c_2}{\delta_1}\right)y-\left(\dfrac{\delta_1c_1+c_2}{\delta_1^2}\right)x}
\right),
\end{align*}
\end{Small}
\hspace{-1mm}nous nous ramenons à $c_1=0$, resp. $c_1=c_2=0$, c'est-à-dire à
\begin{align*}
& \omega=x^2y\mathrm{d}x+(x^3+c_2xy^2+y^3)(x\mathrm{d}y-y\mathrm{d}x),&&
\text{resp.}\hspace{1mm}
\omega=x^2y\mathrm{d}x+(x^3+\delta_1xy+y^3)(x\mathrm{d}y-y\mathrm{d}x)\hspace{1mm};
\end{align*}
d'où les deux premiers modèles annoncés.
\vspace{2mm}

\noindent Envisageons maintenant la possibilité ($\mathfrak{b}$). Dans ce cas $\omega$ s'écrit (\cite{CM82})
\begin{small}
\begin{align*}
& x^2y^2\left(\lambda_0\frac{\mathrm{d}x}{x}+\lambda_1\frac{\mathrm{d}y}{y}+\mathrm{d}\left(\frac{\delta_1x^2+\delta_2y^2}{xy}\right)\right)
+(c_0\,x^3+c_1x^2y+c_2xy^2+c_3y^3)(x\mathrm{d}y-y\mathrm{d}x),&& \lambda_i,\delta_i,c_i\in\mathbb{C}.
\end{align*}
\end{small}
\hspace{-1mm}Ici $A(x,y)=y(\delta_1x^2+\lambda_0\,xy-\delta_2y^2)$\, et \,$B(x,y)=x(\delta_2y^2+\lambda_1xy-\delta_1x^2)$. Toujours d'après le Lemme \ref{lem:exclure-cône-4-droites}, $xy$ divise $\mathrm{pgcd}(A,B)$ et ne divise pas $C$, ce qui équivaut à $\delta_1=\delta_2=0$\, et \,$c_0c_3\neq0.$

\noindent Le feuilletage $\F$ étant de degré trois la somme $\lambda_0+\lambda_1$ est non nulle; un des coefficients $\lambda_i$ est alors non nul et nous pouvons évidemment le normaliser à $1.$ Comme les droites du cône~tangent ({\it i.e.} $x=0$ et $y=0$) jouent un rôle symétrique, il suffit de traiter l'éventualité $\lambda_1=1.$ Ainsi $\F$ est donné par
\begin{align*}
&\hspace{1cm}\omega=xy(x\mathrm{d}y+\lambda_0y\mathrm{d}x)+(c_0\,x^3+c_1x^2y+c_2xy^2+c_3y^3)(x\mathrm{d}y-y\mathrm{d}x),&& (\lambda_0+1)c_0c_3\neq0.
\end{align*}
Soit $\alpha$ dans $\mathbb{C}$ tel que $\alpha^3=\dfrac{1}{c_3c_0^2}$; posons $\beta=c_0\alpha^2.$ En faisant agir $\left(\alpha\,x,\beta y\right)$ sur $\omega$, nous pouvons supposer que $c_0=c_3=1$; par suite
\begin{align*}
&\hspace{1cm}\omega=xy(x\mathrm{d}y+\lambda_0y\mathrm{d}x)+(x^3+c_1x^2y+c_2xy^2+y^3)(x\mathrm{d}y-y\mathrm{d}x),&& \lambda_0+1\neq0.
\end{align*}
Si $\lambda_0=0$, resp. $\lambda_0\neq0$, en conjuguant $\omega$ par
\begin{align*}
& \left(\frac{x}{1-c_1x},\frac{y}{1-c_1x}\right),&&
\text{resp.}\hspace{1mm}
\left(\frac{x}{1+\left(\dfrac{c_2}{\lambda_0}\right)y-c_1x},\frac{y}{1+\left(\dfrac{c_2}{\lambda_0}\right)y-c_1x}\right),
\end{align*}
nous nous ramenons à $c_1=0$, resp. $c_1=c_2=0$, c'est-à-dire à
\begin{align*}
\hspace{1cm}
&\omega=x^2y\mathrm{d}y+(x^3+c_2xy^2+y^3)(x\mathrm{d}y-y\mathrm{d}x),&&\\
\hspace{1cm}
\text{resp.}\hspace{1mm}
&\omega=xy(x\mathrm{d}y+\lambda_0y\mathrm{d}x)+(x^3+y^3)(x\mathrm{d}y-y\mathrm{d}x),\quad \lambda_0(\lambda_0+1)\neq0\hspace{1mm};
\end{align*}
d'où les deux derniers modèles annoncés.
\end{proof}
\newpage
\hfill

\begin{lem}\label{lem:cône-1-droite}
{\sl Si le cône tangent de $\omega$ en $(0,0)$ est réduit à une seule droite, alors $\omega$ est, à isomorphisme près, de l'un des types suivants
\begin{itemize}
\item[\texttt{1. }]\hspace{1mm} $x^3\mathrm{d}x+y^2(c\hspace{0.1mm}x+y)(x\mathrm{d}y-y\mathrm{d}x),\hspace{1mm}c\in\mathbb{C}$;

\item[\texttt{2. }]\hspace{1mm} $x^3\mathrm{d}x+y(x+c\hspace{0.1mm}xy+y^2)(x\mathrm{d}y-y\mathrm{d}x),\hspace{1mm}c\in\mathbb{C}$;

\item[\texttt{3. }]\hspace{1mm} $x^3\mathrm{d}x+(x^2+c\hspace{0.1mm}xy^2+y^3)(x\mathrm{d}y-y\mathrm{d}x),\hspace{1mm}c\in\mathbb{C}$.
\end{itemize}
}
\end{lem}

\begin{proof}[\sl D\'emonstration]
Nous pouvons supposer que le cône tangent est la droite $x=0$; alors $\omega$ s'écrit
\begin{small}
\begin{align*}
& x^4\left(\lambda\frac{\mathrm{d}x}{x}
+\mathrm{d}\left(\frac{\delta_1x^2y+\delta_2xy^2+\delta_3y^3}{x^3}\right)\right)
+(c_0\,x^3+c_1x^2y+c_2xy^2+c_3y^3)(x\mathrm{d}y-y\mathrm{d}x),&&
\lambda,\delta_i,c_i\in\mathbb{C}.
\end{align*}
\end{small}
\hspace{-1mm}Nous avons donc
\begin{small}
\begin{align*}
&A(x,y)=\lambda\,x^3-\delta_1x^2y-2\delta_2xy^2-3\delta_3y^3,&&
B(x,y)=x(\delta_1x^2+2\delta_2xy+3\delta_3y^2),&&
C(x,y)=\sum_{i=0}^{3}c_i\hspace{0.1mm}x^{3-i}y^{i}.
\end{align*}
\end{small}
\hspace{-1mm}D'après le Lemme \ref{lem:exclure-cône-4-droites}, $x$ divise $A$ et $B$ mais ne divise pas $C$; par suite~$\delta_3=0$ et $c_3\neq0$. Quitte à faire agir l'homothétie $\left(\frac{1}{c_3}x,\frac{1}{c_3}y\right),$ nous pouvons effectuer la normalisation~$c_3=1$.

\noindent Le feuilletage $\F$ étant de degré trois le coefficient $\lambda$ est non nul et nous pouvons supposer que $\lambda=1.$ Ainsi $\F$ est décrit par
\begin{align*}
&& \omega=x^3\mathrm{d}x+(\delta_1x^2+2\delta_2xy+c_0\,x^3+c_1x^2y+c_2xy^2+y^3)(x\mathrm{d}y-y\mathrm{d}x).
\end{align*}
Nous avons les trois possibilités suivantes à étudier
\begin{itemize}
\item[--] $\delta_2\neq0$;

\item[--] $\delta_1=\delta_2=0$;

\item[--] $\delta_2=0,\delta_1\neq0$.
\end{itemize}
\vspace{2mm}

\textbf{\textit{1.}} Si $\delta_2\neq0$, alors en conjuguant $\omega$ par $\left(\alpha^2x,\alpha^{3/2}y-\alpha\delta_1x\right)$, où $\alpha=2\delta_2,$ nous nous ramenons à $\delta_1=0$ et $\delta_2=\frac{1}{2}.$ Par suite $\F$ est donné par
\begin{align*}
&& \omega=x^3\mathrm{d}x+(xy+c_0\,x^3+c_1x^2y+c_2xy^2+y^3)(x\mathrm{d}y-y\mathrm{d}x).
\end{align*}
La conjugaison par le difféomorphisme $\left(\dfrac{x}{1+c_0y-c_1x},\dfrac{y}{1+c_0y-c_1x}\right)$ permet de supposer que $c_0=c_1=0$; par conséquent
\begin{align*}
&& \omega=x^3\mathrm{d}x+y(x+c_2xy+y^2)(x\mathrm{d}y-y\mathrm{d}x)\hspace{1mm};
\end{align*}
d'où le deuxième modèle annoncé.
\newpage
\hfill

\textbf{\textit{2.}} Si $\delta_1=\delta_2=0$ la $1$-forme $\omega$ s'écrit
\begin{align*}
&& x^3\mathrm{d}x+(c_0\,x^3+c_1x^2y+c_2xy^2+y^3)(x\mathrm{d}y-y\mathrm{d}x).
\end{align*}

\noindent Soit $\alpha$ dans $\mathbb{C}$ tel que $3\alpha^2+2c_2\alpha+c_1=0.$ En faisant agir $\left(x,y+\alpha\,x\right)$ sur $\omega$, nous pouvons supposer que $c_1=0.$ Puis la conjugaison par le difféomorphisme $\left(\dfrac{x}{1+c_0y},\dfrac{y}{1+c_0y}\right)$ permet d'annuler $c_0$; d'où le premier modèle annoncé.
\vspace{2mm}

\textbf{\textit{3.}} Lorsque $\delta_2=0$ et $\delta_1\neq0$, la forme $\omega$ s'écrit \begin{align*}
&& x^3\mathrm{d}x+(\delta_1x^2+c_0\,x^3+c_1x^2y+c_2xy^2+y^3)(x\mathrm{d}y-y\mathrm{d}x).
\end{align*}
En faisant agir $\left(\delta_1^4x,\delta_1^3y\right)$ sur $\omega$, nous pouvons supposer que $\delta_1=1.$ Puis en conjuguant par $\left(\dfrac{x}{1-c_1y-(c_0+c_1)x},\dfrac{y}{1-c_1y-(c_0+c_1)x}\right)$, nous nous ramenons à $c_0=c_1=0$, c'est-à-dire\vspace{1mm} au troisième modèle annoncé.
\end{proof}

\begin{proof}[\sl D\'emonstration de la Proposition~\ref{pro:v=3-mu=13}] Il suffit de choisir des coordonnées affines $(x,y)$ telles que le point $(0,0)$ soit singulier de $\F$ et d'utiliser les Lemmes \ref{lem:exclure-cône-4-droites}, \ref{lem:cône-3-droites}, \ref{lem:cône-2-droites}, \ref{lem:cône-1-droite}.
\end{proof}

\noindent Nous sommes maintenant prêt pour décrire à isomorphisme près les feuilletages de $\mathbf{FP}(3)$ dont le lieu singulier est réduit à un point de multiplicité algébrique~$3.$

\begin{pro}\label{pro:overline-omega2}
{\sl Soit $\mathcal{F}$ un feuilletage de degré trois sur $\pp$ ayant une seule singularité. Supposons que cette singularité soit de multiplicité algébrique $3$ et que le $3$-tissu $\Leg\F$ soit plat. Alors $\F$ est linéairement conjugué au feuilletage $\F_2$ décrit par la $1$-forme $$\omegaoverline_{2}=x^{3}\mathrm{d}x+y^{3}(x\mathrm{d}y-y\mathrm{d}x).$$
}
\end{pro}

\begin{proof}[\sl D\'emonstration]
Soit $\omega$ une $1$-forme décrivant $\F$ dans une carte affine $(x,y)$ et soit $(p,q)$ la carte affine de $\pd$ correspondant à la droite $\{px-qy=1\}\subset\pp$. \`{A} conjugaison linéaire près $\omega$ est de l'un des huit types de la Proposition \ref{pro:v=3-mu=13}.
\begin{itemize}
      \smallskip
  \item [--] Si $\omega=x^3\mathrm{d}x+y^2(c\hspace{0.1mm}x+y)(x\mathrm{d}y-y\mathrm{d}x),\,c\in\mathbb{C}$, alors le $3$-tissu $\Leg\F$ est donné par l'équation différentielle $q(q')^3+cq'+1=0$, avec $q'=\frac{\mathrm{d}q}{\mathrm{d}p}$. Le calcul explicite de $K(\Leg\F)$ conduit~à $$K(\Leg\F)=-\dfrac{4c^2(2c^3+27q)}{q^2(4c^3+27q)^2}\mathrm{d}p\wedge\mathrm{d}q\hspace{1mm};$$
      par suite $\Leg\F$ est plat si et seulement si $c=0$, auquel cas $\omega=\omegaoverline_{2}=x^{3}\mathrm{d}x+y^{3}(x\mathrm{d}y-y\mathrm{d}x).$
      \newpage
      \hfill
      \vspace{-1cm}
  \item [--] Si $\omega=x^3\mathrm{d}x+y(x+c\hspace{0.1mm}xy+y^2)(x\mathrm{d}y-y\mathrm{d}x),\,c\in\mathbb{C}$, alors $\Leg\F$ est décrit par l'équation différentielle $F(p,q,w):=qw^3+pw^2+(c-q)w+1=0$, avec $w=\frac{\mathrm{d}q}{\mathrm{d}p}$. Le calcul explicite de $K(\Leg\F)$ montre qu'elle s'écrit sous la forme $$K(\Leg\F)=\frac{\sum\limits_{i+j\leq6}\rho_{i}^{j}(c)p^iq^j}{\Delta(p,q)^2}\mathrm{d}p\wedge\mathrm{d}q,$$
      où $\Delta$ est le $w$-discriminant de $F$ et les $\rho_{i}^{j}$ sont des polynômes en $c$ avec $\rho_{1}^{5}(c)=4\neq0$; d'où $K(\Leg\F)\not\equiv0.$

      \noindent De façon analogue nous vérifions que $\Leg\F$ ne peut être plat lorsque $\F$ est donné par l'une des six dernières $1$-formes de la Proposition \ref{pro:v=3-mu=13}.
\end{itemize}
\end{proof}

\paragraph{Le cas $\nu(\F,m)=3$ et $\mu(\F,m)<13$}

Nous commençons par démontrer trois lemmes.

\begin{lem}\label{lem:v=3-leg-plat}
{\sl Soient $\mathcal{F}$ un feuilletage de degré trois sur $\pp$, $m$ un point singulier de $\F$ et $\omega$ une $1$-forme décrivant~$\F$. Supposons que cette singularité soit de multiplicité algébrique $3$ et que le $3$-tissu $\Leg\F$ soit plat. Alors
\begin{itemize}
  \item [--] ou bien $\F$ est homogène;
  \item [--] ou bien le $3$-jet de $\omega$ en $m$ n'est pas saturé.
\end{itemize}
}
\end{lem}

\begin{rem}\label{rem:caractérisation-feuilletage-homogène}
Notons qu'un feuilletage de degré $d$ sur $\pp$ est homogène si et seulement s'il possède une singularité de multiplicité algébrique maximale ({\it i.e.} égale à $d$) et une droite invariante ne passant pas par cette singularité.
\end{rem}

\begin{proof}[\sl D\'emonstration]
Choisissons un système de coordonnées homogènes $[x:y:z]\in\pp$ dans lequel $m=[0:0:1]$. La condition $\nu(\F,m)=3$ assure que toute $1$-forme $\omega$ définissant $\F$ dans la carte affine $(x,y)$ est du type $\omega=\theta_3+C_3(x,y)(x\mathrm{d}y-y\mathrm{d}x)$, où $\theta_3$ (resp. $C_3$) est une $1$-forme (resp. un polynôme) homogène de degré $3$; la $1$-forme $\theta_3$ représente le $3$-jet de $\omega$ en $(0,0).$

\noindent Supposons que $\theta_3$ soit saturé; nous allons prouver que $\F$ est nécessairement homogène. Notons $\mathcal{H}$ le feuilletage homogène de degré trois sur $\pp$ défini par $\theta_3$; $\mathcal{H}$ est bien défini grâce à l'hypothèse sur $\theta_3.$ Considérons la famille d'homothéties $\varphi=\varphi_{\varepsilon}=(\varepsilon\,x,\varepsilon\hspace{0.1mm}y).$ Nous avons
$$
\frac{1}{\varepsilon^4}\varphi^*\omega=\theta_3+\varepsilon\,C_3(x,y)(x\mathrm{d}y-y\mathrm{d}x)
$$
qui tend vers $\theta_3$ lorsque $\varepsilon$ tend vers $0$; il en résulte que $\mathcal{H}\in\overline{\mathcal{O}(\F)}.$ Le $3$-tissu $\Leg\F$ est par~hypothèse plat; il en est donc de même pour le $3$-tissu $\Leg\mathcal{H}.$ Le feuilletage $\mathcal{H}$ est alors linéairement conjugué à l'un des onze feuilletages homogènes donnés par le Théorème~\ref{thm:Class-Homog3-Plat}.
\newpage
\hfill

\noindent Ainsi en consultant la Table~\ref{tab:CS(lambda)}, nous constatons immédiatement que $\mathcal{H}$ possède au moins une singularité $m_0$ non-dégénérée vérifiant $\mathrm{BB}(\mathcal{H},m_0)\not\in\{4,\frac{16}{3}\}.$ Soit $(\F_{\varepsilon})_{\varepsilon\in\mathbb{C}}$ la famille de feuilletages définis par $\omega_{\varepsilon}=\theta_3+\varepsilon\,C_3(x,y)(x\mathrm{d}y-y\mathrm{d}x).$ D'après ce qui précède, pour $\varepsilon$ non nul $\F_{\varepsilon}$ est dans $\mathcal{O}(\F)$ et pour $\varepsilon$ nul nous avons $\F_{\varepsilon=0}=\mathcal{H}.$ La singularité $m_0$ de $\mathcal{H}$ est \og stable\fg; il existe une famille $(m_\varepsilon)_{\varepsilon\in\mathbb{C}}$ de singularités non-dégénérées de $\F_{\varepsilon}$ telle que $m_{\varepsilon=0}=m_0.$ Les $\F_{\varepsilon}$ étant conjugués pour~$\varepsilon$ non nul, $\mathrm{BB}(\F_{\varepsilon},m_\varepsilon)$ est localement constant; par suite $\mathrm{BB}(\F_{\varepsilon},m_\varepsilon)=\mathrm{BB}(\mathcal{H},m_0)$ pour $\varepsilon$ petit. En particulier $\F$ possède une singularité $m'$ non-dégénérée vérifiant $\mathrm{BB}(\F,m')=\mathrm{BB}(\mathcal{H},m_0)$ de sorte que $\mathrm{BB}(\F,m')\not\in\{4,\frac{16}{3}\}.$ D'après la première assertion du Lemme \ref{lem:2-droite-invar} par le point $m'$ passent exactement deux droites invariantes par $\F$, dont au moins une est nécessairement distincte de $(mm')$; ceci implique, d'après la Remarque \ref{rem:caractérisation-feuilletage-homogène}, que $\F$ est homogène.
\end{proof}

\begin{lem}\label{lem:v=3-mu<13-jet-non-saturé}
{\sl Soient $\mathcal{F}$ un feuilletage de degré trois sur $\pp$ ayant une singularité $m$ de multiplicité algébrique $3$ et $\omega$ une $1$-forme décrivant~$\F$. Supposons que le lieu singulier de $\F$ ne soit pas réduit à $m$ et que le $3$-jet de $\omega$ en $m$ ne soit pas saturé. Alors $\omega$ est, à isomorphisme près, du type suivant
$$
y(a_0\,x^2+a_1xy+y^2)\mathrm{d}x+xy(b_0\,x+b_1y)\mathrm{d}y+x(x^2+c_1xy+c_2y^2)(x\mathrm{d}y-y\mathrm{d}x),
$$
où $a_0,a_1,b_0,b_1,c_1,c_2$ sont des nombres complexes tels que le degré du feuilletage associé soit égal à $3.$
}
\end{lem}

\begin{proof}[\sl D\'emonstration]
La condition $\nu(\F,m)=3$ assure l'existence d'un système de coordonnées homogènes $[x:y:z]\in\pp$ dans lequel $m=[0:0:1]$ et $\F$ est défini par une $1$-forme $\omega$ du type $$\omega=A(x,y)\mathrm{d}x+B(x,y)\mathrm{d}y+C(x,y)(x\mathrm{d}y-y\mathrm{d}x),$$ où $A,$ $B$ et $C$ sont des polynômes homogènes de degré $3.$ Comme $J^3_{(0,0)}\omega$ est par hypothèse non saturé, nous pouvons écrire
\begin{align*}
&A(x,y)=(h_0\,x+h_1y)(a_0\,x^2+a_1xy+a_2y^2)&&\hspace{1.5mm} \text{et}  \hspace{1.5mm}&& B(x,y)=(h_0\,x+h_1y)(b_0\,x^2+b_1xy+b_2y^2).
\end{align*}
Écrivons $C(x,y)=\sum_{i=0}^{3}c_i\hspace{0.1mm}x^{3-i}y^{i}$. L'hypothèse $\Sing\F\neq\{m\}$ nous permet de supposer que le point $m'=[0:1:0]$ est singulier de $\F,$ ce qui revient à supposer que $c_3=h_1b_2=0$. Le feuilletage $\F$ étant de degré trois le produit $h_1a_2$ est non nul et par suite $b_2=0$; quitte à poser $h_0=h_0'h_1,\,a_i=a_i'a_2,\,b_i=b_i'a_2,\,c_{\hspace{-0.4mm}j}=c_{\hspace{-0.4mm}j}'h_1a_2$, avec $i\in\{0,1\}$ et $j\in\{0,1,2\}$, nous pouvons supposer que $h_1=a_2=1.$ Ainsi $\omega$ s'écrit
\begin{align*}
&(h_0\,x+y)\left((a_0\,x^2+a_1xy+y^2)\mathrm{d}x+x(b_0\,x+b_1y)\mathrm{d}y\right)+x(c_0\,x^2+c_1xy+c_2y^2)(x\mathrm{d}y-y\mathrm{d}x).
\end{align*}
\newpage
\hfill
\medskip

\noindent La conjugaison par le difféomorphisme $\left(x,y-h_0\,x\right)$ permet d'annuler $h_0$; par conséquent
\begin{align*}
&\omega=y\left((a_0\,x^2+a_1xy+y^2)\mathrm{d}x+x(b_0\,x+b_1y)\mathrm{d}y\right)+x(c_0\,x^2+c_1xy+c_2y^2)(x\mathrm{d}y-y\mathrm{d}x).
\end{align*}
L'égalité $\deg\F=3$ implique que $c_0\neq0$. En faisant agir l'homothétie $\left(\frac{1}{c_0}x,\frac{1}{c_0}y\right)$ sur $\omega$, nous nous ramenons à $c_0=1$, c'est-à-dire au modèle annoncé.
\end{proof}

\begin{lem}\label{lem:v=3-mu<13-leg-plat}
{\sl Soit $\mathcal{F}$ un feuilletage de degré trois sur $\pp$ possédant une singularité $m$ de multiplicité algébrique $3.$ Supposons que le lieu singulier de $\F$ ne soit pas réduit à $m$ et que le $3$-tissu $\Leg\F$ soit plat. Alors toute singularité $m'$ distincte de $m$ est non-dégénérée, la droite $(mm')$ est $\F$-invariante et
\begin{itemize}
  \item [--] soit $\F$ est homogène;
  \item [--] soit $\mathrm{CS}(\F,(mm'),m')\in\{1,3\}$ pour tout $m'\in\Sing\F\smallsetminus\{m\}.$
\end{itemize}
}
\end{lem}

\begin{proof}[\sl D\'emonstration]
Soit $m'$ un point singulier de $\F$ distinct de $m.$ D'après la Remarque~\ref{rems:mu>nu^2-Darboux}~(iii) le point $m$ est l'unique singularité de $\F$ de multiplicité algébrique $3$ et par suite $\nu(\F,m')\leq2.$ Si la singularité $m'$ était dégénérée, alors, d'après la Proposition~\ref{pro:cas-nilpotent-noeud-ordre-2}, le $3$-tissu $\Leg\F$ ne serait pas plat ce qui, par hypothèse, est impossible. Donc $\mu(\F,m')=1.$

\noindent Comme $\deg\F=3,\tau(\F,m)=3$\, et \,$\tau(\F,m')\geq1$, la droite $(mm')$ est invariante par $\F$ (sinon la formule (\ref{equa:sum-tau<d}) donnerait $3\geq\tau(\F,m)+\tau(\F,m')\geq4$).

\noindent Supposons qu'il soit possible de choisir $m'$ de telle sorte que $\mathrm{CS}(\F,(mm'),m')\not\in\{1,3\}$; nous allons montrer que $\F$ est nécessairement homogène. L'égalité $\mu(\F,m')=1$ et la condition $\mathrm{CS}(\F,(mm'),m')\neq1$ impliquent que $\mathrm{BB}(\F,m')\neq4.$
\begin{itemize}
  \item [--] Si $\mathrm{BB}(\F,m')\neq\frac{16}{3}$, alors, d'après la première assertion du Lemme~\ref{lem:2-droite-invar}, par le point~$m'$ passe une droite invariante par $\F$ et distincte de la droite $(mm')$, ce qui entraîne, d'après la Remarque~\ref{rem:caractérisation-feuilletage-homogène}, que $\F$ est homogène;
  \item [--] Si $\mathrm{BB}(\F,m')=\frac{16}{3},$ alors, d'après la seconde assertion du Lemme~\ref{lem:2-droite-invar}, par la singularité $m'$ passe une droite $\ell$ invariante par $\F$ et telle que $\mathrm{CS}(\F,\ell,m')=3$; comme nous avons supposé que $\mathrm{CS}(\F,(mm'),m')\neq3,$ nous en déduisons que $\ell\neq(mm')$, ce qui implique que $\F$ est homogène, en vertu de la même remarque.
\end{itemize}
\end{proof}
\bigskip

\noindent Nous sommes maintenant parés pour décrire à isomorphisme près les feuilletages inhomogènes de $\mathbf{FP}(3)$ dont le lieu singulier contient un point de multiplicité algébrique~$3$ et de nombre de \textsc{Milnor} inférieur strictement à~$13.$
\newpage
\hfill

\begin{pro}\label{pro:overline-omega1-omega4-omega5}
{\sl Soit $\mathcal{F}$ un feuilletage de degré trois sur $\pp.$ Supposons que $\F$ possède une singularité de multiplicité algébrique $3$ et que $\Sing\F$ ne soit pas réduit à cette singularité. Supposons en outre que le $3$-tissu $\Leg\F$ soit plat. Alors ou bien $\F$ est homogène, ou bien $\F$ est, à action d'un automorphisme de $\pp$ près, défini par l'une des $1$-formes suivantes
\begin{itemize}
\item [\texttt{1. }] $\omegaoverline_{1}=y^{3}\mathrm{d}x+x^{3}(x\mathrm{d}y-y\mathrm{d}x)$;

\item [\texttt{2. }] $\omegaoverline_{4}=(x^{3}+y^{3})\mathrm{d}x+x^{3}(x\mathrm{d}y-y\mathrm{d}x)$;

\item [\texttt{3. }] $\omegaoverline_{5}=y^{2}(y\mathrm{d}x+2x\mathrm{d}y)+x^{3}(x\mathrm{d}y-y\mathrm{d}x)$.
\end{itemize}
}
\end{pro}

\begin{proof}[\sl D\'emonstration]
\noindent Supposons que $\F$ ne soit pas homogène; nous devons montrer qu'à conjugaison linéaire près $\F$ est décrit par l'une des trois $1$-formes $\omegaoverline_{1},\omegaoverline_{4},\omegaoverline_{5}$.

\noindent Notons $m$ la singularité de multiplicité algébrique~$3$ de $\F$ et soit $\omega$ une $1$-forme décrivant $\F$ dans une carte affine $(x,y)$ de $\pp.$ Comme $\Leg\F$ est par hypothèse plat, il en résulte, d'après le Lemme \ref{lem:v=3-leg-plat}, que le $3$-jet de $\omega$ en $m$ n'est pas saturé. Par hypothèse nous avons $\Sing\F\neq\{m\}$. Par suite le Lemme~\ref{lem:v=3-mu<13-jet-non-saturé} nous assure que $\omega$ est, à isomorphisme près, du type suivant
$$
y(a_0\,x^2+a_1xy+y^2)\mathrm{d}x+xy(b_0\,x+b_1y)\mathrm{d}y+x(x^2+c_1xy+c_2y^2)(x\mathrm{d}y-y\mathrm{d}x),\quad a_i,b_i,c_{\hspace{-0.4mm}j}\in\mathbb{C}.
$$
Dans cette situation, $m=[0:0:1]$, $m':=[0:1:0]\in\Sing\F$ et la droite $(mm')=(x=0)$ est invariante par $\F$; de plus un calcul montre que $\mathrm{CS}(\F,(mm'),m')=1+b_1.$ Le Lemme~\ref{lem:v=3-mu<13-leg-plat} implique alors que $b_1\in\{0,2\}.$

\noindent Si $b_0\neq0$,\hspace{2mm} resp. $(b_0,b_1)=(0,2)$,\hspace{2mm} resp. $b_0=b_1=0,c_2\neq0$,\hspace{2mm} resp. $b_0=b_1=c_2=0,c_1\neq0$, alors en conjuguant par
\begin{Small}
 \begin{align*}
&\left(\dfrac{b_0^2\,x}{1-c_1b_0\,x},\dfrac{b_0^3y}{1-c_1b_0\,x}\right),&&
\text{resp.}\hspace{1mm}
\left(\frac{x}{1-\left(\dfrac{c_2}{2}\right)x},\frac{y}{1-\left(\dfrac{c_2}{2}\right)x}\right),&&
\text{resp.}\hspace{1mm}
\left(c_2^{-1}x,c_2^{-3/2}\,y\right),&&
\text{resp.}\hspace{1mm}
\left(c_1^{-2}x,c_1^{-3}y\right),
\end{align*}
\end{Small}
\hspace{-1mm}nous nous ramenons à $(b_0,c_1)=(1,0)$,\hspace{2mm} resp. $c_2=0$,\hspace{2mm} resp. $c_2=1$,\hspace{2mm} resp. $c_1=1.$ Il nous suffit donc de traiter les possibilités suivantes
\begin{align*}
& (b_0,b_1,c_1)=(1,0,0),&&\hspace{4mm} (b_0,b_1,c_1)=(1,2,0),&&\hspace{4mm} (b_0,b_1,c_2)=(0,2,0),&&\\
& (b_0,b_1,c_2)=(0,0,1),&&\hspace{4mm} (b_0,b_1,c_1,c_2)=(0,0,1,0),&&\hspace{4mm} (b_0,b_1,c_1,c_2)=(0,0,0,0).
\end{align*}

\noindent Plaçons-nous dans la carte affine $(p,q)$ de $\pd$ associée à la droite $\{py-qx=1\}\subset\pp$; le $3$-tissu $\Leg\F$ est décrit par l'équation différentielle
$$F(p,q,w):=pw^3+(a_1p+b_1q-c_2)w^2+(a_0p+b_0q-c_1)w-1=0,\qquad \text{avec} \qquad w=\frac{\mathrm{d}q}{\mathrm{d}p}.$$
\newpage
\hfill

\noindent Le calcul explicite de $K(\Leg\F)$ montre qu'elle s'écrit sous la forme $$K(\Leg\F)=\frac{\sum\limits_{i+j\leq6}\rho_{i}^{j}p^iq^j}{\Delta(p,q)^2}\mathrm{d}p\wedge\mathrm{d}q,$$ où $\Delta$ est le $w$-discriminant de $F$ et les $\rho_{i}^{j}$ sont des polynômes en les paramètres $a_i,b_i,c_{\hspace{-0.4mm}j}.$
\vspace{2mm}

\textbf{\textit{1.}} Si $(b_0,b_1,c_1)=(1,0,0)$, alors le calcul explicite de $K(\Leg\F)$ conduit à $\rho_{0}^{5}=4c_2$ et
\begin{Small}
$$
\rho_{2}^{1}=a_0^2(a_0+8a_1)c_2^3-(177a_0^2-60a_0-24a_0a_1-84a_0a_1^2+24a_1^2+24a_1^3)c_2^2+(108a_1^2-36a_1-81a_0)c_2+81,
$$
\end{Small}
\hspace{-1mm}de sorte que le système $\rho_{0}^{5}=\rho_{2}^{1}=0$ n'a pas de solutions. Donc ce premier cas n'arrive pas.
\vspace{2mm}

\textbf{\textit{2.}} Lorsque $(b_0,b_1,c_1)=(1,2,0)$, le calcul explicite de $K(\Leg\F)$ nous donne:
\begin{Small}
\begin{align*}
\begin{array}{lll}
\rho_{0}^{6}=-4(6a_0-5a_1+4),\\
\rho_{1}^{5}=-4(12a_0^2+20a_0-10a_0a_1-3a_1^2),\\
\rho_{5}^{0}=32a_0^5c_2-8(a_1^2c_2+2a_1c_2+12)a_0^4+4(a_1^3c_2+4a_1^2+a_1-12)a_0^3+(4a_1^2-5a_1+30)a_0^2a_1^2-4a_0a_1^4\hspace{1mm};
\end{array}
\end{align*}
\end{Small}
\hspace{-1mm}il est facile de voir que le système $\rho_{0}^{6}=\rho_{1}^{5}=\rho_{5}^{0}=0$ n'a pas de solutions. Donc ce deuxième cas n'est pas possible.
\vspace{2mm}

\textbf{\textit{3.}} Quand $(b_0,b_1,c_2)=(0,2,0)$, le calcul explicite de $K(\Leg\F)$ nous fournit:
\begin{align*}
&\rho_{1}^{0}=24c_1^4,&&\rho_{1}^{4}=-256a_0^2,&&\rho_{0}^{4}=64(14a_1+3a_0c_1),
\end{align*}
de sorte que $a_0=a_1=c_1=0.$ Par conséquent, dans ce troisième cas, $\F$ est donné par $$\omegaoverline_{5}=y^2(y\mathrm{d}x+2x\mathrm{d}y)+x^3(x\mathrm{d}y-y\mathrm{d}x)\hspace{1mm};$$
nous vérifions par le calcul que sa transformée de \textsc{Legendre} est plate.
\vspace{2mm}

\textbf{\textit{4.}} Lorsque $(b_0,b_1,c_2)=(0,0,1)$, le calcul explicite de $K(\Leg\F)$ nous donne:
\begin{SMALL}
\begin{align*}
\begin{array}{llll}
\vspace{1mm}
\rho_{5}^{0}=2a_0^2(4a_0-a_1^2)(2a_0^2-6a_0-a_0a_1c_1+2a_1^2),\\
\vspace{1mm}
\rho_{1}^{0}=2\left(c_1^2-4\right)\left((12a_0-a_1^2)c_1^2-(5a_0+18)a_1c_1+a_0^2-24a_0+14a_1^2+27\right),\\
\vspace{1mm}
\rho_{4}^{0}=(24a_0-5a_1^2)a_0^2a_1c_1^2-(60a_0^3-36a_0^2-11a_0^2a_1^2+45a_0a_1^2-8a_1^4)a_0c_1+2a_1(2a_0^4+24a_0^3-54a_0^2-5a_0^2a_1^2+30a_0a_1^2-4a_1^4),\\
\rho_{2}^{0}=(8a_0-a_1^2)a_1c_1^4-(64a_0^2+60a_0-7a_0a_1^2+7a_1^2)c_1^3+3(5a_0^2+52a_0-2a_1^2+27)a_1c_1^2-2a_1(a_0^2+162a_0-48a_1^2-27)\\
\hspace{6mm}-(a_0^3-177a_0^2-81a_0+84a_0a_1^2+108a_1^2+81)c_1\hspace{1mm};
\end{array}
\end{align*}
\end{SMALL}
\hspace{-1mm}le système $\rho_{1}^{0}=\rho_{2}^{0}=\rho_{4}^{0}=\rho_{5}^{0}=0$ équivaut à $(a_0,a_1,c_1)\in\{(1,2,2),(1,-2,-2)\}$ comme le montre un calcul explicite. Si $(a_0,a_1,c_1)=(1,2,2)$, resp. $(a_0,a_1,c_1)=(1,-2,-2)$, alors $\omega$ s'écrit
\begin{align*}
& (x+y)^2(y\mathrm{d}x+x(x\mathrm{d}y-y\mathrm{d}x)),&&
\text{resp.}\hspace{1mm}
(x-y)^2(y\mathrm{d}x+x(x\mathrm{d}y-y\mathrm{d}x)),
\end{align*}
ce qui contredit l'égalité $\deg\F=3$.
\newpage
\hfill

\textbf{\textit{5.}} Lorsque $(b_0,b_1,c_1,c_2)=(0,0,1,0)$, le calcul explicite de $K(\Leg\F)$ montre que:
\begin{small}
\begin{align*}
&\rho_{3}^{0}=-2a_0^2(6a_0+a_0a_1-2a_1^2)(4a_0-a_1^2),&& \rho_{0}^{0}=-81-60a_0+8a_0a_1+81a_1-7a_1^2-a_1^3,\\
&\rho_{1}^{0}=-4(a_1-3)(6a_0^2-9a_0a_1-a_0a_1^2+2a_1^3)\hspace{1mm};
\end{align*}
\end{small}
\hspace{-1mm}le système $\rho_{0}^{0}=\rho_{1}^{0}=\rho_{3}^{0}=0$ est vérifié si et seulement si $(a_0,a_1)=(2,3),$ auquel cas $$\omega=(x+y)\big(y(y+2x)\mathrm{d}x+x^2(x\mathrm{d}y-y\mathrm{d}x)\big),$$ mais ceci contredit l'égalité $\deg\F=3.$
\vspace{2mm}

\textbf{\textit{6.}} Si $(b_0,b_1,c_1,c_2)=(0,0,0,0),$ alors $\omega=y(a_0\,x^2+a_1xy+y^2)\mathrm{d}x+x^3(x\mathrm{d}y-y\mathrm{d}x)$; l'équation différentielle décrivant $\Leg\F$ s'écrit
$$F(p,q,q')=p(q')^3+a_1p(q')^2+a_0pq'-1=0,\quad \text{avec} \quad q'=\frac{\mathrm{d}q}{\mathrm{d}p}.$$
Nous étudions deux éventualités suivant que $a_1$ est nul ou non.
\vspace{2mm}

\textbf{\textit{6.1.}} Lorsque $a_1=0$ le calcul explicite de $K(\Leg\F)$ nous donne $$K(\Leg\F)=-\dfrac{48a_0^4\hspace{0.1mm}p}{(4a_0^3\hspace{0.1mm}p^2+27)^2}\mathrm{d}p\wedge\mathrm{d}q\hspace{1mm};$$
par suite $\Leg\F$ est plat si et seulement si $a_0=0$, auquel cas $\F$ est décrit par
$$\omegaoverline_{1}=y^3\mathrm{d}x+x^3(x\mathrm{d}y-y\mathrm{d}x).$$

\textbf{\textit{6.2.}} Si $a_1\neq0$, alors quitte à conjuguer $\omega$ par $\left(\alpha^2x,\alpha^3y\right)$, où $\alpha=\frac{1}{3}a_1$, nous pouvons supposer que $a_1=3.$ Dans ce cas le calcul explicite de $K(\Leg\F)$ montre que
\begin{align*}
K(\Leg\F)=-\dfrac{12\left(a_0-3\right)\left(a_0^2(4a_0-9)p+27(a_0-2)\right)}{\left(a_0^2(4a_0-9)p^2+54(a_0-2)p+27\right)^2}
\mathrm{d}p\wedge\mathrm{d}q\hspace{1mm};
\end{align*}
par conséquent $\Leg\F$ est plat si et seulement si $a_0=3$, auquel cas $$\omega=y(3x^2+3xy+y^2)\mathrm{d}x+x^3(x\mathrm{d}y-y\mathrm{d}x).$$ Quitte à faire agir le difféomorphisme $\left(x,-x-y\right)$ sur $\omega$, le feuilletage $\F$ est donné dans les coordonnées affines $(x,y)$ par la $1$-forme $$\omegaoverline_{4}=(x^{3}+y^{3})\mathrm{d}x+x^{3}(x\mathrm{d}y-y\mathrm{d}x).$$
\end{proof}
\newpage
\hfill
\medskip

\begin{rem}
Les cinq feuilletages $\mathcal{F}_1,\ldots,\mathcal{F}_{5}$ ont les propriétés suivantes
\vspace{0.2mm}
\begin{itemize}
\item [(i)]   $\#\Sing\F_2=1,$\, $\#\Sing\F_3=13$\, et \,$\#\Sing\F_j=2$ pour $j=1,4,5$;
\item [(ii)]  $\F_j$ est convexe si et seulement si $j\in\{1,3\}$;
\item [(iii)] $\F_j$ possède une singularité radiale d'ordre $2$ si et seulement si $j\in\{1,3,4\}$;
\item [(iv)]  $\F_j$ admet un point d'inflexion double si et seulement si $j\in\{2,4\}.$
\end{itemize}
\vspace{0.2mm}
Les vérifications de ces propriétés sont aisées et laissées au lecteur.
\end{rem}

\begin{rem}\label{rem:non-conjugaison-Hi-Fj}
Les seize feuilletages $\mathcal{H}_1,$ $\ldots,\mathcal{H}_{11},$ $\mathcal{F}_1,$ $\ldots,\mathcal{F}_{5}$ ne sont pas linéairement conjugués. En effet, par construction, les $\F_j$ ne sont pas homogènes et ne sont donc pas conjugués~aux~homogènes~$\mathcal{H}_{\hspace{0.2mm}i}.$ Les $\mathcal{H}_{\hspace{0.2mm}i}$ ne sont pas linéairement conjugués (Théorèmes~\ref{thm:Class-Homog3-Plat}). Enfin le fait que les $\F_j$ ne sont pas linéairement conjugués découle des propriétés (i),~(ii) et (iii) ci-dessus.
\end{rem}

\noindent Le Théorème~\ref{thm:classification} résulte des Théorèmes~\ref{thm:Class-Homog3-Plat}, \ref{thm:Fermat}, des Propositions~\ref{pro:cas-nilpotent-noeud-ordre-2}, \ref{pro:overline-omega2}, \ref{pro:overline-omega1-omega4-omega5} et de la Remarque~\ref{rem:non-conjugaison-Hi-Fj}.

\begin{cor}\label{cor:class-convexe-3}
{\sl \`A automorphisme de $\pp$ près, il y a quatre feuilletages convexes de degré trois sur le plan projectif complexe, à savoir les feuilletages  $\mathcal{H}_{1},\mathcal{H}_{\hspace{0.2mm}3},\F_1$ et $\F_3.$
}
\end{cor}

\begin{proof}[\sl D\'emonstration]
D'après \cite[Corollaire~4.7]{BFM13} tout feuilletage convexe de degré trois sur $\pp$ a une transformée de \textsc{Legendre} plate, et donc est linéairement conjugué à l'un des seize feuilletages donnés par le Théorème~\ref{thm:classification}. Or les seuls feuilletages convexes apparaissant dans ce théorème sont $\mathcal{H}_{1},\mathcal{H}_{\hspace{0.2mm}3},\F_1$ et $\F_3$, d'où le résultat.
\end{proof}

\noindent Le Corollaire~\ref{cor:class-convexe-3} est un analogue en degré $3$ d'un résultat sur les feuilletages de degré $2$ dû à C.~\textsc{Favre} et J.~\textsc{Pereira} (\cite[Proposition~7.4]{FP15}).
\bigskip

\noindent  D.~\textsc{Mar\'{\i}n} et J. \textsc{Pereira} dans \cite{MP13} ont étudié les feuilletages de $\mathbf{F}(d)$ qui sont convexes à diviseur d'inflexion réduit; ils ont montré que l'ensemble formé de tels feuilletages est contenu dans $\mathbf{FP}(d)$, \emph{voir} \cite[Théorème 2]{MP13}. \`{A} notre connaissance les seuls feuilletages convexes à diviseur d'inflexion réduit connus dans la littérature sont ceux qui sont présentés dans \cite[Table~1.1]{MP13}: le feuilletage de \textsc{Fermat} $\F^{d}$ de degré $d$ et les trois feuilletages donnés par les $1$-formes
\[(2x^{3}-y^{3}-1)y\mathrm{d}x+(2y^{3}-x^{3}-1)x\mathrm{d}y\,,\]
\[(y^2-1)(y^2- (\sqrt{5}-2)^2)(y+\sqrt{5}x) \mathrm{d}x-(x^2-1)(x^2- (\sqrt{5}-2)^2)(x+\sqrt{5}y) \mathrm{d}y \, ,\]
\[(y^3-1)(y^3+7x^3+1) y \mathrm{d}x-(x^3-1)(x^3+7 y^3+1) x \mathrm{d}y\,,\]
\newpage
\hfill

\noindent qui sont de degré  $4$, $5$ et $7$ respectivement. Dans \cite[Problème~9.1]{MP13} les auteurs posent la question suivante: y a-t-il d'autres feuilletages convexes à diviseur d'inflexion réduit? Le résultat suivant donne une réponse négative en degré trois à cette question.

\begin{cor}\label{cor:convexe-réduit-3}
{\sl Tout feuilletage convexe de degré $3$ sur $\pp$ à diviseur d'inflexion réduit est linéairement conjugué au feuilletage de \textsc{Fermat} $\F_3$.}
\end{cor}

\begin{proof}[\sl D\'emonstration]
Elle résulte du Corollaire~\ref{cor:class-convexe-3} et du fait que l'unique feuilletage convexe à diviseur d'inflexion réduit apparaissant dans ce corollaire est le feuilletage $\F_3.$
\end{proof}
\bigskip

\section{Orbites sous l'action de $\mathrm{PGL}_3(\C)$}\label{sec:Orbites-Action-PGL-3}


Dans cette section, nous donnons une description des composantes irréductibles de $\mathbf{FP}(3)$: nous commençons par déterminer les dimensions des orbites $\mathcal{O}(\mathcal{H}_{\hspace{0.2mm}i}),\mathcal{O}(\F_j)$ sous l'action de $\mathrm{Aut}(\pp)=\mathrm{PGL}_3(\C)$; puis, nous classifions à isomorphisme près les feuilletages de~$\mathbf{F}(3)$ qui réalisent la dimension minimale des orbites en degré $3$; enfin, nous étudions les adhérences des orbites $\mathcal{O}(\mathcal{H}_{\hspace{0.2mm}i}),\mathcal{O}(\F_j)$ dans $\mathbf{F}(3)$ et nous établissons un résultat (Théorème~\ref{thm:12-Composantes irréductibles}) de description des composantes irréductibles de $\mathbf{FP}(3).$

\Subsection{Isotropies; dimensions des $\mathcal{O}(\mathcal{H}_{\hspace{0.2mm}i})$ et $\mathcal{O}(\F_j)$}

\begin{defin}
Soit $\F$ un feuilletage sur $\pp.$ Le sous-groupe de $\mathrm{Aut}(\pp)$ (resp. $\mathrm{Aut}(\pd)$) qui préserve $\F$ (resp. $\Leg\F$) s'appelle le \textbf{\textit{groupe d'isotropie}} de $\F$ (resp. $\Leg\F$) et est noté $\mathrm{Iso}(\F)$ (resp. $\mathrm{Iso}(\Leg\F)$); $\mathrm{Iso}(\F)$ et $\mathrm{Iso}(\Leg\F)$ sont des groupes algébriques.
\end{defin}

\noindent La Proposition suivante est de nature élémentaire, sa démonstration est laissée au lecteur.
\begin{pro}\label{pro:isotropies}
{\sl Les groupes $\mathrm{Iso}(\mathcal{H}_{\hspace{0.2mm}i})$ et $\mathrm{Iso}(\F_j)$ sont donnés par
\smallskip
\begin{Small}
\begin{itemize}
\item [\texttt{1.}] $\mathrm{Iso}(\mathcal{H}_{1})=\Big\{
                    [\pm\,x:y:\alpha\,z],\hspace{1mm}
                    [\pm\,y:x:\alpha\,z]
                    \hspace{1mm}\big\vert\hspace{1mm}\alpha\in\mathbb{C}^*\Big\}$;
\medskip

\item [\texttt{2.}] $\mathrm{Iso}(\mathcal{H}_{\hspace{0.2mm}2})=\Big\{
                    [\pm\,x:y:\alpha\,z],\hspace{1mm}
                    [\pm\,y:x:\alpha\,z],\hspace{1mm}
                    [\pm\,\mathrm{i}\,x:y:\alpha\,z],\hspace{1mm}
                    [\pm\,\mathrm{i}y:x:\alpha\,z]
                    \hspace{1mm}\big\vert\hspace{1mm}\alpha\in\mathbb{C}^*\Big\}$;
\medskip

\item [\texttt{3.}] $\mathrm{Iso}(\mathcal{H}_{\hspace{0.2mm}3})=\Big\{[\,x:y:\alpha\,z],\hspace{1mm}[\,y:x:\alpha\,z]
                    \hspace{1mm}\big\vert\hspace{1mm}\alpha\in\mathbb{C}^*\Big\}$;
\medskip

\item [\texttt{4.}] $\mathrm{Iso}(\mathcal{H}_{\hspace{0.2mm}4})=\Big\{[\,x:y:\alpha\,z],\hspace{1mm}[\,y:x:\alpha\,z]
                    \hspace{1mm}\big\vert\hspace{1mm}\alpha\in\mathbb{C}^*\Big\}$;
\medskip

\item [\texttt{5.}] $\mathrm{Iso}(\mathcal{H}_{\hspace{0.2mm}5})=\Big\{[\,x:y:\alpha\,z]
                    \hspace{1mm}\big\vert\hspace{1mm}\alpha\in\mathbb{C}^*\Big\}$;
\medskip

\item [\texttt{6.}] $\mathrm{Iso}(\mathcal{H}_{\hspace{0.2mm}6})=\Big\{[\,x:y:\alpha\,z]
                    \hspace{1mm}\big\vert\hspace{1mm}\alpha\in\mathbb{C}^*\Big\}$;
\medskip

\item [\texttt{7.}] $\mathrm{Iso}(\mathcal{H}_{\hspace{0.2mm}7})=\Big\{
                    [\pm\,x:y:\alpha\,z]
                    \hspace{1mm}\big\vert\hspace{1mm}\alpha\in\mathbb{C}^*\Big\}$;
\newpage
\hfill
\medskip

\item [\texttt{8.}] $\mathrm{Iso}(\mathcal{H}_{\hspace{0.2mm}8})=\Big\{
                    [\,x:y:\alpha\,z],\hspace{1mm}
                    [4y-x:y:\alpha\,z]
                    \hspace{1mm}\big\vert\hspace{1mm}\alpha\in\mathbb{C}^*\Big\}$;
\medskip

\item [\texttt{9.}] $\mathrm{Iso}(\mathcal{H}_{\hspace{0.2mm}9})=\Big\{
                    [\,x:y:\alpha\,z],\hspace{1mm}
                    [\,x-y:x:\alpha\,z],\hspace{1mm}
                    [y:y-x:\alpha\,z]
                    \hspace{1mm}\big\vert\hspace{1mm}\alpha\in\mathbb{C}^*\Big\}$;
\medskip

\item [\texttt{10.}] $\mathrm{Iso}(\mathcal{H}_{\hspace{0.2mm}10})=\Big\{
                    [\,x:y:\alpha\,z],\hspace{1mm}
                    [-y:x:\alpha\,z]
                    \hspace{1mm}\big\vert\hspace{1mm}\alpha\in\mathbb{C}^*\Big\}$;
\medskip

\item [\texttt{11.}] $\mathrm{Iso}(\mathcal{H}_{\hspace{0.2mm}11})=\Big\{
                     [\,x:y:\alpha\,z],\hspace{1mm}
                     [y:x:\alpha\,z],\hspace{1mm}
                     [\xi^5\,x:x+\xi\,y:\alpha\,z],\hspace{1mm}
                     [\xi^{-5}\,x:x+\xi^{-1}\,y:\alpha\,z],\hspace{1mm}
                     [\xi^5\,y:y+\xi\,x:\alpha\,z],$
\item[]$
                     \hspace{2.13cm}
                     [\xi^{-5}\,y:y+\xi^{-1}\,x:\alpha\,z],\hspace{1mm}
                     [\xi^{5}\,x-y:x+\xi^{-1}\,y:\alpha\,z],\hspace{1mm}
                     [\xi^{-5}\,x-y:x+\xi\,y:\alpha\,z],$
\item[]$
                     \hspace{2.13cm}
                     [\xi^{5}\,x+\xi^{4}\,y:x:\alpha\,z],\hspace{1mm}
                     [\xi^{-5}\,x+\xi^{-4}\,y:x:\alpha\,z],\hspace{1mm}
                     [\xi^{5}\,y+\xi^{4}\,x:y:\alpha\,z],$
\item[]$
                     \hspace{2.13cm}
                     [\xi^{-5}\,y+\xi^{-4}\,x:y:\alpha\,z]
                     \hspace{1mm}\big\vert\hspace{1mm}\alpha\in\mathbb{C}^*\Big\}$ où $\mathrm{\xi}=\mathrm{e}^{\mathrm{i}\pi/6}$;
\medskip

\item [\texttt{12.}] $\mathrm{Iso}(\mathcal{F}_{1})=\Big\{
                     [\alpha^{2} x:\alpha^{3} y:z+\beta\,x]
                     \hspace{1mm}\big\vert\hspace{1mm} \alpha\in\mathbb{C}^*,\hspace{1mm}\beta\in\mathbb{C}\Big\}$;
\medskip

\item [\texttt{13.}] $\mathrm{Iso}(\mathcal{F}_{2})=\Big\{
                     [\alpha^{4} x:\alpha^{3} y:z+\beta\,x]
                     \hspace{1mm}\big\vert\hspace{1mm} \alpha\in\mathbb{C}^*,\hspace{1mm}\beta\in\mathbb{C}\Big\}$;
\medskip

\item [\texttt{14.}] $\mathrm{Iso}(\mathcal{F}_{3})=\Big\{
                     [\pm\,x:\pm\,y:z],\hspace{1mm}
                     [\pm\,y:\pm\,x:z],\hspace{1mm}
                     [\pm\,x:\pm\,z:y],\hspace{1mm}
                     [\pm\,z:\pm\,x:y],\hspace{1mm}
                     [\pm\,y:\pm\,z:x],\hspace{1mm}
                     [\pm\,z:\pm\,y:x]\Big\}$;
\medskip

\item [\texttt{15.}] $\mathrm{Iso}(\mathcal{F}_{4})=\Big\{
                     [\,x:y:z+\alpha x],\hspace{1mm}
                     [\,\mathrm{j}\,x:y:z+\alpha\,x],\hspace{1mm}
                     [\,\mathrm{j}^2x:y:z+\alpha\,x]
                     \hspace{1mm}\big\vert\hspace{1mm}\alpha\in\mathbb{C}\Big\}$ où $\mathrm{j}=\mathrm{e}^{2\mathrm{i}\pi/3}$;
\medskip

\item [\texttt{16.}] $\mathrm{Iso}(\mathcal{F}_{5})=\Big\{
                     [\alpha^{2} x:\alpha^{3} y:z]
                     \hspace{1mm}\big\vert\hspace{1mm}\alpha\in\mathbb{C}^*\Big\}$.
\end{itemize}
\end{Small}
\vspace{2mm}

\noindent En particulier, les dimensions des $\mathcal{O}(\mathcal{H}_{\hspace{0.2mm}i})$ et $\mathcal{O}(\F_j)$ sont les suivantes
\begin{align*}
&\dim\mathcal{O}(\F_{1})=6,&&&\dim\mathcal{O}(\F_{2})=6,&&& \dim\mathcal{O}(\mathcal{H}_{\hspace{0.2mm}i})=7, i=1,\ldots,11,\\
&
\dim\mathcal{O}(\F_{4})=7,&&& \dim\mathcal{O}(\F_{5})=7,&&& \dim\mathcal{O}(\F_{3})=8.
\end{align*}
}
\end{pro}
\vspace{1mm}

\begin{rem}
Soit $\F$ un feuilletage sur $\pp.$ Si $[a:b:c]$ sont les coordonnées homogènes de $\pd$ associées à la droite $\{ax+by+cz=0\}\subset\pp,$ alors
\begin{align*}
&\mathrm{Iso}(\Leg\F)=\Big\{[a:b:c]\cdot A^{-1}
\hspace{1mm}\big\vert\hspace{1mm}
A\in\mathrm{PGL}_3(\mathbb{C}),\hspace{1mm}[x:y:z]\cdot\transp{A}\in\mathrm{Iso}(\F)
\Big\}.
\end{align*}
Plus précisément, l'isomorphisme $\tau\hspace{1mm}\colon\mathrm{Aut}(\pp)\to\mathrm{Aut}(\pd)$ qui, pour $A$ dans $\mathrm{PGL}_3(\mathbb{C}),$ envoie $[x:y:z]\cdot\transp{A}$ sur $[a:b:c]\cdot A^{-1}$ induit par restriction un isomorphisme de $\mathrm{Iso}(\F)$ sur $\mathrm{Iso}(\Leg\F).$
\end{rem}
\vspace{2mm}

\Subsection{Description des feuilletages $\F$ de degré trois tels que $\dim\mathcal{O}(\F)=6$}

La Proposition~2.3~de~\cite{CDGBM10} affirme que si $\F$ est un feuilletage de degré $d\geq2$ sur $\pp,$ alors la dimension de~$\mathcal{O}(\F)$ est minorée par $6,$ ou, ce qui revient au même, la dimension de~$\mathrm{Iso}(\F)$ est majorée par $2$. Notons que ces bornes sont atteintes pour les feuilletages $\F_{1}^{(d)}$ et $\F_{2}^{(d)}$ définis respectivement en carte affine $z=1$ par les $1$-formes
\begin{align*}
&\omegaoverline_{1}^{(d)}\hspace{1mm}=y^{d}\mathrm{d}x+x^{d}(x\mathrm{d}y-y\mathrm{d}x)
&& \text{et} &&
\omegaoverline_{2}^{(d)}\hspace{1mm}=x^{d}\mathrm{d}x+y^{d}(x\mathrm{d}y-y\mathrm{d}x).
\end{align*}
\newpage
\hfill
\vspace{-6.2mm}

\noindent En effet, il est facile de vérifier que
\begin{align*}
&\left\{
\left(\frac{\alpha^{d-1}x}{1+\beta x},\frac{\alpha^{d}y}{1+\beta x}\right)
\hspace{1mm}\Big\vert\hspace{1mm}
\alpha\in\mathbb{C}^*,\hspace{1mm}\beta\in\mathbb{C}
\right\}
\subset\mathrm{Iso}(\F_{1}^{(d)})\\
&\hspace{-5mm}\text{et}
\\
&\left\{
\left(\frac{\alpha^{d+1}x}{1+\beta x},\frac{\alpha^{d}y}{1+\beta x}\right)
\hspace{1mm}\Big\vert\hspace{1mm}
\alpha\in\mathbb{C}^*,\hspace{1mm}\beta\in\mathbb{C}
\right\}
\subset\mathrm{Iso}(\F_{2}^{(d)}),
\end{align*}
de sorte que $\dim\mathrm{Iso}(\F_{i}^{(d)})\geq2,i=1,2,$ et donc $\dim\mathrm{Iso}(\F_{i}^{(d)})=2.$
\begin{rem}
Par construction on a $\F_{1}^{(3)}=\F_1$ et $\F_{2}^{(3)}=\F_2.$
\end{rem}

\noindent D.~\textsc{Cerveau}, J.~\textsc{D\'eserti}, D.~\textsc{Garba Belko} et R.~\textsc{Meziani} ont montré qu'à automorphisme de $\pp$ près les feuilletages quadratiques $\F_{1}^{(2)}$ et $\F_{2}^{(2)}$ sont les seuls feuilletages qui réalisent la dimension minimale des orbites en degré~$2$ (\cite[Proposition 2.7]{CDGBM10}). En degré $3$ on a le résultat similaire suivant.

\begin{cor}\label{cor:dim-min}
{\sl \`A automorphisme de $\pp$ près, les feuilletages $\mathcal{F}_{1}$ et $\mathcal{F}_{2}$ sont les seuls feuilletages qui réalisent la dimension minimale des orbites en degré $3.$}
\end{cor}

\begin{proof}[\sl D\'emonstration]
Soit $\F$ un feuilletage de degré trois sur $\pp$ tel que $\dim\mathcal{O}(\F)=6.$ Comme $\mathrm{Iso}(\Leg\F)$ est isomorphe à $\mathrm{Iso}(\F),$ nous avons $\dim\mathrm{Iso}(\Leg\F)=\dim\mathrm{Iso}(\F)=8-6=2.$ Soit $m\in\pd\smallsetminus\Delta(\Leg\F)$; désignons par $\W_m$ le germe du $3$-tissu $\Leg\F$ en $m.$ D'après \'{E}. \textsc{Cartan} \cite{Car08} l'égalité $\dim\mathrm{Iso}(\Leg\F)=2$ implique que $\W_m$ est parallélisable et donc plat. Comme la courbure de $\Leg\F$ est holomorphe sur $\pd\smallsetminus\Delta(\Leg\F),$ nous en déduisons que $\Leg\F$ est plat. Donc $\F$ est linéairement conjugué à l'un des seize feuilletages donnés par le Théorème~\ref{thm:classification}. La Proposition~\ref{pro:isotropies} et l'hypothèse $\dim\mathcal{O}(\F)=6$ permettent de conclure.
\end{proof}

\Subsection{Adhérence d'orbites; composantes irréductibles de $\mathbf{FP}(3)$}\label{subsec:Adhérence-Composantes-FP(3)}

Nous commençons par étudier les adhérences des orbites $\mathcal{O}(\mathcal{H}_{\hspace{0.2mm}i})$ et $\mathcal{O}(\F_j)$ dans $\mathbf{F}(3),$ puis nous donnons une description des composantes irréductibles de $\mathbf{FP}(3).$

\noindent La définition suivante nous sera utile.

\begin{defin}[\cite{CDGBM10}]
Soient $\F$ et $\F'$ deux feuilletages de~$\mathbf{F}(3).$ On dit que $\F$ \textbf{\textit{dégénère}} sur $\F'$ si l'adhérence $\overline{\mathcal{O}(\F)}$ (dans~$\mathbf{F}(3)$) de $\mathcal{O}(\F)$ contient~$\mathcal{O}(\F')$ et $\mathcal{O}(\mathcal{F})\not=\mathcal{O}(\mathcal{F}').$
\end{defin}

\begin{rems}\label{rems:orbt-cnvx-plt}
Soient $\F$ et $\F'$ deux feuilletages tels que $\F$ dégénère sur $\F'$. Alors:
\begin{itemize}
\item [(i)]   $\dim\mathcal{O}(\F')<\dim\mathcal{O}(\F)$;

\item [(ii)]  Si $\Leg\F$ plat, $\Leg\F'$ l'est aussi;

\item [(iii)] $\deg\IinvF\leq\deg\mathrm{I}_{\mathcal{F}'}^{\mathrm{inv}}$ ce qui équivaut à~$\deg\ItrF\geq\deg\mathrm{I}_{\mathcal{F}'}^{\hspace{0.2mm}\mathrm{tr}}$. En particulier si $\F$ est convexe, $\F'$ l'est aussi.
\end{itemize}
\end{rems}

\noindent Comme nous l'avons signalé au paragraphe~\S\ref{subsec:cas-non-dégénéré}, D.~\textsc{Mar\'{\i}n} et J. \textsc{Pereira} dans~\cite{MP13} ont montré que l'adhérence de l'orbite $\mathcal{O}(\F_3)$ de $\F_3$ est une composante irréductible de $\mathbf{FP}(3).$ Le~point~\textbf{\textit{2.}} de la proposition suivante en donne une description plus précise.

\begin{pro}\label{pro:adh-F1-F2-F3}
{\sl
\textbf{\textit{1.}} Les orbites $\mathcal{O}(\F_1)$ et $\mathcal{O}(\F_2)$ sont fermées.

\noindent\textbf{\textit{2.}} $\overline{\mathcal{O}(\F_3)}=\mathcal{O}(\F_1)\cup\mathcal{O}(\mathcal{H}_{1})\cup\mathcal{O}(\mathcal{H}_{\hspace{0.2mm}3}) \cup\mathcal{O}(\F_3).$
}
\end{pro}

\begin{proof}[\sl D\'emonstration]
La première assertion résulte du Corollaire~\ref{cor:dim-min} et de la Remarque~\ref{rems:orbt-cnvx-plt}~(i).

\noindent D'après le Corollaire~\ref{cor:class-convexe-3} et la Remarque~\ref{rems:orbt-cnvx-plt}~(iii), $\F_3$ ne peut dégénérer que sur $\F_1,\mathcal{H}_{1}$ et $\mathcal{H}_{\hspace{0.2mm}3}.$ Montrons qu'il en est ainsi. Considérons la famille d'homothéties $\varphi=\varphi_{\varepsilon}=\left(\frac{x}{\varepsilon},\frac{y}{\varepsilon}\right).$ Nous~avons
\begin{align*}
-\varepsilon^4\varphi^*\omegaoverline_{3}=(y^3-\varepsilon^2y)\mathrm{d}x+(\varepsilon^2x-x^3)\mathrm{d}y
\end{align*}
qui tend vers $\omega_{1}$ lorsque $\varepsilon$ tend vers $0.$ Le feuilletage $\F_{3}$ dégénère donc sur $\mathcal{H}_{1}.$

\noindent Dans la carte affine $x=1$, $\F_1,$ resp. $\F_3,$ est décrit par
\begin{align*}
& \thetaoverline_1=\mathrm{d}y-y^3\mathrm{d}z, && \text{resp. } \thetaoverline_3=(y^3-y)\mathrm{d}z-(z^3-z)\mathrm{d}y\hspace{1mm};
\end{align*}
considérons la famille d'automorphismes $\sigma=\left(\frac{y}{\varepsilon},2+6\varepsilon^{2}z\right).$ Un calcul direct conduit à
\begin{align*}
-\frac{\varepsilon}{6}\sigma^*\thetaoverline_3=(1+11\varepsilon^{2}z+36\varepsilon^{4}z^{2}+36\varepsilon^{6}z^{3})\mathrm{d}y+(\varepsilon^{2}y-y^{3})\mathrm{d}z
\end{align*}
ce qui tend vers $\thetaoverline_1$ lorsque $\varepsilon$ tend vers $0.$ Ainsi $\F_3$ dégénère sur $\F_{1}.$

\noindent En coordonnées homogènes, $\mathcal{H}_{\hspace{0.2mm}3}$, resp. $\F_3,$ est défini par
\begin{align*}
&\Omega_3=z\,y^2(3x+y)\mathrm{d}x-z\,x^2(x+3y)\mathrm{d}y+xy(x^2-y^2)\mathrm{d}z,&&\\
\text{resp}.\hspace{1.5mm}
&\Omegaoverline_3=x^3(y\mathrm{d}z-z\mathrm{d}y)+y^3(z\mathrm{d}x-x\mathrm{d}z)+z^3(x\mathrm{d}y-y\mathrm{d}x)\hspace{1mm};
\end{align*}
en posant $\psi=\left[x-y:2\varepsilon\,z-x-y:x+y\right]$ nous obtenons
\begin{align*}
\frac{1}{8\varepsilon}\psi^*\Omegaoverline_3=
z\,y(y-\varepsilon\,z)(3x+y-2\varepsilon\,z)\mathrm{d}x-z\,x(x-\varepsilon\,z)(x+3y-2\varepsilon\,z)\mathrm{d}y+xy(x^2-y^2)\mathrm{d}z
\end{align*}
qui tend vers $\Omega_3$ lorsque $\varepsilon$ tend vers $0.$ Par conséquent $\F_3$ dégénère sur $\mathcal{H}_{\hspace{0.2mm}3}.$
\end{proof}

\begin{rem}
En combinant l'assertion~\textbf{\textit{2.}} de la Proposition~\ref{pro:adh-F1-F2-F3} au Corollaire~\ref{cor:class-convexe-3}, nous constatons que l'ensemble des feuilletages convexes de degré trois de $\pp$ est exactement l'adhérence $\overline{\mathcal{O}(\F_3)}$ de $\mathcal{O}(\F_3)$  et est donc un fermé irréductible de~$\mathbf{F}(3).$
\end{rem}

\noindent Le résultat suivant est une conséquence immédiate du Corollaire~\ref{cor:dim-min} et de la Remarque~\ref{rems:orbt-cnvx-plt}~(i).
\begin{cor}\label{cor:dim-o(F)=<7}
{\sl Soit $\F$ un élément de $\mathbf{F}(3)$ tel que $\dim\mathcal{O}(\F)\leq7.$ Alors
$$\overline{\mathcal{O}(\F)}\subset\mathcal{O}(\F)\cup\mathcal{O}(\F_1)\cup\mathcal{O}(\F_2).$$
}
\end{cor}

\noindent Une condition nécessaire pour qu'un feuilletage de degré trois de $\pp$ dégénère sur le feuilletage $\F_1$ est donnée par la:

\begin{pro}\label{pro:cond-nécess-dgnr-F1}
{\sl Soit $\F$ un élément de $\mathbf{F}(3)$. Si $\F$ dégénère sur $\F_1$, alors $\F$ possède un point singulier $m$ non-dégénéré vérifiant $\mathrm{BB}(\F,m)=4.$
}
\end{pro}

\begin{proof}[\sl D\'emonstration]
Supposons que $\F$ dégénère sur $\F_1$; il existe une famille analytique $(\F_\varepsilon)$ définie par la famille de $1$-formes $(\omega_\varepsilon)$ telle que pour $\varepsilon$ non nul $\F_\varepsilon$ soit dans $\mathcal{O}(\F)$ et pour $\varepsilon$ nul nous ayons $\F_{\varepsilon=0}=\F_1.$ Le point singulier non-dégénéré de $\F_1$, noté~$m_0,$ est \og stable\fg; il existe une famille analytique $(m_\varepsilon)$ de points singuliers non-dégénérés de $\F_\varepsilon$ telle que $m_{\varepsilon=0}=m_0.$ Les~$\F_\varepsilon$ étant conjugués à $\F$ pour $\varepsilon$ non nul, $\F$ admet un point singulier $m$ non-dégénéré tel que
$$
\forall\hspace{1mm}\varepsilon\in\mathbb{C}^*,\hspace{1mm}\mathrm{BB}(\F_\varepsilon,m_\varepsilon)=\mathrm{BB}(\F,m).
$$
Comme $\mu(\F_\varepsilon,m_\varepsilon)=1$ pour tout $\varepsilon$ dans $\mathbb{C}$, la fonction $\varepsilon\mapsto\mathrm{BB}(\F_\varepsilon,m_\varepsilon)$ est continue et est donc constante sur $\mathbb{C}.$ Par suite
\begin{align*}
&&\mathrm{BB}(\F,m)=\mathrm{BB}(\F_{\varepsilon=0},m_{\varepsilon=0})=\mathrm{BB}(\F_1,m_0)=4.
\end{align*}
\end{proof}

\begin{cor}\label{cor:H2-H8-H11-F5}
{\sl
Les feuilletages $\mathcal{H}_{\hspace{0.2mm}2},\mathcal{H}_{\hspace{0.2mm}8},\mathcal{H}_{\hspace{0.2mm}11}$ et $\F_{5}$ ne dégénèrent pas sur $\F_1.$
}
\end{cor}

\noindent Une condition suffisante pour qu'un feuilletage de degré trois de $\pp$ dégénère sur le feuilletage $\F_1$ est donnée par la:

\begin{pro}\label{pro:cond-suff-dgnr-F1}
{\sl Soit $\F$ un élément de $\mathbf{F}(3)$ tel que $\F_1\not\in\mathcal{O}(\F).$ Si $\F$ possède un point singulier $m$ non-dégénéré vérifiant $$\mathrm{BB}(\F,m)=4\qquad \text{et}\qquad \kappa(\F,m)=3,$$
alors $\F$ dégénère sur $\F_1.$
}
\end{pro}

\begin{proof}[\sl D\'emonstration]
Supposons que $\F$ ait une telle singularité $m$. L'égalité $\kappa(\F,m)=3$ assure l'existence d'une droite $\ell_m$ passant par $m$, non invariante par $\F$ et telle que $\Tang(\F,\ell_m,m)=3.$ Choisissons un système de coordonnées affines $(x,y)$ tel que $m=(0,0)$ et $\ell_m=(x=0).$ Le feuilletage $\F$ est défini dans ces coordonnées par une $1$-forme $\omega$ du type
\begin{align*}
\hspace{1cm}& (*x+\beta y+*x^2+*xy+*y^2+*x^3+*x^2y+*xy^2+*y^3)\mathrm{d}x\\
\hspace{1cm}&\hspace{3mm}+(\alpha\,x+ry+*x^2+*xy+s\hspace{0.1mm}y^2+*x^3+*x^2y+*xy^2+\gamma\,y^3)\mathrm{d}y\\
\hspace{1cm}&\hspace{3mm}+(*x^3+*x^2y+*xy^2+*y^3)(x\mathrm{d}y-y\mathrm{d}x),&&
\text{avec}\hspace{1mm} *,r,s,\alpha,\beta,\gamma\in\mathbb{C}.
\end{align*}
Sur l'axe $x=0$, la $2$-forme $\omega\wedge\mathrm{d}x$ s'écrit $(ry+s\hspace{0.1mm}y^2+\gamma\,y^3)\mathrm{d}y\wedge\mathrm{d}x$. L'égalité~$\Tang(\F,\ell_m,m)=3$ se traduit alors par: $r=s=0$ et $\gamma\neq0.$ Les égalités $r=0$, $\mu(\F,m)=1$ et $\mathrm{BB}(\F,m)=4$ impliquent que $\beta=-\alpha\neq0.$ Ainsi $\omega$ est du type
\begin{align*}
\hspace{1cm}& (*x-\alpha y+*x^2+*xy+*y^2+*x^3+*x^2y+*xy^2+*y^3)\mathrm{d}x\\
\hspace{1cm}&\hspace{3mm}+(\alpha\,x+*x^2+*xy+*x^3+*x^2y+*xy^2+\gamma\,y^3)\mathrm{d}y\\
\hspace{1cm}&\hspace{3mm}+(*x^3+*x^2y+*xy^2+*y^3)(x\mathrm{d}y-y\mathrm{d}x),&&
\text{où les }*\in\mathbb{C},\hspace{1mm}\alpha,\gamma\in\mathbb{C}^*.
\end{align*}
Posons $\varphi=\left(\varepsilon^3\hspace{0.1mm}x,\varepsilon\hspace{0.1mm}y\right).$ Soient $i$ et $j$ deux entiers naturels non tous deux nuls. Notons que
\begin{itemize}
\item[\texttt{1. }] $\varphi^*(x^iy^j\mathrm{d}x)=\varepsilon^{3i+j+3}x^iy^j\mathrm{d}x$ est divisible par $\varepsilon^4$ et
                    $\frac{1}{\varepsilon^4}\varphi^*(x^iy^j\mathrm{d}x)$ tend vers $0$ lorsque $\varepsilon$ tend vers $0$ sauf pour $(i,j)=(0,1)$;

\item[\texttt{2. }] $\varphi^*(x^iy^j\mathrm{d}y)=\varepsilon^{3i+j+1}x^iy^j\mathrm{d}y$ est divisible par $\varepsilon^4$ sauf pour $(i,j)=(0,1)$ et
                    $(i,j)=(0,2).$ Si $(i,j)$ n'appartient pas à $\{(0,1),(0,2),(0,3),(1,0)\},$ la forme $\frac{1}{\varepsilon^4}\varphi^*(x^iy^j\mathrm{d}y)$ tend vers $0$ lorsque $\varepsilon$ tend vers $0.$
\end{itemize}
\noindent Nous constatons que
\begin{align*}
\lim_{\varepsilon\to 0}\frac{1}{\varepsilon^4}\varphi^*\omega=\alpha(x\mathrm{d}y-y\mathrm{d}x)+\gamma\,y^3\mathrm{d}y.
\end{align*}
Le feuilletage défini par $\alpha(x\mathrm{d}y-y\mathrm{d}x)+\gamma\,y^3\mathrm{d}y$ est conjugué à $\F_1$ car, comme le montre un calcul immédiat, c'est un feuilletage convexe dont le lieu singulier est formé de deux points. Par suite $\F$ dégénère sur $\F_1.$
\end{proof}

\begin{cor}\label{cor:H1-H3-H5-H7-F4}
{\sl
Les feuilletages $\mathcal{H}_{\hspace{0.2mm}1},\mathcal{H}_{\hspace{0.2mm}3},\mathcal{H}_{\hspace{0.2mm}5},\mathcal{H}_{\hspace{0.2mm}7}$ et $\F_{4}$ dégénèrent sur $\F_1.$
}
\end{cor}

\noindent La réciproque de la Proposition~\ref{pro:cond-suff-dgnr-F1} est fausse comme le montre l'exemple suivant.
\begin{eg}\label{eg:sans-singularité-kappa=3}
Soit $\F$ le feuilletage de degré $3$ sur $\pp$ défini en carte affine $z=1$ par
$$\omega=x\mathrm{d}y-y\mathrm{d}x+(y^2+y^3)\mathrm{d}y.$$
Le lieu singulier de $\F$ est formé des deux points $m=[0:0:1]$ et $m'=[1:0:0]$; de plus
\begin{align*}
&\mu(\F,m)=1,&& \mathrm{BB}(\F,m)=4,&&\kappa(\F,m)=2,&& \mu(\F,m')>1.
\end{align*}
Le feuilletage $\F$ dégénère sur $\F_1$; en effet, en posant
$\varphi=$\begin{small}$\left(\dfrac{1}{\varepsilon^3}x,\dfrac{1}{\varepsilon}y\right)$\end{small}, nous constatons que
\begin{align*}
&\hspace{1.5cm}\lim_{\varepsilon\to 0}\varepsilon^4\varphi^*\omega=x\mathrm{d}y-y\mathrm{d}x+y^3\mathrm{d}y.
\end{align*}
\end{eg}
\vspace{1mm}

\noindent Une condition nécessaire pour qu'un feuilletage de degré trois de $\pp$ dégénère sur le feuilletage $\F_2$ est donnée par la:

\begin{pro}
{\sl Soit $\F$ un élément de $\mathbf{F}(3).$ Si $\F$ dégénère sur $\F_2$, alors $\deg\ItrF\geq2.$
}
\end{pro}

\begin{proof}[\sl D\'emonstration]
Si $\F$ dégénère sur $\F_2,$ alors $\deg\ItrF\geq\deg\mathrm{I}_{\F_2}^{\hspace{0.2mm}\mathrm{tr}}.$ Or un calcul immédiat montre que $\mathrm{I}_{\mathcal{F}_{2}}^{\hspace{0.2mm}\mathrm{tr}}=y^2$ de sorte que $\deg\mathrm{I}_{\mathcal{F}_{2}}^{\hspace{0.2mm}\mathrm{tr}}=2$, d'où l'énoncé.
\end{proof}

\begin{cor}\label{cor:H5-H9}
{\sl
Les feuilletages $\mathcal{H}_{\hspace{0.2mm}5}$ et $\mathcal{H}_{\hspace{0.2mm}9}$ ne dégénèrent pas sur $\F_2$.
}
\end{cor}

\noindent Une condition suffisante pour qu'un feuilletage de degré trois de $\pp$ dégénère sur le feuilletage $\F_2$ est donnée par la:

\begin{pro}\label{pro:cond-suff-dgnr-F2}
{\sl Soit $\F$ un élément de $\mathbf{F}(3)$ tel que $\F_2\not\in\mathcal{O}(\F).$ Si $\F$ possède un point d'inflexion double, alors $\F$ dégénère sur $\F_2.$}
\end{pro}

\begin{proof}[\sl D\'emonstration]
Supposons que $\F$ ait un tel point. Choisissons un système de coordonnées affines $(x,y)$ tel que $m=(0,0)$ soit d'inflexion double pour $\F$ et $x=0$ soit la tangente à la feuille de $\F$ passant par $m.$ Soit $\omega$ une $1$-forme définissant $\F$ dans ces coordonnées; comme $\mathrm{T}_{\hspace{-0.4mm}m}\F=(x=0)$, $\omega$ est du type
\begin{align*}
\hspace{0.5cm}&(\alpha+*x+*y+*x^2+*xy+*y^2+*x^3+*x^2y+*xy^2+*y^3)\mathrm{d}x\\
\hspace{0.5cm}&\hspace{3mm}+(*x+ry+*x^2+*xy+s\hspace{0.1mm}y^2+*x^3+*x^2y+*xy^2+\beta y^3)\mathrm{d}y\\
\hspace{0.5cm}&\hspace{3mm}+(*x^3+*x^2y+*xy^2+*y^3)(x\mathrm{d}y-y\mathrm{d}x),&&
\text{avec}\hspace{1mm} *,r,s,\beta,\in\mathbb{C},\hspace{1mm}\alpha\in\mathbb{C}^*.
\end{align*}
Sur la droite $x=0$, la $2$-forme $\omega\wedge\mathrm{d}x$ s'écrit $(ry+s\hspace{0.1mm}y^2+\beta\,y^3)\mathrm{d}y\wedge\mathrm{d}x$. L'hypothèse que $(0,0)$ est d'inflexion double se traduit donc par: $r=s=0$ et $\beta\neq0.$ Ainsi $\omega$ est du type
\begin{align*}
\hspace{0.5cm}&(\alpha+*x+*y+*x^2+*xy+*y^2+*x^3+*x^2y+*xy^2+*y^3)\mathrm{d}x\\
\hspace{0.5cm}&\hspace{3mm}+(*x+*x^2+*xy+*x^3+*x^2y+*xy^2+\beta y^3)\mathrm{d}y\\
\hspace{0.5cm}&\hspace{3mm}+(*x^3+*x^2y+*xy^2+*y^3)(x\mathrm{d}y-y\mathrm{d}x),&&
\text{où les }*\in\mathbb{C},\hspace{1mm}\alpha,\beta\in\mathbb{C}^*.
\end{align*}
Considérons la famille $\varphi_\varepsilon=\varphi=(\varepsilon^4x,\varepsilon y).$ Notons que
\begin{itemize}
\item[\texttt{1. }] $\varphi^*(x^iy^j\mathrm{d}x)=\varepsilon^{4i+j+4}x^iy^j\mathrm{d}x$ est divisible par $\varepsilon^4$ et
                    $\frac{1}{\varepsilon^4}\varphi^*(x^iy^j\mathrm{d}x)$ tend vers $0$ lorsque $\varepsilon$ tend vers $0$ sauf pour $i=j=0$;

\item[\texttt{2. }] $\varphi^*(x^iy^j\mathrm{d}y)=\varepsilon^{4i+j+1}x^iy^j\mathrm{d}y$ est divisible par $\varepsilon^4$ sauf si
                    $(i,j)\in\{(0,0),(0,1),(0,2)\}.$ Si~$(i,j)\not\in\{(0,0),(0,1),(0,2),(0,3)\},$ la forme $\frac{1}{\varepsilon^4}\varphi^*(x^iy^j\mathrm{d}y)$ tend vers $0$ lorsque $\varepsilon$ tend vers $0.$
\end{itemize}
Nous constatons que
\begin{align*}
\lim_{\varepsilon\to 0}\frac{1}{\varepsilon^4}\varphi^*\omega=\alpha\mathrm{d}x+\beta y^{3}\mathrm{d}y.
\end{align*}
Visiblement $\alpha\mathrm{d}x+\beta y^{3}\mathrm{d}y$ définit un feuilletage conjugué à $\mathcal{F}_2$; par suite $\F$ dégénère sur $\F_2.$
\end{proof}

\begin{cor}\label{cor:H2-H4-H6-H8-F4}
{\sl
Les feuilletages $\mathcal{H}_{\hspace{0.2mm}2},\mathcal{H}_{\hspace{0.2mm}4},\mathcal{H}_{\hspace{0.2mm}6},\mathcal{H}_{\hspace{0.2mm}8}$ et $\F_{4}$ dégénèrent sur $\F_2.$
}
\end{cor}

\begin{eg}[de \textsc{Jouanolou}]
Considérons le feuilletage $\F_J$ de degré $3$ sur $\pp$ défini, dans la carte affine~$z=1,$ par $$\omega_J=(x^3y-1)\mathrm{d}x+(y^3-x^4)\mathrm{d}y\hspace{1mm};$$
cet exemple est dû à \textsc{Jouanolou} (\cite{Jou79}). C'est historiquement le premier exemple explicite de feuilletage sans courbe algébrique invariante (\cite{Jou79}); c'est aussi un feuilletage qui n'admet pas d'ensemble minimal non trivial (\cite{CdF01}). Le point $m=(0,0)$ est d'inflexion double pour $\F_J$ car $\mathrm{T}_{\hspace{-0.4mm}m}\F_J=(x=0)\hspace{2mm} \text{et} \hspace{2mm}\omega_J\wedge\mathrm{d}x\Big|_{x=0}=y^3\mathrm{d}y\wedge\mathrm{d}x$; ainsi $\F_J$ dégénère sur~$\F_2.$
\end{eg}

\noindent La réciproque de la Proposition~\ref{pro:cond-suff-dgnr-F2} est fausse comme le montre l'exemple suivant.
\begin{eg}\label{eg:sans-inflex-double}
Soit $\F$ le feuilletage de degré $3$ sur $\pp$ défini en carte affine $z=1$ par
$$\omega=\mathrm{d}x+(y^2+y^3)\mathrm{d}y.$$
Un calcul élémentaire montre que $\F$ n'a aucun point d'inflexion double. Ce feuilletage dégénère sur $\F_2$; en effet, en posant
$\varphi=$\begin{small}$\left(\dfrac{1}{\varepsilon^4}x,\dfrac{1}{\varepsilon}y\right)$\end{small}, nous obtenons que
\begin{align*}
&\hspace{1.5cm}\lim_{\varepsilon\to 0}\varepsilon^4\varphi^*\omega=\mathrm{d}x+y^3\mathrm{d}y.
\end{align*}
\end{eg}
\vspace{2mm}

\noindent Du Théorème~\ref{thm:classification}, des Propositions~\ref{pro:isotropies}, \ref{pro:adh-F1-F2-F3} et des Corollaires~\ref{cor:class-convexe-3}, \ref{cor:dim-o(F)=<7}, \ref{cor:H2-H8-H11-F5}, \ref{cor:H1-H3-H5-H7-F4}, \ref{cor:H5-H9}, \ref{cor:H2-H4-H6-H8-F4}, nous tirons l'énoncé suivant.

\begin{thm}\label{thm:12-Composantes irréductibles}
{\sl Les adhérences étant prises dans $\mathbf{F}(3),$ nous avons
\begin{align*}
&
\overline{\mathcal{O}(\F_1)}=\mathcal{O}(\F_1),&&
\overline{\mathcal{O}(\F_2)}=\mathcal{O}(\F_2),\\
&
\overline{\mathcal{O}(\F_3)}=\mathcal{O}(\F_1)\cup\mathcal{O}(\mathcal{H}_{1})\cup\mathcal{O}(\mathcal{H}_{\hspace{0.2mm}3})\cup\mathcal{O}(\F_3),&&
\overline{\mathcal{O}(\F_4)}=\mathcal{O}(\F_1)\cup\mathcal{O}(\F_2)\cup\mathcal{O}(\F_4),\\
&
\overline{\mathcal{O}(\mathcal{H}_{1})}=\mathcal{O}(\F_1)\cup\mathcal{O}(\mathcal{H}_{1}),&&
\overline{\mathcal{O}(\mathcal{H}_{\hspace{0.2mm}2})}=\mathcal{O}(\F_2)\cup\mathcal{O}(\mathcal{H}_{\hspace{0.2mm}2}),\\
&
\overline{\mathcal{O}(\mathcal{H}_{\hspace{0.2mm}3})}=\mathcal{O}(\F_1)\cup\mathcal{O}(\mathcal{H}_{\hspace{0.2mm}3}),&&
\overline{\mathcal{O}(\mathcal{H}_{\hspace{0.2mm}8})}=\mathcal{O}(\F_2)\cup\mathcal{O}(\mathcal{H}_{\hspace{0.2mm}8}),\\
&
\overline{\mathcal{O}(\mathcal{H}_{\hspace{0.2mm}5})}=\mathcal{O}(\F_1)\cup\mathcal{O}(\mathcal{H}_{\hspace{0.2mm}5}),&&
\overline{\mathcal{O}(\mathcal{H}_{\hspace{0.2mm}4})}\supset\mathcal{O}(\F_2)\cup\mathcal{O}(\mathcal{H}_{\hspace{0.2mm}4}),\\
&
\overline{\mathcal{O}(\mathcal{H}_{\hspace{0.2mm}7})}\supset\mathcal{O}(\F_1)\cup\mathcal{O}(\mathcal{H}_{\hspace{0.2mm}7}),&&
\overline{\mathcal{O}(\mathcal{H}_{\hspace{0.2mm}6})}\supset\mathcal{O}(\F_2)\cup\mathcal{O}(\mathcal{H}_{\hspace{0.2mm}6}),\\
&
\overline{\mathcal{O}(\mathcal{H}_{\hspace{0.2mm}9})}\subset\mathcal{O}(\F_1)\cup\mathcal{O}(\mathcal{H}_{\hspace{0.2mm}9}),&&
\overline{\mathcal{O}(\mathcal{H}_{\hspace{0.2mm}11})}\subset\mathcal{O}(\F_2)\cup\mathcal{O}(\mathcal{H}_{11}),\\
&
\overline{\mathcal{O}(\mathcal{H}_{\hspace{0.2mm}10})}\subset\mathcal{O}(\F_1)\cup\mathcal{O}(\F_2)\cup\mathcal{O}(\mathcal{H}_{\hspace{0.2mm}10}),&&
\overline{\mathcal{O}(\F_5)}\subset\mathcal{O}(\F_2)\cup\mathcal{O}(\F_5)
\end{align*}
avec
\begin{align*}
&\dim\mathcal{O}(\F_{1})=6,&&&\dim\mathcal{O}(\F_{2})=6,&&& \dim\mathcal{O}(\mathcal{H}_{\hspace{0.2mm}i})=7, i=1,\ldots,11,\\
&
\dim\mathcal{O}(\F_{4})=7,&&& \dim\mathcal{O}(\F_{5})=7,&&& \dim\mathcal{O}(\F_{3})=8.
\end{align*}
En particulier:
\begin{itemize}
  \item  l'ensemble $\mathbf{FP}(3)$ possède exactement douze composantes irréductibles, à savoir $\overline{\mathcal{O}(\F_3)},\,$ $\overline{\mathcal{O}(\F_4)},\,$ $\overline{\mathcal{O}(\F_5)},\,$ $\overline{\mathcal{O}(\mathcal{H}_{\hspace{0.2mm}2})},\,$ $\overline{\mathcal{O}(\mathcal{H}_{\hspace{0.4mm}k})},\,k=4,5,\ldots,11$;
  \item  l'ensemble des feuilletages convexes de degré trois de $\pp$ est exactement l'adhérence $\overline{\mathcal{O}(\F_3)}$ de $\mathcal{O}(\F_3)$ (c'est donc un fermé irréductible de~$\mathbf{F}(3)).$
\end{itemize}
}
\end{thm}
\smallskip

\begin{prob}\label{prob:dgnr-F1-F2}
{\sl
Établir un critère permettant de décider si un feuilletage de degré trois de $\pp$ dégénère ou non sur les feuilletages $\F_1$ et $\F_2.$
}
\end{prob}

\noindent Une réponse à ce problème nous permet de décider si une orbite d'un feuilletage de $\mathbf{F}(3)$ de dimension $7$ est fermée ou non dans~$\mathbf{F}(3),$ en vertu du Corollaire~\ref{cor:dim-o(F)=<7}.

\renewcommand\appendixname{Appendice}
\appendix

\chapter{D\'emonstration de la Proposition 3.13}\label{Dém:pro:cas-nilpotent-noeud-ordre-2}

Nous présentons ici une démonstration de la Proposition \ref{pro:cas-nilpotent-noeud-ordre-2}. Cet énoncé affirme que si un feuilletage de degré trois sur le plan projectif complexe possède une singularité dégénérée de multiplicité algébrique inférieure ou égale à $2$, alors sa transformée de \textsc{Legendre} ne peut être plate.

\noindent Nous allons raisonner par l'absurde en supposant qu'il existe un feuilletage $\F$ de degré trois sur $\pp$ tel que le $3$-tissu $\Leg\F$ soit plat et dont le lieu singulier $\Sing\F$ contient un point $m$ vérifiant $\mu(\F,m)\geq2$\, et \,$\nu(\F,m)\leq2.$ Il s'agit de calculer explicitement la courbure de $\Leg\F$ par la formule~(\ref{equa:Formule-Henaut}) et d'utiliser la condition $K(\Leg\F)\equiv0$ pour arriver à une contradiction avec l'hypothèse $\deg\F=3.$

\noindent Pour ce faire, choisissons un système de coordonnées homogènes $[x:y:z]\in\pp$ tel que $m$ soit l'origine $O$ de la carte affine $z=1$; soit $\omega$ une $1$-forme décrivant $\F$ dans cette carte. Comme $\mu(\F,m)\geq2$ et $\nu(\F,m)\leq2$, à isomorphisme linéaire près, on a les trois possibilités suivantes
\begin{enumerate}
\item[\texttt{1. }] $J^{1}_{(0,0)}\,\omega=x\mathrm{d}y,$ la singularité est dite de type \textbf{\textit{selle-nœud}}, éventualité étudiée au \S\hspace{1mm}\ref{sec:cas-noeud};

\item[\texttt{2. }] $J^{1}_{(0,0)}\,\omega=y\mathrm{d}y,$ la singularité $O$ est dite de type \textbf{\textit{nilpotent}}, ce cas sera l'objet du \S\hspace{1mm}\ref{sec:cas-nilpotent};

\item[\texttt{3. }] $J^{1}_{(0,0)}\,\omega=0$\, et \,$J^{2}_{(0,0)}\,\omega\neq0$ (\emph{voir} \S\hspace{1mm}\ref{sec:cas-ordre-2}).
\end{enumerate}

\section{Étude du cas selle-nœud}\label{sec:cas-noeud}\vspace{2mm}

Supposons que la singularité $O$ soit de type selle-nœud; on peut alors écrire $\omega$ sous la forme $\omega=A(x,y)\mathrm{d}x+B(x,y)\mathrm{d}y+C(x,y)(x\mathrm{d}y-y\mathrm{d}x)$ où
\begin{small}
\begin{align*}
&A(x,y)=\alpha_0x^2+\alpha_1xy+\alpha_2y^2+a_0x^3+a_1x^2y+a_2xy^2+a_3y^3,\quad C(x,y)=c_0x^3+c_1x^2y+c_2xy^2+c_3y^3,\\
&B(x,y)=x+\beta_0x^2+\beta_1xy+\beta_2y^2+b_0x^3+b_1x^2y+b_2xy^2+b_3y^3.
\end{align*}
\end{small}
\hspace{-1mm}Quitte à conjuguer $\omega$ par le difféomorphisme \begin{small}$\left(\dfrac{2x}{2\alpha_1x-\beta_1y+2}, \dfrac{2y}{2\alpha_1x-\beta_1y+2}\right)$\end{small}, on peut supposer que $\alpha_1=\beta_1=0.$

\noindent Nous allons examiner la platitude de $\Leg\F$ dans les trois cas suivants
\begin{itemize}
\item $\alpha_2\neq0$ (\hspace{0.3mm}$\mu(\F,O)=2$);
\item $\alpha_2=0$\, \text{et} \,$a_3\neq0$ (\hspace{0.3mm}$\mu(\F,O)=3$);
\item $\alpha_2=0$\, \text{et} \,$a_3=0$ (\hspace{0.3mm}$\mu(\F,O)\geq4$).
\end{itemize}
Pour cela, plaçons-nous dans la carte affine $(p,q)$ de $\pd$ associée à la droite $\{px-qy=1\}$ de $\pp$, où le $3$-tissu $\Leg\F$ est décrit par l'équation différentielle
\begin{small}
\begin{align*}
&F(p,q,w):=\Big(p^3+\beta_0p^2+b_0p+\alpha_0pq+a_0q+c_0\Big)w^3\\
&\hspace{2cm}+\Big(-2p^2q+b_1p-\alpha_0q^2-\beta_0pq+a_1q+c_1\Big)w^2\\
&\hspace{2cm}+\Big(pq^2+\beta_2p^2+b_2p+\alpha_2pq+a_2q+c_2\Big)w\\
&\hspace{2cm}+b_3p-\alpha_2q^2-\beta_2pq+a_3q+c_3=0,
&&\text{avec}\hspace{1.5mm} w=\frac{\mathrm{d}q}{\mathrm{d}p}.
\end{align*}
\end{small}
\hspace{-1mm}Le calcul explicite de $K(\Leg\F)$ montre qu'elle s'écrit sous la forme $$K(\Leg\F)=\frac{\sum\limits_{i+j\leq23}\rho_{i}^{j}p^iq^j}{R(p,q)^2}\mathrm{d}p\wedge\mathrm{d}q,$$
où $R:=\text{Result}(F,\partial_{w}(F))$ et les $\rho_{i}^{j}$ sont des polynômes en les paramètres $\alpha_0,\alpha_2,\beta_0,\beta_2,a_k,b_k,c_k.$
\vspace{2mm}

\textbf{\textit{1.}} Si $\alpha_2\neq0$, alors, quitte à conjuguer $\omega$ par la transformation linéaire \begin{small}$\left(x,\frac{1}{\alpha_2}y\right)$\end{small}, on peut supposer que $\alpha_2=1.$ Dans ce cas l'hypothèse $K(\Leg\F)\equiv0$ implique que
\begin{align*}
&(\alpha_0,\beta_0,\beta_2)=(0,a_2,0),&& (b_0,b_1,b_2,b_3)=(\tfrac{1}{4}a_2^2-a_1,-\tfrac{1}{2}a_2,-\tfrac{3}{4},0),\\
&(a_0,a_3)=(0,1),&& (c_0,c_1,c_2,c_3)=(-\tfrac{1}{2}a_1a_2,-\tfrac{1}{4}a_2^2-\tfrac{1}{2}a_1,-\tfrac{1}{2}a_2,-\tfrac{1}{4})\hspace{1mm};
\end{align*}
en effet
\begin{SMALL}
\begin{align*}
\hspace{-5.8cm}\alpha_2=1
\Rightarrow
\left\{
\begin{array}{lll}
\rho_{21}^{2}=-2\beta_2^4=0\\
\rho_{4}^{17}=-\alpha_0^5a_0=0\\
\rho_{3}^{17}=-2\alpha_0^4(4\alpha_0^3+a_0^2)=0
\end{array}
\right.
\Leftrightarrow\hspace{2mm}
\left\{
\begin{array}{ll}
\beta_2=0\\
\alpha_0=0
\end{array}
\right.
\end{align*}
\end{SMALL}
\hspace{-1mm}
\begin{SMALL}
\begin{align*}
\hspace{-2.4cm}(\alpha_2,\beta_2,\alpha_0)=(1,0,0)
\Rightarrow
\left\{
\begin{array}{lll}
\rho_{19}^{4}=-32b_3^2=0\\
\rho_{10}^{13}=-4\beta_0a_0^2=0\\
\rho_{11}^{12}=-32a_0^2+2\beta_0a_0(\beta_0^2-4a_1-4b_0)=0
\end{array}
\right.
\Leftrightarrow\hspace{2mm}
\left\{
\begin{array}{ll}
b_3=0\\
a_0=0
\end{array}
\right.
\end{align*}
\end{SMALL}
\hspace{-1mm}
\begin{SMALL}
\begin{align*}
\hspace{-3.3cm}(\alpha_2,\beta_2,\alpha_0,b_3,a_0)=(1,0,0,0,0)
\Rightarrow
\rho_{17}^{6}=-2(4a_3+4b_2-1)^2=0
\Leftrightarrow
b_2=\frac{1}{4}-a_3
\end{align*}
\end{SMALL}
\hspace{-1mm}
\begin{SMALL}
\begin{align*}
\hspace{-2.5cm}\left\{
\begin{array}{ll}
(\alpha_2,\beta_2,\alpha_0,b_3,a_0)=(1,0,0,0,0)\\
b_2=\frac{1}{4}-a_3
\end{array}
\right.
\Rightarrow
\rho_{17}^{4}=-\frac{5}{2}(2a_3+4c_3-1)^2=0
\Leftrightarrow
c_3=\frac{1}{4}-\frac{1}{2}a_3
\end{align*}
\end{SMALL}
\hspace{-1mm}
\begin{SMALL}
\begin{align*}
\hspace{-4.3cm}\left\{
\begin{array}{ll}
(\alpha_2,\beta_2,\alpha_0,b_3,a_0)=(1,0,0,0,0)\\
(b_2,c_3)=(\frac{1}{4}-a_3,\frac{1}{4}-\frac{1}{2}a_3)
\end{array}
\right.
\Rightarrow
\rho_{17}^{2}=-3(a_3-1)^4=0
\Leftrightarrow
a_3=1
\end{align*}
\end{SMALL}
\hspace{-1mm}
\begin{SMALL}
\begin{align*}
\hspace{-2.4cm}\left\{
\begin{array}{ll}
(\alpha_2,\beta_2,\alpha_0,b_3,a_0)=(1,0,0,0,0)\\
(a_3,b_2,c_3)=(1,-\frac{3}{4},-\frac{1}{4})
\end{array}
\right.
\Rightarrow
\rho_{15}^{8}=-8(2a_2+2b_1-\beta_0)^2=0
\Leftrightarrow
b_1=\frac{1}{2}\beta_0-a_2
\end{align*}
\end{SMALL}
\hspace{-1mm}
\begin{SMALL}
\begin{align*}
\hspace{-3.4cm}\left\{
\begin{array}{lll}
(\alpha_2,\beta_2,\alpha_0,b_3,a_0)=(1,0,0,0,0)\\
(a_3,b_2,c_3)=(1,-\frac{3}{4},-\frac{1}{4})\\
b_1=\frac{1}{2}\beta_0-a_2
\end{array}
\right.
\Rightarrow
\rho_{15}^{6}=-10(a_2+2c_2)^2=0
\Leftrightarrow
c_2=-\frac{1}{2}a_2
\end{align*}
\end{SMALL}
\hspace{-1mm}
\begin{SMALL}
\begin{align*}
\hspace{-2.4cm}\left\{
\begin{array}{lll}
(\alpha_2,\beta_2,\alpha_0,b_3,a_0)=(1,0,0,0,0)\\
(a_3,b_2,c_3)=(1,-\frac{3}{4},-\frac{1}{4})\\
(b_1,c_2)=(\frac{1}{2}\beta_0-a_2,-\frac{1}{2}a_2)
\end{array}
\right.
\Rightarrow
\rho_{13}^{10}=-2(4a_1+4b_0-\beta_0^2)^2=0
\Leftrightarrow
b_0=\frac{1}{4}\beta_0^2-a_1
\end{align*}
\end{SMALL}
\hspace{-1mm}
\begin{SMALL}
\begin{align*}
\left\{
\begin{array}{llll}
(\alpha_2,\beta_2,\alpha_0,b_3,a_0)=(1,0,0,0,0)\\
(a_3,b_2,c_3)=(1,-\frac{3}{4},-\frac{1}{4})\\
(b_1,c_2)=(\frac{1}{2}\beta_0-a_2,-\frac{1}{2}a_2)\\
b_0=\frac{1}{4}\beta_0^2-a_1
\end{array}
\right.
\Rightarrow
\rho_{13}^{8}=-\frac{5}{2}(2a_1+2\beta_0a_2-\beta_0^2+4c_1)^2=0
\Leftrightarrow
c_1=\frac{1}{4}\beta_0^2-\frac{1}{2}a_1-\frac{1}{2}\beta_0a_2
\end{align*}
\end{SMALL}
\hspace{-1mm}
\begin{SMALL}
\begin{align*}
\hspace{-3cm}\left\{
\begin{array}{llll}
(\alpha_2,\beta_2,\alpha_0,b_3,a_0)=(1,0,0,0,0)\\
(a_3,b_2,c_3)=(1,-\frac{3}{4},-\frac{1}{4})\\
(b_1,c_2)=(\frac{1}{2}\beta_0-a_2,-\frac{1}{2}a_2)\\
(b_0,c_1)=(\frac{1}{4}\beta_0^2-a_1,\frac{1}{4}\beta_0^2-\frac{1}{2}a_1-\frac{1}{2}\beta_0a_2)
\end{array}
\right.
\Rightarrow
\rho_{13}^{6}=-3(a_2-\beta_0)^4=0
\Leftrightarrow
\beta_0=a_2
\end{align*}
\end{SMALL}
\hspace{-1mm}et
\begin{SMALL}
\begin{align*}
\hspace{-2.5cm}\left\{
\begin{array}{llll}
(\alpha_2,\beta_2,\alpha_0,b_3,a_0)=(1,0,0,0,0)\\
(a_3,b_2,c_3)=(1,-\frac{3}{4},-\frac{1}{4})\\
(\beta_0,b_1,c_2)=(a_2,-\frac{1}{2}a_2,-\frac{1}{2}a_2)\\
(b_0,c_1)=(\frac{1}{4}a_2^2-a_1,-\frac{1}{4}a_2^2-\frac{1}{2}a_1)
\end{array}
\right.
\Rightarrow
\rho_{11}^{10}=-10(a_1a_2+2c_0)^2=0
\Leftrightarrow
c_0=-\frac{1}{2}a_1a_2.
\end{align*}
\end{SMALL}
\hspace{-1mm}Par suite la $1$-forme $\omega$ définissant $\F$ s'écrit
\begin{align*}
\frac{1}{4}\Big(a_2x+y+2\Big)\Big(2y^2\mathrm{d}x+(a_2x^2-xy+2x)\mathrm{d}y-(2a_1x^2+a_2xy+y^2)(x\mathrm{d}y-y\mathrm{d}x)\Big),
\end{align*}
ce qui conduit à $\deg\F=2$: contradiction.
\vspace{2mm}

\textbf{\textit{2.}} Cas où $\alpha_2=0$ et $a_3\neq0.$ Quitte à faire agir la transformation linéaire  \begin{small}$\left(x,\sqrt{\frac{1}{a_3}}y\right)$\end{small} sur $\omega,$ on peut supposer que $a_3=1.$ Dans ce cas le système $\rho_{20}^{3}=\rho_{19}^{3}=\rho_{17}^{5}=\rho_{17}^{3}=\rho_{17}^{1}=0$ n'a aucune solution; en effet
\begin{Small}
\begin{align*}
&(\alpha_2,a_3)=(0,1)
\Rightarrow
\rho_{20}^{3}=\beta_2^3=0
\Leftrightarrow
\beta_2=0,\\
&(\alpha_2,a_3,\beta_2)=(0,1,0)
\Rightarrow
\rho_{19}^{3}=96b_3^2=0
\Leftrightarrow
b_3=0,\\
&(\alpha_2,a_3,\beta_2,b_3)=(0,1,0,0)
\Rightarrow
\rho_{17}^{5}=96(b_2+1)^2=0
\Leftrightarrow
b_2=-1,\\
&(\alpha_2,a_3,\beta_2,b_3,b_2)=(0,1,0,0,-1)
\Rightarrow
\rho_{17}^{3}=112c_3^2=0
\Leftrightarrow
c_3=0,\\
&(\alpha_2,a_3,\beta_2,b_3,b_2,c_3)=(0,1,0,0,-1,0)
\Rightarrow
\rho_{17}^{1}=8\neq0.
\end{align*}
\end{Small}
\hspace{-1mm}Il s'en suit que ce deuxième cas est aussi impossible.
\vspace{2mm}

\textbf{\textit{3.}} Cas où $(\alpha_2,a_3)=(0,0).$ Le calcul explicite de $K(\Leg\F)$ montre que $\rho_{4}^{17}=-\alpha_0^5a_0,$ de sorte que $\alpha_0a_0=0.$
\`{A} conjugaison près par des transformations linéaires du type $(\lambda\hspace{0.1mm}x,\mu\hspace{0.1mm}y)$, on est dans l'une des situations suivantes
\begin{align*}
& (\alpha_2,a_3,\beta_2,\alpha_0,a_0)=(0,0,1,1,0), &&(\alpha_2,a_3,\beta_2,\alpha_0)=(0,0,1,0),\\
& (\alpha_2,a_3,\beta_2,b_3)=(0,0,0,0),&& (\alpha_2,a_3,\beta_2,b_3)=(0,0,0,1).
\end{align*}

\textbf{\textit{3.1.}} Lorsque $(\alpha_2,a_3,\beta_2,\alpha_0,a_0)=(0,0,1,1,0)$, l'hypothèse $K(\Leg\F)$ implique que
\begin{align*}
& (a_1,a_2)=(-2b_3,2),&& (b_0,b_1,b_2)=(b_3,-1-\beta_0b_3,\beta_0-b_3^2),\\
& (c_0,c_1,c_2,c_3)=(0,-b_3^2,2b_3,-1)\hspace{1mm};
\end{align*}
en effet
\begin{SMALL}
\begin{align*}
\hspace{-0.9cm}(\alpha_2,a_3,\beta_2,\alpha_0,a_0)=(0,0,1,1,0)
\Rightarrow
\left\{
\begin{array}{llll}
\rho_{19}^{4}=2a_2+4c_3=0\\
\rho_{20}^{2}=-12a_2-28c_3-4=0\\
\rho_{18}^{5}=-(4a_2+8c_3+6)b_3+3a_1+6c_2=0\\
\rho_{19}^{3}=(14a_2+56c_3+76)b_3-22a_1-46c_2=0
\end{array}
\right.
\Leftrightarrow\hspace{2mm}
\left\{
\begin{array}{llll}
a_2=2\\
c_3=-1\\
a_1=-2b_3\\
c_2=2b_3
\end{array}
\right.
\end{align*}
\end{SMALL}
\hspace{-1mm}et
\begin{SMALL}
\begin{align*}
\hspace{0.2cm}\left\{
\begin{array}{ll}
(\alpha_2,a_3,\beta_2,\alpha_0,a_0)=(0,0,1,1,0)\\
(a_2,c_3,a_1,c_2)=(2,-1,-2b_3,2b_3)
\end{array}
\right.
\Rightarrow
\left\{
\begin{array}{lllll}
\rho_{17}^{6}=8(\beta_0-b_2+c_1)=0\\
\rho_{18}^{4}=-6(12\beta_0-b_3^2-12b_2+11c_1)=0\\
\rho_{16}^{7}=-(34\beta_0-24b_2+24c_1)b_3-10\\
\hspace{0.8cm}-10b_1+10c_0=0\\
\rho_{17}^{5}=-16b_3^3+4(81\beta_0-56b_2+52c_1)b_3\\
\hspace{0.8cm}+100+100b_1-88c_0=0\\
\rho_{15}^{8}=32\beta_0b_3^2+(32b_1-32c_0+44)b_3+8\beta_0^2\\
\hspace{0.8cm}+(8c_1-24b_2)\beta_0+16b_2^2-16b_2c_1\\
\hspace{0.8cm}-12b_0=0
\end{array}
\right.
\Leftrightarrow\hspace{2mm}
\left\{
\begin{array}{lllll}
b_2=\beta_0-b_3^2\\
c_1=-b_3^2\\
b_1=-1-\beta_0b_3\\
c_0=0\\
b_0=b_3.
\end{array}
\right.
\end{align*}
\end{SMALL}
\hspace{-1mm}Donc la $1$-forme $\omega$ décrivant $\F$ s'écrit
\begin{align*}
\Big(y^2-b_3\hspace{0.1mm}xy+x\Big)\Big(x\mathrm{d}x+(\beta_0\hspace{0.1mm}x+b_3\hspace{0.1mm}y+1)\mathrm{d}y
+(b_3\hspace{0.1mm}x-y)(x\mathrm{d}y-y\mathrm{d}x)\Big),
\end{align*}
mais ceci contredit l'égalité $\deg\F=3.$
\vspace{2mm}

\textbf{\textit{3.2.}} Lorsque $(\alpha_2,a_3,\beta_2,\alpha_0)=(0,0,1,0)$, la platitude de $\Leg\F$ entraîne que $a_i=c_i=0$ pour $i=0,1,2,3$; en effet
\begin{SMALL}
\begin{align*}
\hspace{-1.9cm}(\alpha_2,a_3,\beta_2,\alpha_0)=(0,0,1,0)
\Rightarrow
\left\{
\begin{array}{llll}
\rho_{19}^{4}=2a_2+4c_3=0\\
\rho_{20}^{2}=-12a_2-28c_3=0\\
\rho_{18}^{5}=-(4a_2+8c_3)b_3+3a_1+6c_2=0\\
\rho_{19}^{3}=(14a_2+56c_3)b_3-22a_1-46c_2=0
\end{array}
\right.
\Leftrightarrow\hspace{2mm}
\left\{
\begin{array}{llll}
a_2=0\\
c_3=0\\
a_1=0\\
c_2=0
\end{array}
\right.
\end{align*}
\end{SMALL}
\hspace{-1mm}et
\begin{SMALL}
\begin{align*}
(\alpha_2,a_3,\beta_2,\alpha_0,a_2,c_3,a_1,c_2)=(0,0,1,0,0,0,0,0)
\Rightarrow
\left\{
\begin{array}{lll}
\rho_{17}^{6}=4a_0+8c_1=0\\
\rho_{18}^{4}=-30a_0-66c_1=0\\
\rho_{16}^{7}=-(12a_0+24c_1)b_3+10c_0=0
\end{array}
\right.
\Leftrightarrow\hspace{2mm}
\left\{
\begin{array}{lll}
a_0=0\\
c_1=0\\
c_0=0.
\end{array}
\right.
\end{align*}
\end{SMALL}
\hspace{-1mm}Ainsi, la $1$-forme $\omega$ s'écrit $(x+\beta_0x^2+y^2+b_0x^3+b_1x^2y+b_2xy^2+b_3y^3)\mathrm{d}y$, ce qui implique que $\deg\F=0$: contradiction.
\vspace{2mm}

\textbf{\textit{3.3.}} Quand $(\alpha_2,a_3,\beta_2,b_3)=(0,0,0,0)$, le calcul explicite de $K(\Leg\F)$ conduit à
\begin{align*}
& \rho_{17}^{1}=2(12b_2c_3)^2 && \text{et} && \rho_{17}^{2}=8c_3(16b_2^3+27c_3^2)\hspace{1mm};
\end{align*}
le système $\rho_{17}^{1}=\rho_{17}^{2}=0$ équivaut à $c_3=0.$ Les égalités $\alpha_2=a_3=\beta_2=b_3=c_3=0$ impliquent que $\deg\F<3$: contradiction.
\vspace{3mm}

\textbf{\textit{3.4.}} Lorsque $(\alpha_2,a_3,\beta_2,b_3)=(0,0,0,1)$, la platitude de $\Leg\F$ implique que $\alpha_0=a_i=c_i=0$ pour $i=0,1,2,3$; en effet
\begin{SMALL}
\begin{align*}
\hspace{-1.9cm}(\alpha_2,a_3,\beta_2,b_3)=(0,0,0,1)
\Rightarrow
\left\{
\begin{array}{lllll}
\rho_{18}^{5}=-8c_3=0\\
\rho_{18}^{4}=28a_2=0\\
\rho_{17}^{6}=16(\alpha_0-c_2)=0\\
\rho_{17}^{5}=8(6a_2b_2-6\beta_0c_3+a_1)=0\\
\rho_{18}^{3}=2(184b_2c_3+141\alpha_0-75c_2)=0
\end{array}
\right.
\Leftrightarrow\hspace{2mm}
\left\{
\begin{array}{lllll}
c_3=0\\
a_2=0\\
a_1=0\\
c_2=0\\
\alpha_0=0
\end{array}
\right.
\end{align*}
\end{SMALL}
\hspace{-1mm}et
\begin{SMALL}
\begin{align*}
(\alpha_2,a_3,\beta_2,b_3,c_3,a_2,a_1,c_2,\alpha_0)=(0,0,0,1,0,0,0,0,0)
\Rightarrow
\left\{
\begin{array}{lll}
\rho_{16}^{7}=-24c_1=0\\
\rho_{16}^{6}=36a_0=0\\
\rho_{15}^{8}=-16b_2c_1-32c_0=0
\end{array}
\right.
\Leftrightarrow\hspace{2mm}
\left\{
\begin{array}{lll}
c_1=0\\
a_0=0\\
c_0=0.
\end{array}
\right.
\end{align*}
\end{SMALL}
\hspace{-1mm}Par conséquent, la $1$-forme $\omega$ s'écrit $(x+\beta_0x^2+b_0x^3+b_1x^2y+b_2xy^2+y^3)\mathrm{d}y$; de sorte que $\deg\F=0$: contradiction.

\newpage
\section{Étude du cas nilpotent}\label{sec:cas-nilpotent}

Supposons que la singularité $O$ soit de type nilpotent; on peut donc écrire $\omega$ sous la forme $\omega=A(x,y)\mathrm{d}x+B(x,y)\mathrm{d}y+C(x,y)(x\mathrm{d}y-y\mathrm{d}x)$ où
\begin{small}
\begin{align*}
&A(x,y)=\alpha_0x^2+\alpha_1xy+\alpha_2y^2+a_0x^3+a_1x^2y+a_2xy^2+a_3y^3,\quad C(x,y)=c_0x^3+c_1x^2y+c_2xy^2+c_3y^3,\\
&B(x,y)=y+\beta_0x^2+\beta_1xy+\beta_2y^2+b_0x^3+b_1x^2y+b_2xy^2+b_3y^3.
\end{align*}
\end{small}
\hspace{-1mm}Soit $(p,q)$ la carte affine de $\pd$ associée à la droite $\{py-qx=1\}\subset\pp$; le $3$-tissu $\Leg\F$ est donné par l'équation différentielle
\begin{small}
\begin{align*}
&F(p,q,w):=\Big(p^2q+\alpha_2p^2+\beta_2pq+a_3p+b_3q-c_3\Big)w^3\\
&\hspace{2cm}+\Big(-2pq^2+\alpha_1p^2+(\beta_1-\alpha_2)pq-\beta_2q^2+a_2p+b_2q-c_2\Big)w^2\\
&\hspace{2cm}+\Big(q^3+\alpha_0p^2+(\beta_0-\alpha_1)pq-\beta_1q^2+a_1p+b_1q-c_1\Big)w\\
&\hspace{2cm}-\alpha_0pq-\beta_0q^2+a_0p+b_0q-c_0=0,
&&\text{où}\hspace{1.5mm} w=\frac{\mathrm{d}q}{\mathrm{d}p}.
\end{align*}
\end{small}
\hspace{-1mm}Le calcul explicite de $K(\Leg\F)$ montre qu'elle s'écrit sous la forme $$K(\Leg\F)=\frac{\sum\limits_{i+j\leq23}\rho_{i}^{j}p^iq^j}{R(p,q)^2}\mathrm{d}p\wedge\mathrm{d}q,$$
où $R:=\text{Result}(F,\partial_{w}(F))$ et les $\rho_{i}^{j}$ sont des polynômes en les coefficients de $A,B,C,$ avec $\rho_{17}^{6}=2\alpha_0^5.$ La platitude de $\Leg\F$ entraîne donc que $\alpha_0=0.$
\noindent Puis en faisant agir le difféomorphisme
\begin{small}$\left(\dfrac{2x}{2\alpha_2x-\beta_2y+2}, \dfrac{2y}{2\alpha_2x-\beta_2y+2}\right)$\end{small} sur $\omega$, on peut supposer que $\alpha_2=\beta_2=0.$ Ainsi
\begin{small}
\begin{align*}
\omega=\alpha_1xy\mathrm{d}x+(y+\beta_0x^2+\beta_1xy)\mathrm{d}y+\mathrm{d}x\sum_{i=0}^{3}a_i\hspace{0.1mm}x^{3-i}y^{i}
+\mathrm{d}y\sum_{i=0}^{3}b_i\hspace{0.1mm}x^{3-i}y^{i}+(x\mathrm{d}y-y\mathrm{d}x)\sum_{i=0}^{3}c_i\hspace{0.1mm}x^{3-i}y^{i}.
\end{align*}
\end{small}
\hspace{-1mm}Dans ce cas le calcul explicite de $K(\Leg\F)$ donne
\begin{Small}
\begin{align*}
\begin{array}{lllll}
\smallskip
\rho_{12}^{11}=\Big(4a_0-(\alpha_1+\beta_0)^2\Big)\Big((5\alpha_1-\beta_0)a_0-\alpha_1(\alpha_1+\beta_0)^2\Big),\\
\smallskip
\rho_{13}^{9}=-2\Big(4a_0-(\alpha_1+\beta_0)^2\Big)\Big(15a_0^2-2\beta_0(\alpha_1+4\beta_0)a_0-\alpha_1^2(\alpha_1+\beta_0)^2\Big),\\
\smallskip
\rho_{14}^{7}=(99\alpha_1+297\beta_0)a_0^3-3(\alpha_1+\beta_0)(68\beta_0^2+19\alpha_1\beta_0+19\alpha_1^2)a_0^2\\
\hspace{8mm}+(5\alpha_1^3-21\alpha_1^2\beta_0+24\alpha_1\beta_0^2+32\beta_0^3)(\alpha_1+\beta_0)^2a_0+\alpha_1^2(\alpha_1+4\beta_0)(\alpha_1+\beta_0)^4,\\
\smallskip
\rho_{15}^{5}=-8\alpha_1a_0\Big((54\beta_0+18\alpha_1)a_0^2+(-12\beta_0^3-49\alpha_1^2\beta_0-7\alpha_1^3-48\alpha_1\beta_0^2)a_0
+\alpha_1(\alpha_1^2+6\alpha_1\beta_0+4\beta_0^2)(\alpha_1+\beta_0)^2\Big),\\
\smallskip
\rho_{16}^{3}=4\alpha_1a_0\Big(81a_0^3-3\alpha_1(10\alpha_1+27\beta_0)a_0^2+\alpha_1^2(3\alpha_1^2+16\alpha_1\beta_0+7\beta_0^2)a_0
+\alpha_1^3\beta_0(\alpha_1+\beta_0)^2\Big)\hspace{1mm};
\end{array}
\end{align*}
\end{Small}
\hspace{-1mm}le système $\rho_{12}^{11}=\rho_{13}^{9}=\rho_{14}^{7}=\rho_{15}^{5}=\rho_{16}^{3}=0$ est vérifié si et seulement si on est dans l'une des situations suivantes
\begin{align*}
& a_0=0,\hspace{1mm}\alpha_1=0\hspace{1mm};                                                     && a_0=0,\hspace{1mm}\beta_0=-\alpha_1,\hspace{1mm}\alpha_1\neq0\hspace{1mm};\\
& a_0=\alpha_1^2,\hspace{1mm}\beta_0=\alpha_1,\hspace{1mm}\alpha_1\neq0\hspace{1mm};&&
a_0=\frac{\alpha_1^2}{2},\hspace{1mm}\beta_0=\frac{\alpha_1}{2},\hspace{1mm}\alpha_1\neq0.
\end{align*}
\`{A} conjugaison près par des transformations linéaires du type $(\lambda\hspace{0.1mm}x,\mu\hspace{0.1mm}y)$, on se ramène aux cinq cas suivants
\begin{align*}
& (a_0,\alpha_1,\beta_0)=(0,0,0), && (a_0,\alpha_1,\beta_0)=(0,0,1), && (a_0,\alpha_1,\beta_0)=(0,1,-1),\\
& (a_0,\alpha_1,\beta_0)=(1,1,1), && (a_0,\alpha_1,\beta_0)=(\tfrac{1}{2},1,\tfrac{1}{2}).
\end{align*}

\textbf{\textit{1.}} Lorsque $(a_0,\alpha_1,\beta_0)=(0,0,0),$ l'hypothèse $K(\Leg\F)\equiv0$ implique que $a_1=0$; en effet
\begin{small}
\begin{align*}
(a_0,\alpha_1,\beta_0)=(0,0,0)
\Rightarrow
\left\{
\begin{array}{lll}
\rho_{10}^{12}=-4a_1(a_1+b_0)(10a_1+9b_0)=0\\
\rho_{11}^{9}=-16a_1^2(5a_1+9b_0)(a_1+b_0)=0\\
\rho_{12}^{6}=-2a_1^2(20a_1^3+234a_1^2b_0+459a_1b_0^2+243b_0^3)=0
\end{array}
\right.
\Leftrightarrow\hspace{2mm}
a_1=0.
\end{align*}
\end{small}
\hspace{-1mm}On envisage deux cas suivant que $b_0$ est nul ou non.
\vspace{2mm}

\textbf{\textit{1.1.}} Lorsque $b_0=0,$ l'hypothèse sur $\Leg\F$ d'être plat entraîne que $c_0=0$; en effet
\begin{align*}
(a_0,\alpha_1,\beta_0,a_1,b_0)=(0,0,0,0,0)\Rightarrow \rho_{10}^{9}=-78c_0^3=0\Leftrightarrow c_0=0.
\end{align*}
Les égalités $\beta_0=a_0=b_0=c_0=0$ impliquent que $\deg\F<3$: contradiction.
\vspace{2mm}

\textbf{\textit{1.2.}} Si $b_0\neq0$, alors, quitte à faire agir la transformation linéaire diagonale $(x,b_0y)$, on peut effectuer la normalisation $b_0=1.$ Dans ce cas l'hypothèse $K(\Leg\F)\equiv0$ implique que $a_2=a_3=c_i=0,i=0,1,2,3$; en effet
\begin{small}
\begin{align*}
\left\{
\begin{array}{ll}
(a_0,\alpha_1,\beta_0)=(0,0,0)\\
(a_1,b_0)=(0,1)
\end{array}
\right.
\Rightarrow
\left\{
\begin{array}{ll}
\rho_{10}^{11}=-78c_0=0\\
\rho_{10}^{10}=(141\beta_1+156c_0)c_0+153a_2=0
\end{array}
\right.
\Leftrightarrow\hspace{2mm}
\left\{
\begin{array}{ll}
c_0=0\\
a_2=0
\end{array}
\right.
\end{align*}
\end{small}
\hspace{-1mm}
\begin{small}
\begin{align*}
\hspace{-1.4cm}\left\{
\begin{array}{lll}
(a_0,\alpha_1,\beta_0)=(0,0,0)\\
(a_1,b_0)=(0,1)\\
(c_0,a_2)=(0,0)
\end{array}
\right.
\Rightarrow
\left\{
\begin{array}{ll}
\rho_{9}^{12}=84c_1=0\\
\rho_{8}^{14}=-28\beta_1c_1-12a_3=0
\end{array}
\right.
\Leftrightarrow\hspace{2mm}
\left\{
\begin{array}{ll}
c_1=0\\
a_3=0
\end{array}
\right.
\end{align*}
\end{small}
\hspace{-1mm}et
\begin{small}
\begin{align*}
\left\{
\begin{array}{llll}
(a_0,\alpha_1,\beta_0)=(0,0,0)\\
(a_1,b_0)=(0,1)\\
(c_0,a_2)=(0,0)\\
(c_1,a_3)=(0,0)
\end{array}
\right.
\Rightarrow
\left\{
\begin{array}{ll}
\rho_{8}^{13}=198c_2=0\\
\rho_{7}^{14}=(74\beta_1^2+152b_1)c_2+336c_3=0
\end{array}
\right.
\Leftrightarrow\hspace{2mm}
\left\{
\begin{array}{ll}
c_2=0\\
c_3=0.
\end{array}
\right.
\end{align*}
\end{small}
\hspace{-1mm}Ainsi la forme $\omega$ s'écrit $(y+\beta_1xy+x^3+b_1x^2y+b_2xy^2+b_3y^3)\mathrm{d}y$, ce qui conduit à $\deg\F=0$: contradiction.
\vspace{2mm}

\textbf{\textit{2.}} Quand $(a_0,\alpha_1,\beta_0)=(0,0,1),$ l'hypothèse $K(\Leg\F)\equiv0$ implique que $a_i=c_i=0$ pour $i=0,1,2,3$; en effet
\begin{small}
\begin{align*}
\left\{
\begin{array}{lll}
a_0=0\\
\alpha_1=0\\
\beta_0=1
\end{array}
\right.
\Rightarrow
\left\{
\begin{array}{lll}
\rho_{12}^{10}=-4a_1=0\\
\rho_{11}^{11}=(56a_1+56b_0-22\beta_1)a_1+6a_2+12c_0=0\\
\rho_{12}^{9}=-16(7a_1+6b_0-3\beta_1)a_1-12a_2-12c_0=0
\end{array}
\right.
\Leftrightarrow\hspace{2mm}
\left\{
\begin{array}{lll}
a_1=0\\
a_2=0\\
c_0=0
\end{array}
\right.
\end{align*}
\end{small}
\hspace{-1mm}
\begin{small}
\begin{align*}
\hspace{-1.5cm}\left\{
\begin{array}{ll}
(a_0,\alpha_1,\beta_0)=(0,0,1)\\
(a_1,a_2,c_0)=(0,0,0)
\end{array}
\right.
\Rightarrow
\left\{
\begin{array}{ll}
\rho_{10}^{12}=14a_3+30c_1=0\\
\rho_{11}^{10}=-22a_3-54c_1=0
\end{array}
\right.
\Leftrightarrow\hspace{2mm}
\left\{
\begin{array}{ll}
a_3=0\\
c_1=0
\end{array}
\right.
\end{align*}
\end{small}
\hspace{-1mm}et
\begin{small}
\begin{align*}
\left\{
\begin{array}{lll}
(a_0,\alpha_1,\beta_0)=(0,0,1)\\
(a_1,a_2,c_0)=(0,0,0)\\
(a_3,c_1)=(0,0)
\end{array}
\right.
\Rightarrow
\left\{
\begin{array}{ll}
\rho_{9}^{13}=50c_2=0\\
\rho_{8}^{14}=(89\beta_1-136b_0)c_2+72c_3=0
\end{array}
\right.
\Leftrightarrow\hspace{2mm}
\left\{
\begin{array}{ll}
c_2=0\\
c_3=0.
\end{array}
\right.
\end{align*}
\end{small}
\hspace{-1mm}Par suite $\omega$ s'écrit $(y+x^2+\beta_1xy+b_0x^3+b_1x^2y+b_2xy^2+b_3y^3)\mathrm{d}y$, ce qui contredit l'égalité $\deg\F=3$.
\vspace{2mm}

\textbf{\textit{3.}} Lorsque $(a_0,\alpha_1,\beta_0)=(1,1,1)$, la platitude de $\Leg\F$ implique que
\begin{align*}
\left(\beta_1,b_0,b_1,b_2,b_3,c_0,c_1,c_2,c_3\right)=\left(2a_1,a_1,a_1^2-a_2,-a_3,0,-a_2,-a_1a_2-a_3,-a_1a_3,0\right);
\end{align*}
en effet
\begin{SMALL}
\begin{align*}
\hspace{-2cm}\left\{
\begin{array}{lll}
a_0=1\\
\alpha_1=1\\
\beta_0=1
\end{array}
\right.
\Rightarrow
\left\{
\begin{array}{llll}
\rho_{12}^{10}=-12a_1+12\beta_1-12b_0=0\\
\rho_{13}^{8}=124a_1-122\beta_1+120b_0=0\\
\rho_{11}^{11}=-16a_2-16b_1+456\beta_1b_0-184a_1^2\\
\hspace{0.8cm}-248b_0^2-204\beta_1^2+392\beta_1a_1-432a_1b_0=0\\
\rho_{12}^{9}=180a_2+174b_1+6c_0-1586\beta_1b_0+772a_1^2\\
\hspace{0.8cm}+876b_0^2+667\beta_1^2-1477\beta_1a_1+1636a_1b_0=0
\end{array}
\right.
\Leftrightarrow\hspace{2mm}
\left\{
\begin{array}{llll}
\beta_1=2a_1\\
b_0=a_1\\
b_1=a_1^2-a_2\\
c_0=-a_2
\end{array}
\right.
\end{align*}
\end{SMALL}
\hspace{-1mm}et
\begin{SMALL}
\begin{align*}
\left\{
\begin{array}{lllllll}
a_0=1\\
\alpha_1=1\\
\beta_0=1\\
\beta_1=2a_1\\
b_0=a_1\\
b_1=a_1^2-a_2\\
c_0=-a_2
\end{array}
\right.
\Rightarrow
\left\{
\begin{array}{lllll}
\rho_{10}^{12}=-20a_3-20b_2=0\\
\rho_{11}^{10}=8a_1a_2+244a_3+236b_2+8c_1=0\\
\rho_{9}^{13}=-32a_1a_3-32a_1b_2-24b_3=0\\
\rho_{10}^{11}=306b_3+10c_2+2(6a_1a_2+267a_3+256b_2+6c_1)a_1=0\\
\rho_{9}^{12}=4(a_1a_2+c_1+95a_3+90b_2)a_1^2-(8a_2^2-696b_3-16c_2)a_1\\
\hspace{8mm}-(344a_3+336b_2+8c_1)a_2+12c_3=0
\end{array}
\right.
\Leftrightarrow
\left\{
\begin{array}{lllll}
b_2=-a_3\\
b_3=0\\
c_1=-a_1a_2-a_3\\
c_2=-a_1a_3\\
c_3=0.
\end{array}
\right.
\end{align*}
\end{SMALL}
\hspace{-1mm}Il en résulte que la $1$-forme $\omega$ décrivant $\F$ s'écrit
$$
\Big(y+a_1xy+x^2\Big)\Big(x\mathrm{d}x+(1+a_1x)\mathrm{d}y-(a_2x+a_3y)(x\mathrm{d}y-y\mathrm{d}x)\Big)\hspace{1mm};
$$
d'où $\deg\F=1$: contradiction.
\vspace{2mm}

\textbf{\textit{4.}} Quand $(a_0,\alpha_1,\beta_0)=(\frac{1}{2},1,\frac{1}{2}),$ la platitude de $\Leg\F$ conduit à:
\begin{small}
\begin{align*}
&\left(a_3,\beta_1,b_0,b_1,b_2,b_3\right)=
\left(\frac{1}{4}a_1(a_1^2+2a_2),\frac{5}{2}a_1,\frac{3}{4}a_1,\frac{25}{16}a_1^2-\frac{7}{8}a_2,-\frac{3}{16}a_1(a_1^2+2a_2),-\frac{1}{64}(a_1^2+2a_2)^2\right),\\
&\left(c_0,c_1,c_2,c_3\right)=\left(\frac{1}{16}a_1^2-\frac{3}{8}a_2,-a_1a_2,-\frac{3}{64}(5a_1^2+2a_2)(a_1^2+2a_2),-\frac{1}{32}(a_1^2+2a_2)^2a_1\right);
\end{align*}
\end{small}
\hspace{-1mm}en effet
\begin{SMALL}
\begin{align*}
\hspace{-3cm}(a_0,\alpha_1,\beta_0)=(\frac{1}{2},1,\frac{1}{2})
\Rightarrow
\left\{
\begin{array}{ll}
\rho_{12}^{10}=-\frac{5}{2}a_1-\frac{15}{4}b_0+\frac{17}{8}\beta_1=0\\
\rho_{13}^{8}=17a_1+24b_0-14\beta_1=0
\end{array}
\right.
\Leftrightarrow\hspace{2mm}
\left\{
\begin{array}{ll}
\beta_1=\frac{5}{2}a_1\\
b_0=\frac{3}{4}a_1
\end{array}
\right.
\end{align*}
\end{SMALL}
\hspace{-1mm}
\begin{SMALL}
\begin{align*}
\hspace{-0.5cm}\left\{
\begin{array}{ll}
(a_0,\alpha_1,\beta_0)=(\frac{1}{2},1,\frac{1}{2})\\
(\beta_1,b_0)=(\frac{5}{2}a_1,\frac{3}{4}a_1)
\end{array}
\right.
\Rightarrow
\left\{
\begin{array}{ll}
\rho_{11}^{11}=\frac{407}{16}a_1^2-\frac{99}{8}a_2-\frac{33}{2}b_1+\frac{11}{2}c_0=0\\
\rho_{12}^{9}=-\frac{213}{2}a_1^2+\frac{105}{2}a_2+69b_1-21c_0=0
\end{array}
\right.
\Leftrightarrow\hspace{2mm}
\left\{
\begin{array}{ll}
b_1=\frac{25}{16}a_1^2-\frac{7}{8}a_2\\
c_0=\frac{1}{16}a_1^2-\frac{3}{8}a_2
\end{array}
\right.
\end{align*}
\end{SMALL}
\hspace{-1mm}
\begin{SMALL}
\begin{align*}
\hspace{1cm}\left\{
\begin{array}{lll}
(a_0,\alpha_1,\beta_0)=(\frac{1}{2},1,\frac{1}{2})\\
(\beta_1,b_0)=(\frac{5}{2}a_1,\frac{3}{4}a_1)\\
(b_1,c_0)=(\frac{25}{16}a_1^2-\frac{7}{8}a_2,\frac{1}{16}a_1^2-\frac{3}{8}a_2)
\end{array}
\right.
\Rightarrow
\left\{
\begin{array}{lll}
\rho_{10}^{12}=\frac{1}{32}a_1(3a_1^2+302a_2)\\
\hspace{0.8cm}-\frac{189}{8}a_3-31b_2+\frac{37}{4}c_1=0\\
\rho_{11}^{10}=-\frac{3}{64}a_1(21a_1^2+706a_2)\\
\hspace{0.8cm}+\frac{1593}{16}a_3+\frac{255}{2}b_2-\frac{249}{8}c_1=0\\
\rho_{12}^{8}=\frac{21}{16}a_1(3a_1^2+38a_2)\\
\hspace{0.8cm}-213a_3-263b_2+42c_1=0
\end{array}
\right.
\Leftrightarrow\hspace{2mm}
\left\{
\begin{array}{lll}
a_3=\frac{1}{4}a_1(a_1^2+2a_2)\\
b_2=-\frac{3}{16}a_1(a_1^2+2a_2)\\
c_1=-a_1a_2
\end{array}
\right.
\end{align*}
\end{SMALL}
\hspace{-1mm}
\begin{SMALL}
\begin{align*}
\hspace{0.8cm}\left\{
\begin{array}{lllllll}
(a_0,\alpha_1,\beta_0)=(\frac{1}{2},1,\frac{1}{2})\\
(\beta_1,b_0)=(\frac{5}{2}a_1,\frac{3}{4}a_1)\\
b_1=\frac{25}{16}a_1^2-\frac{7}{8}a_2\\
c_0=\frac{1}{16}a_1^2-\frac{3}{8}a_2\\
a_3=\frac{1}{4}a_1(a_1^2+2a_2)\\
b_2=-\frac{3}{16}a_1(a_1^2+2a_2)\\
c_1=-a_1a_2
\end{array}
\right.
\Rightarrow
\left\{
\begin{array}{ll}
\rho_{9}^{13}=\frac{3}{64}(a_1^2+2a_2)(53a_1^2-4a_2)\\
\hspace{0.8cm}-\frac{189}{4}b_3+\frac{55}{4}c_2=0\\
\rho_{10}^{11}=-\frac{15}{256}(a_1^2+2a_2)(135a_1^2-32a_2)\\
\hspace{0.8cm}+\frac{3225}{16}b_3-\frac{755}{16}c_2=0
\end{array}
\right.
\Leftrightarrow\hspace{2mm}
\left\{
\begin{array}{ll}
b_3=-\frac{1}{64}(a_1^2+2a_2)^2\\
c_2=-\frac{3}{64}(5a_1^2+2a_2)(a_1^2+2a_2)
\end{array}
\right.
\end{align*}
\end{SMALL}
\hspace{-1mm}et
\begin{SMALL}
\begin{align*}
\left\{
\begin{array}{lllllllll}
(a_0,\alpha_1,\beta_0)=(\frac{1}{2},1,\frac{1}{2})\\
(\beta_1,b_0)=(\frac{5}{2}a_1,\frac{3}{4}a_1)\\
b_1=\frac{25}{16}a_1^2-\frac{7}{8}a_2\\
c_0=\frac{1}{16}a_1^2-\frac{3}{8}a_2\\
a_3=\frac{1}{4}a_1(a_1^2+2a_2)\\
b_2=-\frac{3}{16}a_1(a_1^2+2a_2)\\
c_1=-a_1a_2\\
b_3=-\frac{1}{64}(a_1^2+2a_2)^2\\
c_2=-\frac{3}{64}(5a_1^2+2a_2)(a_1^2+2a_2)
\end{array}
\right.
\Rightarrow
\begin{array}{l}
\rho_{8}^{14}=\frac{19}{32}(a_1^2+2a_2)^2a_1+19c_3=0
\end{array}
\Leftrightarrow\hspace{2mm}
\begin{array}{l}
c_3=-\frac{1}{32}(a_1^2+2a_2)^2a_1.
\end{array}
\end{align*}
\end{SMALL}
\hspace{-1mm}Il s'en suit que la $1$-forme $\omega$ s'écrit
\begin{align*}
\hspace{-2mm}\frac{1}{64}\Big(4x^2+8a_1xy+\lambda\,y^2+8y\Big)\Big(8x\mathrm{d}x+(8+12a_1x-\lambda\,y)\mathrm{d}y
+(\mu\,x-2\lambda\,a_1y)(x\mathrm{d}y-y\mathrm{d}x)\Big),
\end{align*}
où $\lambda=a_1^2+2a_2$\, et \,$\mu=a_1^2-6a_2$. D'où $\deg\F=1$: contradiction.
\vspace{2mm}

\textbf{\textit{5.}} Examinons le cas où $(a_0,\alpha_1,\beta_0)=(0,1,-1).$ Écrivons
\begin{align*}
&b_0=\delta_0-a_1,&&c_0=\varepsilon_0-\tfrac{1}{4}a_1^2,&&b_1=\delta_1+\tfrac{1}{4}a_1^2-a_2,&&\beta_1=\varepsilon_1+a_1,\\
&b_2=\delta_2-a_3,&&c_1=\varepsilon_2-\tfrac{1}{2}a_1a_2,&& c_2=\varepsilon_3-\tfrac{1}{2}a_1a_3;
\end{align*}
l'hypothèse $K(\Leg\F)\equiv0$ implique que $\delta_0=\delta_1=\delta_2=\varepsilon_0=\varepsilon_1=\varepsilon_2=\varepsilon_3=b_3=a_2=c_3=0$; en effet
\begin{Small}
\begin{align*}
\hspace{-6.3cm}(a_0,\alpha_1,\beta_0)=(0,1,-1)
\Rightarrow
\rho_{14}^{5}=-24\delta_0^2=0
\Leftrightarrow\hspace{2mm}
\delta_0=0
\end{align*}
\end{Small}
\hspace{-1mm}
\begin{Small}
\begin{align*}
\hspace{-1.1cm}(a_0,\alpha_1,\beta_0,\delta_0)=(0,1,-1,0)
\Rightarrow
\left\{
\begin{array}{lll}
\rho_{14}^{3}=8\varepsilon_0(\varepsilon_1a_1-2\delta_1-2\varepsilon_0)=0\\
\rho_{13}^{5}=8\varepsilon_0(2\varepsilon_1^2+3\varepsilon_1a_1-6\delta_1-2\varepsilon_0)=0\\
\rho_{14}^{2}=-4\varepsilon_0(\varepsilon_1a_1^2-2\delta_1a_1-2\varepsilon_0a_1-4\varepsilon_1\varepsilon_0)=0
\end{array}
\right.
\Leftrightarrow\hspace{2mm}
\varepsilon_0=0
\end{align*}
\end{Small}
\hspace{-1mm}
\begin{Small}
\begin{align*}
\hspace{-0.2cm}\left\{
\begin{array}{ll}
(a_0,\alpha_1,\beta_0)=(0,1,-1)\\
(\delta_0,\varepsilon_0)=(0,0)
\end{array}
\right.
\Rightarrow
\left\{
\begin{array}{ll}
\rho_{10}^{11}=3\varepsilon_1^2(6\delta_1-\varepsilon_1^2-3\varepsilon_1a_1)=0\\
\rho_{9}^{13}=-2(\varepsilon_1^2-3\varepsilon_1a_1+6\delta_1)(\varepsilon_1^2+2\varepsilon_1a_1-4\delta_1)=0
\end{array}
\right.
\Leftrightarrow\hspace{2mm}
\left\{
\begin{array}{ll}
\delta_1=0\\
\varepsilon_1=0
\end{array}
\right.
\end{align*}
\end{Small}
\hspace{-1mm}
\begin{Small}
\begin{align*}
\hspace{-3cm}\left\{
\begin{array}{ll}
(a_0,\alpha_1,\beta_0)=(0,1,-1)\\
(\delta_0,\delta_1,\varepsilon_0,\varepsilon_1)=(0,0,0,0)
\end{array}
\right.
\Rightarrow
\left\{
\begin{array}{ll}
\rho_{7}^{15}=16\delta_2(3\delta_2-7\varepsilon_2)=0\\
\rho_{8}^{13}=32\varepsilon_2(\varepsilon_2-9\delta_2)=0
\end{array}
\right.
\Leftrightarrow\hspace{2mm}
\left\{
\begin{array}{ll}
\delta_2=0\\
\varepsilon_2=0
\end{array}
\right.
\end{align*}
\end{Small}
\hspace{-1mm}
\begin{Small}
\begin{align*}
\left\{
\begin{array}{lll}
(a_0,\alpha_1,\beta_0)=(0,1,-1)\\
(\delta_0,\delta_1,\delta_2)=(0,0,0)\\
(\varepsilon_0,\varepsilon_1,\varepsilon_2)=(0,0,0)
\end{array}
\right.
\Rightarrow
\left\{
\begin{array}{lll}
\rho_{5}^{17}=-8b_3(a_2^2+14\varepsilon_3-6b_3)=0\\
\rho_{6}^{15}=2\varepsilon_3(16\varepsilon_3-144b_3+a_2^2)-54b_3a_2^2=0\\
\rho_{7}^{13}=8\varepsilon_3(10\varepsilon_3-36b_3+3a_2^2)-8b_3(11a_2^2+6b_3)=0
\end{array}
\right.
\Leftrightarrow\hspace{2mm}
\left\{
\begin{array}{ll}
\varepsilon_3=0\\
b_3=0
\end{array}
\right.
\end{align*}
\end{Small}
\hspace{-1mm}et
\begin{Small}
\begin{align*}
\left\{
\begin{array}{lll}
(a_0,\alpha_1,\beta_0)=(0,1,-1)\\
(\delta_0,\delta_1,\delta_2,b_3)=(0,0,0,0)\\
(\varepsilon_0,\varepsilon_1,\varepsilon_2,\varepsilon_3)=(0,0,0,0)
\end{array}
\right.
\Rightarrow
\left\{
\begin{array}{ll}
\rho_{8}^{11}=2a_2^4=0\\
\rho_{6}^{13}=(12a_3^2-159a_1c_3)a_2^2+60a_2a_3c_3+48c_3^2=0
\end{array}
\right.
\Leftrightarrow\hspace{2mm}
\left\{
\begin{array}{ll}
a_2=0\\
c_3=0.
\end{array}
\right.
\end{align*}
\end{Small}
\hspace{-1mm}Donc la $1$-forme $\omega$ s'écrit
\begin{align*}
\frac{1}{4}\Big(a_1x+2\Big)\Big(2\hspace{0.2mm}xy\mathrm{d}x-(2\hspace{0.2mm}x^2-a_1xy-2y)\mathrm{d}y
-(a_1x^2+2\hspace{0.2mm}a_3y^2)(x\mathrm{d}y-y\mathrm{d}x)\Big)\hspace{1mm};
\end{align*}
mais ceci contredit l'égalité $\deg\F=3.$

\newpage
\section{Cas $1$-jet nul et $2$-jet non nul}\label{sec:cas-ordre-2}

Supposons que la singularité $O$ soit de multiplicité algébrique $2$, {\it i.e.} que $J^{1}_{(0,0)}\,\omega=0$\, et \,$J^{2}_{(0,0)}\,\omega\neq0$; $\omega$ s'écrit alors sous la forme $$\omega=A_2(x,y)\mathrm{d}x+B_2(x,y)\mathrm{d}y+A_3(x,y)\mathrm{d}x+B_3(x,y)\mathrm{d}y+C_3(x,y)(x\mathrm{d}y-y\mathrm{d}x),$$ où $A_2,B_2$ (resp. $A_3,B_3,C_3$) sont des polynômes homogènes de degré $2$ (resp. $3$). Les polynômes $A_2$ et $B_2$ ne sont pas tous deux nuls car $J^{2}_{(0,0)}\,\omega=A_2(x,y)\mathrm{d}x+B_2(x,y)\mathrm{d}y$.

\noindent Dans la suite, nous nous placerons systématiquement dans la carte affine $(p,q)$ de $\pd$ associée à la droite $\{px-qy=1\}\subset\pp$, où le $3$-tissu $\Leg\F$ est défini par l'équation différentielle
\begin{small}
\begin{align*}
F(p,q,w):=(pw-q)\left(qA_2(w,1)+pB_2(w,1)\right)+qA_3(w,1)+pB_3(w,1)+C_3(w,1)=0,\quad\text{où}\hspace{1mm} w=\frac{\mathrm{d}q}{\mathrm{d}p}.
\end{align*}
\end{small}
\hspace{-1mm}Nous adopterons les notations suivantes
\begin{small}
\begin{align*}
&\hspace{-2mm}A_2=\sum_{i=0}^{2}\alpha_i\hspace{0.1mm}x^{2-i}y^{i},&&
B_2=\sum_{i=0}^{2}\beta_i\hspace{0.1mm}x^{2-i}y^{i},&&
A_3=\sum_{i=0}^{3}a_i\hspace{0.1mm}x^{3-i}y^{i},&&
B_3=\sum_{i=0}^{3}b_i\hspace{0.1mm}x^{3-i}y^{i},&&
C_3=\sum_{i=0}^{3}c_i\hspace{0.1mm}x^{3-i}y^{i}.
\end{align*}
\end{small}
\hspace{-1mm}Le calcul explicite de $K(\Leg\F)$ montre qu'elle s'écrit sous la forme $$K(\Leg\F)=\frac{\sum\limits_{j+k\leq17}\rho_{\hspace{-0.3mm}j}^{\hspace{0.1mm}k}p^jq^k}{R(p,q)^2}\mathrm{d}p\wedge\mathrm{d}q,$$
où $R:=\text{Result}(F,\partial_{w}(F))$ et les $\rho_{\hspace{-0.3mm}j}^{\hspace{0.1mm}k}$ sont des polynômes en les paramètres $\alpha_i,\beta_i,a_i,b_i,c_i.$

\noindent Signalons que les $\rho_{\hspace{-0.3mm}j}^{\hspace{0.1mm}k}$ ont des expressions très compliquées; il est très difficile et peut-être impossible de contrôler la condition $K(\Leg\F)\equiv0$ \og à la main \fg, comme nous l'avons pu faire dans les cas précédents où la singularité $O$ était selle-nœud ou nilpotente.

\noindent Pour s'affranchir de cette difficulté, nous allons envisager quatre éventualités, chacune donnant lieu à une sous-section. Il s'agit de réduire le nombre des paramètres $\alpha_i, \beta_i$ en raisonnant suivant la nature du cône tangent $xA_2+yB_2$ de $J^{2}_{(0,0)}\,\omega$\, qui, a priori, peut être nul, une droite, deux droites ou trois droites.

\DeactivateToc
\subsection{Cas $2$-jet à cône tangent nul}\label{subsec:$2$-jet-cône-tangent-nul}

\`{A} conjugaison linéaire près $\omega$ s'écrit
\begin{align*}
& y(x\mathrm{d}y-y\mathrm{d}x)+\mathrm{d}x\sum_{i=0}^{3}a_i\hspace{0.1mm}x^{3-i}y^{i}+\mathrm{d}y\sum_{i=0}^{3}b_i\hspace{0.1mm}x^{3-i}y^{i}
+(x\mathrm{d}y-y\mathrm{d}x)\sum_{i=0}^{3}c_i\hspace{0.1mm}x^{3-i}y^{i},\qquad a_i,b_i,c_i\in\mathbb{C}.
\end{align*}
En faisant agir le difféomorphisme \begin{small}$\left(\dfrac{2x}{a_2x+a_3y+2}, \dfrac{2y}{a_2x+a_3y+2}\right)$\end{small} sur $\omega$, on peut supposer que $a_2=a_3=0.$ Dans ce cas le calcul explicite de $K(\Leg\F)$ donne $\rho_{2}^{11}=108a_0^5,$ donc $a_0=0$.

\noindent Le feuilletage $\F$ étant de degré trois, $\vert b_{0}\vert+\vert c_{0}\vert$ est non nul. Suivant que le paramètre $b_0$ est nul ou non, on se ramène à $(b_0,c_0)=(0,1)$ ou $b_0=1.$
\vspace{2mm}

\textbf{\textit{1.}} Lorsque $(b_0,c_0)=(0,1),$ l'hypothèse $K(\Leg\F)\equiv0$ implique que $a_1=b_1=b_2=b_3=0$; en effet
\begin{small}
\begin{align*}
&(a_2,a_3,a_0,b_0,c_0)=(0,0,0,0,1)\Rightarrow\rho_{5}^{6}=24a_1^3=0\Leftrightarrow a_1=0,\\
&(a_2,a_3,a_0,b_0,c_0,a_1)=(0,0,0,0,1,0)\Rightarrow\rho_{7}^{4}=-8b_1=0\Leftrightarrow b_1=0,\\
&(a_2,a_3,a_0,b_0,c_0,a_1,b_1)=(0,0,0,0,1,0,0)\Rightarrow\rho_{8}^{3}=-16b_2=0\Leftrightarrow b_2=0,\\
&(a_2,a_3,a_0,b_0,c_0,a_1,b_1,b_2)=(0,0,0,0,1,0,0,0)\Rightarrow\rho_{9}^{2}=-24b_3=0\Leftrightarrow b_3=0.
\end{align*}
\end{small}
\hspace{-1mm}Ainsi $\omega$ s'écrit $(y+x^3+c_1x^2y+c_2xy^2+c_3y^3)(x\mathrm{d}y-y\mathrm{d}x)$; d'où $\deg\F=0$: contradiction.
\vspace{2mm}

\textbf{\textit{2.}} Quand $b_0=1,$ l'hypothèse $K(\Leg\F)\equiv0$ entraîne que $a_1+1=b_1=b_2=b_3=0$; en effet
\begin{small}
\begin{align*}
(a_2,a_3,a_0,b_0)=(0,0,0,1)
\Rightarrow
\left\{
\begin{array}{ll}
\rho_{5}^{7}=16a_1^2(a_1+1)(3a_1+4)=0\\
\rho_{7}^{6}=2(1+a_1)(12a_1^2+26a_1+15)=0
\end{array}
\right.
\Leftrightarrow\hspace{2mm}
a_1=-1
\end{align*}
\end{small}
\hspace{-1mm}et
\begin{small}
\begin{align*}
\hspace{0.5cm}(a_2,a_3,a_0,b_0,a_1)=(0,0,0,1,-1)
\Rightarrow
\left\{
\begin{array}{lll}
\rho_{8}^{5}=4b_1=0\\
\rho_{9}^{4}=8b_1(b_1-c_0)+6b_2=0\\
\rho_{10}^{3}=16b_2(b_1-c_0)+8b_3=0
\end{array}
\right.
\Leftrightarrow\hspace{2mm}
\left\{
\begin{array}{lll}
b_1=0\\
b_2=0\\
b_3=0.
\end{array}
\right.
\end{align*}
\end{small}
\hspace{-1mm}Par suite $\omega$ s'écrit $(y+x^2+c_0x^3+c_1x^2y+c_2xy^2+c_3y^3)(x\mathrm{d}y-y\mathrm{d}x)$; ceci contredit l'égalité $\deg\F=3.$

\subsection{Cas $2$-jet à cône tangent réduit à une seule droite}

On peut supposer que le cône tangent de $J^{2}_{(0,0)}\,\omega$\, est la droite $x=0$; il s'en suit que $\omega$ s'écrit (\cite{CM82})
\begin{Small}
\begin{align*}
& x^3\left(\lambda\frac{\mathrm{d}x}{x}+\mathrm{d}\left(\frac{\delta_1 xy+\delta_2 y^2}{x^2}\right)
\right)+\mathrm{d}x\sum_{i=0}^{3}a_i\hspace{0.1mm}x^{3-i}y^{i}+\mathrm{d}y\sum_{i=0}^{3}b_i\hspace{0.1mm}x^{3-i}y^{i}
+(x\mathrm{d}y-y\mathrm{d}x)\sum_{i=0}^{3}c_i\hspace{0.1mm}x^{3-i}y^{i},&& \lambda,\delta_i,a_i,b_i,c_i\in\mathbb{C}.
\end{align*}
\end{Small}
\hspace{-1mm}On a $xA_2(x,y)+yB_2(x,y)=\lambda\,x^3$; le paramètre $\lambda$ est donc non nul et on peut le supposer égal~à~$1.$ Ainsi
\begin{small}
\begin{align*}
\omega=x^2\mathrm{d}x+(\delta_1x+2\delta_2y)(x\mathrm{d}y-y\mathrm{d}x)
+\mathrm{d}x\sum_{i=0}^{3}a_i\hspace{0.1mm}x^{3-i}y^{i}
+\mathrm{d}y\sum_{i=0}^{3}b_i\hspace{0.1mm}x^{3-i}y^{i}
+(x\mathrm{d}y-y\mathrm{d}x)\sum_{i=0}^{3}c_i\hspace{0.1mm}x^{3-i}y^{i}.
\end{align*}
\end{small}
\hspace{-1mm}Si $\delta_1\neq0,$ resp. $\delta_2\neq0,$ en conjuguant $\omega$ par $\left(x,\frac{1}{\delta_1}y\right)$, resp. $\left(x,\sqrt{\frac{1}{2\delta_2}}y-\frac{\delta_1}{2\delta_2}x\right)$, on se ramène à $\delta_1=1,$ resp. $(\delta_1,\delta_2)=(0,\frac{1}{2}).$ \noindent En outre, en faisant agir le difféomorphisme
\begin{small}
\begin{align*}
\left(\dfrac{x}{\left(b_0-2\delta_1a_0\right)y-a_0\hspace{0.1mm}x+1}, \dfrac{y}{\left(b_0-2\delta_1a_0\right)y-a_0\hspace{0.1mm}x+1}\right)
\end{align*}
\end{small}
\hspace{-1mm}sur $\omega$, on se ramène à $a_0=b_0=0$ sans altérer le~$2$-jet de $\omega$ en $(0,0).$ Il suffit donc de traiter les possibilités suivantes
\begin{align*}
& (\delta_1,\delta_2,a_0,b_0)=(0,0,0,0),&&\hspace{4mm} (\delta_1,\delta_2,a_0,b_0)=(1,0,0,0),&&\hspace{4mm} (\delta_1,\delta_2,a_0,b_0)=(0,\tfrac{1}{2},0,0).
\end{align*}

\textbf{\textit{1.}} Lorsque $(\delta_1,\delta_2,a_0,b_0)=(0,\tfrac{1}{2},0,0),$ l'hypothèse $K(\Leg\F)\equiv0$ implique que $a_i=b_i=c_i=0,i=0,1,2,3$; en effet
\begin{SMALL}
\begin{align*}
\left\{
\begin{array}{llll}
\vspace{3mm}
\delta_1=0\\
\vspace{3mm}
\delta_2=\frac{1}{2}\\
\vspace{3mm}
a_0=0\\
b_0=0
\end{array}
\right.
\Rightarrow
\left\{
\begin{array}{llllll}
\rho_{10}^{7}=-12b_3=0\\
\rho_{9}^{8}=-10(a_3+b_2)=0\\
\rho_{8}^{9}=-86b_3-8(a_2+b_1)=0\\
\rho_{7}^{10}=-6a_1-75a_3-72b_2=0\\
\rho_{6}^{11}=-60a_2-58b_1-126b_3=0\\
\rho_{5}^{12}=-45a_1-129a_3-106b_2=0
\end{array}
\right.
\Leftrightarrow\hspace{2mm}
\left\{
\begin{array}{llllll}
b_3=0\\
b_2=0\\
b_1=0\\
a_3=0\\
a_2=0\\
a_1=0
\end{array}
\right.
\Rightarrow
\left\{
\begin{array}{llll}
\vspace{3mm}
\rho_{9}^{7}=-24c_3=0\\
\vspace{3mm}
\rho_{8}^{8}=-20c_2=0\\
\vspace{3mm}
\rho_{7}^{9}=-16(c_1+12c_3)=0\\
\rho_{6}^{10}=-12c_0-155c_2=0
\end{array}
\right.
\Leftrightarrow\hspace{2mm}
\left\{
\begin{array}{llll}
\vspace{3mm}
c_3=0\\
\vspace{3mm}
c_2=0\\
\vspace{3mm}
c_1=0\\
c_0=0.
\end{array}
\right.
\end{align*}
\end{SMALL}
\hspace{-1mm}Ainsi $\omega$ s'écrit $x^2\mathrm{d}x+y(x\mathrm{d}y-y\mathrm{d}x)$; d'où $\deg\F=2$: contradiction.
\vspace{2mm}

\textbf{\textit{2.}} Quand $(\delta_1,\delta_2,a_0,b_0)=(1,0,0,0),$ l'hypothèse $K(\Leg\F)\equiv0$ entraîne que $$a_2+b_1=a_3=b_2=b_3=c_3=0\hspace{1mm};$$ en effet, $\rho_{7}^{10}=-8b_3$ et
\begin{SMALL}
\begin{align*}
\left\{
\begin{array}{lllll}
\delta_1=1\\
\delta_2=0\\
a_0=0\\
b_0=0\\
b_3=0
\end{array}
\right.
\Rightarrow
\left\{
\begin{array}{llll}
\vspace{1.3mm}
\rho_{6}^{11}=-6(a_3+b_2)=0\\
\vspace{1.3mm}
\rho_{5}^{12}=-4a_2-12a_3-4b_1-10b_2=0\\
\vspace{1.3mm}
\rho_{7}^{9}=4(4a_2-9a_3+4b_1)a_3+8(2a_2-3a_3+2b_1+2b_2)b_2=0\\
\rho_{6}^{10}=-(28a_2+126a_3+b_1+105b_2)a_3+(2a_2+35b_1+15b_2)b_2+8(a_2+b_1)^2-12c_3=0
\end{array}
\right.
\Leftrightarrow
\left\{
\begin{array}{llll}
\vspace{1.3mm}
a_3=0\\
\vspace{1.3mm}
b_2=0\\
\vspace{1.3mm}
b_1=-a_2\\
c_3=0.
\end{array}
\right.
\end{align*}
\end{SMALL}
\hspace{-1mm}Les égalités $\delta_2=a_3=b_3=c_3=0$ impliquent que $\deg\F\leq2$, ce qui est absurde.
\vspace{2mm}

\textbf{\textit{3.}} Lorsque $(\delta_1,\delta_2,a_0,b_0)=(0,0,0,0),$ la platitude de $\Leg\F$ implique que $b_2=b_3=0$; en effet, $\rho_{6}^{9}=78b_3^3$ et
\begin{Small}
\begin{align*}
(\delta_1,\delta_2,a_0,b_0,b_3)=(0,0,0,0,0)
\Rightarrow
\left\{
\begin{array}{ll}
\rho_{3}^{12}=4a_3b_2(9a_3-b_2)=0\\
\rho_{7}^{6}=-2b_2^2\Big(27(9a_3+10b_2)a_3^2+(45a_3-2b_2)b_2^2\Big)=0
\end{array}
\right.
\Leftrightarrow
b_2=0.
\end{align*}
\end{Small}
\hspace{-1mm}Ainsi, dans ce dernier cas, le feuilletage $\F$ est décrit par
\begin{align*}
\omega=x^2\mathrm{d}x+y(a_1x^2+a_2xy+a_3y^2)\mathrm{d}x+b_1x^2y\mathrm{d}y+(c_0\,x^3+c_1x^2y+c_2xy^2+c_3y^3)(x\mathrm{d}y-y\mathrm{d}x).
\end{align*}

\noindent Maintenant nous allons examiner les trois éventualités suivantes
\begin{align*}
& a_3\neq0\hspace{1mm};&& (a_2,a_3)=(0,0)\hspace{1mm};&& a_3=0, a_2\neq0.
\end{align*}

\textbf{\textit{3.1.}} Si $a_3\neq0,$ alors, quitte à conjuguer $\omega$ par l'homothétie $\left(\frac{1}{a_3}x,\frac{1}{a_3}y\right)$, on peut supposer que $a_3=1.$ Dans ce cas la platitude de $\Leg\F$ implique que $b_1=c_i=0,i=0,1,2,3;$ en effet

\begin{small}
\begin{align*}
&a_3=1\Rightarrow\rho_{2}^{13}=24b_1=0\Leftrightarrow b_1=0,\\
&(a_3,b_1)=(1,0)\Rightarrow\rho_{3}^{11}=-120c_3=0\Leftrightarrow c_3=0,\\
&(a_3,b_1,c_3)=(1,0,0)\Rightarrow\rho_{2}^{12}=-48c_2=0\Leftrightarrow c_2=0,\\
&(a_3,b_1,c_3,c_2)=(1,0,0,0)\Rightarrow\rho_{3}^{10}=-396c_1=0\Leftrightarrow c_1=0,\\
&(a_3,b_1,c_3,c_2,c_1)=(1,0,0,0,0)\Rightarrow\rho_{2}^{11}=-168c_0=0\Leftrightarrow c_0=0.
\end{align*}
\end{small}
\hspace{-1mm}Par suite $\omega$ s'écrit $(x^2+a_1x^2y+a_2xy^2+y^3)\mathrm{d}x,$ ce qui contredit l'égalité $\deg\F=3.$
\vspace{2mm}

\textbf{\textit{3.2.}} Lorsque $(a_2,a_3)=(0,0),$ la platitude de $\Leg\F$ conduit à $c_3=0$; en effet
\begin{small}
\begin{align*}
(a_2,a_3)=(0,0)
\Rightarrow
\left\{
\begin{array}{ll}
\rho_{2}^{11}=24b_1c_3^2=0\\
\rho_{3}^{9}=-24c_3^2(2b_1^2+5c_3)=0
\end{array}
\right.
\Leftrightarrow
c_3=0.
\end{align*}
\end{small}
\hspace{-1mm}Les égalités $a_3=c_3=0$ implique que $\deg\F<3$: contradiction.
\vspace{2mm}

\textbf{\textit{3.3.}} Cas où $a_3=0$\, et \,$a_2\neq0.$ La conjugaison par l'homothétie $\left(\frac{1}{a_2}x,\frac{1}{a_2}y\right)$ permet d'effectuer la normalisation $a_2=1.$ Dans ce cas le calcul explicite de $K(\Leg\F)$ donne
\begin{Small}
\begin{align*}
&\rho_{2}^{11}=\Big(4c_3+1\Big)\Big((6c_3+1)b_1+c_3\Big),\\
&\rho_{3}^{9}=-2\Big(4c_3+1\Big)\Big((6c_3-1)b_1^2+14c_3b_1+c_3(15c_3+8)\Big),\\
&\rho_{5}^{5}=8c_3b_1\Big((6c_3+1)b_1^3-(11c_3+2)b_1^2+(36c_3^2-12c_3-4)b_1+6c_3(9c_3+2)\Big),\\
&\rho_{4}^{7}=-(18c_3+3)b_1^3+(204c_3^2+27c_3-4)b_1^2+9c_3(22c_3^2+39c_3+8)b_1+c_3(9c_3+4)(33c_3+8)\hspace{1mm};
\end{align*}
\end{Small}
\hspace{-1mm}il est facile de voir que les seules solutions du système $\rho_{2}^{11}=\rho_{3}^{9}=\rho_{5}^{5}=\rho_{4}^{7}=0$, sous la condition $c_3\neq0$, sont
\begin{align*}
& (b_1,c_3)=(-\tfrac{1}{2},-\tfrac{1}{4})&& \text{et} && (b_1,c_3)=(-\tfrac{2}{3},-\tfrac{2}{9}).
\end{align*}

\textbf{\textit{3.3.1.}} Lorsque $(b_1,c_3)=(-\tfrac{1}{2},-\tfrac{1}{4})$, l'hypothèse $K(\Leg\F)\equiv0$ implique que $$a_1+2c_2=c_0=c_1=0\hspace{1mm};$$ en effet
\begin{Small}
\begin{align*}
\left\{
\begin{array}{llll}
a_3=0\\
a_2=1\\
b_1=-\frac{1}{2}\\
c_3=-\frac{1}{4}
\end{array}
\right.
\Rightarrow
\left\{
\begin{array}{lll}
\smallskip
\rho_{3}^{8}=-\frac{1}{16}(a_1+2c_2)=0\\
\smallskip
\rho_{4}^{5}=-\frac{1}{64}\Big((a_1+2c_2)(176a_1+207c_2)-c_1\Big)=0\\
\rho_{3}^{6}=-\frac{1}{16}\Big((a_1+2c_2)(61a_1^2+84a_1c_2-33c_1-15c_2^2)-\frac{1}{2}c_0\Big)=0
\end{array}
\right.
\Leftrightarrow
\left\{
\begin{array}{lll}
a_1=-2c_2\\
c_1=0\\
c_0=0.
\end{array}
\right.
\end{align*}
\end{Small}
\hspace{-1mm}Ainsi $\omega$ s'écrit $\frac{1}{4}\Big(y^2-4c_2xy+2\hspace{0.1mm}x\Big)\Big(2\hspace{0.1mm}x\mathrm{d}x-y(x\mathrm{d}y-y\mathrm{d}x)\Big),$ ce qui contredit l'égalité $\deg\F=3.$
\vspace{2mm}

\textbf{\textit{3.3.2.}} Lorsque $(b_1,c_3)=(-\tfrac{2}{3},-\tfrac{2}{9})$, la platitude de $\Leg\F$ implique que $a_1=c_i=0$ pour $i=0,1,2$; en effet
\begin{Small}
\begin{align*}
\left\{
\begin{array}{llll}
a_3=0\\
a_2=1\\
b_1=-\frac{2}{3}\\
c_3=-\frac{2}{9}
\end{array}
\right.
\Rightarrow
\left\{
\begin{array}{ll}
\smallskip
\rho_{2}^{10}=-\frac{2}{243}(20a_1+33c_2)=0\\
\rho_{3}^{8}=\frac{64}{729}(7a_1+12c_2)=0\\
\end{array}
\right.
\Leftrightarrow
\left\{
\begin{array}{ll}
a_1=0\\
c_2=0
\end{array}
\right.
\end{align*}
\end{Small}
\hspace{-1mm}et
\begin{Small}
\begin{align*}
\hspace{-1.5cm}\left\{
\begin{array}{llllll}
a_3=0\\
a_2=1\\
b_1=-\frac{2}{3}\\
c_3=-\frac{2}{9}\\
a_1=0\\
c_2=0
\end{array}
\right.
\Rightarrow
\left\{
\begin{array}{ll}
\vspace{5mm}
\rho_{4}^{5}=-\frac{256}{243}c_1=0\\
\rho_{3}^{6}=\frac{128}{243}c_0=0
\end{array}
\right.
\Leftrightarrow
\left\{
\begin{array}{ll}
c_1=0\\
c_0=0.
\end{array}
\right.
\end{align*}
\end{Small}
\hspace{-1mm}Donc $\omega$ s'écrit $\frac{1}{9}\Big(y^2+3x\Big)\Big(3x\mathrm{d}x-2y(x\mathrm{d}y-y\mathrm{d}x)\Big)$, mais ceci contredit l'égalité $\deg\F=3.$

\subsection{Cas $2$-jet à cône tangent formé de deux droites}

On peut supposer que le cône tangent de $J^{2}_{(0,0)}\,\omega$\, est de la forme $\{x^2y=0\}$; dans ce cas, d'après \cite{CM82}, la forme $\omega$ s'écrit
\begin{Small}
\begin{align*}
& x^2y\left(\lambda_0\frac{\mathrm{d}x}{x}+\lambda_1\frac{\mathrm{d}y}{y}+\delta\mathrm{d}\left(\frac{y}{x}\right)\right)
+\mathrm{d}x\sum_{i=0}^{3}a_i\hspace{0.1mm}x^{3-i}y^{i}
+\mathrm{d}y\sum_{i=0}^{3}b_i\hspace{0.1mm}x^{3-i}y^{i}
+(x\mathrm{d}y-y\mathrm{d}x)\sum_{i=0}^{3}c_i\hspace{0.1mm}x^{3-i}y^{i},&& \lambda_i,\delta,a_i,b_i,c_i\in\mathbb{C}.
\end{align*}
\end{Small}
\hspace{-1mm}On a $xA_2(x,y)+yB_2(x,y)=(\lambda_0+\lambda_1)x^2y$; la somme $\lambda_0+\lambda_1$ est alors non nulle. Par suite, on est dans l'une des situations suivantes
\begin{align*}
& \delta=\lambda_1=0,\hspace{1mm}\lambda_0\neq0\hspace{1mm};&&
  \delta=\lambda_0=0,\hspace{1mm}\lambda_1\neq0\hspace{1mm};&&
  \delta=0,\hspace{1mm}\lambda_0\lambda_1(\lambda_0+\lambda_1)\neq0\hspace{1mm}; \\
& \delta\neq0,\hspace{1mm}\lambda_1=0,\hspace{1mm} \lambda_0\neq0\hspace{1mm};&&
  \delta\lambda_1(\lambda_0+\lambda_1)\neq0.
\end{align*}
Si l'un des paramètres $\lambda_i$ est non nul, on peut évidemment le supposer égal à $1$; puis, si $\delta\neq0$ en faisant agir la transformation linéaire $(\delta\hspace{0.1mm}x,y)$ sur $\omega$, on peut effectuer la normalisation $\delta=1.$ Il suffit donc d'étudier les cas suivants
\begin{align*}
& (\delta,\lambda_0,\lambda_1)=(0,1,0)\hspace{1mm};&&
  (\delta,\lambda_0,\lambda_1)=(0,0,1)\hspace{1mm};&&
  (\delta,\lambda_1)=(0,1),\hspace{1mm}\lambda_0(\lambda_0+1)\neq0\hspace{1mm}; \\
& (\delta,\lambda_0,\lambda_1)=(1,1,0)\hspace{1mm};&&
  (\delta,\lambda_1)=(1,1),\hspace{1mm}\lambda_0+1\neq0.
\end{align*}

\textbf{\textit{1.}} Si $(\delta,\lambda_0,\lambda_1)=(0,1,0),$ alors $\omega$ s'écrit
\begin{align*}
& \omega=xy\mathrm{d}x+\mathrm{d}x\sum_{i=0}^{3}a_i\hspace{0.1mm}x^{3-i}y^{i}
+\mathrm{d}y\sum_{i=0}^{3}b_i\hspace{0.1mm}x^{3-i}y^{i}
+(x\mathrm{d}y-y\mathrm{d}x)\sum_{i=0}^{3}c_i\hspace{0.1mm}x^{3-i}y^{i}.
\end{align*}
La conjugaison par le difféomorphisme
\begin{small}
$\left(\dfrac{2\hspace{0.1mm}x}{2-2\hspace{0.1mm}a_1x-a_2y}, \dfrac{2y}{2-2\hspace{0.1mm}a_1x-a_2y}\right)$
\end{small}
permet de supposer que $a_1=a_2=0.$

\noindent On envisage deux cas suivant que $a_0$ est nul ou non.
\vspace{2mm}

\textbf{\textit{1.1.}} Lorsque $a_0=0,$ l'hypothèse $K(\Leg\F)\equiv0$ implique que $b_0=c_0=0$; en effet
\begin{small}
\begin{align*}
(a_1,a_2,a_0)=(0,0,0)
\Rightarrow
\left\{
\begin{array}{ll}
\rho_{4}^{11}=-b_0^3=0\\
\rho_{1}^{11}=-4c_0^2(2b_0^2+c_0)=0
\end{array}
\right.
\Leftrightarrow
\left\{
\begin{array}{ll}
b_0=0\\
c_0=0.
\end{array}
\right.
\end{align*}
\end{small}
\hspace{-1mm}Les égalités $a_0=b_0=c_0=0$ entraînent que $\deg\F<3$: contradiction.
\vspace{2mm}

\textbf{\textit{1.2.}} Si $a_0\neq0,$ alors, quitte à conjuguer $\omega$ par $\left(\sqrt{\frac{1}{a_0}}x,y\right)$, on peut supposer que $a_0=1.$ Dans ce cas l'hypothèse $K(\Leg\F)\equiv0$ entraîne que
\begin{align*}
& (b_0,b_2,b_3)=(0,0,0)&& \text{et}  &&(c_0,c_1,c_2,c_3)=(b_1,0,b_1^2,a_3b_1)\hspace{1mm};
\end{align*}
en effet
\begin{SMALL}
\begin{align*}
\left\{
\begin{array}{lll}
a_1=0\\
a_2=0\\
a_0=1
\end{array}
\right.
\Rightarrow
\left\{
\begin{array}{lllllll}
\rho_{3}^{12}=-2b_0^2=0\\
\rho_{4}^{11}=2b_2-b_0^3=0\\
\rho_{5}^{10}=3b_3+4b_0b_2=0\\
\rho_{1}^{13}=4(2b_0^2+b_1-c_0)=0\\
\rho_{2}^{12}=4b_0^3+(9b_1-10c_0)b_0+17b_2-c_1=0\\
\rho_{1}^{12}=-8b_0^2b_1-(50a_3-80b_2-10c_1)b_0+44b_3+4c_0(b_1-2c_0)+4c_2=0\\
\rho_{4}^{10}=b_0^3b_1+(12a_3+57b_2-c_1)b_0^2+4b_2(2b_1+c_0)+78b_0b_3+3a_3b_1-3c_3=0
\end{array}
\right.
\Leftrightarrow\hspace{2mm}
\left\{
\begin{array}{lllllll}
b_0=0\\
b_2=0\\
b_3=0\\
c_0=b_1\\
c_1=0\\
c_2=b_1^2\\
c_3=a_3b_1.
\end{array}
\right.
\end{align*}
\end{SMALL}
\hspace{-1mm}Ainsi $\omega$ s'écrit $\Big(xy+x^3+b_1xy^2+a_3y^3\Big)\Big(\mathrm{d}x+b_1(x\mathrm{d}y-y\mathrm{d}x)\Big)$; d'où $\deg\F=0$: contradiction.
\vspace{2mm}

\textbf{\textit{2.}} Lorsque $(\delta,\lambda_0,\lambda_1)=(1,1,0),$ la $1$-forme $\omega$ s'écrit
\begin{align*}
& \omega=xy\mathrm{d}x+y(x\mathrm{d}y-y\mathrm{d}x)+\mathrm{d}x\sum_{i=0}^{3}a_i\hspace{0.1mm}x^{3-i}y^{i}
+\mathrm{d}y\sum_{i=0}^{3}b_i\hspace{0.1mm}x^{3-i}y^{i}
+(x\mathrm{d}y-y\mathrm{d}x)\sum_{i=0}^{3}c_i\hspace{0.1mm}x^{3-i}y^{i}.
\end{align*}
Quitte à conjuguer $\omega$ par le difféomorphisme
\begin{Small}
$\left(\dfrac{2\hspace{0.1mm}x}{2-2\hspace{0.1mm}a_1x-(2\hspace{0.1mm}a_1+a_2)y}, \dfrac{2y}{2-2\hspace{0.1mm}a_1x-(2\hspace{0.1mm}a_1+a_2)y}\right)$
\end{Small},
on peut supposer que $a_1=a_2=0.$ En calculant $\rho_{4}^{11},\,\rho_{1}^{11}$ et en raisonnant comme dans \textbf{\textit{1.1.}}, on constate que $a_0\neq0$. En faisant agir l'homothétie $\left(\frac{1}{a_0}x,\frac{1}{a_0}y\right)$ sur $\omega,$ on se ramène à $a_0=1.$ Dans ce cas la platitude de $\Leg\F$ implique que
\begin{align*}
&a_3=-2,&& (b_0,b_1,b_2,b_3)=(1,1,2,0)&& \text{et}  &&(c_0,c_1,c_2,c_3)=(0,1,1,1)\hspace{1mm};
\end{align*}
en effet
\begin{SMALL}
\begin{align*}
\hspace{-4.3cm}\left\{
\begin{array}{lll}
a_1=0\\
a_2=0\\
a_0=1
\end{array}
\right.
\Rightarrow
\left\{
\begin{array}{lll}
\rho_{2}^{13}=1-b_0=0\\
\rho_{12}^{3}=-8b_0^2b_3=0\\
\rho_{4}^{11}=(1-b_0)(b_0^2+10b_0+7)+2b_2-4=0
\end{array}
\right.
\Leftrightarrow\hspace{2mm}
\left\{
\begin{array}{lll}
b_0=1\\
b_3=0\\
b_2=2
\end{array}
\right.
\end{align*}
\end{SMALL}
\hspace{-1mm}et
\begin{SMALL}
\begin{align*}
\hspace{4mm}\left\{
\begin{array}{llllll}
a_1=0\\
a_2=0\\
a_0=1\\
b_0=1\\
b_3=0\\
b_2=2
\end{array}
\right.
\Rightarrow
\left\{
\begin{array}{llllll}
\rho_{5}^{10}=-3(a_3+2b_1)=0\\
\rho_{6}^{9}=-2(11a_3+14b_1+8)=0\\
\rho_{1}^{13}=4(b_1-c_0-1)=0\\
\rho_{2}^{12}=25b_1-25c_0-c_1-24 =0\\
\rho_{11}^{3}=64a_3^2+8(2b_1+31)a_3+16(2b_1-c_3+16)=0\\
\rho_{10}^{4}=2(70a_3+29b_1-6c_0+255)a_3+8(b_1+15)b_1-4(24c_3+6c_0+3c_2-139)=0
\end{array}
\right.
\Leftrightarrow\hspace{2mm}
\left\{
\begin{array}{llllll}
b_1=1\\
a_3=-2\\
c_0=0\\
c_1=1\\
c_3=1\\
c_2=1.
\end{array}
\right.
\end{align*}
\end{SMALL}
\hspace{-1mm}Donc $\omega$ s'écrit $\Big(x^2+xy+y^2+y\Big)\Big(x\mathrm{d}x+(y+1)(x\mathrm{d}y-y\mathrm{d}x)\Big)$, mais ceci contredit l'égalité $\deg\F=3.$
\vspace{2mm}

\textbf{\textit{3.}} Si $(\delta,\lambda_0,\lambda_1)=(0,0,1),$ la forme $\omega$ s'écrit
\begin{align*}
& x^2\mathrm{d}y+\mathrm{d}x\sum_{i=0}^{3}a_i\hspace{0.1mm}x^{3-i}y^{i}+\mathrm{d}y\sum_{i=0}^{3}b_i\hspace{0.1mm}x^{3-i}y^{i}
+(x\mathrm{d}y-y\mathrm{d}x)\sum_{i=0}^{3}c_i\hspace{0.1mm}x^{3-i}y^{i}.
\end{align*}
Quitte à conjuguer $\omega$ par
\begin{small}
$\left(\dfrac{2\hspace{0.1mm}x}{2-b_0\hspace{0.1mm}x-2b_1y}, \dfrac{2y}{2-b_0\hspace{0.1mm}x-2b_1y}\right)$
\end{small},
on peut supposer que $b_0=b_1=0.$ Dans ce cas la platitude de $\Leg\F$ implique que $a_3=0$; en effet
\begin{Small}
\begin{align*}
(b_0,b_1)=(0,0)
\Rightarrow
\left\{
\begin{array}{lll}
\smallskip
\rho_{12}^{3}=-30a_3b_3^2=0\\
\smallskip
\rho_{10}^{5}=8a_1b_3^2+(4a_2b_2+8c_3)b_3-2a_3(15a_3^2+4b_2a_3+b_2^2)=0\\
\rho_{11}^{3}=-456a_0b_3^3-(24c_2+48a_2^2+88a_1b_2+672a_1a_3)b_3^2\\
\hspace{8mm}-(32a_2b_2^2+384a_3c_3+432a_2a_3^2+320a_2a_3b_2+64b_2c_3)b_3\\
\hspace{8mm}+16a_3b_2^2(4a_3+b_2)=0
\end{array}
\right.
\Rightarrow
a_3=0.
\end{align*}
\end{Small}
\hspace{-1mm}Ainsi
\begin{align*}
& \omega=x^2\mathrm{d}y+x(a_0\,x^2+a_1xy+a_2y^2)\mathrm{d}x+y^2(b_2x+b_3y)\mathrm{d}y+(x\mathrm{d}y-y\mathrm{d}x)\sum_{i=0}^{3}c_i\hspace{0.1mm}x^{3-i}y^{i}.
\end{align*}
Le feuilletage $\F$ étant de degré trois, la somme $\vert b_{3}\vert+\vert c_{3}\vert$ est non nulle. Si $b_3\neq0$, resp. $b_3=0,$ en conjuguant $\omega$ par $\left(\frac{1}{b_3}x,\frac{1}{b_3}y\right)$, resp. $\left(\sqrt{\frac{1}{c_3}}x,\sqrt{\frac{1}{c_3}}y\right),$ on se ramène à $b_3=1,$ resp. $c_3=1.$ Il suffit donc d'examiner les possibilités suivantes
\begin{align*}
& b_3=1 && \text{et} && (b_3,c_3)=(0,1).
\end{align*}

\textbf{\textit{3.1.}} Lorsque $(b_3,c_3)=(0,1),$ le calcul explicite de $K(\Leg\F)$ montre que
\begin{Small}
\begin{align*}
\rho_{10}^{2}=-14b_2^5a_0-(44a_1a_2-2c_1)b_2^4-(24a_2c_2+68a_1+48a_2^3)b_2^3-(12c_2+126a_2^2)b_2^2-198b_2a_2-108\hspace{1mm};
\end{align*}
\end{Small}
\hspace{-1mm}on en déduit que $b_2\neq0$. La conjugaison par la transformation linéaire $\left(b_2^3\hspace{0.1mm}x,b_2y\right)$ permet de supposer que~$b_2=1$. Dans ce cas la platitude de $\Leg\F$ implique que
\begin{align*}
&(a_0,a_1,a_2)=(0,0,-1)&& \text{et} && (c_0,c_1,c_2)=(0,0,-1)\hspace{1mm};
\end{align*}
en effet
\begin{SMALL}
\begin{align*}
\left\{
\begin{array}{lll}
b_3=0\\
c_3=1\\
b_2=1
\end{array}
\right.
\Rightarrow
\left\{
\begin{array}{lllllll}
\rho_{5}^{8}=4a_0\hspace{0.1mm}a_2^4=0\\
\rho_{10}^{3}=a_1-5a_2+c_2-4=0\\
\rho_{11}^{1}=-4(2a_1+8a_2^2+15a_2+2c_2+9)=0\\
\rho_{9}^{4}=2a_0+8a_1+6a_1a_2-2a_2^2+2a_2c_2=0\\
\rho_{7}^{6}=4(5a_2^2+12a_2+12)a_0+8(a_2+2)a_1a_2^2=0\\
\rho_{5}^{7}=8(5a_2^2+12a_2+12)a_0^2-2(8a_2^3-a_1a_2-8a_2c_2+14a_1)a_0a_2^2+4a_2^4c_0=0\\
\rho_{10}^{2}=-14a_0-68a_1-2(24a_2^2+63a_2+22a_1+12c_2+99)a_2+2c_1-12c_2-108=0
\end{array}
\right.
\Leftrightarrow\hspace{2mm}
\left\{
\begin{array}{llllll}
a_0=0\\
a_1=0\\
c_0=0\\
c_1=0\\
a_2=-1\\
c_2=-1.
\end{array}
\right.
\end{align*}
\end{SMALL}
\hspace{-1mm}Donc $\omega$ s'écrit $\Big(y^2-xy+x\Big)\Big(x\mathrm{d}y+y(x\mathrm{d}y-y\mathrm{d}x)\Big)$, mais ceci contredit l'égalité $\deg\F=3.$
\vspace{2mm}

\textbf{\textit{3.2.}} Lorsque $b_3=1$, la platitude de $\Leg\F$ entraîne que
\begin{align*}
& (a_0,a_2)=(0,0)&& \text{et} && (c_0,c_1,c_2,c_3)=(-a_1^2,0,-a_1b_2,-a_1)\hspace{1mm};
\end{align*}
en effet
\begin{Small}
\begin{align*}
\hspace{-4.3cm}b_3=1
\Rightarrow
\left\{
\begin{array}{ll}
\rho_{11}^{4}=-4a_2=0\\
\rho_{9}^{6}=4(8a_0+a_2^2)=0
\end{array}
\right.
\Leftrightarrow\hspace{2mm}
\left\{
\begin{array}{ll}
a_2=0\\
a_0=0
\end{array}
\right.
\end{align*}
\end{Small}
\hspace{-1mm}et
\begin{Small}
\begin{align*}
\hspace{-1cm}\left\{
\begin{array}{lll}
b_3=1\\
a_2=0\\
a_0=0
\end{array}
\right.
\Rightarrow
\left\{
\begin{array}{llll}
\rho_{10}^{5}=8(a_1+c_3)=0\\
\rho_{9}^{5}=16(a_1^2-a_1c_3+2c_0)=0\\
\rho_{10}^{4}=8(a_1b_2^2+b_2c_2+2c_1)=0\\
\rho_{11}^{3}=-8(11a_1b_2+8b_2c_3+3c_2)=0
\end{array}
\right.
\Leftrightarrow\hspace{2mm}
\left\{
\begin{array}{llll}
c_3=-a_1\\
c_0=-a_1^2\\
c_2=-a_1b_2\\
c_1=0.
\end{array}
\right.
\end{align*}
\end{Small}
\hspace{-1mm}Par suite $\omega$ s'écrit $\Big(a_1x^3+b_2xy^2+y^3+x^2\Big)\Big(\mathrm{d}y-a_1(x\mathrm{d}y-y\mathrm{d}x)\Big)$, ce qui contredit l'égalité $\deg\F=3.$
\vspace{2mm}

\textbf{\textit{4.}} Lorsque $(\delta,\lambda_1)=(0,1)$\, et \,$\lambda_0(\lambda_0+1)\neq0,$ la $1$-forme $\omega$ s'écrit
\begin{small}
\begin{align*}
& x(x\mathrm{d}y+\lambda_0y\mathrm{d}x)+\mathrm{d}x\sum_{i=0}^{3}a_i\hspace{0.1mm}x^{3-i}y^{i}
+\mathrm{d}y\sum_{i=0}^{3}b_i\hspace{0.1mm}x^{3-i}y^{i}
+(x\mathrm{d}y-y\mathrm{d}x)\sum_{i=0}^{3}c_i\hspace{0.1mm}x^{3-i}y^{i},\qquad \lambda_0(\lambda_0+1)\neq0.
\end{align*}
\end{small}
\hspace{-1mm}Quitte à conjuguer $\omega$ par le difféomorphisme
\begin{small}
$\left(\dfrac{2\lambda_0\hspace{0.2mm}x}{2-\lambda_0\hspace{0.1mm}b_0\hspace{0.2mm}x-a_2y}, \dfrac{2\lambda_0y}{2-\lambda_0\hspace{0.1mm}b_0\hspace{0.2mm}x-a_2y}\right)$
\end{small},
on peut supposer que $a_2=b_0=0.$ Dans ce cas le calcul explicite de $K(\Leg\F)\equiv0$ donne
\begin{small}
\begin{align*}
&\rho_{5}^{12}=4a_0\lambda_0^4(\lambda_0+1)^2,&&\rho_{6}^{11}=2a_1\lambda_0^4(\lambda_0+1)^2,&&\rho_{11}^{5}=54(b_3\lambda_0)^2(\lambda_0+1)^3\hspace{1mm};
\end{align*}
\end{small}
\hspace{-1mm}donc $a_0=a_1=b_3=0$ et $\omega$ s'écrit
\begin{small}
\begin{align*}
& \omega=x(x\mathrm{d}y+\lambda_0y\mathrm{d}x)+a_3y^3\mathrm{d}x+xy(b_1x+b_2y)\mathrm{d}y
+(x\mathrm{d}y-y\mathrm{d}x)\sum_{i=0}^{3}c_i\hspace{0.1mm}x^{3-i}y^{i},\qquad \lambda_0(\lambda_0+1)\neq0.
\end{align*}
\end{small}
\hspace{-1mm}Nous allons traiter deux possibilités suivant que le paramètre $a_3$ est nul ou non.
\vspace{2mm}

\textbf{\textit{4.1.}} Lorsque $a_3=0$, l'hypothèse $K(\Leg\F)\equiv0$ implique que $b_2=c_3=0$; en effet
\begin{Small}
\begin{align*}
\hspace{-2cm}a_3=0
\Rightarrow
\left\{
\begin{array}{ll}
\vspace{3mm}
\rho_{11}^{3}=-8b_2^4\,\lambda_0(7\lambda_0+1)=0\\
\rho_{9}^{7}=-2(b_2\lambda_0)^2(2\lambda_0^2+5\lambda_0+1)(\lambda_0+1)^2=0
\end{array}
\right.
\Leftrightarrow
b_2=0
\end{align*}
\end{Small}
\hspace{-1mm}et
\begin{Small}
\begin{align*}
\hspace{-4.4cm}\left\{
\begin{array}{ll}
a_3=0\\
b_2=0
\end{array}
\right.
\Rightarrow
\rho_{9}^{5}=54c_3^2\,\lambda_0^3\,(\lambda_0+1)^2=0
\Leftrightarrow
c_3=0.
\end{align*}
\end{Small}
\hspace{-1mm}Les égalités $a_3=c_3=0$ entraînent que $\deg\F<3$: contradiction.
\vspace{2mm}

\textbf{\textit{4.2.}} Si $a_3\neq0$, alors, quitte à faire agir l'homothétie $\left(\frac{1}{a_3}x,\frac{1}{a_3}y\right)$ sur $\omega,$ on peut supposer que $a_3=1.$ Dans ce cas la platitude de $\Leg\F$ implique que
\begin{align*}
&\lambda_0=-\frac{1}{2},&& (b_1,b_2)=(0,-2)&& \text{et} && c_i=0, i=0,1,2,3\hspace{1mm};
\end{align*}
en effet
\begin{SMALL}
\begin{align*}
\hspace{4mm}a_3=1
\Rightarrow
\left\{
\begin{array}{lll}
\smallskip
\rho_{8}^{9}=\lambda_0^2\Big(\lambda_0+1\Big)^2\Big(\lambda_0(2\lambda_0^2+\lambda_0+1)b_2-4\lambda_0^2-4\lambda_0-2\Big)=0\\
\smallskip
\rho_{9}^{7}=-2\lambda_0^2\Big(\lambda_0+1\Big)^2\Big((2\lambda_0^2+5\lambda_0+1)b_2^2-(8\lambda_0-4\lambda_0^2+7)b_2-12\lambda_0-6\Big)=0\\
\rho_{11}^{3}=-8b_2^2\Big(\lambda_0(7\lambda_0+1)b_2^2+2(4\lambda_0^3+13\lambda_0^2+\lambda_0-1)b_2+4(\lambda_0-1)(2\lambda_0+1)(\lambda_0+2)\Big)=0
\smallskip
\end{array}
\right.
\Leftrightarrow
\left\{
\begin{array}{ll}
\vspace{4mm}
\lambda_0=-\frac{1}{2}\\
b_2=-2
\end{array}
\right.
\end{align*}
\end{SMALL}
\hspace{-1mm}et
\begin{SMALL}
\begin{align*}
\hspace{-4.9cm}\left\{
\begin{array}{lll}
a_3=1\\
\lambda_0=-\frac{1}{2}\\
b_2=-2
\end{array}
\right.
\Rightarrow
\left\{
\begin{array}{lllll}
\rho_{5}^{11}=\frac{3}{32}c_0=0\\
\rho_{6}^{10}=\frac{7}{128}c_1=0\\
\rho_{9}^{6}=\frac{1}{4}(34b_1-9c_3)=0\\
\rho_{8}^{8}=\frac{1}{128}(46b_1-15c_3)=0\\
\rho_{8}^{7}=-\frac{1}{16}(54b_1^2-23b_1c_3-12c_2)=0
\end{array}
\right.
\Leftrightarrow\hspace{2mm}
\left\{
\begin{array}{lllll}
c_0=0\\
c_1=0\\
b_1=0\\
c_3=0\\
c_2=0.
\end{array}
\right.
\end{align*}
\end{SMALL}
\hspace{-1mm}Ainsi la $1$-forme $\omega$ s'écrit $(x-2y^2)(x\mathrm{d}y-\frac{1}{2}y\mathrm{d}x)$; d'où $\deg\F=1$: contradiction.
\vspace{2mm}

\textbf{\textit{5.}} Considérons l'éventualité $(\delta,\lambda_1)=(1,1),\hspace{1mm}\lambda_0+1\neq0.$ Dans ce cas $\omega$ s'écrit
\begin{small}
\begin{align*}
& x(x\mathrm{d}y+\lambda_0y\mathrm{d}x)+y(x\mathrm{d}y-y\mathrm{d}x)
+\mathrm{d}x\sum_{i=0}^{3}a_i\hspace{0.1mm}x^{3-i}y^{i}
+\mathrm{d}y\sum_{i=0}^{3}b_i\hspace{0.1mm}x^{3-i}y^{i}
+(x\mathrm{d}y-y\mathrm{d}x)\sum_{i=0}^{3}c_i\hspace{0.1mm}x^{3-i}y^{i},\quad \lambda_0+1\neq0.
\end{align*}
\end{small}
\hspace{-1mm}Quitte à conjuguer $\omega$ par le difféomorphisme
\begin{Small}
\begin{align*}
\left(\dfrac{2\hspace{0.1mm}x}{\left(a_2+\lambda_0\hspace{0.1mm}a_3\right)x+a_3y+2}, \dfrac{2y}{\left(a_2+\lambda_0\hspace{0.1mm}a_3\right)x+a_3y+2}\right),
\end{align*}
\end{Small}
\hspace{-1mm}on peut supposer que $a_2=a_3=0.$ Dans ce cas la platitude de $\Leg\F$ implique que $b_i=c_i=0$ pour $i=1,2,3$; en effet
\begin{SMALL}
\begin{align*}
\left\{
\begin{array}{ll}
a_2=0\\
a_3=0
\end{array}
\right.
\Rightarrow
\left\{
\begin{array}{lll}
\smallskip
\rho_{14}^{3}=-8\Big(\lambda_0+1\Big)^2b_3=0\\
\smallskip
\rho_{13}^{4}=-2\Big(\lambda_0+1\Big)^2\Big((13\lambda_0+31)b_3+3b_2\Big)=0\\
\rho_{12}^{5}=-2\Big(\lambda_0+1\Big)^2\Big((13\lambda_0^2+41\lambda_0+60)b_3+(10\lambda_0+23)b_2+2b_1\Big)=0
\smallskip
\end{array}
\right.
\Leftrightarrow\hspace{2mm}
\left\{
\begin{array}{lll}
b_3=0\\
b_2=0\\
b_1=0
\end{array}
\right.
\end{align*}
\end{SMALL}
\hspace{-1mm}et
\begin{SMALL}
\begin{align*}
\hspace{4mm}\left\{
\begin{array}{lllll}
a_2=0\\
a_3=0\\
b_3=0\\
b_2=0\\
b_1=0
\end{array}
\right.
\Rightarrow
\left\{
\begin{array}{lll}
\smallskip
\rho_{13}^{3}=-16\Big(\lambda_0+1\Big)^2c_3=0\\
\smallskip
\rho_{12}^{4}=-4\Big(\lambda_0+1\Big)^2\Big((16\lambda_0+33)c_3+3c_2\Big)=0\\
\rho_{11}^{5}=-2\Big(\lambda_0+1\Big)^2\Big((49\lambda_0^2+143\lambda_0+150)c_3+(23\lambda_0+47)c_2+4c_1\Big)=0
\smallskip
\end{array}
\right.
\Leftrightarrow\hspace{2mm}
\left\{
\begin{array}{lll}
c_3=0\\
c_2=0\\
c_1=0.
\end{array}
\right.
\end{align*}
\end{SMALL}
\hspace{-1mm}Ainsi le feuilletage $\F$ est décrit dans la carte affine $y=1$ par la $1$-forme
\begin{small}
\begin{align*}
\theta=z^2\mathrm{d}x+x(c_0\hspace{0.1mm}x^2-a_1xz-\lambda_0\hspace{0.1mm}z^2)\mathrm{d}x+x^2[(a_1+b_0)x+(\lambda_0+1)z]\mathrm{d}z
-a_0\hspace{0.1mm}x^3(z\mathrm{d}x-x\mathrm{d}z).
\end{align*}
\end{small}
\hspace{-1mm}Remarquons que le point $[0:1:0]\in\pp$ est singulier de $\F$ de multiplicité algébrique $2$ et que $J^{2}_{(0,0)}\theta=z^2\mathrm{d}x$; on est donc dans une situation analogue à celle du cas \textbf{\textit{3.}}, qui a été exclu auparavant. Par suite l'éventualité \textbf{\textit{5.}} n'arrive pas.

\subsection{Cas $2$-jet à cône tangent formé de trois droites}

On peut supposer que le cône tangent de $J^{2}_{(0,0)}\,\omega$\, est formé des droites $x=0$, $y=0$ et $y-x=0$; il s'en suit que $\omega$ s'écrit~(\cite{CM82})
\begin{SMALL}
\begin{align*}
& xy(y-x)\left(\lambda_0\frac{\mathrm{d}x}{x}+\lambda_1\frac{\mathrm{d}y}{y}+\lambda_2\frac{\mathrm{d}(y-x)}{y-x}\right)
+\mathrm{d}x\sum_{i=0}^{3}a_i\hspace{0.1mm}x^{3-i}y^{i}
+\mathrm{d}y\sum_{i=0}^{3}b_i\hspace{0.1mm}x^{3-i}y^{i}
+(x\mathrm{d}y-y\mathrm{d}x)\sum_{i=0}^{3}c_i\hspace{0.1mm}x^{3-i}y^{i},&& \lambda_i,a_i,b_i,c_i\in\mathbb{C}.
\end{align*}
\end{SMALL}
\hspace{-1mm}On a $xA_2(x,y)+yB_2(x,y)=(\lambda_0+\lambda_1+\lambda_2)xy(y-x)$; la somme $\lambda_0+\lambda_1+\lambda_2$ est donc non nulle. En particulier les $\lambda_i$ ne peuvent pas être simultanément nuls; comme les droites du cône tangent de $J^{2}_{(0,0)}\,\omega$\, jouent un rôle symétrique, il nous suffit de traiter les trois possibilités suivantes
\begin{align*}
& (\lambda_0,\lambda_1,\lambda_2)=(1,0,0),&&
   \lambda_0=1,\hspace{0.1mm}\lambda_2=0,\hspace{0.1mm}\lambda_1(\lambda_1+1)\neq0,&&\\
&  \hspace{0.2mm}\lambda_2=1,\hspace{0.1mm}\lambda_0\lambda_1(\lambda_0+\lambda_1+1)\neq0.
\end{align*}

\textbf{\textit{1.}} Commençons par étudier l'éventualité $(\lambda_0,\lambda_1,\lambda_2)=(1,0,0).$ Dans ce cas $\omega$ s'écrit
\begin{small}
\begin{align*}
& y(y-x)\mathrm{d}x+\mathrm{d}x\sum_{i=0}^{3}a_i\hspace{0.1mm}x^{3-i}y^{i}
+\mathrm{d}y\sum_{i=0}^{3}b_i\hspace{0.1mm}x^{3-i}y^{i}
+(x\mathrm{d}y-y\mathrm{d}x)\sum_{i=0}^{3}c_i\hspace{0.1mm}x^{3-i}y^{i}.
\end{align*}
\end{small}
\hspace{-1mm}Quitte à conjuguer $\omega$ par le difféomorphisme \begin{small}$\left(\dfrac{x}{a_1x-b_1y+1}, \dfrac{y}{a_1x-b_1y+1}\right)$\end{small}, on peut supposer que $a_1=b_1=0.$ L'égalité $\deg\F=3$ implique que $\vert a_{0}\vert+\vert b_{0}\vert+\vert c_{0}\vert\neq0$. Si $a_0$ est nul, on trouve que: $\rho_{4}^{11}=b_0^3$\, et \,$\rho_{1}^{11}=4c_0^2(c_0-2b_0^2)$\, de sorte que $b_0=c_0=0$, ce qui est absurde. Donc $a_0\neq0$; l'homothétie $\left(\frac{1}{a_0}x,\frac{1}{a_0}y\right)$ nous permet de supposer que $a_0=1.$ Dans ce cas, la platitude de $\Leg\F$ entraîne que $b_0=b_2=b_3=c_0=c_1=c_2=c_3=0$; en effet

\begin{Small}
\begin{align*}
\hspace{0.7cm}(a_0,a_1,b_1)=(1,0,0)
\Rightarrow
\left\{
\begin{array}{lll}
\rho_{2}^{13}=b_0=0\\
\rho_{4}^{11}=(b_0-6)b_0^2-2b_2=0\\
\rho_{5}^{10}=-3b_0^3+(7-4b_2)b_0+5b_2-3b_3=0
\end{array}
\right.
\Leftrightarrow\hspace{2mm}
\left\{
\begin{array}{lll}
b_0=0\\
b_2=0\\
b_3=0
\end{array}
\right.
\end{align*}
\end{Small}
\hspace{-1mm}et

\begin{Small}
\begin{align*}
\left\{
\begin{array}{lll}
(a_0,a_1,b_1)=(1,0,0)
\\
(b_0,b_2,b_3)=(0,0,0)
\end{array}
\right.
\Rightarrow
\left\{
\begin{array}{llll}
\rho_{1}^{13}=4c_0=0\\
\rho_{2}^{12}=c_1-12c_0=0\\
\rho_{5}^{9}=-4(3c_0+2c_1+2c_2)=0\\
\rho_{4}^{10}=4c_0+7c_1+6c_2+3c_3=0
\end{array}
\right.
\Leftrightarrow\hspace{2mm}
\left\{
\begin{array}{llll}
c_0=0\\
c_1=0\\
c_2=0\\
c_3=0.
\end{array}
\right.
\end{align*}
\end{Small}
Ainsi $\omega=(x^3+a_2xy^2+a_3y^3-xy+y^2)\mathrm{d}x$; d'où $\deg\F=0$: contradiction.
\vspace{2mm}

\textbf{\textit{2.}} Examinons l'éventualité: $\lambda_0=1,\hspace{0.1mm}\lambda_2=0,\hspace{0.1mm}\lambda_1(\lambda_1+1)\neq0.$ Dans ce cas $\omega$ s'écrit
\begin{small}
\begin{align*}
& (y-x)(\lambda_1x\mathrm{d}y+y\mathrm{d}x)+\mathrm{d}x\sum_{i=0}^{3}a_i\hspace{0.1mm}x^{3-i}y^{i}
+\mathrm{d}y\sum_{i=0}^{3}b_i\hspace{0.1mm}x^{3-i}y^{i}
+(x\mathrm{d}y-y\mathrm{d}x)\sum_{i=0}^{3}c_i\hspace{0.1mm}x^{3-i}y^{i},\qquad\lambda_1(\lambda_1+1)\neq0.
\end{align*}
\end{small}
\hspace{-1mm}En conjuguant $\omega$ par le difféomorphisme \begin{small}
$\left(\dfrac{2\hspace{0.1mm}\lambda_1x}{b_0\hspace{0.1mm}x-\lambda_1a_3y+2\hspace{0.1mm}\lambda_1}, \dfrac{2\hspace{0.1mm}\lambda_1y}{b_0\hspace{0.2mm}x-\lambda_1a_3y+2\hspace{0.1mm}\lambda_1}\right)$
\end{small}, on se ramène à $a_3=b_0=0.$ Dans ce cas la platitude de $\Leg\F$ implique que $a_0=a_1=b_2=b_3=0$ et que $b_1=\lambda_1a_2$; en effet

\begin{Small}
\begin{align*}
\hspace{0.7cm}(a_3,b_0)=(0,0)
\Rightarrow
\left\{
\begin{array}{llll}
\vspace{0.5mm}
\rho_{5}^{12}=-4\lambda_1^3(\lambda_1+1)^2a_0=0\\
\vspace{0.5mm}
\rho_{16}^{1}=4\lambda_1^6(\lambda_1+1)^2b_3=0\\
\vspace{0.5mm}
\rho_{6}^{11}=-\lambda_1^3(\lambda_1+1)^2(15\lambda_1a_0+2a_1-13a_0)=0\\
\rho_{15}^{2}=\lambda_1^5(\lambda_1+1)^2(2\lambda_1b_2-13\lambda_1b_3+15b_3)=0
\end{array}
\right.
\Leftrightarrow\hspace{2mm}
\left\{
\begin{array}{llll}
\vspace{0.5mm}
a_0=0\\
\vspace{0.5mm}
b_3=0\\
\vspace{0.5mm}
a_1=0\\
b_2=0
\end{array}
\right.
\end{align*}
\end{Small}
\hspace{-1mm}et
\begin{Small}
\begin{align*}
\hspace{1.2cm}\left\{
\begin{array}{ll}
(a_0,a_1,a_3)=(0,0,0)\\
(b_0,b_2,b_3)=(0,0,0)
\end{array}
\right.
\Rightarrow
\left\{
\begin{array}{ll}
\rho_{8}^{9}=\lambda_1^2(\lambda_1+1)^3(3\lambda_1+2)(\lambda_1a_2-b_1)=0\\
\rho_{13}^{4}=\lambda_1^4(\lambda_1+1)^3(2\lambda_1+3)(\lambda_1a_2-b_1)=0
\end{array}
\right.
\Leftrightarrow\hspace{2mm}
b_1=\lambda_1a_2.
\end{align*}
\end{Small}
\hspace{-1mm}Ainsi $\F$ est décrit par
\begin{small}
\begin{align*}
&\omega=(y-x+a_2xy)(\lambda_1x\mathrm{d}y+y\mathrm{d}x)+(c_0\,x^3+c_1x^2y+c_2xy^2+c_3y^3)(x\mathrm{d}y-y\mathrm{d}x),\qquad\lambda_1(\lambda_1+1)\neq0.
\end{align*}
\end{small}
\hspace{-1mm}On envisage maintenant les trois éventualités suivantes
\begin{align*}
& (a_2,\lambda_1)=(0,1),&&\hspace{1cm} a_2=0,\,\lambda_1(\lambda_1+1)(\lambda_1-1)\neq0,&&\hspace{1cm} \lambda_1(\lambda_1+1)a_2\neq0.
\end{align*}

\textbf{\textit{2.1.}} Lorsque $(a_2,\lambda_1)=(0,1),$ le calcul explicite de $K(\Leg\F)$ conduit à
\begin{align*}
&\rho_{6}^{10}=-8(c_0+c_1) &&\text{et}&& \rho_{14}^{2}=8(c_2+c_3),
\end{align*}
donc $c_1=-c_0,\,c_3=-c_2$ et $\omega$ s'écrit
\begin{align*}
(y-x)\left[x\mathrm{d}y+y\mathrm{d}x-(c_0\,x^2+c_2y^2)(x\mathrm{d}y-y\mathrm{d}x)\right],
\end{align*}
ce qui contredit l'égalité $\deg\F=3.$
\vspace{2mm}

\textbf{\textit{2.2.}} Quand $a_2=0$ et $\lambda_1(\lambda_1+1)(\lambda_1-1)\neq0,$ la nullité de la courbure de $\Leg\F$ implique que $c_i=0,i=0,1,2,3$; en effet
\begin{Small}
\begin{align*}
a_2=0
\Rightarrow
\left\{
\begin{array}{ll}
\vspace{0.5mm}
\rho_{5}^{11}=-4\lambda_1^2(\lambda_1-1)(\lambda_1+1)^2c_0=0\\
\vspace{0.5mm}
\rho_{15}^{1}=-4\lambda_1^6(\lambda_1-1)(\lambda_1+1)^2c_3=0
\end{array}
\right.
\Leftrightarrow\hspace{2mm}
\left\{
\begin{array}{ll}
\vspace{0.5mm}
c_0=0\\
\vspace{0.5mm}
c_3=0
\end{array}
\right.
\end{align*}
\end{Small}
\hspace{-1mm}et
\begin{Small}
\begin{align*}
\hspace{0.7cm}
(a_2,c_0,c_3)=(0,0,0)
\Rightarrow
\left\{
\begin{array}{ll}
\vspace{0.5mm}
\rho_{5}^{9}=-2\lambda_1^2(\lambda_1+1)^2c_1^2=0\\
\vspace{0.5mm}
\rho_{13}^{1}=2\lambda_1^6(\lambda_1+1)^2c_2^2=0
\end{array}
\right.
\Leftrightarrow\hspace{2mm}
\left\{
\begin{array}{ll}
\vspace{0.5mm}
c_1=0\\
\vspace{0.5mm}
c_2=0.
\end{array}
\right.
\end{align*}
\end{Small}
\hspace{-1mm}Par suite $\omega$ s'écrit $(y-x)(\lambda_1x\mathrm{d}y+y\mathrm{d}x),$ ce qui contredit l'égalité $\deg\F=3.$
\vspace{2mm}

\textbf{\textit{2.3.}} Si $\lambda_1(\lambda_1+1)a_2\neq0,$ alors, en faisant agir l'homothétie $\left(\frac{1}{a_2}x,\frac{1}{a_2}y\right)$ sur $\omega,$ on se ramène à $a_2=1.$ Dans ce cas l'hypothèse sur $\Leg\F$ d'être plat entraîne que $c_i=0,$ $i=0,1,2,3$; en effet

\begin{Small}
\begin{align*}
a_2=1
\Rightarrow
\left\{
\begin{array}{llll}
\vspace{0.5mm}
\rho_{0}^{9}=-8(\lambda_1+1)(3\lambda_1c_1+2c_0)c_0^3=0\\
\vspace{0.5mm}
\rho_{0}^{10}=-4(\lambda_1+1)(\lambda_1c_1+c_0)c_0^3=0\\
\vspace{0.5mm}
\rho_{14}^{1}=-4\lambda_1^6(\lambda_1+1)(5\lambda_1^2-7)c_3=0\\
\rho_{15}^{1}=-4\lambda_1^6(\lambda_1-1)(\lambda_1+1)^2c_3=0
\end{array}
\right.
\Leftrightarrow\hspace{2mm}
\left\{
\begin{array}{ll}
\vspace{0.5mm}
c_0=0
\\
\\
\\
\vspace{0.5mm}
c_3=0
\end{array}
\right.
\end{align*}
\end{Small}
\hspace{-1mm}et
\begin{Small}
\begin{align*}
\hspace{0.7cm}
(a_2,c_0,c_3)=(1,0,0)
\Rightarrow
\left\{
\begin{array}{ll}
\vspace{0.5mm}
\rho_{5}^{9}=-2\lambda_1^2(\lambda_1+1)^2c_1^2=0\\
\vspace{0.5mm}
\rho_{13}^{1}=2(\lambda_1+1)^2\lambda_1^6c_2^2=0
\end{array}
\right.
\Leftrightarrow\hspace{2mm}
\left\{
\begin{array}{ll}
\vspace{0.5mm}
c_1=0\\
\vspace{0.5mm}
c_2=0.
\end{array}
\right.
\end{align*}
\end{Small}
\hspace{-1mm}Il s'en suit que $\omega$ s'écrit $(y-x+xy)(\lambda_1x\mathrm{d}y+y\mathrm{d}x)$, mais ceci contredit l'égalité $\deg\F=3.$
\vspace{2mm}

\textbf{\textit{3.}} Pour finir considérons l'éventualité: $\lambda_2=1$\, et \,$\lambda_0\lambda_1(\lambda_0+\lambda_1+1)\neq0.$ Dans ce cas $\omega$ s'écrit
\begin{SMALL}
\begin{align*}
& xy\mathrm{d}(y-x)+(y-x)(\lambda_1x\mathrm{d}y+\lambda_0y\mathrm{d}x)+\mathrm{d}x\sum_{i=0}^{3}a_i\hspace{0.1mm}x^{3-i}y^{i}
+\mathrm{d}y\sum_{i=0}^{3}b_i\hspace{0.1mm}x^{3-i}y^{i}
+(x\mathrm{d}y-y\mathrm{d}x)\sum_{i=0}^{3}c_i\hspace{0.1mm}x^{3-i}y^{i},\qquad\lambda_0\lambda_1(\lambda_0+\lambda_1+1)\neq0.
\end{align*}
\end{SMALL}
\hspace{-1mm}La conjugaison par le difféomorphisme
\begin{Small}
\begin{align*}
\left(\dfrac{x}{\left(\dfrac{b_0}{2\hspace{0.1mm}\lambda_1}\right)x-\left(\dfrac{a_3}{2\hspace{0.1mm}\lambda_0}\right)y+1},
\dfrac{y}{\left(\dfrac{b_0}{2\hspace{0.1mm}\lambda_1}\right)x-\left(\dfrac{a_3}{2\hspace{0.1mm}\lambda_0}\right)y+1}\right)
\end{align*}
\end{Small}
\hspace{-1mm}permet de supposer que $a_3=b_0=0.$ Le calcul explicite de $K(\Leg\F)$ donne
\begin{align*}
& \rho_{5}^{12}=-4\lambda_1^3(\lambda_0+1)^4(\lambda_0+\lambda_1+1)^2a_0
&&\text{et}&&
\rho_{16}^{1}=4\lambda_0\lambda_1^2(\lambda_1+1)^4(\lambda_0+\lambda_1+1)^2b_3\hspace{1mm};
\end{align*}
comme $\lambda_0\lambda_1(\lambda_0+\lambda_1+1)\neq0$ et comme les coordonnées $x,y$ jouent un rôle symétrique, il suffit de traiter seulement les deux cas suivants
\begin{itemize}
  \item [$\bullet$] $\lambda_0=-1$ et $\lambda_1\neq0$;
  \item [$\bullet$] $\lambda_0\lambda_1(\lambda_0+1)(\lambda_1+1)(\lambda_0+\lambda_1+1)\neq0.$
\end{itemize}
\vspace{2mm}

\textbf{\textit{3.1.}} Quand $\lambda_0=-1$ et $\lambda_1\neq0,$ la platitude de $\Leg\F$ implique que $a_i=c_i=0$ pour $i=0,1,2$; en effet

\begin{SMALL}
\begin{align*}
\hspace{-2mm}\left\{
\begin{array}{lll}
a_3=0\\
b_0=0\\
\lambda_0=-1
\end{array}
\right.
\Rightarrow
\left\{
\begin{array}{lll}
\rho_{7}^{10}=-48\lambda_1^7a_0=0\\
\rho_{8}^{9}=2\lambda_1^6(11\lambda_1^2+106\lambda_1+15)a_0-16\lambda_1^7a_1=0\\
\rho_{9}^{8}=-\lambda_1^5(367\lambda_1^2+2\lambda_1^4+85\lambda_1^3+167\lambda_1+3)a_0+8\lambda_1^6(10\lambda_1+\lambda_1^2-2)a_1+8\lambda_1^7a_2=0
\end{array}
\right.
\Leftrightarrow\hspace{2mm}
\left\{
\begin{array}{lll}
a_0=0\\
a_1=0\\
a_2=0
\end{array}
\right.
\end{align*}
\end{SMALL}
\hspace{-1mm}et
\begin{SMALL}
\begin{align*}
\left\{
\begin{array}{llllll}
a_3=0\\
b_0=0\\
\lambda_0=-1\\
a_0=0\\
a_1=0\\
a_2=0
\end{array}
\right.
\Rightarrow
\left\{
\begin{array}{lll}
\rho_{7}^{9}=-48\lambda_1^7c_0=0\\
\rho_{8}^{8}=2\lambda_1^6(11\lambda_1^2-5+90\lambda_1)c_0-24\lambda_1^7c_1=0\\
\rho_{9}^{7}=-2\lambda_1^5(\lambda_1+1)(\lambda_1^3+34\lambda_1^2+61\lambda_1-12)c_0+4\lambda_1^6(3\lambda_1^2+32\lambda_1-23)c_1+32\lambda_1^7c_2=0
\end{array}
\right.
\Leftrightarrow\hspace{2mm}
\left\{
\begin{array}{lll}
c_0=0\\
c_1=0\\
c_2=0.
\end{array}
\right.
\end{align*}
\end{SMALL}
\hspace{-1mm}Ainsi le feuilletage $\F$ est décrit dans la carte affine $x=1$ par la $1$-forme
\begin{small}
\begin{align*}
& \theta=-(b_1y-\lambda_1z)\left(y\mathrm{d}z-z\mathrm{d}y)+y(c_3y^2+b_2yz+\left(\lambda_1+1\right)z^2\right)\mathrm{d}y-y^2(b_2y+\lambda_1z)\mathrm{d}z
         -b_3y^3(y\mathrm{d}z-z\mathrm{d}y).
\end{align*}
\end{small}
\hspace{-1mm}Remarquons que le point $[1:0:0]\in\pp$ est singulier de $\F$ de multiplicité algébrique $2$ et que le cône tangent de $J^{2}_{(0,0)}\theta$ est identiquement nul; d'après le paragraphe~\S\ref{subsec:$2$-jet-cône-tangent-nul}, $\F$ ne peut admettre un tel point singulier. Donc le cas \textbf{\textit{3.1.}} n'arrive pas.
\vspace{2mm}

\textbf{\textit{3.2.}} Lorsque $\lambda_0\lambda_1(\lambda_0+1)(\lambda_1+1)(\lambda_0+\lambda_1+1)\neq0,$ l'hypothèse $K(\Leg\F)\equiv0$ entraîne que $a_0=a_1=b_2=b_3=0$; en effet
\begin{Small}
\begin{align*}
\left\{
\begin{array}{ll}
a_3=0\\
b_0=0
\end{array}
\right.
\Rightarrow
\left\{
\begin{array}{ll}
\rho_{5}^{12}=-4\lambda_1^3(\lambda_0+1)^4(\lambda_0+\lambda_1+1)^2a_0=0\\
\rho_{16}^{1}=4\lambda_0\lambda_1^2(\lambda_1+1)^4(\lambda_0+\lambda_1+1)^2b_3=0
\end{array}
\right.
\Leftrightarrow\hspace{2mm}
\left\{
\begin{array}{ll}
a_0=0\\
b_3=0
\end{array}
\right.
\end{align*}
\end{Small}
\hspace{-1mm}et
\begin{Small}
\begin{align*}
\hspace{1.4cm}\left\{
\begin{array}{ll}
(a_0,a_3)=(0,0)\\
(b_0,b_3)=(0,0)
\end{array}
\right.
\Rightarrow
\left\{
\begin{array}{ll}
\rho_{6}^{11}=-2\lambda_1^3(\lambda_0+1)^4(\lambda_0+\lambda_1+1)^2a_1=0\\
\rho_{15}^{2}=2\lambda_0\lambda_1^2(\lambda_1+1)^4(\lambda_0+\lambda_1+1)^2b_2=0
\end{array}
\right.
\Leftrightarrow\hspace{2mm}
\left\{
\begin{array}{ll}
a_1=0\\
b_2=0.
\end{array}
\right.
\end{align*}
\end{Small}
\hspace{-1mm}Ainsi le feuilletage $\F$ est donné par
\begin{align*}
& \omega=xy\mathrm{d}(y-x)+(y-x)(\lambda_1x\mathrm{d}y+\lambda_0y\mathrm{d}x)+xy(a_2y\mathrm{d}x+b_1x\mathrm{d}y)
+(x\mathrm{d}y-y\mathrm{d}x)\sum_{i=0}^{3}c_i\hspace{0.1mm}x^{3-i}y^{i},
\end{align*}
avec $\lambda_0\lambda_1(\lambda_0+1)(\lambda_1+1)(\lambda_0+\lambda_1+1)\neq0.$ Dans ce cas le calcul explicite de $K(\Leg\F)$ montre que
\begin{align*}
& \rho_{8}^{9}=\lambda_1^2(\lambda_0+1)(\lambda_0+\lambda_1+1)^2\sigma_1,&& \rho_{9}^{8}=-\lambda_1^2(\lambda_0+\lambda_1+1)^2\sigma_2,&&\\
& \rho_{12}^{5}=\lambda_1^2(\lambda_0+\lambda_1+1)^2\sigma_3,&& \rho_{13}^{4}=\lambda_1^2(\lambda_1+1)(\lambda_0+\lambda_1+1)^2\sigma_4,
\end{align*}
avec

\begin{SMALL}
\begin{align*}
\begin{array}{llll}
\sigma_1=\Big(-2+2\lambda_1\lambda_0^3+3\lambda_1^2\lambda_0+5\lambda_1\lambda_0^2-6\lambda_0+\lambda_1+5\lambda_1^2\lambda_0^2
+3\lambda_1^3\lambda_0+4\lambda_0\lambda_1-\lambda_1^3-2\lambda_1^2-2\lambda_0^3-6\lambda_0^2\Big)a_2\\
\vspace{0.3mm}
\hspace{0.8cm}-\Big(\lambda_0+1\Big)\Big(2\lambda_0^3+6\lambda_0^2+5\lambda_1\lambda_0^2+2\lambda_0\lambda_1+6\lambda_0
+3\lambda_1^2\lambda_0+2-\lambda_1^2-3\lambda_1\Big)b_1,
\vspace{2.5mm}
\\
\sigma_2=\Big(5\lambda_1\lambda_0^4-10+38\lambda_1\lambda_0^3+\lambda_0^4+6\lambda_1^3\lambda_0^2-2\lambda_1^4\lambda_0+7\lambda_1^2\lambda_0
+55\lambda_1\lambda_0^2-29\lambda_0+13\lambda_1^2\lambda_0^3-6\lambda_1+28\lambda_1^2\lambda_0^2+4\lambda_1^3\lambda_0+16\lambda_0\lambda_1\\
\vspace{0.3mm}
\hspace{0.8cm}-2\lambda_1^4-10\lambda_1^3-8\lambda_1^2-7\lambda_0^3-27\lambda_0^2\Big)a_2-\Big(\lambda_0+1\Big)\Big(5\lambda_0^4+19\lambda_0^3
+13\lambda_1\lambda_0^3+6\lambda_1^2\lambda_0^2+8\lambda_1\lambda_0^2+33\lambda_0^2+29\lambda_0-2\lambda_1^3\lambda_0\\
\vspace{0.3mm}
\hspace{0.8cm}-9\lambda_0\lambda_1+2\lambda_1^2\lambda_0+10-12\lambda_1^2-4\lambda_1-2\lambda_1^3\Big)b_1,
\vspace{2.5mm}
\\
\sigma_3=\Big(\lambda_1+1\Big)\Big(-5\lambda_1^4-19\lambda_1^3-13\lambda_1^3\lambda_0-8\lambda_1^2\lambda_0-33\lambda_1^2-6\lambda_1^2\lambda_0^2-29\lambda_1+9\lambda_0\lambda_1-2\lambda_1\lambda_0^2
+2\lambda_1\lambda_0^3-10+2\lambda_0^3+12\lambda_0^2\\
\vspace{0.3mm}
\hspace{0.8cm}+4\lambda_0\Big)a_2+\Big(-10+5\lambda_1^4\lambda_0+13\lambda_1^3\lambda_0^2+\lambda_1^4-10\lambda_0^3-7\lambda_1^3-27\lambda_1^2-29\lambda_1+6\lambda_1^2\lambda_0^3
+28\lambda_1^2\lambda_0^2-2\lambda_0^4-6\lambda_0+4\lambda_1\lambda_0^3\\
\vspace{0.3mm}
\hspace{0.8cm}+16\lambda_0\lambda_1+7\lambda_1\lambda_0^2+38\lambda_1^3\lambda_0+55\lambda_1^2\lambda_0-8\lambda_0^2-2\lambda_1\lambda_0^4\Big)b_1,
\vspace{2.5mm}
\\
\sigma_4=\Big(\lambda_1+1\Big)\Big(2\lambda_1^3+6\lambda_1^2+5\lambda_1^2\lambda_0+2\lambda_0\lambda_1+3\lambda_1\lambda_0^2+6\lambda_1+2-3\lambda_0-\lambda_0^2\Big)a_2+\Big(2\lambda_1^3-2\lambda_1^3\lambda_0-5\lambda_1^2\lambda_0^2-5\lambda_1^2\lambda_0+6\lambda_1^2\\
\vspace{0.3mm}
\hspace{0.8cm}-3\lambda_1\lambda_0^2+6\lambda_1-3\lambda_1\lambda_0^3-4\lambda_0\lambda_1+2-\lambda_0+2\lambda_0^2+\lambda_0^3\Big)b_1.
\end{array}
\end{align*}
\end{SMALL}
\hspace{-1mm}Grâce à la linéarité des $\sigma_i$ en $(a_2,b_1),$ on vérifie que le système $\sigma_i=0,i=1,\ldots,4,$ sous la condition $\lambda_0\lambda_1(\lambda_0+1)(\lambda_1+1)(\lambda_0+\lambda_1+1)\neq0,$  est équivalent à $(a_2,b_1)=(0,0).$ Ainsi
\begin{Small}
\begin{align*}
& \omega=xy\mathrm{d}(y-x)+(y-x)(\lambda_1x\mathrm{d}y+\lambda_0y\mathrm{d}x)+(x\mathrm{d}y-y\mathrm{d}x)\sum_{i=0}^{3}c_i\hspace{0.1mm}x^{3-i}y^{i},\qquad \lambda_0\lambda_1(\lambda_0+1)(\lambda_1+1)(\lambda_0+\lambda_1+1)\neq0.
\end{align*}
\end{Small}
\hspace{-1mm}Dans ce cas le calcul explicite de $K(\Leg\F)$ donne
\begin{Small}
\begin{align*}
& \rho_{5}^{11}=4\lambda_1^2(\lambda_0+1)^4(\lambda_0-\lambda_1+1)(\lambda_0+\lambda_1+1)^2c_0
&&\hspace{2mm} \text{et} &&\hspace{2mm}
\rho_{15}^{1}=4\lambda_1^2(\lambda_1+1)^4(\lambda_0-\lambda_1-1)(\lambda_0+\lambda_1+1)^2c_3\hspace{1mm};
\end{align*}
\end{Small}
\hspace{-1mm}puisque $\lambda_0\lambda_1(\lambda_0+1)(\lambda_1+1)(\lambda_0+\lambda_1+1)\neq0$ et puisque les coordonnées $x,y$ jouent un rôle symétrique, il suffit d'examiner seulement les deux cas suivants
\begin{itemize}
  \item [$\bullet$] $\lambda_0=\lambda_1-1$\, et \,$\lambda_1(\lambda_1-1)(\lambda_1+1)\neq0$;
  \item [$\bullet$] $\lambda_0\lambda_1(\lambda_0+1)(\lambda_1+1)(\lambda_0+\lambda_1+1)(\lambda_0-\lambda_1+1)(\lambda_0-\lambda_1-1)\neq0.$
\end{itemize}
\vspace{3mm}

\textbf{\textit{3.2.1.}} Lorsque $\lambda_0=\lambda_1-1$\, et \,$\lambda_1(\lambda_1-1)(\lambda_1+1)\neq0,$ la platitude de $\Leg\F$ implique que $c_i=0,i=0,1,2,3$; en effet
\begin{SMALL}
\begin{align*}
\lambda_0=\lambda_1-1
\Rightarrow
\left\{
\begin{array}{ll}
\rho_{15}^{1}=-32\lambda_1^4(\lambda_1+1)^4c_3=0\\
\rho_{13}^{1}=8\lambda_1^3\Big(\lambda_1+1\Big)^2\Big(\lambda_1(\lambda_1+1)^2c_2^2-(42\lambda_1^2-12\lambda_1+6)c_2c_3-\lambda_1(\lambda_1^2+62\lambda_1-38)c_3^2\\
\vspace{0.3mm}
\hspace{0.8cm}-4(\lambda_1+1)(5\lambda_1c_1-2\lambda_1c_0+c_1-2c_0)c_3\Big)=0
\end{array}
\right.
\Leftrightarrow\hspace{2mm}
\left\{
\begin{array}{ll}
c_3=0\\
\vspace{0.3mm}
\\
c_2=0
\end{array}
\right.
\end{align*}
\end{SMALL}
\hspace{-1mm}et
\begin{Small}
\begin{align*}
\left\{
\begin{array}{lll}
\lambda_0=\lambda_1-1\\
c_2=0\\
c_3=0
\end{array}
\right.
\Rightarrow
\left\{
\begin{array}{lll}
\rho_{6}^{10}=-8\lambda_1^8\Big((\lambda_1+1)c_0+\lambda_1c_1\Big)=0\\
\rho_{8}^{8}=8\lambda_1^6\Big((35\lambda_1^3-128\lambda_1^2+18\lambda_1-35)c_0+\lambda_1(31\lambda_1^2-105\lambda_1+19)c_1\Big)=0\\
\rho_{9}^{7}=-16\lambda_1^5\Big((6\lambda_1^4-46\lambda_1^3-113\lambda_1^2+37\lambda_1-10)c_0+\lambda_1(6\lambda_1^3-24\lambda_1^2\\
\vspace{0.3mm}
\hspace{0.8cm}-103\lambda_1+58)c_1\Big)=0
\end{array}
\right.
\Leftrightarrow\hspace{1mm}
\left\{
\begin{array}{ll}
c_0=0
\\
\vspace{0.3mm}
\\
c_1=0.
\end{array}
\right.
\end{align*}
\end{Small}
\hspace{-1mm}Ainsi $\omega$ s'écrit\, $xy\mathrm{d}(y-x)+(y-x)(\lambda_1x\mathrm{d}y+\lambda_1y\mathrm{d}x-y\mathrm{d}x)$, ce qui contredit l'égalité $\deg\F=3.$
\vspace{2mm}

\textbf{\textit{3.2.2.}} Quand $\lambda_0\lambda_1(\lambda_0+1)(\lambda_1+1)(\lambda_0+\lambda_1+1)(\lambda_0-\lambda_1+1)(\lambda_0-\lambda_1-1)\neq0,$ l'hypothèse $K(\Leg\F)\equiv0$ entraîne que $c_i=0,i=0,1,2,3$; en effet
\begin{Small}
\begin{align*}
&\quad \rho_{5}^{11}=4\lambda_1^2(\lambda_0+1)^4(\lambda_0-\lambda_1+1)(\lambda_0+\lambda_1+1)^2c_0,&&\quad
\rho_{15}^{1}=4\lambda_1^2(\lambda_1+1)^4(\lambda_0-\lambda_1-1)(\lambda_0+\lambda_1+1)^2c_3
\end{align*}
\end{Small}
\hspace{-1mm}et
\begin{Small}
\begin{align*}
\hspace{-3cm}\left\{
\begin{array}{ll}
c_0=0\\
c_3=0
\end{array}
\right.
\Rightarrow
\left\{
\begin{array}{ll}
\rho_{5}^{9}=-2\lambda_1^2(\lambda_0+1)^4(\lambda_0+\lambda_1+1)^2c_1^2=0\\
\rho_{13}^{1}=2\lambda_1^2(\lambda_1+1)^4(\lambda_0+\lambda_1+1)^2c_2^2=0
\end{array}
\right.
\Leftrightarrow\hspace{2mm}
\left\{
\begin{array}{ll}
c_1=0\\
c_2=0.
\end{array}
\right.
\end{align*}
\end{Small}
\smallskip
\hspace{-1mm}Donc $\omega$ s'écrit\, $xy\mathrm{d}(y-x)+(y-x)(\lambda_1x\mathrm{d}y+\lambda_0y\mathrm{d}x)$, mais ceci contredit l'égalité $\deg\F=3.$

\ActivateToc

\backmatter

\end{document}